\documentclass{amsart}

\usepackage[top=1in, bottom=1in, left=1in, right=1in]{geometry}
\usepackage{graphicx,hyperref,multicol,amsthm,amssymb,tikz,mathtools,xcolor,tikz,tikz-3dplot,ytableau}

\newcommand{\Wsp}[0]{\mathcal{W}^{\mathfrak{sp}_2}_{\infty}}
\newcommand{\vir}[0]{\text{Vir}^c}
\newcommand{\aff}{V^k(\mathfrak{sl}_2)}
\newcommand{\nlcalg}{\cL^{\fr{sp}_2}_{\infty}}

\newcommand{\Winf}{\mathcal{W}_{\infty}}
\newcommand{\Wev}{\mathcal{W}^{\text{ev}}_{\infty}}

\numberwithin{equation}{section}

\newtheorem{definition}{Definition}[section]
\newtheorem{lemma}{Lemma}[section]
\newtheorem{remark}{Remark}[section]
\newtheorem{proposition}{Proposition}[section]
\newtheorem{theorem}{Theorem}[section]

\newtheorem{corollary}{Corollary}[section]

\newtheorem{conjecture}{Conjecture}[section]

\newcommand{\B}[1]{\textbf{#1}}

\newcommand{\T}[1]{\text{#1}}

\newcommand{\fr}[1]{\mathfrak{#1}}

\def\cO{{\cal O}}

\def\cG{{\cal G}}

\def\cA{{\cal A}}

\def\ra{\rightarrow}

\def\cA{{\mathcal A}}

\def\cC{{\mathcal C}}
\def\cD{{\mathcal D}}
\def\cE{{\mathcal E}}
\def\cF{{\mathcal F}}
\def\cG{{\mathcal G}}
\def\cH{{\mathcal H}}
\def\cI{{\mathcal I}}

\def\cL{{\mathcal L}}

\def\cN{{\mathcal N}}
\def\cO{{\mathcal O}}

\def\cS{{\mathcal S}}

\def\cV{{\mathcal V}}
\def\cW{{\mathcal W}}

\def\ga{{\mathfrak a}}
\def\gb{{\mathfrak b}}

\def\gf{{\mathfrak f}}
\def\gg{{\mathfrak g}}
\def\gh{{\mathfrak h}}

\def\gl{{\mathfrak l}}

\def\go{{\mathfrak o}}
\def\gp{{\mathfrak p}}

\def\gs{{\mathfrak s}}

\title{Building blocks for $\cW$-algebras of classical types}
\author{Thomas Creutzig}
\address{Department Mathematik, FAU Erlangen}
\email{creutzigt@math.fau.de}
\thanks{T. Creutzig is supported by DFG project Projektnummer 551865932.}

\author{Volodymyr Kovalchuk}
\address{Department Mathematik, FAU Erlangen}
\email{v.dexter97@gmail.com}
\thanks{V. Kovalchuk is supported by the Alexander von Humboldt Foundation.}

\author{Andrew R. Linshaw}
\address{Department of Mathematics, University of Denver}
\email{andrew.linshaw@du.edu}
\thanks{A. Linshaw is supported by NSF Grant DMS-2401382 and Simons Foundation Grant MPS-TSM-00007694.}

\thanks{We thank S. Nakatsuka for pointing out a mistake in the classification of indecomposable nilpotents in types $B$ and $D$ in an earlier version.}
\begin{document}
	\maketitle

	\pagestyle{plain}

\noindent {ABSTRACT. The universal $2$-parameter vertex algebra $\cW_{\infty}$ of type $\cW(2,3,4,\dots)$ serves as a classifying object for vertex algebras of type $\cW(2,3,\dots,N)$ for some $N$ in the sense that under mild hypothesis, all such vertex algebras arise as quotients of $\cW_{\infty}$. There is an $\mathbb{N} \times \mathbb{N}$ family of such $1$-parameter vertex algebras which, after tensoring with a Heisenberg algebra, are known as $Y$-algebras. They were introduced by Gaiotto and Rap\v{c}\'ak and are expected to be the building blocks for all $\cW$-algebras in type $A$, i.e., every $\cW$-(super)algebra of type $A$ is an extension of a tensor product of finitely many $Y$-algebras. Similarly, the orthosymplectic $Y$-algebras are $1$-parameter quotients of a universal $2$-parameter vertex algebra $\cW^{\text{ev}}_{\infty}$ of type $\cW(2,4,6,\dots)$, which is a classifying object for vertex algebras of type $\cW(2,4,\dots, 2N)$ for some $N$. Unlike type $A$, these algebras are not all the building blocks for $\cW$-algebras of types $B$, $C$, and $D$. In this paper, we construct a new universal $2$-parameter vertex algebra of type $\cW(1^3, 2, 3^3, 4, 5^3,6,\dots)$ which we denote by $\cW^{\gs\gp_2}_{\infty}$ since it contains a copy of the affine vertex algebra $V^k(\gs\gp_2)$. We identify $8$ infinite families of $1$-parameter quotients of $\cW^{\gs\gp_2}_{\infty}$ which are analogues of the $Y$-algebras. We regard $\cW^{\gs\gp_2}_{\infty}$ as a fundamental object on equal footing with $\cW_{\infty}$ and $\cW^{\text{ev}}_{\infty}$, and we give some heuristic reasons for why we expect the $1$-parameter quotients of these three objects to be the building blocks for all $\cW$-algebras of classical types. Finally, we prove that $\cW^{\gs\gp_2}_{\infty}$ has many quotients which are strongly rational. This yields new examples of strongly rational $\cW$-superalgebras.}

\section{Introduction}
$\cW$-algebras are an important class of vertex algebras that have been studied in both the mathematics and physics literature for nearly 40 years. For any Lie (super)algebra $\mathfrak{g}$ and nilpotent element $f$ in the even part of $\mathfrak{g}$, the $\cW$-algebra $\cW^k(\mathfrak{g}, f)$ at level $k\in \mathbb{C}$, is defined via the generalized Drinfeld-Sokolov reduction \cite{KRW}. They are a common generalization of affine vertex algebras and the Virasoro algebra, as well as the $\cN=1$, $2$, and $4$ superconformal algebras. When $f$ is principal nilpotent, $\cW^k(\mathfrak{g},f)$ is called a principal $\cW$-algebra and is denoted by $\cW^k(\mathfrak{g})$; they appear in many settings including integrable systems \cite{B89,BM,GD,DS}, conformal field theory to higher spin gravity duality \cite{GG}, Alday-Gaiotto-Tachikawa correspondence \cite{AGT,SV,BFN}, and the quantum geometric Langlands program \cite{Fre07,Gai16,AFO,CG,FG,AF}. In general, $\cW^k(\gg,f)$ can be regarded as a chiralization of the finite $\cW$-algebra $\cW^{\text{fin}}(\gg,f)$ defined by Premet \cite{Pre}, which is a quantization of the coordinate ring of the Slodowy slice $S_f\subseteq \gg \cong \gg^*$.

Principal $\cW$-algebras satisfy {\it Feigin-Frenkel duality}, which is a vertex algebra isomorphism $\cW^k(\gg) \cong \cW^{k'}(^L \gg)$. Here $^L\gg$ is the Langlands dual Lie algebra, $h^{\vee}$,  $^L h^{\vee}$ are the dual Coxeter numbers of $\gg$, $^L\gg$, and $(k+h^{\vee})(k' + ^Lh^{\vee}) = r$ where $r$ is the lacity \cite{FF}. For $\gg$ simply-laced, there is another duality called the {\it coset realization} which was proven \cite{ACL}; see \cite{CN} for a short and recent proof. For generic values of $\ell$, we have a vertex algebra isomorphism 
	\begin{equation} \label{eq:cosetrealization} \cW^{\ell}(\gg)\cong  \text{Com}(V^{k+1}(\gg),V^k(\gg)\otimes L_1(\gg)),\qquad \ell +h^{\vee}=\frac{k+h^{\vee}}{(k+1)+ h^{\vee}},\end{equation} which descends to an isomorphism of simple vertex algebras $\cW_{\ell}(\gg)\cong \text{Com}(L_{k+1}(\gg), L_k(\gg)\otimes L_1(\gg))$ for all admissible levels $k$. This was a longstanding conjecture \cite{BBSS,FaLu,FKW92,KW89}, vastly generalizing the Goddard-Kent-Olive construction of the Virasoro algebra \cite{GKO}. Accordingly, we will call the cosets on the right hand side of \eqref{eq:cosetrealization} {\it GKO cosets}. Note that when $k$ is an admissible level for $\widehat{\gg}$, $\ell$ is a nondegenerate admissible level for $\widehat{\gg}$, so that $\cW_{\ell}(\gg)$ is strongly rational (i.e., lisse and rational) \cite{Ar1, Ar2}. There is a different coset realization for type $B$ (and type $C$, by Feigin-Frenkel duality) appearing in \cite{CL4}: we have
\begin{equation} \label{cosetBC} \cW^{\ell}(\gs\go_{2n+1})  \cong \text{Com}(V^k(\gs\gp_{2n}), V^k(\go\gs\gp_{1|2n})), \qquad \ell + h^{\vee}_{\gs\go_{2n+1}} =  \frac{k + h^{\vee}_{\go\gs\gp_{1|2n}}}{k + h^{\vee}_{\gs\gp_{2n}}}. 
\end{equation} However, no such coset realization is known for types $F$ and $G$, even conjecturally.

	Non-principal $\cW$-algebras are not as well understood, but they have become increasingly important in physics in recent years; see for example \cite{ArMT,CG,CH,GR,GRZ,ProI,ProII,PR}. In \cite{GR}, Gaiotto and Rap\v{c}\'ak introduced a family of vertex algebras $Y_{N_1, N_2, N_3}[\psi]$ called $Y$-algebras, which are indexed by three integers $N_1, N_2, N_3\geq 0$ and a complex parameter $\psi$. They are associated to interfaces of twisted $\cN=4$ supersymmetric gauge theories with gauge groups $U(N_1)$,  $U(N_2)$, and $U(N_3)$. The interfaces satisfy a permutation symmetry which is expected to induce a corresponding symmetry on the vertex algebras. This led Gaiotto and Rap\v{c}\'ak to conjecture a triality of isomorphisms of $Y$-algebras.

	The $Y$-algebras with one label zero are (up to a Heisenberg algebra) the affine cosets of a family of non-principal $\cW$-(super)algebras of type $A$ which are known as {\it hook-type}. For $n\geq 1$ and $m\geq 0$, we define $\cW^{\psi}(n,m) = \cW^k(\gs\gl_{n+m}, f_{n,1^m})$ where $f_{n,1^m}$ corresponds to the hook-type partition $(n,1,1,\dots, 1)$ of $n+m$, and $\psi = k + n+m$. It has affine subalgebra $V^{\psi-m-1}(\gg\gl_m)$, and affine coset 
$$\cC^{\psi}(n,m) = \text{Com}(V^{\psi-m-1}(\gg\gl_m), \cW^{\psi}(n,m)).$$ When $n=0$, we need a different definition; we set
$$\cC^{\psi}(0,m) = \text{Com}(V^{k-1}(\gg\gl_m), V^k(\gs\gl_m) \otimes \cS(m)),$$ where $\cS(m)$ is the rank $m$ $\beta\gamma$-system, which has an action of $L_{-1}(\gg\gl_m)$.

Similarly, for $n\geq 1$, $m\geq 0$, and $m\neq n$, we define $\cV^{\psi}(n,m) = \cW^k(\gs\gl_{n|m}, f_{n|1^m})$, where $f_{n|1^m}$ is principal in the subalgebra $\gs\gl_n$ and trivial in $\gs\gl_m$, and $\psi = k+n-m$. It has affine subalgebra $V^{-\psi-m+1}(\gg\gl_m)$, and affine coset
$$\cD^{\psi}(n,m) = \text{Com}(V^{-\psi-m+1}(\gg\gl_m), \cV^{\psi}(n,m)).$$
For $m = n \geq 2$, we need a slightly different definition: we define $\cD^{\psi}(n,n) = \text{Com}(V^{-\psi-n+1}(\gs\gl_n), \cV^{\psi}(n,n))^{\text{GL}_1}$, where $\cV^{\psi}(n,n) = \cW^k(\gp\gs\gl_{n|n}, f_{n|1^n})$. When $n=0$, we set
$$\cD^{\psi}(0,m) = \text{Com}(V^{-k+1}(\gg\gl_m), V^{-k}(\gs\gl_m) \otimes \cE(m)),$$ where $\cE(m)$ is the rank $m$ $bc$-system, which has an action of $L_{1}(\gg\gl_m)$. We also set $\cD^{\psi}(1,1) \cong \cA(1)^{\text{GL}_1}$ where $\cA(1)$ is the rank $1$ symplectic fermion algebra. Following the notation in \cite{CL3}, the triality theorem is that for integers $n\geq m \geq 0$, we have isomorphisms of $1$-parameter vertex algebras
\begin{equation} \label{typeAtriality} \cD^\psi(n, m)  \cong \cC^{\psi^{-1}}(n-m, m) \cong \cD^{\psi'}(m, n),\qquad  \frac{1}{\psi} +\frac{1}{\psi'} =1.\end{equation}
The case $\cD^\psi(n, 0)  \cong \cC^{\psi^{-1}}(n, 0)$ is just Feigin-Frenkel duality, so we may regard the isomorphisms $\cD^\psi(n, m)  \cong \cC^{\psi^{-1}}(n-m, m)$ as generalizations of Feigin-Frenkel duality. Similarly, 
the case $\cD^\psi(n, 0)  \cong \cD^{\psi'}(0, n)$ is exactly the coset realization; using the fact that $\cE(n)$ is an extension $L_1(\gg\gl_n)$, $\cD^{\psi'}(0, n)$ is easily seen to be isomorphic to the GKO coset appearing in \eqref{eq:cosetrealization}. Therefore the isomorphisms $\cD^\psi(n, m)  \cong \cD^{\psi'}(m, n)$ can be regarded as generalizations of the coset realization.

The key idea in the proof of \eqref{typeAtriality} that $\cC^\psi(n, m)$ and $\cD^\psi(n, m)$ can be realized explicitly as simple $1$-parameter quotients of the universal $2$-parameter vertex algebra $\cW_{\infty}$ of type $\cW(2,3,\dots)$. This is a classifying object for vertex algebras of type $\cW(2,3,\dots, N)$ for some $N$ satisfying some mild hypothesis. It was known to physicists since the early 1990s, and was constructed rigorously by the third author in \cite{Lin}. After tensoring with a Heisenberg algebra, $\cW_{\infty}$ has some better properties and is related to several other structures arising in different contexts. For example, up to suitable completions its associative algebra of modes is isomorphic to the Yangian of $\widehat{\gg\gl_1}$ \cite{AS,MO,Ts}, as well as the algebra ${\bf SH}^c$ defined in \cite{SV} as a certain limit of degenerate double affine Hecke algebras of $\mathfrak{gl}_n$. This identification allowed Schiffmann and Vasserot to define an action of $\cW^{\ell}(\gs\gl_n)$ on the equivariant cohomology of the moduli space of $U_n$-instantons \cite{SV}.

Recently, it has been conjectured that the $Y$-algebras are the building blocks for {\it all} $\cW$-algebras in type $A$ in the sense that any such $\cW$-algebra is an extension of a tensor product of finitely many $Y$-algebras; see \cite[Conjecture B]{CFLN}. This is based on \cite[Conjecture A]{CFLN}, which says that the quantum Drinfeld-Solokov reduction can be carried out in stages; see \cite{FKN,FN,GJ,GJ2} for related results and conjectures in the setting of both finite and affine $\cW$-algebras. In \cite{CFLN}, this conjecture was proven at the level of graded characters and was also verified by computer for all $\cW$-algebras in type $A$ of rank at most $4$. In \cite{FFFN}, a more refined picture connecting all the $\cW$-algebras of $\gs\gl_4$ via partial and inverse reduction was presented. In these examples, it was also shown that reduction by stages gives more than just an isomorphism of vertex algebras; the total reduction functor and iterated reduction functors are isomorphic on the Kazhdan-Lusztig category.

Note that the strongly rational vertex algebras $\cW_{\ell}(\gs\gl_n)$ for all $n\geq 2$ and nondegenerate admissible levels $\ell$ for $\widehat{\gs\gl}_n$, are quotients of $\cW_{\infty}$. It is expected that these are all the strongly rational quotients $\cW_{\infty}$, but this is still an open question. The rationality of all the exceptional $\cW$-algebras in type $A$ was proven in \cite{AvE}, and the Kac-Wakimoto-Arakawa rationality conjecture was proven in full generality by McRae in \cite{McR}. From the perspective of building blocks, the rationality of such $\cW$-algebras can also be explained by the observation that at the special levels where rationality occurs, each of the building blocks should be isomorphic to one of the strongly rational algebras $\cW_{\ell}(\gs\gl_n)$.

	In \cite{GR}, Gaiotto and Rap\v{c}\'ak also introduced the orthosymplectic $Y$-algebras, which can be realized as affine cosets of $\cW$-(super)algebras in types $B$, $C$, and $D$. They conjectured a similar triality of isomorphisms which includes as special cases Feigin-Frenkel duality, as well as the coset realizations \eqref{eq:cosetrealization} for type $D$ and \eqref{cosetBC} for type $B$. This conjecture was proven by the first and third authors in \cite{CL4} by realizing these algebras explicitly as simple $1$-parameter quotients of the universal $2$-parameter vertex algebra $\cW^{\text{ev}}_{\infty}$ of type $\cW(2,4,6,\dots)$. This algebra was also known to physicists \cite{CGKV}, and was constructed by Kanade and the third author in \cite{KL}. We mention that $\cW^{\text{ev}}_{\infty}$ has several families of quotients that are strongly rational. First, the algebras $\cW_{\ell}(\gs\go_{2n+1})$ and $\cW_{r}(\gs\go_{2n})^{\mathbb{Z}_2}$, where $\ell$ and $r$ are nondegenerate admissible for $\widehat{\gs\go}_{2n+1}$ and $\widehat{\gs\go}_{2n}$, respectively, are all strongly rational quotients of $\cW^{\text{ev}}_{\infty}$. Also, another special case of orthosymplectic triality is a coset realization of principal $\cW$-algebras of $\go\gs\gp_{1|2n}$:
\begin{equation}\label{GKOtypeB}  \cW^{\psi' - n-1/2}(\go\gs\gp_{1|2n}) \cong \text{Com}(V^{-2\psi -2n +2}(\gs\go_{2n+1}), V^{-2\psi -2n+1}(\gs\go_{2n+1}) \otimes \cF(2n+1)), \qquad \frac{1}{\psi} + \frac{1}{\psi'} = 2.
\end{equation} Here $\cF(2n+1)$ is the rank $2n+1$ free fermion algebra. This result was proven in \cite{CL4} by first establishing the isomorphism of $\mathbb{Z}_2$-orbifolds of both sides, which are $1$-parameter quotients of $\cW^{\text{ev}}_{\infty}$, and then proving the uniqueness of the corresponding simple current extension. We regard the right hand side of \eqref{GKOtypeB} as the GKO coset of type $B$, and it was also conjectured in \cite{CL4} that this coset is strongly rational whenever the level $-2\psi -2n +1$ is admissible for $\widehat{\gs\go}_{2n}$. This would then imply the strong rationality of the corresponding $\mathbb{Z}_2$-orbifolds.

By analogy with type $A$, one might expect that the building blocks for $\cW$-(super)algebras of types $B$, $C$, and $D$ are the orthosymplectic $Y$-algebras, but it is readily seen that these are some, but not all, of the necessary building blocks. The reduction by stages conjecture says that if a nilpotent $f$ can be decomposed as $f = f_1 + f_2$ with $f_1 \geq f_2$ in a sense we shall define in Section \ref{sec:reduction}, then the reduction $H_f(V^k(\gg))$ coincides with $H_{f_2}(H_{f_1}(V^k(\gg)))$. We call a nilpotent $f$ {\it indecomposable} if it cannot be written in this way. The reason that only $Y$-algebras appear in \cite[Conjecture A]{CFLN} is that the only indecomposable nilpotents in type $A$ are the hook-type nilpotents. But in types $B$, $C$, and $D$, there are more indecomposable nilpotents. This suggests that more building blocks will be needed.

\subsection{Main result}
Our main result in this paper, which is a paraphrasing of Theorems \ref{thm:induction}, \ref{one-parameter quotients theorem}, and Corollary \ref{Wsp freely generated}, is the following
\begin{theorem} There exists a unique $2$-parameter vertex algebra $\cW^{\gs\gp_2}_{\infty}$ with the following features:
\begin{enumerate}
\item It is defined over a localization of the polynomial ring $\mathbb{C}[c,k]$ and is freely generated of type
\begin{equation} \label{sp2starting} \cW(1^3, 2, 3^3, 4, 5^3, 6,\dots),\end{equation} and is weakly generated by the fields in weights up to $4$.
\item The three fields in weight $1$ generate a copy of the affine vertex algebra $V^k(\gs\gp_2)$.
\item The field in weight $2$ is a conformal vector with central charge $c$.
\item The fields in each even weight $2,4,6,\dots$ transform as the trivial $\gs\gp_2$-module.
\item The three fields in each odd weight $3,5,7,\dots$ transform as the adjoint $\gs\gp_2$-module. 
\end{enumerate}
Moreover, $\cW^{\gs\gp_2}_{\infty}$ serves as a classifying object for vertex algebras with these properties; any vertex algebra with a strong generating set of type \eqref{sp2starting} (not necessarily minimal) satisfying the above conditions, is a quotient of $\cW^{\gs\gp_2}_{\infty}$.
\end{theorem} 
There are $8$ infinite families of $1$-parameter quotients of $\cW^{\gs\gp_2}_{\infty}$ which are either $\cW$-algebras of a type we call {\it $\gs\gp_2$-rectangular}, or (orbifolds of) affine cosets of $\cW$-algebras of a type we call {\it $\gs\gp_2$-rectangular with a tail}. They are denoted by $\cC^{\psi}_{XY}(n,m)$ where $X = B,C$ and $Y = B,C,D,O$, and $n,m \in \mathbb{N}$. This list is quite parallel to the list of orthosymplectic $Y$-algebras; accordingly we call them $Y$-algebras of type $C$. We will give some heuristic reasons based on decompositions of nilpotents in types $B$, $C$, and $D$ for why we expect the $Y$-algebras of type $C$ to account for all the new building blocks for $\cW$-algebras of orthosymplectic types. In particular, we expect that every $\cW$-(super)algebra of type $B$, $C$, or $D$, is an extension of a tensor product of finitely many vertex algebras which are quotients of either $\cW^{\text{ev}}_{\infty}$ or $\cW^{\gs\gp_2}_{\infty}$.

The construction of $\cW^{\gs\gp_2}_{\infty}$ is similar to the constructions of $\cW_{\infty}$ and $\cW^{\text{ev}}_{\infty}$ in \cite{Lin,KL}, but is much more involved. We regard it as a fundamental object on equal footing with $\cW_{\infty}$ and $\cW^{\text{ev}}_{\infty}$, but it has inexplicably never appeared before either in the mathematics or physics literature. However, $\cW^{\gs\gp_2}_{\infty}$ does differ from $\cW_{\infty}$ and $\cW^{\text{ev}}_{\infty}$ in two important ways. First, the vertex algebras $\cC^{\psi}_{XY}(n,m)$ are all distinct, and there are no triality isomorphisms among them. By contrast, the $8$ families of orthosymplectic $Y$-algebras are organized into groups of three algebras that are isomorphic. Second, it is expected (although still an open question) that the $Y$-algebras and orthosymplectic $Y$-algebras are the complete list of simple, strongly finitely $1$-parameter quotients of $\cW_{\infty}$ and $\cW^{\text{ev}}_{\infty}$, respectively. However, $\cW^{\gs\gp_2}_{\infty}$ admits additional infinite families of simple, strongly finitely generated $1$-parameter quotients, which we denote by $\cC^{\ell}(n)$. Here $\ell$ is a complex parameter related to the central charge, $n \in \frac{1}{2}\mathbb{Z}$ is fixed, and $\cC^{\ell}(n)$ contains the simple quotient $L_n(\gs\gp_2)$. They are defined as follows:
 \begin{equation} \begin{split} \label{intro:diagcoset}
		\cC^{\ell}\big(-\frac{n}{2}\big)  & = \T{Com}(V^{\ell}(\fr{so}_n), V^{\ell+2}(\fr{so}_n)\otimes \cS(n))^{\mathbb{Z}_2}, \ n\in \mathbb{Z}_{\geq 2},
\\ \cC^{\ell}(n) & = \T{Com}(V^{\ell}(\fr{sp}_{2n}), V^{\ell-1}(\fr{sp}_{2n})\otimes \cE(2n)),\ n \in \mathbb{Z}_{\geq 1},
\\ \cC^{\ell}\big(n-\frac{1}{2}\big) &  = \T{Com}(V^{\ell}(\fr{osp}_{1|2n}), V^{\ell-1}(\fr{osp}_{1|2n})\otimes \cS(1)\otimes \cE(2n))^{\mathbb{Z}_2}, \ n \in \mathbb{Z}_{\geq 1}.
\end{split}
	\end{equation}
We regard $\cC^{\ell}(n)$ and $\cC^{\ell}\left(n-\frac{1}{2}\right)$ as the GKO cosets for $\gs\gp_{2n}$ and $\go\gs\gp_{1|2n}$, respectively. Note that when $\ell-1$ is admissible for $\widehat{\gs\gp}_{2n}$, we have an embedding $L_{\ell}(\fr{sp}_{2n}) \rightarrow L_{\ell-1}(\fr{sp}_{2n})\otimes \cE(2n)$, and the simple quotient $\cC_{\ell}(n)$ is just the coset $\T{Com}(L_{\ell}(\fr{sp}_{2n}), L_{\ell-1}(\fr{sp}_{2n})\otimes \cE(2n))$. We will prove the following result (Theorem \ref{thm:coset-rat}).
 \begin{theorem} \label{rationality:intro} 
Let $\ell -1$ be an admissible level for $\gs\gp_{2n}$. Then $\cC_\ell(n)$ is strongly rational. 
\end{theorem}
 We regard the family of algebras $\cC_\ell(n)$ as the analogue of the families of strongly rational quotients of $\cW_{\infty}$ and $\cW^{\text{ev}}_{\infty}$ mentioned above.  The proof of Theorem \ref{rationality:intro} is based on realizing $\cC_\ell(n)$ as an extension of $\cW_{\ell_1}(\gs\gp_{2n}) \otimes \cW_{\ell_2}(\gs\gp_{2n})$ for $\ell_1 = -(n+1) + \frac{\ell+n+1}{2\ell+2n+1}$ and $\ell_2 = -(n+1) + \frac{\ell+n}{2\ell+2n+1}$, which is strongly rational because $\ell_1, \ell_2$ are both nondegenerate (co)-admissible levels. As in type $A$, we expect that the rationality of exceptional $\cW$-algebras of types $B$, $C$, and $D$ can be explained using building blocks, i.e., such an exceptional $\cW$-algebra can be realized as an extension of a tensor product of quotients of $\cW^{\text{ev}}_{\infty}$ or $\cW^{\gs\gp_2}_{\infty}$ which are all rational. Using this perspective, we will give a new proof of the rationality of certain exceptional $\cW$-algebras, as well as the rationality of a new family of orthosymplectic $\cW$-superalgebras. 
  	
A key step in the proof of Theorem \ref{rationality:intro} is to show that for each $n \in \mathbb{N}$, $\cC^\ell(n)$ is an extension of $\cW^{\ell_1}(\gs\gp_{2n}) \otimes \cW^{\ell_2}(\gs\gp_{2n})$ as a $1$-parameter vertex algebra; see Theorem \ref{twocopiesofW(sp)}. Since $\cC^\ell(n)$ is a $1$-parameter quotient of $\cW^{\gs\gp_2}_{\infty}$ and both $\cW^{\ell_1}(\gs\gp_{2n})$ and $\cW^{\ell_2}(\gs\gp_{2n})$ are $1$-parameter quotients of $\cW^{\text{ev}}_{\infty}$, it is natural to ask whether $\cW^{\gs\gp_2}_{\infty}$ is an extension of two commuting copies of $\cW^{\text{ev}}_{\infty}$ in a way that is compatible with these maps for all $n\geq 1$. It can be checked by computer calculation that $\cW^{\gs\gp_2}_{\infty}$ does {\it not} have this property. However, there exists a completion $\tilde{\cW}^{\gs\gp_2}_{\infty}$ which is indeed an extension of two commuting copies of $\cW^{\text{ev}}_{\infty}$ with this compatibility property; see Theorem \ref{completion:twocopies}. Finally, we consider the quantum Hamiltonian reduction of $H_f(\cW^{\gs\gp_2}_{\infty})$ with respect to the nonzero nilpotent $f$ in $\gs\gp_2$. We expect that as a $2$-parameter vertex algebra, $H_f(\cW^{\gs\gp_2}_{\infty})$ is freely generated of type $\cW(2^3, 4, 3^3, 5, 6^3, 7,\dots)$, and is an extension of two commuting copies of $\cW^{\text{ev}}_{\infty}$. This is closely related to the fact that $\tilde{\cW}^{\gs\gp_2}_{\infty}$ is an extension of this structure.

\subsection{Organization} In Section \ref{sec:VOA} we recall free field algebras, affine vertex algebras, and $\cW$-algebras following the notation in the papers \cite{CL3,CL4} of the first and third authors. In Section \ref{sec:reduction} we define a notion of decomposition of a nilpotent and recall a conjecture on reduction by stages, which motivates our main construction. In Section \ref{sect:YtypeC} we introduce the $\gs\gp_2$-rectangular $\cW$-algebras as well as those with a tail, and define the $8$ families of $Y$-algebras of type $C$ which have strong generating type \eqref{sp2starting}. In Section \ref{sect:Diagonal}, we introduce the $4$ families of diagonal cosets \eqref{intro:diagcoset} that also have this strong generating type. In Section \ref{sect:main}, we construct the universal $2$-parameter vertex algebra $\cW^{\gs\gp_2}_{\infty}$ of this type, which is our main result. In Section \ref{sect:1paramquot}, we prove that the $Y$-algebras of type $C$ and the diagonal cosets indeed are $1$-parameter quotients of $\cW^{\gs\gp_2}_{\infty}$. In Section \ref{sect:recon}, we prove a reconstruction theorem (Theorem \ref{thm:reconstruction}) that says that the full OPE algebra of an $\gs\gp_2$-rectangular $\cW$-algebra with a tail, which is an extension of an affine vertex algebra and a $Y$-algebra of type $C$, is uniquely and constructively determined by the conformal weight and parity of the extension fields and the zero mode action of the Lie algebra on the extension fields. In Section \ref{sect:rational}, we prove Theorem \ref{rationality:intro}, which allows us to deduce new rationality results for $\cW$-superalgebras. Finally, in Section \ref{sect:completion} we explore the relationship between $\cW^{\gs\gp_2}_{\infty}$ and the tensor product of two copies of $\cW^{\text{ev}}_{\infty}$. In particular, we will define the completion $\tilde{\cW}^{\gs\gp_2}_{\infty}$ and prove Theorem \ref{completion:twocopies}.

\section{Vertex algebras} \label{sec:VOA} We will assume that the reader is familiar with vertex algebras, and we use the same notation as the paper \cite{CL3} of the first and third authors. We will make use of the following well-known identities which hold in any vertex algebra $\cA$.
\begin{equation}\label{conformal identity}
			(\partial a)_{(r)}b=-ra_{(r-1)}b,\quad r\in\mathbb{Z}.
		\end{equation}
		\begin{equation}\label{skew-symmetry}
			a_{(r)}b=(-1)^{|a||b|+r+1}b_{(r)}a + \sum_{i = 1}^{\infty}\frac{(-1)^{|a||b|+r+i+1}}{i!}\partial^i \left(b_{(r+i)}a\right),\quad r\in\mathbb{Z}.
		\end{equation}
		\begin{equation}\label{quasi-associativity}
			:a (:bc:):\ =\ :(:ab:)c: +\sum_{i=0}^{\infty}\frac{1}{(i+1)!}\left(:\!\partial^{i+1}(a) (b_{(i)}c)\!:+(-1)^{|a||b|}:\!\partial^{i+1}(b) (a_{(i)}c)\!:\right).
		\end{equation}
		\begin{equation}\label{quasi-derivation}
			a_{(r)}:bc:\ =\ :(a_{(r)}b)c:+(-1)^{|a||b|}: b(a_{(r)}c):+\sum_{i=1}^r\binom{r}{i}(a_{(r-i)}b)_{(i-1)}c, \quad r\geq 0.
		\end{equation}
		\begin{equation}\label{Jacobi}
			a_{(r)}(b_{(s)}c) = (-1)^{|a||b|}b_{(s)}(a_{(r)}c) + \sum_{i=0}^r \binom{r}{i}(a_{(i)}b) _{(r+s-i)}c, \quad r,s\geq 0.
		\end{equation}
		Identities \eqref{Jacobi} are known as {\it Jacobi identities}, and we often denote them using the shorthand $J_{r,s}(a,b,c)$. We use $J(a,b,c)$ to denote the set of all Jacobi identities $\{J_{r,s}(a,b,c)|r,s\geq 0\}$.

	\subsection{Free field algebras}
	A {\it free field algebra} is a vertex superalgebra $\cV$ with weight grading 
	\[\cV = \bigoplus_{d\in \frac{1}{2}\mathbb{Z}_{\geq 0}}\cV[d],\quad \cV[0] = \mathbb{C}\B{1},\] 
	with strong generators $\{X^i | i\in I\}$ satisfying OPE relations
	\[X^i(z)X^j(w)\sim a_{i,j} (z-w)^{-\Delta(X^i)-\Delta(X^j)}, \quad a_{i,j}\in\mathbb{C}, \quad a_{i,j} =0 \T{ if }\Delta(X^i)+\Delta(X^j)\not\in\mathbb{Z}.\]
	Note $\cV$ is not assumed to have a conformal structure. In \cite{CL3}, the first and third authors introduced the following families of free field algebras.
	\begin{enumerate} 
	\item Even algebras of orthogonal type $\cO_{\text{ev}}(n,k)$ for $n\geq 1 $ and even $k \geq 2$. When $k =2$, $\cO_{\text{ev}}(n,2)$ is just the rank $n$ Heisenberg algebra $\cH(n)$.
	\item Odd algebras of orthogonal type $\cO_{\text{odd}}(n,k)$ for $n\geq 1 $ and odd $k \geq 1$. When $k = 1$, $\cO_{\text{odd}}(n,1)$ is just the rank $n$ free fermion algebra $\cF(n)$.
	\item Even algebras of symplectic type $\cS_{\text{ev}}(n,k)$ for $n\geq 1$ and odd $k \geq 1$. When  $k = 1$, $\cS_{\text{odd}}(n,1)$ is just the rank $n$ $\beta\gamma$-system $\cS(n)$.
	\item Odd algebras of symplectic type $\cS_{\text{odd}}(n,k)$ for $n\geq 1$ and even $k \geq 2$. When  $k = 2$, $\cS_{\text{ev}}(n,2)$ is just the rank $n$ symplectic fermion algebra $\cA(n)$.
	\end{enumerate}
	We refer the reader to \cite{CL3} for the construction and key properties of these algebras.

\subsection{Affine vertex superalgebras}
Let $\gg$ be a simple, finite-dimensional Lie superalgebra with normalized Killing form $( \cdot | \cdot )$.
	Let $\{q^{\alpha} | \alpha \in S\}$ be a basis of $\gg$ which is homogeneous with respect to parity.
	Define the corresponding structure constants $\{f^{\alpha,\beta}_{\gamma} | \alpha,\beta,\gamma \in S\}$ by 
	\[[q^{\alpha},q^{\beta}] = \sum_{\gamma \in S} f^{\alpha,\beta}_{\gamma} q^{\gamma}.\]
	The affine vertex algebra $V^k(\gg)$ of $\gg$ at level $k$ is strongly generated by the fields $\{X^{\alpha} | \alpha \in S \}$, satisfying 
\begin{equation} \label{OPE:affine} X^{\alpha}(z)X^{\beta}(w) \sim k (q^{\alpha}|q^{\beta}) (z-w)^{-2} + \sum_{\gamma \in S}f^{\alpha,\beta}_{\gamma}X^{\gamma}(w)(z-w)^{-1}.\end{equation}
	We define $X_{\alpha}$ to be the field corresponding to $q_{\alpha}$ where $\{q_{\alpha} | \alpha \in S\}$ is the dual basis of $\gg$ with respect to $(\cdot|\cdot)$.
	The Sugawara conformal vector $L^{\gg}$ with central charge $c^{\fr{g}}$ is given by
	\begin{equation}\label{sugwara}
		L^{\gg}=\frac{1}{2(k+h^{\vee})}\sum_{\alpha \in S}(-1)^{|\alpha|} :\!X_{\alpha}X^{\alpha}\!:, \quad c^{\fr{g}}= \frac{k\ \T{sdim}\gg}{k+h^{\vee}}.
	\end{equation}
	Fields $X^{\alpha}(z)$ and $X_{\alpha}(z)$ are primary with respect to $L^{\gg}$ and have conformal weight $1$.

	Let $\cV$ be a vertex algebra equipped with a homomorphism
	 $V^k(\gg)\to \cV$.
	 We continue using the notation $\{X^{\alpha}|\alpha \in S\}$ to denote the image of the generators.
	In particular, $\gg$ acts on $\cV$ by derivation via the zero modes $\{X^{\alpha}_{(0)}\}$. Let $P=\T{Span}\{P^i|i=1,\dots,n\}$ denote some irreducible $\gg$-submodule arising in $\cV$, and $\rho:\gg\to \T{End}(V)$ denote the corresponding action.
	We say that $P$ is affine primary if it satisfies
	\[X^{\alpha}(z)P^i(w)\sim (\rho(\alpha)P^i) (w)(z-w)^{-1}, \quad i=1,\dots,n.\]
	In all our examples these will be either trivial, standard or adjoint representations.

	Let $(V,g)$ be the vector superspace $\mathbb{C}^{n|2m}$ of dimension $n|2m$, equipped with a supersymmetric bilinear form $g$.
	The Lie superalgebra $\fr{osp}_{n|2m}$ is the Lie subsuperalgebra of $\fr{gl}_{n|2m}$ preserving $g$, i.e.,
	\[\fr{osp}_{n|2m}= \{A \in \fr{gl}_{n|2m} |\ (Av|w)+(-1)^{|A||v|}(v|Aw) = 0,\ \forall v ,w \in V\}.\]
	The even part of $\go\gs\gp_{n|2m}$ is the semisimple Lie algebra $\fr{so}_n\oplus \fr{sp}_{2m}$, and the odd part transforms as $\mathbb{C}^n\otimes  \mathbb{C}^{2m}$ under $\fr{so}_n\oplus \fr{sp}_{2m}$. 
	
	Choose an orthonormal basis $\{P_i|i=1,\dots,n\}$ for the even subspace $V_{\bar 0}$, and a symplectic basis $\{Q_i,Q_{-i}|i=1,\dots,m\}$ for the odd subspace $V_{\bar{1}}$, so that
	\begin{equation*}
		g(P_i,P_j) =\delta_{i,j}:=\begin{cases}
			1, & i=j,\\
			0,& i\neq j .
		\end{cases}\quad g(Q_i,Q_j) =\omega_{i,j}:=
		\begin{cases}
			\delta_{i,-j},& i\geq 1,\\
			-\delta_{i,-j},& i\leq 1.
		\end{cases}
	\end{equation*}
	In particular, the above pairings identify $\mathbb{C}^{n|2m}$ with the dual space $(\mathbb{C}^{n|2m})^{*}$, equipped with the dual basis $\{P_i^*|i=1,\dots,n\}\cup \{Q_i^*,Q_{-i}^*|i=1,\dots,m\}$.
	
	First, we consider the orthogonal Lie subalgebra $\fr{so}_n$. 
	Define maps $\{E_{i,j}|1 \leq  i < j\leq n \}$ by their action on $\mathbb{C}^{n|0}$ as endomorphisms
	\begin{equation}\label{soStandard}
		\fr{so}_n \to \fr{gl}_{n|0},\quad E_{i,j}\mapsto P_j^*\otimes P_i - P_i^*\otimes P_j.
	\end{equation}
This defines the standard representation of $\fr{so}_{n}$. In this basis, the Lie bracket of $\fr{so}_n$ has the form
	\[[E_{i,j},E_{p,q}] = -\delta_{j,q} E_{i,p}+ \delta_{j,p} E_{i,q} + 
	\delta_{i,q}E_{j,p} - \delta_{i,p}E_{j, q}.\]
	The normalized Killing form and the dual Coxeter number are 
	\begin{equation}
		(E_{i,j}|E_{p,q}) =\delta_{i,p}\delta_{j,q}-\delta_{i,q}\delta_{j,p}, \quad h^{\vee}_{\fr{so}_n} = n-2.
	\end{equation}
	
	Next we consider the symplectic Lie subalgebra $\fr{sp}_{2m}$. 
	Define maps $\{G_{i,j}|-m \leq  i,j \leq m, i,j\neq 0\}$ that act on $\mathbb{C}^{0|2m}$ as
	endomorphisms
	\begin{equation}\label{spStandard}
		\fr{sp}_{2m} \to \fr{gl}_{0|2m},\quad G_{i,j}\mapsto Q_j^*\otimes Q_i + Q_i^*\otimes Q_j.
	\end{equation}
	This defines the standard representation of $\fr{sp}_{2n}$.
	In this basis, the Lie bracket of $\fr{sp}_{2m}$ has the form
	\begin{equation}\label{orthonormalC}
		\begin{split}
			[G_{i,j},G_{p,q}] =\omega_{j,q}G_{i,p}+ \omega_{j,p}G_{i,q} + 
			\omega_{i,q}G_{j,p}+ \omega_{i,p}G_{j, q}.
		\end{split}
	\end{equation}
	The normalized Killing form and the dual Coxeter number are 
	\begin{equation}
		(G_{i,j}|G_{p,q}) =\omega_{i,p}\omega_{j,q}+\omega_{i,q}\omega_{j,q}, \quad h^{\vee}_{\fr{sp}_{2m}} = m+1.
	\end{equation}

	Lastly, we write down the basis for the odd part of $\fr{osp}_{n|2m}$.
	Define maps $\{X_{i,j}|1\leq i \leq n,-m\leq j\leq m,j\neq 0\}$ that act on  $\mathbb{C}^{n|2m}$ as endomorphisms
	\begin{equation}\label{ospStandard}
		X_{i,j}=P_i ^*\otimes Q_j + Q_{j}^*\otimes P_i.
	\end{equation}
	In this basis, the Lie bracket given by
	\begin{equation}
		[X_{i,j},X_{p,q}] = \delta_{i,p} Q_{j,q} + \omega_{j,q}E_{i,p}.
	\end{equation}

	\subsection{$\cW$-superalgebras}\label{sec:W}
	Let $\gg$ be a simple, finite-dimensional Lie (super)algebra as above, and let $f$ be a nilpotent element in the even part of $\gg$. Associated to $\gg$, $f$ and a complex number $k$, is the $\cW$-(super)algebra $\cW^k(\gg,f)$ of level $k$. The definition is due to Kac, Roan, and Wakimoto \cite{KRW}, and it generalizes the definition for $f$ a principal nilpotent and $\gg$ a Lie algebra given by Feigin and Frenkel \cite{FF}. We may complete $f$ to an $\fr{sl}_2$-triple $ \{f,h,e\}$ satisfying
	\[[h,e] =2 e, \quad [h,f] = -2f, \quad [e,f] =h.\]
	The semisimple element $x= \frac{1}{2}h$ induces a $\frac{1}{2}\mathbb{Z}$-grading on $\gg$ as follows.
	\begin{equation}\label{decomposition of g over sl2}
		\fr{g} = \bigoplus_{j\in \frac{1}{2}\mathbb{Z}}\fr{g}_j, \quad \gg_j = \{ a \in \gg | [x,a] = j a\}.
	\end{equation}
	We may assume that the basis $S$ of $\gg$ has the form $S = \bigcup S_k$, where $S_k$ is a basis of $\gg_k$. Write $S_+= \bigcup\{ S_k | k <0\}$ and $S_-= \bigcup\{S_k | k>0\}$ for bases of respective subspaces 
	\[ \gg_+ =\bigoplus_{j\in \frac{1}{2}\mathbb{Z}_{>0}}\fr{g}_j ,\quad  \gg_- =\bigoplus_{j\in \frac{1}{2}\mathbb{Z}_{<0}}\fr{g}_j.\]
Denote by $F(\gg_+)$ the algebra of charged fermions associated to the vector superspace $\gg_+ \oplus \gg_+^*$. It is strongly generated by fields $\{\varphi_{\alpha}, \varphi^{\alpha} | \alpha \in S_+\}$, where $\varphi_{\alpha}$ and $\varphi^{\alpha}$ have opposite parity to $q^{\alpha}$. They satisfy
	\[\varphi_{\alpha}(z)\varphi^{\beta}(w) \sim \delta_{\alpha,\beta}(z-w)^{-1},\quad\varphi_{\alpha}(z)\varphi_{\beta}(w)\sim 0 \sim \varphi^{\alpha}(z)\varphi^{\beta}(w).\]
	We give  $F(\gg_+)$ the conformal vector and associated central charge
		\begin{equation}\label{charge Fermions}
		L^{\T{ch}}=\sum_{\alpha \in S_+} (1-m_{\alpha}):\partial \varphi^{\alpha}\varphi_{\alpha}\!:-m_{\alpha}:\!\varphi^{\alpha}\partial \varphi_{\alpha}:,\quad c^{\T{ch}}= - \sum_{\alpha \in S_+} (-1)^{|\alpha|}(12m_{\alpha}^2-12m_{\alpha}+2).
	\end{equation}
	The fields $\varphi_{\alpha}(z)$ and $\varphi^{\alpha}(z)$ are primary with respect to $L^{\T{ch}}$ and have conformal weights of $1-m_{\alpha}$ and $m_{\alpha}$, respectively. Since $f \in \gg_{-1}$, it endows $\gg_{\frac{1}{2}}$ with a skew-symmetric bilinear form
	\begin{equation}\label{bilinear form}
		\langle a,b\rangle = (f|[a,b]).
	\end{equation}
	Denote by $F(\gg_{\frac{1}{2}})$ the algebra of neutral fermions associated to $\gg_{\frac{1}{2}}$. It has strong generators $\{\Phi_{\alpha}| \alpha \in S_{\frac{1}{2}}\}$, where $\Phi_{\alpha}$ and $q^{\alpha}$ have the same parity, satisfying
	\[\Phi_{\alpha}(z)\Phi_{\beta}(w) \sim \langle q^{\alpha},q^{\beta} \rangle (z-w)^{-1}.\]
	We give $F(\gg_{\frac{1}{2}})$ the conformal vector $L^{\T{ne}}$ with central charge $c^{\T{ne}}$, where
	\[ L^{\T{ne}}=\frac{1}{2}\sum_{\alpha \in S_{\frac{1}{2}}} :\!\partial \Phi^{\alpha}\Phi_{\alpha}\!:,\quad c^{\T{ne}}=-\frac{1}{2}\T{sdim}\gg_{\frac{1}{2}}.\]
	Here $\Phi^{\alpha}$ is dual to $\Phi_{\alpha}$ with respect to the bilinear form (\ref{bilinear form}), and $\Phi^{\alpha}(z), \Phi_{\alpha}(z)$ are primary of conformal weight of $\frac{1}{2}$ with respect to $L^{\T{ne}}$.
	
	As in \cite{KRW}, define $C^k(\gg,f)= V^k(\gg)\otimes F(\gg_+)\otimes F(\gg_{\frac{1}{2}})$. 
	It admits a $\mathbb{Z}$-grading by charge
	\[C^k(\gg,f)=\bigoplus_{j \in\mathbb{Z}} C_j, \]
	by giving the $\varphi_{\alpha}$ charge $-1$, the $\varphi^{\alpha}$ charge 1, and all others 0.
	There is an odd field $d = d_{\T{st}} + d_{\T{tw}}$ of charge $-1$, where
	\begin{equation}\label{Qfield}
		\begin{split}
			d_{\T{st}} &= \sum_{\alpha\in S_+} (-1)^{|\alpha|}:\!X^{\alpha}\varphi^{\alpha}\!: - \frac{1}{2}\sum_{\alpha,\beta,\gamma \in S_+} (-1)^{|\alpha||\gamma|}f^{\alpha,\beta}_{\gamma}:\!\varphi_{\gamma}\varphi^{\alpha}\varphi^{\beta}\!:,\\
			d_{\T{tw}} &=\sum_{\alpha\in S_+}(f|q^{\alpha})\varphi^{\alpha} + \sum_{\alpha \in S_{\frac{1}{2}}}:\!\varphi^{\alpha}\Phi_{\alpha}\!:.
		\end{split}
	\end{equation}
	It satisfies $d(z)d(w) \sim  0$, so the zero mode $d_{(0)}$ is a square-zero map. This endows the vertex algebra $C^k(\gg,f)$ with a structure of $\mathbb{Z}$-graded homology complex, and the $\cW$-algebra is defined to
	\[\cW^k(\gg,f): =H(C^k(\gg,f),d_{(0)}).\]
	It has a conformal vector represented by $L = L^{\gg}+\partial x + L^{\T{ch}} + L^{\T{ne}}$ with central charge 
	\begin{equation}\label{central charge W algebra}
		c(\gg,f,k) = \frac{k\ \T{sdim}(\gg)}{k+h^{\vee}}-12k(x|x) - \sum_{\alpha \in S_+} (-1)^{|\alpha|}(12m_{\alpha}^2-12m_{\alpha}+2)-\frac{1}{2}\T{sdim}(\gg_{\frac{1}{2}}),
	\end{equation}
	where $m_{\alpha} = j$ if $\alpha \in S_j$. 
	
	The following fields feature prominently in the description of $\cW$-algebras. 
	\begin{equation}\label{XtoJ} J^{\alpha}= X^{\alpha} + \sum_{\beta,\gamma\in S_+}(-1)^{|\gamma|}f^{\alpha,\beta}_{\gamma}:\!\varphi_{\gamma}\varphi^{\beta}\!: + \frac{(-1)^{|\alpha}}{2}\sum_{\beta,\gamma \in S_+}f^{\beta,\alpha}_{\gamma}:\!\Phi_{\beta}\Phi^{\gamma}\!:.\end{equation}
	
	Denote by $\gg^f$ the centralizer of $f$ in $\gg$, and let $\ga = \gg^f \cap \gg_0$, which is a Lie subsuperalgebra of $\gg$. The fields $\{J^{\alpha}|q^{\alpha} \in \fr{a}\}$ close under OPEs, and generate an affine vertex algebra of type $\fr{a}$, with its level shifted.
	\begin{theorem}[\cite{KW}, Thm. 2.1]
		\begin{equation}
			J^{\alpha}(z)J^{\beta}(w)\sim (k(q^{\alpha}|q^{\beta})+k^{\Gamma}(q^{\alpha},q^{\beta}))(z-w)^{-2} + f^{\alpha,\beta}_{\gamma} J^{\gamma}(w)(z-w)^{-1},
		\end{equation}
		where
		\[k^{\Gamma}(q^{\alpha},q^{\beta})=\frac{1}{2}\left(\kappa_{\gg}(q^{\alpha},q^{\beta})-\kappa_{\gg_0}(q^{\alpha},q^{\beta})-\kappa_{\frac{1}{2}}(q^{\alpha},q^{\beta})\right),\]
		with $\kappa_{\frac{1}{2}}$ the supertrace of $\gg_0$ on $\gg_{\frac{1}{2}}$.
	\end{theorem}
	
 The key structural theorem is the following.
\begin{theorem} \cite[Theorem 4.1]{KW} \label{thm:kacwakimoto}
Let $\gg$ be a simple finite-dimensional Lie superalgebra with an invariant bilinear
form $( \ \ |  \ \ )$, and let $x, f$ be a pair of even elements of $\gg$ such that ${\rm ad}\ x$ is diagonalizable with
eigenvalues in $\frac{1}{2} \mathbb Z$ and $[x,f] = -f$. Suppose that all eigenvalues of ${\rm ad}\ x$ on $\gg^f$ are non-positive, so that
$\gg^f = \bigoplus_{j\leq 0} \gg^f_j$. Then
 \begin{enumerate}
\item For each $q^\alpha \in \gg^f_{-j}$ with $j\geq 0$, there exists a  $d_{(0)}$-closed field $K^\alpha$ of conformal weight
$1 + j$, with respect to $L$.
\item The homology classes of the fields $K^\alpha$, where $\{q^\alpha\}$ is a basis of $\gg^f$, strongly and freely generate $\cW^k(\gg, f)$.
\item $H_0(C(\gg, f, k), d_0) = \cW^k(\gg, f)$ and $H_j(C(\gg, f, k), d_0) = 0$ if $j \neq 0$.
\end{enumerate}
\end{theorem}
One can also consider the reduction of a module, i.e., for a $V^k(\gg)$-module $M$, the homology of the complex $H(M \otimes F(\gg_+) \otimes F(\gg_{\frac{1}{2}}), d_0)$ is a $\cW^k(\gg, f)$-module that we denote by $H_{f}(M)$, that is 
\begin{equation}\label{Wmodule}
H_{f}(M) := H(M \otimes F(\gg_+) \otimes F(\gg_{\frac{1}{2}}), d_0).
\end{equation} In this notation, $\cW^k(\gg,f) = H_f(V^k(\gg))$.

\subsection{Large level limits of $\cW$-algebras}
Suppose that $\gg$ is a simple Lie (super)algebra and $( \  |  \ )$ is nondegenerate. Let  $x, f$, and 
$\gg^f$ be as in Theorem \ref{thm:kacwakimoto}. In \cite{CL3}, the first and third authors defined a certain large level limit $\cW^{\T{free}}(\gg,f) = \lim_{k\rightarrow \infty} \cW^k(\gg,f)$, which is a simple vertex algebra that can be regarded as the $0^{\T{th}}$-order approximation to $\cW^k(\gg,f)$,

	\begin{theorem}[\cite{CL3}, Thm 3.5 and Cor. 3.4]\label{thm:largelevellimit} 
		$\cW^{\T{free}}(\gg,f)$ is a free field algebra with strong generators $\{X^{\alpha}| q^{\alpha} \in \gg^f\}$ and OPEs
		\begin{equation}\label{freeField}
			X^{\alpha}(z)X^{\beta}(w) \sim (z-w)^{-2k}\delta_{j,k}B_k(q^{\alpha},q^{\beta})
		\end{equation}
		for $q^{\alpha} \in \gg^{f}_{-k}$ and $q^{\beta} \in \gg^{f}_{-j}$, where
		\[B_k:\gg_{-k}^{f}\times \gg^f_{-k} \to \mathbb{C}, \quad B_k(a,b):=(-1)^{2k}((\T{ad} (f))^{2k}b|a). \]
		Moreover, $\cW^{\infty}(\gg,f)$ decomposes as a tensor product of the standard free field algebras. 
		Specifically, let us refer to $2k$ in (\ref{freeField}) as pole order, and $X=\T{Span}\{X^{\alpha} | q^{\alpha} \in \gg^f_{-k}\}$.
		Then,
		\begin{itemize}
			\item if pole order is even and $(\cdot|\cdot)$ is symmetric, $X$ generates an even algebra of orthogonal type.
			\item if pole order is odd and $(\cdot|\cdot)$ is symmetric, $X$ generates an odd algebra of orthogonal type.
			\item if pole order is odd and $(\cdot|\cdot)$ is skew-symmetric, $X$ generates an even algebra of symplectic type.
			\item if pole order is even and $(\cdot|\cdot)$ is skew-symmetric, $X$ generates an odd algebra of symplectic type.
		\end{itemize}
	\end{theorem}
	This theorem is very useful for deducing the strong generating type of cosets of $\cW$-algebras by affine subalgebras. In the limit, the coset becomes a certain orbifold of a free field algebra whose structure can be analyzed using classical invariant theory. A strong generating for this orbifold will give rise to a strong generating set for the coset at generic levels.

	\subsection{Vertex algebras over commutative rings} \label{subsection:voaring}
	We will often work with vertex algebras over a commutative ring $R$ which is a finitely generated $\mathbb{C}$-algebra, or a localization of such a ring. We will use the notation and setup of \cite[Section 3]{Lin}. Let $\cV$ be a vertex algebra over $R$ with conformal weight grading
\begin{equation} \label{RVOAgrading} \cV = \bigoplus_{d \in \frac{1}{2} \mathbb{Z}_{\geq 0}} \cV[d],\qquad \cV[0] \cong R.\end{equation} An ideal $\cI \subseteq \cV$ is called {\it graded} if $$\cI = \bigoplus_{d \in \frac{1}{2} \mathbb{Z}_{\geq 0}} \cI[d],\qquad \cI[d] = \cI \cap \cV[d].$$ 
We recall the notion of simplicity from \cite{Lin}
\begin{definition} \label{VOAsimplicity}
$\cV$ is simple if there are no proper graded ideals $\cI$ such that $\cI[0] = \{0\}$.
\end{definition}
Note that this coincides with the usual notion of simplicity in the case that $R$ is a field. If $I\subseteq R$ is an ideal, we may regard $I$ as a subset of $\cV[0] \cong R$. Let $I \cdot \cV$ denote the vertex algebra ideal generated by $I$. Then $\cV^I = \cV / (I \cdot \cV)$ is a vertex algebra over $R/I$. Even if $\cV$ is simple as a vertex algebra over $R$, $\cV^I$ need not be simple as a vertex algebra over $R/I$.

Let $R$ be the coordinate ring of some variety $X$, and let $\cV$ be a simple vertex algebra over $R$ with weight grading \eqref{RVOAgrading}. Let $I\subseteq R$ be an ideal such that $\cV^I$ is {\it not} simple, i.e., $\cV^I$ has a maximal proper graded ideal $\mathcal{I}$ such that $\mathcal{I}[0] = \{0\}$. Then the quotient 
$$\cV_{I} = \cV^I / \cI$$ is a simple vertex algebra over $R/I$. Letting $Y\subseteq X$ be the closed subvariety corresponding to $I$, we can regard $\cV_I$ as a simple vertex algebra defined over $Y$. As in \cite{Lin}, we define the {\it degeneration poset} of all ideals $I\subseteq R$ for which $\cV^I$ is not simple over $R/I$. For ideals $I_1, I_2$ in the degeneration poset, we have the corresponding simple vertex algebras  $\cV_{I_1} = \mathcal{V}^{I_1} / \mathcal{I}_1$ and $\mathcal{V}_{I_2} = \mathcal{V}^{I_2} / \mathcal{I}_2$ over $R/I_1$ and $R / I_2$, respectively. Let $Y_1, Y_2 \subseteq X$ be the closed subvarieties corresponding to $I_1, I_2$, and let $p \in Y_1 \cap Y_2$. Let $\mathcal{V}^p_{I_1}$ and $ \mathcal{V}^p_{I_2}$ be the vertex algebras over $\mathbb{C}$ obtained by evaluating at $p$. As above, $\mathcal{V}^p_{I_1}$ and $ \mathcal{V}^p_{I_2}$ need not be simple, and we denote their simple quotients by $\mathcal{V}_{I_1,p}$ and $\mathcal{V}_{I_2,p}$. Then $p$ corresponds to a maximal ideal $M_p \subseteq R$ containing both $I_1$ and $I_2$, so we have isomorphisms
\begin{equation}\label{tautologicalsio}  \mathcal{V}_{I_1,p} \cong \mathcal{V}_M \cong \mathcal{V}_{I_2,p}.\end{equation} Typically, the existence of such an isomorphism is nontrivial because $\cV_{I_1}$ and $\cV_{I_2}$ can be very different away from this intersection point.

Our main example $\cW^{\gs\gp_2}_{\infty}$ will be defined over a localization $R = D^{-1} \mathbb{C}[c,k]$ of the polynomial ring $\mathbb{C}[c,k]$, where $D$ is finitely generated multiplicatively closed set; see Theorem \ref{thm:induction}. We will identify certain $1$-parameter vertex algebras $\cC^{\psi}_{XY}(n,m)$ for $n,m \in \mathbb{N}$, and $\cC^{\ell}(n)$ for $n \in \frac{1}{2} \mathbb{Z}$, with quotients of $\cW^{\gs\gp_2}_{\infty}$ along prime ideals $I \subseteq R$, after a suitable localization. All the vertex algebras $\cC^{\psi}_{XY}(n,m)$ will be defined either as $\cW$-algebras $\cW^{\psi}_{XY}(n,m)$, affine cosets of such $\cW$-algebras, or $\mathbb{Z}_2$-orbifolds of such $\cW$-algebras or their cosets. Similarly, $\cC^{\ell}(n)$ is either an affine coset of an affine vertex algebra tensored with a free field algebra, or the $\mathbb{Z}_2$-orbifold of such a coset. The defining equation for the ideal $I$ is given by the formula for the central charge $c$ as a rational function of the level $k$ of the affine $\gs\gp_2$-subalgebra.

We will always regard $\cC^{\psi}_{XY}(n,m)$ and $\cC^{\ell}(n)$ as $1$-parameter vertex algebras, where $\psi$ and $\ell$ are regarded as {\it formal} parameters. They are defined over the localization of $R/I$ obtained by inverting all denominators of structure constants of $\cW^{\gs\gp_2}_{\infty}$ after replacing $c$ and $k$ with the corresponding functions of $\psi$ (respectively $\ell$). However, there is a subtle point which we wish to emphasize. At a given point $\psi_0 \in \mathbb{C}$, the specialization $$\cC^{\psi_0}_{XY}(n,m) := \cC^{\psi}_{XY}(n,m)  / (\psi-\psi_0) \cC^{\psi}_{XY}(n,m)$$ makes sense as long as $\psi$ is not in the (finite) set of poles of denominators of structure constants. However, $\cC^{\psi_0}_{XY}(n,m) $ can be a proper subalgebra of the \lq\lq honest" coset (or orbifold) obtained by first specializing the $\cW$-algebra $\cW^{\psi}_{XY}(n,m)$ to $\psi = \psi_0$, and then taking its coset (or orbifold), even though for generic values of $\psi$ these coincide. The same can happen for $\cC^{\ell_0}(n) := \cC^{\ell}(n) / (\ell - \ell_0) \cC^{\ell}(n)$ of $\cC^{\ell}(n)$ at a particular level $\ell_0 \in \mathbb{C}$.  The issue is that $\cW^{\gs\gp_2}_{\infty}$ is weakly generated by the fields in weight at most $4$, so the same holds for $\cC^{\psi_0}_{XY}(n,m)$ or $\cC^{\ell_0}(n)$, and this weak generation property of the honest coset (or orbifold) can fail at special values $\psi_0$ or $\ell_0$. At the moment, there is no constructive algorithm for determining the set of points where weak generation fails, and a priori these sets need not be finite. So even though we have isomorphisms \eqref{tautologicalsio} between simple quotients of any of the above $1$-parameter quotients of $\cW^{\gs\gp_2}_{\infty}$ at intersection points in their truncation curves, it is a difficult problem to show that at such intersection points, both vertex algebras coincide with the honest cosets (or orbifolds).

With some effort, we are able to show that in the case $\cC^{\psi}_{BC}(0,m) = \cW^{\psi-2m-2}(\gs\gp_{2(2m+1)}, f_{2m+1, 2m+1})$, there are only finitely many values of $\psi$ where the weak generation property can fail, and we describe them explicitly; see Theorem \ref{weakgeneration:Walgebras}. In Section \ref{sect:rational}, using a different approach that involves fusion categories for rational $\cW$-algebras, we will prove that for $n$ a positive integer and $\ell-1$ an admissible level for $\gs\gp_{2n}$, the simple quotient $\cC_{\ell}(n)$ is strongly rational and the weak generation property always holds, so that $\cC_{\ell}(n)$ is always a quotient of $\cW^{\gs\gp_2}_{\infty}$. In Section \ref{sect:rational}, we also determined for which levels $\psi$ the weak generation property holds for $\cC^{\psi}_{CB}(0,m)$ and $\cC^{\psi}_{BO}(0,m)$. Combining this with the rationality of $\cC_{\ell}(n)$, and using the isomorphisms \eqref{tautologicalsio} at intersection points of truncation curves, we will give a new proof of the rationality of certain exceptional $\cW$-algebras, and we will prove the rationality of $\cW_r(\go\gs\gp_{1|2(2m+1)}, f_{2m+1,2m+1})$ for $r = -(2m+\frac{3}{2}) + \frac{3+2k+4m}{1(2m+1)}$ for all but finitely many values of $k$ when $2m+1, 2k+1$ are coprime. For the other $1$-parameter families $\cC^{\psi}_{XY}(n,m)$ and $\cC^{\ell}(n)$, we leave open the problem of effectively determining for which levels the weak generation property holds.

\section{Indecomposable nilpotents and reduction by stages} \label{sec:reduction}

In this section, we provide some motivation for our main result, which is the construction of a new universal $2$-parameter vertex algebra that is freely generated of type \eqref{sp2starting}. Let $\mathfrak{g}$ be a simple Lie algebra, $f \in \mathfrak{g}$ a nilpotent element, and suppose that $f$ has a decomposition $f = f_1 + f_2$, where $f_1, f_2$ are both nilpotents in $\mathfrak{g}$. Complete both $f_1$ and $f_2$ to $\mathfrak{sl}_2$-triples in $\mathfrak{g}$, and let $\mathfrak{g}^{\natural}(f_1)$ and $\mathfrak{g}^{\natural}(f_2)$ be the centralizers of these $\mathfrak{sl}_2$-triples in $\mathfrak{g}$. Suppose further that 
\begin{equation} \label{decomp:firstattempt} f_2 \in \mathfrak{g}^{\natural}(f_1),\qquad f_1 \in \mathfrak{g}^{\natural}(f_2),\end{equation} and consider the reduction $H_{f_1}(\mathfrak{g})$. Since this has affine subalgebra $V^{k'}(\mathfrak{g}^{\natural}(f_1))$ and $f_2 \in \mathfrak{g}^{\natural}(f_1)$, it is meaningful to consider the reduction $H_{f_2}(H_{f_1}(\mathfrak{g}))$. For $i = 1,2$, let $\mathfrak{g}_i$ be the centralizer of $\mathfrak{g}^{\natural}(f_i)$ in $\mathfrak{g}$, and let $d_i$ be the rank of the largest simple component of $\mathfrak{g}_i$. We say that $f_1 \geq f_2$ if $d_1 \geq d_2$.

\begin{conjecture} \label{reductionstages}
If $f = f_1 +f_2$, $f_1, f_2$ satisfy \eqref{decomp:firstattempt}, and $f_1 \geq f_2$, then 
\begin{equation} \label{red:stages} H_{f}(\mathfrak{g}) \cong  H_{f_2}(H_{f_1}(\mathfrak{g})).\end{equation}
\end{conjecture}

Since $\text{Com}(V^{k'}(\gg^{\natural}(f_1)), H_{f_1}(\mathfrak{g}))$ is unchanged by applying $H_{f_2}$, Conjecture \ref{reductionstages} implies that $ H_{f_2}(H_{f_1}(\mathfrak{g}))$ is a conformal extension of 
$$\cW^{k'}(\gg^{\natural}(f_1), f_2) \otimes \text{Com}(V^{k'}(\gg^{\natural}(f_1)), \cW^k(\mathfrak{g}, f_1)).$$

In type $A$, this conjecture and some of its consequences appeared in \cite{CFLN}. In particular, every nilpotent $f \in \mathfrak{sl}_N$ can be written as a sum of hook-type nilpotents, which satisfy the condition \eqref{decomp:firstattempt} pairwise. If $f$ corresponds to a partition of $N$ with $m$ parts of size at least $2$, then $\cW^k(\mathfrak{sl}_N, f)$ would be a conformal extension of a tensor product of $m$ $Y$-algebras.
	
There may be different decompositions of a nilpotent $f$ such that \eqref{decomp:firstattempt} is satisfied. 
For example, consider $f_{n,n} \in \gs\gl_{2n}$, which corresponds to the partition with two parts of size $n$. We have $\gs\gl_{2n}^{\natural}(f_{n,n}) \cong \gs\gl_2$, which contains a nilpotent $f'$ which is conjugate to the standard minimal nilpotent $f_{2, 1^{2n-2}} \in \gs\gl_{2n}$. By abuse of notation, we write $f' = f_{2, 1^{2n-2}}$. Then $f_{n,n} + f_{2, 1^{2n-2}}$ is in the same conjugacy class as $f_{n+1, n-1}$, which also decomposes as the sum of hook-type nilpotents $f_{n+1, n-1} = f_{n+1, 1^{n-1}} + f_{n-1, 1^{n+1}}$. 
This suggests that 
\begin{equation} \label{fks:typeA} H_{f_{2, 1^{2n-2}}} (H_{f_{n,n}} (\gs\gl_{2n}))  \cong H_{f_{n+1, n-1}}(\gs\gl_{2n}) \cong H_{f_{n-1}} (H_{f_{n+1}}(\gs\gl_{2n})).\end{equation} The first isomorphism in \eqref{fks:typeA} has been proven by the second author with Fasquel and Nakatsuka in \cite{FKN}.
\begin{remark}
One can also ask about the relation between $H_{f}(\mathfrak{g})$ and  $H_{f_2}(H_{f_1}(\mathfrak{g}))$ when $f_2 \geq f_1$. In that case one expects that the two quantum Hamiltonian reduction differ by some free field algebra that depends on the nilpotent element. This was studied in the case of type $A$ in \cite{CDM} and in particular a connection between this type of iterated quantum Hamiltonian reductions and cohomologies of affine Laumon spaces was noticed. Our understanding of iterated reductions has been inspired by \cite{CDM}.
\end{remark}

In order to motivate our main construction, we define a slightly more restrictive notion of decomposition of a nilpotent. It has the property that in type $A$, the indecomposable nilpotents are exactly the hook-type nilpotents. 
\begin{definition} \label{def:indecomposable}
Let $\mathfrak{g}$ be a simple Lie algebra, and $f \in \mathfrak{g}$ a nilpotent element. We say that $f$ is decomposable if we can write $f = f_1 + f_2$, where $f_1, f_2$ are nilpotents with the following properties. 
\begin{enumerate} 
\item If we complete both $f_1$ and $f_2$ to $\mathfrak{sl}_2$ triples in $\mathfrak{g}$, and let $\mathfrak{g}^{\natural}(f_1)$ and $\mathfrak{g}^{\natural}(f_2)$ be the centralizers of these $\mathfrak{sl}_2$ triples in $\mathfrak{g}$, then 
$$f_2 \in \mathfrak{g}^{\natural}(f_1),\qquad f_1 \in \mathfrak{g}^{\natural}(f_2).$$
\item If $f_1$ further decomposes as $f_1 = f_{11} + f_{12}$ and this decomposition satisfies the above property, i.e., $f_{12} \in \mathfrak{g}^{\natural}(f_{11})$ and $f_{11} \in \mathfrak{g}^{\natural}(f_{12})$, then both $f_{11}$ and $f_{12}$ commute with $f_{2}$. 
\item Similarly, if $f_2$ further decomposes as $f_2 = f_{21} + f_{22}$ and $f_{22} \in \mathfrak{g}^{\natural}(f_{21})$ and $f_{21} \in \mathfrak{g}^{\natural}(f_{22})$, then both $f_{21}$ and $f_{22}$ commute with $f_{1}$.
\end{enumerate} 
If no such decomposition of $f$ exists, we say that $f$ is indecomposable. 
\end{definition}

According to this definition, for $n_1 > n_2$, $f_{n_1, n_2} = f_{n_1, 1^{n_2}} + f_{n_2, 1^{n_1}}$ is a decomposition. But $f_{n+1, n-1} = f_{n,n} + f_{2,1^{2n -2}}$ is not a decomposition because $f_{n,n}$ further decomposes as $f_{n, 1^n} + f_{n,1^n}$ but neither copy of $f_{n, 1^n}$ commutes with $f_{2, 1^{2n-2}}$. In fact, it is straightforward to check that the indecomposable nilpotents in type $A$ are exactly the hook-type nilpotents, and that the decomposition of any nilpotent in type $A$ as a sum of indecomposables is unique up to conjugacy. We now list the indecomposable nilpotents in the other classical Lie types.

	\subsection{Indecomposable nilpotents in type $C$}
	Recall that the nilpotent orbits in $\fr{sp}_{2N}$ are in bijection with partitions of $2N$, where the odd parts occurs with even multiplicity \cite[Theorem 5.1.3]{CMc}. The following nilpotents are indecomposable according to the above definition.
	\begin{enumerate}
		\item For $1\leq n \leq N$, let $f_{2n} \in \gs\gp_{2N}$ correspond to the partition $(2n,1^{2N-2n})$. This is a hook-type nilpotent as defined in \cite{CL4}, i.e., it is the principal nilpotent in $\gs\gp_{2n}\subseteq \gs\gp_{2N}$.
		
		\item For $1 \leq n < \frac{N}{2}$, let $f_{2n+1, 2n+1} \in \gs\gp_{2N}$ correspond to the partition $(2n+1, 2n+1, 2N-2n-2)$.
	\end{enumerate}

	\subsection{Indecomposable nilpotents in type $B$}
	Recall that nilpotent orbits in $\fr{so}_{2N+1}$ are in bijection with partitions of $2N+1$, where the even parts occurs with even multiplicity \cite[Theorem 5.1.2]{CMc}. The following nilpotents are indecomposable.
	\begin{enumerate}
		\item For $ 1 \leq n \leq N$, let $f_{2n+1} \in \gs\go_{2N+1}$ correspond to the partition $(2n+1, 1^{2N-2n})$ of $2N+1$. This is of hook-type, i.e., the principal nilpotent in $\gs\go_{2n+1}\subseteq \gs\go_{2N+1}$.
		
		\item For $ 1 \leq n \leq \frac{N}{2}$, let $f_{2n, 2n} \in \gs\go_{2N+1}$ correspond to the partition $(2n, 2n, 1^{2N-4n+1})$ of $2N+1$.
	\end{enumerate}

\subsection{Indecomposable nilpotents in type $D$}
Recall that the nilpotent orbits in $\fr{so}_{2N}$ are in bijection with partitions of $2n$, where the even parts occurs with even multiplicity, except that very even partitions (those with only even parts, each having even multiplicity) correspond to two orbits \cite[Theorem 5.1.4]{CMc}.
The following are indecomposable:
	\begin{enumerate}
		\item For $1 \leq n \leq N$, let $f_{2n-1} \in \gs\go_{2N}$ correspond to the partition $(2n-1, 1^{2N-2n+1})$ of $2N$. This is of hook-type, i.e., the principal nilpotent in $\gs\go_{2n-1}\subseteq \gs\go_{2N}$.
		
		\item For $1 \leq n \leq \frac{N}{2}$, let $f_{2n, 2n} \in \mathfrak{so}_{2N}$ correspond to the partition $(2n, 2n, 1^{2N-4n})$ of $2N$.
\end{enumerate}

	\subsection{Generating Types} We now consider the generating type of the corresponding $\cW$-algebras in the cases where $\gg^{\natural}$ is as small as possible.
\begin{enumerate}
		\item $\cW^k(\gs\gp_{2N}, f_{2N})$ is the principal $\cW$-algebra and has type $\cW(2,4,\dots, 2N)$.
		\item $\cW^k(\gs\gp_{2(2N+1)}, f_{2N+1, 2N+1})$ has type $\cW(1^3, 2, 3^3, 4, \dots, (2N-1)^3, 2N, (2N+1)^3)$.
		\item $\cW^k(\gs\go_{2N}, f_{2N-1})$ is the principal $\cW$-algebra and has type $\cW(2,4,\dots, 2N-2, N)$
		\item $\cW^k(\gs\go_{4N}, f_{2N, 2N})$ has type $\cW(1^3,2, 3^3, 4,\dots, (2N-1)^3, 2N)$.
		\item $\cW^k(\gs\go_{2N+1}, f_{2N+1})$ is the principal $\cW$-algebra  has type $\cW(2,4,\dots, 2N)$
		\item $\cW^k(\gs\go_{4N+1}, f_{2N, 2N})$ has type 
		$\cW\big(1^3,2, 3^3, 4,\dots, (2N-1)^3, 2N, \left( \frac{2N+1}{2}\right)^2\big)$.
\end{enumerate}
	
	We already know the existence of a universal $2$-parameter vertex algebra of type $\cW(2,4,6,\dots)$ that admits $\cW^k(\gs\gp_{2n+1})$ and $\cW^k(\gs\go_{2n})^{\mathbb{Z}_2}$ as $1$-parameter quotients. Examples (2) and (4) above suggest the existence of a universal vertex algebra which is freely generated of type \eqref{sp2starting} admitting all these algebras as $1$-parameter quotients. In fact, Example (6) has an action of $\mathbb{Z}_2$, and it is not difficult to see that the $\mathbb{Z}_2$-orbifold has this generating type. It can be seen easily that there is no family of decomposable nilpotents in types $B$, $C$, or $D$, whose unifying algebra is of the form \eqref{sp2starting} This suggests that the universal object of this type is genuinely new, and is not an extension of two commuting copies of $\cW^{\text{ev}}_{\infty}$. Moreover, the above analysis suggests that the quotients of this algebra are the remaining building blocks for $\cW$-algebras of classical types.
	
	In the next two sections, we enumerate $12$ infinite families of $1$-parameter vertex algebras with strong generating type $\cW(1^3, 2, 3^3, 4,\dots)$. In this notation, this strong generating set need not be minimal; in fact all our examples are strongly generated by a finite subset of these generators. First, in Section \ref{sect:YtypeC} we give $8$ $\mathbb{N} \times \mathbb{N}$ families which contain the $\cW$-algebras $\cW^k(\gs\go_{4n}, f_{2n,2n})$ and $\cW^k(\gs\gp_{2(2n+1)}, f_{2n+1, 2n+1})$ as special cases. These families are all either $\cW$-algebras or (orbifolds of) cosets of $\cW$-algebras, and the list is quite parallel to the orthosymplectic $Y$-algebras as given in \cite{CL4}. Therefore we call these the $Y$-algebras of type $C$. The remaining $4$ families arise in a different way as diagonal cosets, and have the property that the level $k$ of the affine subalgebra of $\gs\gp_2$ is a fixed integer or half integer, and also is simple.

	\section{$Y$-algebras of Type $C$} \label{sect:YtypeC}
	
	Here we define 8 families of $\cW$-(super)algebras that we need in a unified framework. 
	First, let $\gg$ be a simple, orthosymplectic Lie (super)algebra which contains a Lie sub(super)algebra $\fr{b}\oplus\fr{sp}_2\oplus \fr{a}$ of full rank. Moreover, we assume that $\gg$ has a decomposition as a $\fr{b}\oplus\fr{sp}_2\oplus \fr{a}$-module 
	\begin{equation}\label{g decomposition}
		\begin{split}
			\gg\cong&\, \fr{b}\oplus\fr{sp}_2\oplus \fr{a}\oplus 	\rho_{\gg}^{\tau}\otimes \mathbb{C}^3\otimes \mathbb{C}\oplus \rho_{\fr{b}}\otimes \mathbb{C}^2\otimes \rho_{\fr{a}},\\
		\end{split}
	\end{equation}
	where
		\[\rho^{\tau}_{\fr{g}}:=
		\begin{cases}
			\rho_{\omega_2},& \quad \fr{b} = \fr{sp}_{2n},\\
			\rho_{2\omega_1},& \quad \fr{b} = \fr{so}_{2n+1},
		\end{cases}\]
	with the following properties.
	\begin{enumerate}
		\item $\fr{b}$ is either $\fr{sp}_{2m}$ or $\fr{so}_{2m+1}$.
		\item $\fr{a}$ is either $\mathfrak{so}_{2n}$, $\mathfrak{so}_{2n+1}$, $\mathfrak{sp}_{2n}$, or $\mathfrak{osp}_{1|2n}$.
		\item $\rho_{\fr{a}}$ and $\rho_{\fr{b}}$ transform as the standard representations of $\fr{a}$ and $\fr{b}$ respectively, and have the same parity, which can be even or odd.
	\end{enumerate}

	Note that if $\fr{a} = \fr{osp}_{1|2n}$, $\rho_{\fr{a}}$ even means that $\rho_{\fr{a}}\cong \mathbb{C}^{2n|1}$ as a vector superspace, whereas $\rho_{\fr{a}}$ odd means that $\rho_{\fr{a}}\cong \mathbb{C}^{1|2n}$.
	If $\fr{g} = \fr{osp}_{m|2n}$ we use the following convention for its dual Coxeter number $h^{\vee}$.
	\begin{equation}
		h^{\vee} = 
		\begin{cases}
			m-2n-2 & \T{type } C\\
			n+1-\frac{1}{2}m & \T{type } B\\
		\end{cases},
		\quad \T{sdim}(\fr{osp}_{m|2n}) = \frac{(m-2n)(m-2n-1)}{2}.
	\end{equation}
	In this notation, type $B$ (respectively $C$) means that the subalgebra $\fr{b} \subseteq \fr{g}$ is of type $B$ (respectively $C$), and the bilinear form on $\fr{osp}_{m|2n}$ is normalized so that it coincides with the normalized Killing form on $\fr{b}$. The cases that we need are recorded in Table \ref{tab:Walgebras}.
	
	\begin{table}[h]
		\centering
		\caption{$\fr{sp}_2$-rectangular $\cW$-algebras with a tail.}
		\label{tab:Walgebras}
		\begin{tabular}{|l|l|l|l|l|l|}
			\hline
			Case & $\gg$& $\fr{a}$ & $\fr{b}$ & $\rho_{\fr{a}}\otimes\mathbb{C}^2\otimes\rho_{\fr{b}}$& $a$ \\
			\hline
			$CD$ & $\fr{so}_{2(2m)+2n}$ & $\fr{so}_{2n}$ & $\fr{sp}_{2m}$ & Even& $\psi-2n$\\
			$CB$ & $\fr{so}_{2(2m)+2n+1}$ & $\fr{so}_{2n+1}$ & $\fr{sp}_{2m}$ & Even& $\psi-2n-1$\\
			$CC$ & $\fr{osp}_{2(2m)|2n}$ & $\fr{sp}_{2n}$ & $\fr{sp}_{2m}$ & Odd &$-\frac{\psi}{2}-n$\\
			$CO$ & $\fr{osp}_{2(2m)+1|2n}$ & $\fr{osp}_{1|2n}$ & $\fr{sp}_{2m}$ & Odd&$-\frac{\psi}{2}-n + \frac{1}{2}$ \\
			$BD$ & $\fr{osp}_{2n|2(2m+1)}$ & $\fr{so}_{2n}$ & $\fr{so}_{2m+1}$ & Odd& $-2\psi-2n+4$\\
			$BB$ & $\fr{osp}_{2n+1|2(2m+1)}$ & $\fr{so}_{2n+1}$ & $\fr{so}_{2m+1}$ & Odd &$-2\psi-2n+3$\\
			$BC$ & $\fr{sp}_{2(2m+1)+2n}$ & $\fr{sp}_{2n}$ & $\fr{so}_{2m+1}$ & Even & $\psi-n-2$\\
			$BO$ & $\fr{osp}_{1|2(2m+1)+2n}$ & $\fr{osp}_{1|2n}$ & $\fr{so}_{2m+1}$ & Even &$\psi-n-\frac{3}{2}$\\
			\hline
		\end{tabular}
	\end{table}

	Let $f_{\fr{b}}$ be the nilpotent element which is principal in $\gb$ and trivial in $\fr{sp}_2\oplus \fr{a}$. 
	The corresponding $\cW$-algebras $\cW^{\ell}(\gg,f_{\fr{b}})$ will be called {\it $\fr{sp}_2$-rectangular} when $n = 0$ so that $\ga$ is trivial, and {\it $\fr{sp}_2$-rectangular with a tail} otherwise. The \lq\lq tail" refers to the affine subalgebra $V^a(\ga) \subseteq \cW^{\ell}(\gg,f_{\fr{b}})$, which has level $a$.

	Let $\rho_d$ denotes the $(d+1)$-dimensional representation of the $\fr{sl}_2$-triple $\{f,h,e\}$, and define two $\fr{sl}_2$-modules 
	\begin{equation}\label{evenodd}
		\begin{split}
			\T{Even}(m)= \bigoplus_{i=1}^m \rho_{4i-2},\quad \T{Odd}(m) =  \bigoplus_{i=1}^m \rho_{4i}.
		\end{split}
	\end{equation}
	Recall the decomposition (\ref{g decomposition}).
	Taking $f$ to be principal in $\fr{b}$ we have the following isomorphisms of $\fr{sl}_2$-modules.
	\begin{equation}
		\rho^{\fr{so}_{2n+1}}_{\omega_2} \cong \rho^{\fr{sp}_{2n}}_{2\omega_1}\cong  \T{Even}(n),\quad \rho^{\fr{so}_{2n+1}}_{2\omega_1} \cong \T{Odd}(n),\quad \rho^{\fr{sp}_{2n}}_{\omega_2} \cong\T{Odd}(n-1).\\
	\end{equation}
	To compute the central charge (\ref{central charge W algebra}) we have to evaluate the contribution of each term in the expression
	\[c=c_{\fr{g}}+c_{\T{dilaton}} + c_{\T{ghost}}.\]
	So we have to evaluate $c_{\T{dilaton}}$ and  $c_{\T{ghost}}$ from charged fermions arising from $\gg_{+}$.
	
	First, we evaluate the dilaton contribution. In types $B$ and $C$ this reduces to the sum of odd and even squares, respectively.
	\begin{equation*}
		\begin{split}
			c_{\T{dilaton}} =& -\ell\times\begin{cases}
				(m+2)(2m)(2m+1), \quad &\fr{b} = \fr{so}_{2m+1},\\
				(2m-1)(2m)(2m+1), \quad &\fr{b} = \fr{sp}_{2m}.\\
			\end{cases}\\
		\end{split}
	\end{equation*}
	Next, consider the decomposition $\gg=\bigoplus_{d}\rho_{d}$ into irreducible $\fr{sl}_2$-modules.
	Each $\rho_d$ gives rise to a field of conformal weight $\frac{d+1}{2}$ in $\cW^{\ell}(\gg,f_{\fr{b}})$.
	Using (\ref{charge Fermions}), the corresponding ghosts give rise to a central charge contribution 
	\begin{equation}\label{ghosts}
		c_d = -\frac{(d-1)(d^2-2d-1)}{2}.
	\end{equation}
	Using the decomposition (\ref{g decomposition}), we find that charged fermion contribution consists of three terms
	\begin{equation}\label{ghost}
		c_{\T{ghost}}=c_{\T{even}}+3c_{\T{odd}}+2\T{sdim}(\rho_{\fr{g}})c_{d_{\fr{b}}}.
	\end{equation}
	To compute $c_{\T{Even}}$ and $c_{\T{Odd}}$, we use (\ref{ghosts}) and (\ref{evenodd}) to obtain
	\[c_{\T{Even}}=\sum_{i=0}^n c_{4i+3}=6m^2-8m^4, \quad c_{\T{Odd}}=\sum_{i=0}^n c_{4i+1}=-2m(m+1)(4m(m+1)-1).\]
	The third term in (\ref{ghost}) is computed by applying formula (\ref{ghosts}).
	
	It follows from the above discussion that $\cW^{\ell}(\gg,f_{\fr{b}})$ is of type 
	\begin{equation}
		\begin{split}
			\cW\left(1^{3+\T{dim}(\fr{a})},2,3^3,4,\dots,(2m-1)^3,2m, (m+\frac{1}{2})^{2\T{dim}(\rho_{\fr{a}})}\right),\quad &\fr{b}= \fr{sp}_{2m},\\
			\cW\left(1^{3+\T{dim}(\fr{a})},2,3^3,4,\dots,2m,(2m+1)^3, \left(m+1\right)^{2\T{dim}(\rho_{\fr{a}})}\right),\quad &\fr{b}=\fr{so}_{2m+1}.
		\end{split}
	\end{equation}
	The affine subalgebra is $V^k(\fr{sp}_2)\otimes V^{a}(\fr{a})$ for some level $a$, recorded in Table \ref{tab:Walgebras}.
	We will always replace $\ell$ with the {critically shifted level} $\psi = \ell + h^{\vee}_{\fr{g}}$, where $h^{\vee}_{\fr{g}}$ is the dual Coxeter number of $\fr{g}$. We now describe the examples we need in greater detail.

	\subsection{Case CD}
	For $\gg =  \fr{so}_{4m+2n}$ with $m\geq 1$ and $n\geq 0$, we have an isomorphism of $\fr{sp}_{2m}\oplus \fr{sp}_2\oplus\fr{so}_{2n}$-modules
	\[\fr{so}_{4m+2n}\cong\fr{sp}_{2m}\oplus\fr{sp}_2\oplus\fr{so}_{2n} \oplus \rho_{\omega_2}\otimes \mathbb{C}^3\otimes\mathbb{C} \oplus \mathbb{C}^{2m}\otimes \mathbb{C}^2\otimes \mathbb{C}^{2n},\]
	and the critically shifted level is $\psi =\frac{k+2 m+n}{m}$.
	We define 
	\[\cW^{\psi}_{CD}(n,m) : = \cW^{\frac{k+2 m+n}{m} - h^{\vee}}(\fr{so}_{4m+2n},f_{\fr{sp}_{2m}}), \qquad h^{\vee} = 4 m + 2 n - 2,\] where $f_{\gs\gp_{2m}} = f_{2m, 2m}$ in the notation of Section \ref{sec:reduction}. The affine subalgebra is $V^{k}(\fr{sp}_2) \otimes V^{\psi-2n}(\fr{so}_{2n})$, and we define
	\[\cC^{\psi}_{CD}(n,m) : = \T{Com}(V^{\psi-2n}(\fr{so}_{2n}),\cW^{\psi}_{CD}(n,m))^{\mathbb{Z}_2}.\]
	The conformal element $L-L^{\fr{sp}_2} - L^{\fr{so}_{2n}}$ has central charge
	\begin{equation}\label{CD}
		c_{CD}=-\frac{k (2 k+1) (2 k m-k+2 m-n) (2 k m+k+4 m+n)}{(k+2) (k+n) (k+2 m+n)}.
	\end{equation}
	The free field limit of $\cW^{\psi}_{CD}(n,m)$ is
	\[\cO_{\T{ev}}(2n^2-n,2)\otimes \left( \bigotimes_{i=1}^{m}\cO_{\T{ev}}(1,4i)\otimes \bigotimes_{i=1}^{m}\cO_{\T{ev}}(3,4i-2) \right) \otimes \cS_{\T{ev}}(2n,2m+1).\]

	\begin{lemma}\label{lemma:CD}
		For $m\geq 1$ and $n\geq 0$, $\cC^{\psi}_{CD}(n,m)$ is of type 
		$\cW(1^3,2,3^3,4,\dots)$
		as a 1-parameter vertex algebra.
		Equivalently, this holds for generic values of $\psi$.
	\end{lemma}
	
	\begin{proof} By \cite[Lemma 4.2]{CL3}, $\cC^{\psi}_{CD}(n,m)$ has large level limit 
	$$ \left( \bigotimes_{i=1}^{m}\cO_{\T{ev}}(1,4i)\otimes \bigotimes_{i=1}^{m}\cO_{\T{ev}}(3,4i-2) \right) \otimes \cS_{\T{ev}}(2n,2m+1)^{\text{O}_{2n}}.$$ The first factor $\left( \bigotimes_{i=1}^{m}\cO_{\T{ev}}(1,4i)\otimes \bigotimes_{i=1}^{m}\cO_{\T{ev}}(3,4i-2) \right)$ has strong generating type $\cW(1^3,2,3^3,4,\dots, (2m-1)^{3}, 2m)$. Using classical invariant theory, the second factor $\cS_{\T{ev}}(2n,2m+1)^{\text{O}_{2n}}$ is easily seen to have an infinite strong generating set of type $\cW((2m+1)^3, 2m+2, (2m+3)^3, 2m+4,\dots)$, of which only finitely many are needed. The precise minimal strong generating type can be found using the methods of \cite{CL3,CL4} but we omit this since it is not needed in this paper. \end{proof}
	
	\subsection{Case CB}
	For $\gg =  \fr{so}_{4m+2n+1}$ with $m\geq 1$ and $n\geq 0$, we have an isomorphism of $\fr{sp}_{2m}\oplus \fr{sp}_2\oplus\fr{so}_{2n+1}$-modules
	\[\fr{so}_{4m+2n+1}\cong\fr{sp}_{2m}\oplus\fr{sp}_2\oplus\fr{so}_{2n+1} \oplus \rho_{\omega_2}\otimes \mathbb{C}^3\otimes\mathbb{C}  \oplus \mathbb{C}^{2m}\otimes \mathbb{C}^2\otimes \mathbb{C}^{2n+1},\]
	and the critically shifted level is $\psi =\frac{2 k+4 m+2 n+1}{2 m}$.
	We define
	\[\cW^{\psi}_{CB}(n,m) : = \cW^{\frac{2 k+4 m+2 n+1}{2 m} - h^{\vee}}(\fr{so}_{4m+2n+1},f_{\fr{sp}_{2m}}),\qquad h^{\vee} = 4m+2n-1,\] where $f_{\gs\gp_{2m}} = f_{2m, 2m}$. The affine subalgebra is $V^{k}(\fr{sp}_2) \otimes V^{\psi-2n-1}(\fr{so}_{2n+1})$, and we define
	\[\cC^{\psi}_{CB}(n,m) : = \T{Com}(V^{\psi-2n-1}(\fr{so}_{2n+1}),\cW^{\psi}_{CB}(n,m))^{\mathbb{Z}_2}.\]
	The conformal element $L-L^{\fr{sp}_2} - L^{\fr{so}_{2n+1}}$ has central charge
	\begin{equation}\label{CB}
		c_{CB}=-\frac{k (2 k+1) (4 k m-2 k+4 m-2 n-1) (4 k m+2 k+8 m+2 n+1)}{(k+2) (2 k+2 n+1)
			(2 k+4 m+2 n+1)}.
	\end{equation}
	The free field limit of $\cW^{\psi}_{CB}(n,m)$ is
	\[\cO_{\T{ev}}(2n^2+n,2)\otimes \left( \bigotimes_{i=1}^{m}\cO_{\T{ev}}(1,4i)\otimes \bigotimes_{i=1}^{m}\cO_{\T{ev}}(3,4i-2) \right) \otimes \cS_{\T{ev}}(2n+1,2m+1).\] As above, $\cC^{\psi}_{CB}(n,m)$ has large level limit
	$\left( \bigotimes_{i=1}^{m}\cO_{\T{ev}}(1,4i)\otimes \bigotimes_{i=1}^{m}\cO_{\T{ev}}(3,4i-2) \right) \otimes \cS_{\T{ev}}(2n+1,2m+1)^{\text{O}_{2n+1}}$, and $\cS_{\T{ev}}(2n+1,2m+1)^{\text{O}_{2n+1}}$ has strong generating type $\cW((2m+1)^3, 2m+2, (2m+3)^3, 2m+4,\dots)$, so we obtain

	\begin{lemma}\label{lemma:CB}
		For For $m\geq 1$ and $n\geq 0$, $\cC^{\psi}_{CB}(n,m)$ is of type 
		$\cW(1^3,2,3^3,4,\dots)$
		as a 1-parameter vertex algebra.
		Equivalently, this holds for generic values of $\psi$.
	\end{lemma}
	
	\subsection{Case CC}
	For $\gg = \fr{osp}_{4m|2n}$ with $n\geq 1$ and $m\geq 1$, we have an isomorphism of $\fr{sp}_{2m}\oplus \fr{sp}_2\oplus\fr{sp}_{2n}$-modules
	\[\fr{osp}_{4m|2n}\cong\fr{sp}_{2m}\oplus\fr{sp}_2\oplus\fr{sp}_{2n} \oplus \rho_{\omega_2}\otimes \mathbb{C}^3\otimes\mathbb{C}  \oplus \mathbb{C}^{2m}\otimes \mathbb{C}^2\otimes \mathbb{C}^{0|2n},\]
	and the critically shifted level is $\psi = \frac{k+2 m-n}{m}$. 
	We define 
	\[\cW^{\psi}_{CC}(n,m) : = \cW^{\frac{k+2 m-n}{m}-h^{\vee}}(\fr{osp}_{4m|2n},f_{\fr{sp}_{2m}}), \qquad h^{\vee} = 4 m - 2 n - 2,\] where $f_{\gs\gp_{2m}}$ coincides with the nilpotent $f_{2m, 2m}$ in the Lie subalgebra $\gs\go_{4m}$. The affine subalgebra is $V^{k}(\fr{sp}_2)\otimes V^{-\frac{\psi}{2}-n}(\fr{sp}_{2n})$, and we define
	\[\cC^{\psi}_{CC}(n,m) : = \T{Com}(V^{-\frac{\psi}{2}-n}(\fr{sp}_{2n}),\cW^{\psi}_{CC}(n,m)),\]
	The conformal element $L-L^{\fr{sp}_2} - L^{\fr{sp}_{2n}}$ has central charge
	\begin{equation}\label{CC}
		c_{CC}=-\frac{k (2 k+1) (2 k m+k+4 m-n) (2 k m-k+2 m+n)}{(k+2) (k-n) (k+2 m-n)}.
	\end{equation}
	The free field limit of $\cW^{\psi}_{CC}(n,m)$ is
	\[\cO_{\T{ev}}(2n^2+n,2)\otimes \left( \bigotimes_{i=1}^{m}\cO_{\T{ev}}(1,4i)\otimes \bigotimes_{i=1}^{m}\cO_{\T{ev}}(3,4i-2) \right) \otimes \cO_{\T{odd}}(4n,2m+1).\]
As above, $\cC^{\psi}_{CC}(n,m)$ has large level limit
	$\left(\bigotimes_{i=1}^{m}\cO_{\T{ev}}(1,4i)\otimes \bigotimes_{i=1}^{m}\cO_{\T{ev}}(3,4i-2) \right) \otimes \cO_{\T{odd}}(4n,2m+1)^{\text{Sp}_{2n}}$, and $\cO_{\T{odd}}(4n,2m+1)^{\text{Sp}_{2n}}$ has strong generating type $\cW((2m+1)^3, 2m+2, (2m+3)^3, 2m+4,\dots)$, so we obtain

	\begin{lemma}\label{lemma:CC}
		For $n\geq 1$ and $m\geq 1$, $\cC^{\psi}_{CC}(n,m)$ is of type 
		$\cW(1^3,2,3^3,4,\dots)$
		as a 1-parameter vertex algebra.
		Equivalently, this holds for generic values of $\psi$.
	\end{lemma}
	
	\subsection{Case CO}
	For $\gg = \fr{osp}_{4m+1|2n}$ with $n\geq 1$ and $m\geq 1$, 
	we have an isomorphism of $\fr{sp}_{2m}\oplus \fr{sp}_2\oplus\fr{osp}_{1|2n}$-modules
	\[\fr{osp}_{4m+1|2n}\cong\fr{sp}_{2m}\oplus\fr{sp}_2\oplus\fr{osp}_{1|2n} \oplus \rho_{\omega_2}\otimes \mathbb{C}^3\otimes\mathbb{C}  \oplus \mathbb{C}^{2m}\otimes \mathbb{C}^2\otimes \mathbb{C}^{1|2n},\]
	and the critically shifted level is $\psi = \frac{2 k+4 m-2 n+1}{2 m}$. 
	We define 
	\[\cW^{\psi}_{CO}(n,m) : = \cW^{ \frac{2 k+4 m-2 n+1}{2 m} - h^{\vee}}(\fr{osp}_{4m+1|2n},f_{\fr{sp}_{2m}})), \qquad h^{\vee} = 4m-2n-1,\] where $f_{\gs\gp_{2m}}$ coincides with the nilpotent $f_{2m, 2m}$ in the Lie subalgebra $\gs\go_{4m+1}$. The affine subalgebra is $V^k(\fr{sp}_2)\otimes V^{-\frac{\psi}{2}-n+\frac{1}{2}}(\fr{osp}_{1|2n})$, and we define
	\[\cC^{\psi}_{CO}(n,m) : = \T{Com}(V^{-\frac{\psi}{2}-n+\frac{1}{2}}(\fr{osp}_{1|2n}),\cW^{\psi}_{CO}(n,m))^{\mathbb{Z}_2},\]
	The conformal element $L-L^{\fr{sp}_2} - L^{\fr{osp}_{1|2n}}$ has central charge
	\begin{equation}\label{CO}
		c_{CO}= - \frac{k (2 k+1) (4 k m+2 k+8 m-2 n+1) (4 k m-2 k+4 m+2 n-1)}{(k+2) (2 k-2 n+1)(2 k+4 m-2 n+1)}.
	\end{equation}
	The free field limit of $\cW^{\psi}_{CO}(n,m)$ is
	\begin{equation*}
	\begin{split}
		\cO_{\T{ev}}(2n^2+n,2)\otimes \cS_{\T{odd}}(n,2) &\otimes \left( \bigotimes_{i=1}^{m}\cO_{\T{ev}}(1,4i)\otimes \bigotimes_{i=1}^{m}\cO_{\T{ev}}(3,4i-2) \right)\\
		&\otimes  \cO_{\T{odd}}(4n,2m+1)\otimes \cS_{\T{ev}}(1,2m+1).\
	\end{split}
	\end{equation*}
	As above, $\cC^{\psi}_{CO}(n,m)$ has large level limit
	$$ \big(\bigotimes_{i=1}^{m}\cO_{\T{ev}}(1,4i)\otimes \bigotimes_{i=1}^{m}\cO_{\T{ev}}(3,4i-2) \big) \otimes  \big(\cO_{\T{odd}}(4n,2m+1)\otimes \cS_{\T{ev}}(1,2m+1)\big)^{\text{Osp}_{1|2n}},$$ and $\big(\cO_{\T{odd}}(4n,2m+1)\otimes \cS_{\T{ev}}(1,2m+1)\big)^{\text{Osp}_{1|2n}},$ has strong generating type $\cW((2m+1)^3, 2m+2, (2m+3)^3, 2m+4,\dots)$, so we obtain
	
	\begin{lemma}\label{lemma:CO}
		For $n\geq 1$ and $m\geq 1$, $\cC^{\psi}_{CO}(n,m)$ is of type 
		$\cW(1^3,2,3^3,4,\dots)$
		as a 1-parameter vertex algebra.
		Equivalently, this holds for generic values of $\psi$.
	\end{lemma}

	\subsection{Case BD}
	For $\gg =\fr{osp}_{2n|2(2m+1)}$ with $n\geq 1$ and $m\geq 0$,
	we have an isomorphism of $\fr{so}_{2m+1}\oplus \fr{sp}_2\oplus\fr{so}_{2n}$-modules
	\[\fr{osp}_{2n|2(2m+1)}\cong\fr{so}_{2m+1}\oplus\fr{sp}_2\oplus\fr{so}_{2n} \oplus \rho_{2\omega_1}\otimes \mathbb{C}^3 \otimes\mathbb{C} \oplus \mathbb{C}^{2m+1}\otimes \mathbb{C}^2\otimes \mathbb{C}^{0|2n},\]
	and the critically shifted level is $\psi = \frac{k+2 m-n+2}{2 m+1}$. 
	We define 
	\[\cW^{\psi}_{BD}(n,m) : = \cW^{\frac{k+2 m-n+2}{2 m+1}-h^{\vee}}(\fr{osp}_{2n|2(2m+1)},f_{\fr{so}_{2m+1}}), \qquad h^{\vee} = 2m-n+2,\] where $f_{\gs\go_{2m+1}}$ coincides with the nilpotent $f_{2m+1, 2m+1}$ in the Lie subalgebra $\gs\gp_{2(2m+1)}$. The affine subalgebra is $V^k(\fr{sp}_2)\otimes V^{-2\psi-2n+4}(\fr{so}_{2n})$. Note that in the extremal case $m=0$, $\cW^{\psi}_{BD}(n,m) = V^{-2\psi-2n+4} (\go\gs\gp_{2n|2})$.
	We define
	\[\cC^{\psi}_{BD}(n,m) : = \T{Com}(V^{-2\psi-2n+4}(\fr{so}_{2n}), \cW^{\psi}_{BD}(n,m))^{\mathbb{Z}_2}.\]
	The conformal element $L-L^{\fr{sp}_2} - L^{\fr{so}_{2n}}$ has central charge
	\begin{equation}\label{BD}
		c_{BD}=-\frac{k (2 k+1) (2 k m+2 k+4 m-n+3) (2 k m+2 m+n)}{(k+2) (k-n+1) (k+2 m-n+2)}.
	\end{equation}
	The free field limit of $\cW^{\psi}_{BD}(n,m)$ is
	\[\cO_{\T{ev}}(2n^2-n,2) \otimes \left( \bigotimes_{i=1}^{m}\cO_{\T{ev}}(1,4i)\otimes \bigotimes_{i=1}^{m+1}\cO_{\T{ev}}(3,4i-2) \right)\otimes \cS_{\T{odd}}(2n,2m+2).\]
Then $\cC^{\psi}_{BD}(n,m)$ has large level limit 
$\left( \bigotimes_{i=1}^{m}\cO_{\T{ev}}(1,4i)\otimes \bigotimes_{i=1}^{m+1}\cO_{\T{ev}}(3,4i-2) \right)\otimes \cS_{\T{odd}}(2n,2m+2)^{\text{O}_{2n}}$. The first factor $\left( \bigotimes_{i=1}^{m}\cO_{\T{ev}}(1,4i)\otimes \bigotimes_{i=1}^{m+1}\cO_{\T{ev}}(3,4i-2) \right)$ has strong generating type $\cW(1^3,2,3^3,4,\dots, 2m, (2m+1)^3)$. Using classical invariant theory, it is easy to check that $\cS_{\T{odd}}(2n,2m+2)^{\text{O}_{2n}}$ has strong generating type $\cW(2m+2, (2m+3)^3, 2m+4, (2m+5)^3,\dots)$, so we obtain
	
	\begin{lemma}\label{lemma:BD}
		For $n\geq 1$ and $m\geq 0$, $\cC^{\psi}_{BD}(n,m)$ is of type 
		$\cW(1^3,2,3^3,4,\dots)$
		as a 1-parameter vertex algebra.
		Equivalently, this holds for generic values of $\psi$.
	\end{lemma}
	
	\subsection{Case BB}
	For $\gg =\fr{osp}_{2n+1|2(2m+1)}$ with $n\geq 0$ and $m\geq 0$, 
	we have an isomorphism of $\fr{so}_{2m+1}\oplus \fr{sp}_2\oplus\fr{so}_{2n+1}$-modules
	\[\fr{osp}_{2n+1|2(2m+1)}\cong\fr{so}_{2m+1}\oplus\fr{sp}_2\oplus\fr{so}_{2n+1} \oplus \rho_{2\omega_1}\otimes \mathbb{C}^3\otimes\mathbb{C}  \oplus \mathbb{C}^{2m+1}\otimes \mathbb{C}^2\otimes \mathbb{C}^{0|2n+1},\]
	and the critically shifted level is $\psi = \frac{2 k+4 m-2 n+3}{2 (2 m+1)}$. 
	We define 
	\[\cW^{\psi}_{BB}(n,m) : = \cW^{ \frac{2 k+4 m-2 n+3}{2 (2 m+1)}-h^{\vee}}(\fr{osp}_{2n+1|2(2m+1)},f_{\fr{so}_{2m+1}}), \qquad h^{\vee} = 2m-n+\frac{3}{2},\] where $f_{\gs\go_{2m+1}}$ coincides with the nilpotent $f_{2m+1, 2m+1}$ in the Lie subalgebra $\gs\gp_{2(2m+1)}$. The affine subalgebra is $V^k(\fr{sp}_2)\otimes V^{-2\psi-2n+3}(\fr{so}_{2n+1})$. Again, in the case $m=0$, $\cW^{\psi}_{BB}(n,0) = V^{-2\psi-2n+3}(\go\gs\gp_{2n+1|2})$. We define
\[\cC^{\psi}_{BB}(n,m) : = \T{Com}(V^{-2\psi-2n+3}(\fr{so}_{2n+1}),  \cW^{\psi}_{BB}(n,m))^{\mathbb{Z}_2}.\]
	The conformal element $L-L^{\fr{sp}_2} - L^{\fr{so}_{2n+1}}$ has central charge
	\begin{equation}\label{BB}
		c_{BB}= - \frac{k (2 k+1) (4 k m+4 k+8 m-2 n+5) (4 k m+4 m+2 n+1)}{(k+2) (2 k-2 n+1) (2
			k+4 m-2 n+3)}
	\end{equation}
	The free field limit of $\cW^{\psi}_{BB}(n,m)$ is
	\[\cO_{\T{ev}}(2n^2+n,2) \otimes \left( \bigotimes_{i=1}^{m}\cO_{\T{ev}}(1,4i)\otimes \bigotimes_{i=1}^{m+1}\cO_{\T{ev}}(3,4i-2) \right)\otimes \cS_{\T{odd}}(2n+1,2m+2).\]
As above, $\cC^{\psi}_{BB}(n,m)$ has large level limit $\left( \bigotimes_{i=1}^{m}\cO_{\T{ev}}(1,4i)\otimes \bigotimes_{i=1}^{m+1}\cO_{\T{ev}}(3,4i-2) \right)\otimes \cS_{\T{odd}}(2n+1,2m+2)^{\text{O}_{2n+1}}$, and $\cS_{\T{odd}}(2n+1,2m+2)^{\text{O}_{2n+1}}$ has strong generating type $\cW(2m+2, (2m+3)^3, 2m+4, (2m+5)^3,\dots)$, so we obtain
\begin{lemma}\label{lemma:BB}
		For $n\geq 0$ and $m\geq 0$, $\cC^{\psi}_{BB}(n,m)$ is of type 
		$\cW(1^3,2,3^3,4,\dots)$
		as a 1-parameter vertex algebra.
		Equivalently, this holds for generic values of $\psi$.
	\end{lemma}

	\subsection{Case BC}
	For $\gg =\fr{sp}_{2(2m+1)+2n}$ with $n\geq 0$ and $m\geq 0$,
	we have an isomorphism of $\fr{so}_{2m+1}\oplus \fr{sp}_2\oplus\fr{sp}_{2n}$-modules
	\[\fr{sp}_{2(2m+1)+2n}\cong\fr{so}_{2m+1}\oplus\fr{sp}_2\oplus\fr{sp}_{2n} \oplus \rho_{2\omega_1}\otimes \mathbb{C}^3\otimes\mathbb{C}  \oplus \mathbb{C}^{2m+1}\otimes \mathbb{C}^2\otimes \mathbb{C}^{2n},\]
	and the critically shifted level is $\psi = \frac{k+2 m+n+2}{2 m+1}$. 
	We define 
	\[\cW^{\psi}_{BC}(n,m) : = \cW^{ \frac{k+2 m+n+2}{2 m+1}-h^{\vee}}(\fr{sp}_{2(2m+1)+2n},f_{\fr{so}_{2m+1}}),\qquad h^{\vee} = 2m+n+2,\] where $f_{\gs\go_{2m+1}}$ coincides with the nilpotent $f_{2m+1, 2m+1}$. The affine subalgebra is $V^k(\fr{sp}_2)\otimes V^{\psi-n-2}(\fr{sp}_{2n})$, and we define
	\[\cC^{\psi}_{BC}(n,m) : = \T{Com}(V^{\psi-n-2}(\fr{sp}_{2n}), \cW^{\psi}_{BC}(n,m)).\]
	The conformal element $L-L^{\fr{sp}_2} - L^{\fr{sp}_{2n}}$ has central charge
	\begin{equation}\label{BC}
		c_{BC}=-\frac{k (2 k+1) (2 k m+2 m-n) (2 k m+2 k+4 m+n+3)}{(k+2) (k+n+1) (k+2 m+n+2)}.
	\end{equation}
	The free field limit of $\cW^{\psi}_{BC}(n,m)$ is
	\[\cO_{\T{ev}}(2n^2+n,2)\otimes \cS_{\T{odd}}(n,2) \otimes \left( \bigotimes_{i=1}^{m}\cO_{\T{ev}}(1,4i)\otimes \bigotimes_{i=1}^{m+1}\cO_{\T{ev}}(3,4i-2) \right) \otimes \cO_{\T{ev}}(4n,2m+2).\]
	As above, $\cC^{\psi}_{BC}(n,m)$ has large level limit $\left( \bigotimes_{i=1}^{m}\cO_{\T{ev}}(1,4i)\otimes \bigotimes_{i=1}^{m+1}\cO_{\T{ev}}(3,4i-2) \right) \otimes \cO_{\T{ev}}(4n,2m+2)^{\text{Sp}_{2n}}$, and $ \cO_{\T{ev}}(4n,2m+2)^{\text{Sp}_{2n}}$ has strong generating type $\cW(2m+2, (2m+3)^3, 2m+4, (2m+5)^3,\dots)$, so we obtain
	\begin{lemma}\label{lemma:BC}
		For $n\geq 0$ and $m\geq 0$, $C^{\psi}_{BC}(n,m)$ is of type 
		$\cW(1^3,2,3^3,4,\dots)$
		as a 1-parameter vertex algebra.
		Equivalently, this holds for generic values of $\psi$.
	\end{lemma}
	
	\subsection{Case BO}
	For $\gg =\fr{osp}_{1|2(2m+1)+2n}$ with $n\geq 0$ and $m\geq 0$, 
	we have an isomorphism of $\fr{so}_{2m+1}\oplus \fr{sp}_2\oplus\fr{osp}_{1|2n}$-modules
	\[\fr{osp}_{1|2(2m+1)+2n}\cong\fr{so}_{2m+1}\oplus\fr{sp}_2\oplus\fr{osp}_{1|2n} \oplus \rho_{2\omega_1}\otimes \mathbb{C}^3 \otimes \mathbb{C} \oplus \mathbb{C}^{2m+1}\otimes \mathbb{C}^2\otimes \mathbb{C}^{2n|1},\]
	and the critically shifted level is $\psi = \frac{2 k+4 m+2 n+3}{2 (2 m+1)}$. 
	We define 
	\[\cW^{\psi}_{BO}(n,m) : = \cW^{\frac{2 k+4 m+2 n+3}{2 (2 m+1)} - h^{\vee}}(\fr{osp}_{1|2(2m+1)+2n},f_{\fr{so}_{2m+1}}),\qquad h^{\vee} =2m+n+\frac{3}{2}, \] where $f_{\gs\go_{2m+1}}$ coincides with the nilpotent $f_{2m+1, 2m+1}$ in the Lie subalgebra $\gs\gp_{2(2m+1)+2n}$. The affine subalgebra is $V^k(\fr{sp}_2)\otimes V^{\psi-n-\frac{3}{2}}(\fr{osp}_{1|2n})$, and we define
\[\cC^{\psi}_{BO}(n,m) : = \T{Com}(V^{\psi-n-\frac{3}{2}}(\fr{osp}_{1|2n}),  \cW^{\psi}_{BO}(n,m))^{\mathbb{Z}_2}.\]
	The conformal element $L-L^{\fr{sp}_2} - L^{\fr{osp}_{1|2n}}$ has central charge
	\begin{equation}\label{BO}
		c_{BO}=-\frac{k (2 k+1) (4 k m+4 m-2 n+1) (4 k m+4 k+8 m+2 n+5)}{(k+2) (2 k+2 n+1) (2
			k+4 m+2 n+3)}
	\end{equation}
	The free field limit of $\cW^{\psi}_{BO}(n,m)$ is
	\begin{equation*}
		\begin{split}
		\cO_{\T{ev}}(2n^2+n,2)\otimes \cS_{\T{odd}}(n,2) &\otimes \left( \bigotimes_{i=1}^{m}\cO_{\T{ev}}(1,4i)\otimes \bigotimes_{i=1}^{m+1}\cO_{\T{ev}}(3,4i-2) \right)\\
		&  \otimes \cO_{\T{ev}}(2n,2m+2)\otimes \cS_{\T{odd}}(1,2m+2).
		\end{split}
	\end{equation*}
	As above, $\cC^{\psi}_{BO}(n,m)$ has large level limit 
	$$\big( \bigotimes_{i=1}^{m}\cO_{\T{ev}}(1,4i)\otimes \bigotimes_{i=1}^{m+1}\cO_{\T{ev}}(3,4i-2) \big)
\otimes \big(\cO_{\T{ev}}(2n,2m+2)\otimes \cS_{\T{odd}}(1,2m+2)\big)^{\text{Osp}_{1|2n}},$$
and $\big(\cO_{\T{ev}}(2n,2m+2)\otimes \cS_{\T{odd}}(1,2m+2)\big)^{\text{Osp}_{1|2n}}$ has strong generating type $\cW(2m+2, (2m+3)^3, 2m+4, (2m+5)^3,\dots)$, so we obtain

	\begin{lemma}\label{lemma:BO}
		For $n\geq 0$ and $m\geq 0$, $C^{\psi}_{BO}(n,m)$ is of type 
		$\cW(1^3,2,3^3,4,\dots)$
		as a 1-parameter vertex algebra.
		Equivalently, this holds for generic values of $\psi$.
	\end{lemma}
	
	\begin{remark}
		Let $X$ be either $B$ or $C$.
		Then we have the following relations among the central charges.
		\[c_{XD}(n,m)=c_{XC}(-n,m),\quad c_{XB}(n,m)=c_{XO}(-n,m).\]
		This suggests a heuristic that type $\fr{so}_{2n+1}$ is the negative of $\fr{osp}_{1|2n}$, and $\fr{so}_{2n}$ is the negative of $\fr{sp}_{2n}$. In fact, a similar feature continues to hold for the diagonal cosets introduced in Section \ref{sect:Diagonal}.
	\end{remark}

	A consequence of Theorem \ref{thm:largelevellimit} and Lemmas \ref{lemma:CD}-\ref{lemma:BO} is the following.
	\begin{proposition}\label{extension features}
		Let $X= B$ or $C$ and let $Y = B$, $C$, $D$, or $O$.
		\begin{enumerate}
			\item $\cC^{\psi}_{XY}(n,m)$ is simple for generic values of $\psi$.
			\item For $\cW^{\psi}_{CY}(n,m)$, without loss of generality we may replace the strong generating field in each even weight $2,4,\dots, 2M$, and the three strong generating fields in each odd weight $1,3,\dots, 2M-1$, with elements of the same weight in the coset $\cC^{\psi}_{CY}(n,m)$. Similarly, for $\cW^{\psi}_{BY}(n,m)$, we may replace the strong generating field in each even weight $2,4,\dots, 2M$, and the three strong generating fields in each odd weight $1,3,\dots, 2M+1$, with elements of the same weight in the coset $\cC^{\psi}_{BY}(n,m)$.

			\item Let $U\cong \mathbb{C}^2\otimes \rho_{\fr{a}}$, where $\rho_{\fr{a}}$ is the standard representation of $\fr{a}$. It is spanned by fields $P^{1,j},P^{-1,j}$, where $j$ runs over a basis of $\rho_{\fr{a}}$.
			Then $U$ has a supersymmetric bilinear form 
			\[\langle \ ,\ \rangle:  U \to \mathbb{C}, \quad \begin{cases}
				\langle a,b\rangle = a_{(2m)}b,&\quad X=C,\\
				\langle a,b\rangle = a_{(2m+1)}b & \quad X=B.
			\end{cases}  \]
			This form is nondegenerate and coincides with the standard pairing on $\mathbb{C}^2\otimes \rho_{\fr{a}}$.
			Hence, without loss of generality, we may normalize the fields in $U$ as
			\begin{equation}\label{normalization}
				P^{\mu,i}(z)P^{\nu,j}(w)\sim 
				\begin{cases}
					\delta_{i,j}\delta_{\mu+\nu,0}(z-w)^{-2m-1}+\dotsb,&\quad X=C,\\
					\delta_{i,j}\delta_{\mu+\nu,0}(z-w)^{-2m-2}+\dotsb,&\quad X=B.\\
				\end{cases}
			\end{equation}
			Here, the remaining terms lie $V^{a}(\fr{a}) \otimes \cC^{\psi}_{XY}(n,m)$.
		\end{enumerate}
	\end{proposition}

\section{Families with $\gs\gp_2$-level Constant}\label{sect:Diagonal}
Unlike $\cW_{\infty}$ and $\cW^{\text{ev}}_{\infty}$ where the $Y$-algebras are expected to account for all the simple, strongly finitely generated $1$-parameter quotients, $\cW^{\gs\gp_2}_{\infty}$ admits at least $4$ more infinite families of such $1$-parameter quotients. These algebras all contain the simple affine vertex algebra $L_k(\gs\gp_2)$ for a fixed $k$.

	\subsection{Cases B and D}
	For $n\geq 2$, let $\cS(n)$ denote the $\beta\gamma$-system of rank $n$. There is a standard homomorphism
	\begin{equation} \varphi:V^{-2}(\fr{so}_n) \hookrightarrow \cS(n),\end{equation} given in terms of the basis \eqref{soStandard} by 
		\begin{equation} E_{i,j} \mapsto :\!\beta^i\gamma^j\!:- :\!\beta^j\gamma^i \!:, \quad 1\leq i<j\leq n.
	\end{equation}
	The commutant of $V^{-2}(\fr{so}_n)$ inside $\cS(n)$ is $L_{-\frac{n}{2}}(\fr{sp}_2)$ \cite[Thm. 5.1]{LinSS}, and is generated by the following $\T{O}_n$-invariants

\begin{equation}
X =\sum_{i=1}^n :\!\beta^i\beta^i\!:, \qquad Y = \sum_{i=1}^n :\!\gamma^i\gamma^i\!:,\qquad H = \sum_{i=1}^n :\!\beta^i\gamma^i\!:.
\end{equation}
	The weight $\frac{1}{2}$ space is spanned by $\{\beta^i|i=1,\dots,n\}\cup\{\gamma^i|i=1,\dots,n\}$, and transforms as $\mathbb{C}^2\otimes \mathbb{C}^n$ under $\fr{sp}_2\oplus\fr{so}_n$ .
	We have a diagonal action
	\begin{equation}
		\begin{split}
			V^{\ell}(\fr{so}_n) \hookrightarrow V^{\ell+2}(\fr{so}_n)\otimes \cS(n),\quad E_{i,j} \mapsto E_{i,j}\otimes {\bf 1}+{\bf 1} \otimes \varphi(E_{i,j}).
		\end{split}
	\end{equation}
	Define
	\begin{equation}\label{diagonalBD}
		\cC^{\ell}\left(-\frac{n}{2}\right) = \T{Com}(V^{\ell}(\fr{so}_n), V^{\ell+2}(\fr{so}_n)\otimes \cS(n))^{\mathbb{Z}_2}.
	\end{equation}
	The free field limit of $\cC^{\ell}\left(-\frac{n}{2}\right) $ is the invariant algebra $\cS(n)^{\T{O}_n}$.
	By Weyl's first fundamental theorem for standard representation of $\T{O}_{n}$ \cite{W}, $\cS(n)^{\T{O}_n}$ has a strong generating set consisting of the $\T{O}_n$-invariants
	\begin{equation}
		\begin{split}
			X^{p,q} =&\sum_{i=1}^n\!:\partial^{p} \beta^i\partial^{q}\beta^{i}\!: \quad p\geq q\geq 0, \\
			Y^{p,q} =&\sum_{i=1}^n:\!\partial^{p} \gamma^i\partial^{q}\gamma^{i}\!:\quad p\geq q\geq 0,\\
			H^{p,q}  =&\sum_{i=1}^n:\!\partial^{p} \beta^i\partial ^{q}\gamma^{i}\!:,\quad p,q\geq 0.
		\end{split}
	\end{equation}

	\begin{lemma}\label{lemma:orbifoldB and D} In the case $n=1$, we have $\cS(1)^{\mathbb{Z}_2} \cong L_{-\frac{1}{2}}(\gs\gp_2)$. For $n\geq 2$, $\cS(n)^{\text{O}_n}$ is a simple vertex algebra of type $\cW(1^3,2,3^3,4,\dots)$ containing $L_{-\frac{n}{2}}(\gs\gp_2)$.
	\end{lemma}
	
	\begin{proof} The simplicity follows from \cite{DLM} for all $n\geq 1$. In the case $n=1$, it is easy to find decoupling relations expressing the strong generators $X^{p,q}$, $Y^{p,q}$, and $H^{p,q}$ for all $p,q$, as normally ordered polynomials in $X^{0,0}$, $Y^{0,0}$, and $H^{0,0}$. Therefore $X^{0,0}$, $Y^{0,0}$, and $H^{0,0}$ are strong generators and they clearly generate $L_{-\frac{1}{2}}(\gs\gp_2)$. For $n\geq 2$, using the fact that $\partial A^{p,q} = A^{p+1,q} + A^{p,q+1}$ for $A = X,Y,H$, we see that the sets $\{\partial^i X^{2j,0}, \partial^i Y^{2j,0}, \partial^i H^{j,0}|\ i,j \geq 0\}$ and $\{X^{p,q}, Y^{p,q}, H^{r,s}|\ p \geq q \geq 0,\ r,s \geq 0\}$ span the same vector space. Therefore $\{\partial^i X^{2j,0}, \partial^i Y^{2j,0}, \partial^i H^{j,0}|\ i,j \geq 0\}$ is a strong generating set, and it is of type $\cW(1^3,2,3^3,4,\dots)$. By \cite[Theorem 6.6]{Lin2}, only finitely many of these strong generators are needed, but for our purposes we do not need to compute the minimal strong generating set.
	\end{proof}

	\begin{lemma}\label{lemma:B and D}
For $n\in\mathbb{Z}_{\geq 2}$, $\cC^{\ell}\left(-\frac{n}{2}\right)$ is a simple, $1$-parameter vertex algebra of type $\cW(1^3,2,3^3,4,\dots)$ containing $L_{-\frac{n}{2}}(\gs\gp_2)$.
	\end{lemma}
	
	\begin{proof} This is immediate from \cite[Lemma 3.2 and Thm. 6.10]{CL2}. 
	\end{proof} 
	
\begin{remark} \label{rem:casek=-1/2} In Section \ref{sect:1paramquot}, we will see that for all $n\geq 2$ and $n\neq 2,4,8$, $\cC^{\ell}\left(-\frac{n}{2}\right)$ is a $1$-parameter quotient of $\cW^{\gs\gp_2}_{\infty}$. For $n = 2,4,8$, so that $\cC^{\ell}\left(-\frac{n}{2}\right)$ contains $L_{-1}(\gs\gp_2)$, $L_{-2}(\gs\gp_2)$, $L_{-4}(\gs\gp_2)$, respectively, $\cC^{\ell}\left(-\frac{n}{2}\right)$ is not a quotient $\cW^{\gs\gp_2}_{\infty}$ because $\cW^{\gs\gp_2}_{\infty}$ is not defined for these values of the level $k$. However, we will see that by suitably rescaling the generators, we can extend  $\cW^{\gs\gp_2}_{\infty}$ to these values of $k$, and $\cC^{\ell}\left(-\frac{n}{2}\right)$ is indeed a $1$-parameter quotient of the extended algebra. The case $n=1$ is on a different footing. We will see that $\cW^{\gs\gp_2}_{\infty}$ can be extended to include the level $k = -\frac{1}{2}$, but the $1$-parameter quotient of the extended algebra is not the same as $\cS(1)^{\mathbb{Z}_2} \cong L_{-\frac{1}{2}}(\gs\gp_2)$, which is the meaning of $\cC^{\ell}\left(-\frac{n}{2}\right)$ for $n=1$. \end{remark}

	\subsection{Case C}
	Let $\cE(2n)$ denote the $bc$-system of rank $2n$, which is isomorphic to the free fermion algebra $\cF(4n)$. There is a standard embedding \begin{equation}\label{embedC}
		\begin{split}
			\varphi: L_1(\fr{sp}_{2n})\hookrightarrow \cE(2n),
		\end{split} 	\end{equation}  given in terms of the basis \eqref{spStandard} by
	\begin{equation*}
		\begin{split}
			G_{i,j} &\mapsto :\!b^{i}c^{-j}\!:+ :\!b^{j}c^{-i}\!:,\quad  1\leq i\leq j\leq n,\\
			G_{-i,-j} &\mapsto :\!b^{-i}c^{j}\!:+ :\!b^{-j}c^{i}\!: \quad  1\leq i\leq j\leq n,\\
			G_{-i,j} &\mapsto :\!b^ic^{j}\!:- :\!b^{-j}c^{-i}\!:,\quad 1\leq i,j\leq n.
		\end{split}
	\end{equation*}
	The commutant of $L_{1}(\fr{sp}_{2n})$ inside $\cE(2n)$ is isomorphic to $L_{n}(\fr{sp}_2)$ \cite{ORS}, and it is generated by the following $\fr{sp}_{2n}$-invariants.
	\begin{equation} \label{sl2generators}
		X = \sum_{i=1}^n :\!b^i b^{-i}\!:, \qquad Y = \sum_{i=1}^n :\!c^{i}c^{-i}\!:, \qquad H = \sum_{i=1}^n :\!b^ic^i\!:+:\!b^{-i}c^{-i}\!:.
	\end{equation}
	The weight $\frac{1}{2}$ space is spanned by $\{b^i,b^{-i}|i=1,\dots,n\}\cup\{c^i,c^{-i}|i=1,\dots,n\}$, and transforms under $\fr{sp}_2\oplus\fr{sp}_{2n}$ as $\mathbb{C}^2\otimes \mathbb{C}^{2n}$.
	Therefore, we have a diagonal action
	\begin{equation}
		\begin{split}
			V^{\ell}(\fr{sp}_{2n}) \hookrightarrow V^{\ell-1}(\fr{sp}_{2n})\otimes \cE(2n),\quad G_{i,j} \mapsto G_{i,j}\otimes {\bf 1} + {\bf 1} \otimes \varphi(G_{i,j}).
		\end{split}
	\end{equation}
	Define
	\begin{equation}\label{diagonalC}
		\cC^{\ell}(n) = \T{Com}(V^{\ell}(\fr{sp}_{2n}), V^{\ell-1}(\fr{sp}_{2n})\otimes \cE(2n)).
	\end{equation}
	The free field limit of $\cC^{\ell}(n)$ is the invariant algebra $\cE(2n)^{\T{Sp}_{2n}}$.
	By Weyl's first fundamental theorem for standard representation of $\T{Sp}_{2n}$ \cite{W} (modified for odd variables), $\cE(2n)^{\T{Sp}_{2n}}$ has a strong generating set consisting of the $\T{Sp}_{2n}$-invariants
		\begin{equation}
		\begin{split}
			X^{p,q}  =&\frac{1}{2}\sum_{i=1}^n:\!\partial^{p} b^i \partial^{q}b^{-i}\!: +:\!\partial^{q} b^i \partial^{p}b^{-i}\!:,\quad p\geq q\geq 0, \\
			Y^{p,q} =&\frac{1}{2}\sum_{i=1}^n:\!\partial^{p} c^i\partial ^{q}c^{-i}\!:+:\!\partial^{q} c^i\partial ^{p}c^{-i}\!:,\quad p\geq q\geq 0,\\
			H^{p,q}  =&\frac{1}{2}\sum_{i=1}^n:\!\partial^{p} b^i\partial ^{q}c^{-i}\!:+:\!\partial^{q} b^i\partial ^{p}c^{-i}\!:,\quad p,q\geq 0.
		\end{split}
	\end{equation}
	As above, removing the redundancy due to differential relations among the above generators, we find that $\cE(2n)^{\T{Sp}_{2n}}$ has a strong generating type $\cW(1^3,2,3^3,4,\dots)$. This continues to hold generically for the cosets (\ref{diagonalC}), so we obtain
	\begin{lemma}\label{lemma:C}
		For $n\in \mathbb{Z}_{\geq 1}$, $\cC^{\ell}(n)$ is a simple, $1$-parameter vertex algebra of type $\cW(1^3,2,3^3,4,\dots)$ containing $L_{n}(\gs\gp_2)$.
	\end{lemma}
	
	\begin{remark} \label{rem:casek=0} We will see in Section \ref{sect:1paramquot} that $\cC^{\ell}(n)$ is a $1$-parameter quotient of $\cW^{\gs\gp_2}_{\infty}$ for all $n \geq 1$. The case $n=0$ is on a different footing. Although $\cW^{\gs\gp_2}_{\infty}$ is not defined when the level $k = 0$, it can be extended to this level. The quotient of the extended algebra when $k=0$ is then a nontrivial $1$-parameter vertex algebra. As in the case $k=-\frac{1}{2}$, it does not arise as a diagonal coset.
\end{remark}

	\subsection{Case O}
	By \cite[Cor. 2.1]{CLS}, we have an embedding
	\begin{equation} \varphi: L_1(\fr{osp}_{1|2n}) \rightarrow \cS(1)\otimes \cE(2n) 
	\end{equation}
extending the map \eqref{embedC} on the even part $L_1(\gs\gp_{2n})$ by
	\begin{equation*}
		\begin{split}
			X_{1,-j} &\mapsto :\!\gamma c^{j}\!:- :\!\beta b^{-j}\!:,\quad 1\leq j\leq n,\\
			X_{1,j} &\mapsto :\!\gamma c^{-j}\!:+ :\! \beta b^{j}\!:,\quad 1\leq j\leq n.\\
		\end{split}
	\end{equation*}
	The commutant of $L_{1}(\fr{osp}_{1|2n})$ in $\cE(2n)\otimes \cS(1)$ contains $L_{n-\frac{1}{2}}(\fr{sp}_2)$, which it is generated by the following $\fr{osp}_{1|2n}$-invariants.
	\begin{equation}
		X = -:\!\beta\beta\!:-\sum_{i=1}^n :\!b^i b^{-i}\!:, \qquad
		 Y =\ :\!\gamma\gamma\!:+ \sum_{i=1}^n :\!c^{i}c^{-i}\!:,\qquad H=\  :\!\beta\gamma\!:+ \sum_{i=1}^n :\!b^ic^i\!:+:\!b^{-i}c^{-i}\!:.
	\end{equation}
	The weight $\frac{1}{2}$ space is spanned by odd variables $\{b^i,b^{-i},c^i,c^{-i}|i=1,\dots,n\}$, and even variables $\{\beta,\gamma\}$, that transform under $\fr{sp}_2\oplus\fr{osp}_{1|2n}$ as $\mathbb{C}^2\otimes \mathbb{C}^{1|2n}$.
	Therefore, we have a diagonal action
	\begin{equation}
		\begin{split}
			V^{\ell}(\fr{osp}_{1|2n}) \hookrightarrow V^{\ell-1}(\fr{osp}_{1|2n})\otimes \cE(2n)\otimes \cS(1),\quad 	G_{i,j} \mapsto G_{i,j}\otimes  {\bf 1} +  {\bf 1} \otimes \varphi(G_{i,j}).
		\end{split}
	\end{equation}
	Define
	\begin{equation}\label{diagonalO}
		\cC^{\ell}\left(n-\frac{1}{2}\right) = \T{Com}(V^{\ell}(\fr{osp}_{1|2n}), V^{\ell-1}(\fr{osp}_{1|2n})\otimes \cS(1)\otimes \cE(2n))^{\mathbb{Z}_2}.
	\end{equation}
	The free field limit of $\cC^{\ell}\left(n-\frac{1}{2}\right)$ is the invariant algebra $\left(\cS(1)\otimes\cE(2n)\right)^{\T{Osp}_{1|2n}}$.
	By Sergeev's first fundamental theorem for standard representation $\mathbb{C}^{1|2n}$ of $\T{Osp}_{1|2n}$ \cite{SI,SII}, $\left(\cS(1)\otimes\cE(2n)\right)^{\T{Osp}_{1|2n}}$ has a strong generating set consisting of the $\T{Osp}_{1|2n}$-invariants
	\begin{equation}
		\begin{split}
			X^{p,q} =&:\!\partial^{n}\beta\partial^{m}\beta\!:+\sum_{i=1}^n:\!\partial^{n} b^i \partial^{m}b^{-i}\!: +:\!\partial^{m} b^i \partial^{m}b^{-i}\!:,\quad p\geq q\geq 0, \\
			Y^{p,q} =&:\!\partial^{n}\gamma\partial^{m}\gamma\!:+\sum_{i=1}^n:\!\partial^{n} c^i\partial ^{m}c^{-i}\!:+:\!\partial^{m} c^i\partial ^{n}c^{-i}\!:,\quad p\geq q\geq 0,\\
			H^{p,q} =&:\!\partial^{n}\beta\partial^{m}\gamma\!:+\sum_{i=1}^n:\!\partial^{n} b^i\partial ^{m}c^{-i}\!:+:\!\partial^{m} b^i\partial ^{n}c^{-i}\!:,\quad p,q\geq 0.
		\end{split}
	\end{equation}
	Removing the redundancy due to differential relations among the generators, we find that $\left(\cS(1)\otimes\cE(2n)\right)^{\T{Osp}_{1|2n}}$ has a strong generating type $\cW(1^3,2,3^3,4,\dots)$. This continues to hold generically for the cosets (\ref{diagonalO}), so we obtain
	
	\begin{lemma}\label{lemma:O}
		For $n\in\mathbb{Z}_{\geq 1}$, $\cC^{\ell}\left(n-\frac{1}{2}\right)$ is a simple, $1$-parameter vertex algebra of type $\cW(1^3,2,3^3,4,\dots)$ containing $L_{n-\frac{1}{2}}(\fr{sp}_2)$. \end{lemma}

\section{Universal $2$-parameter vertex algebra $\Wsp$} \label{sect:main}
In this section we will construct the universal $2$-parameter vertex algebra $\Wsp$ of type \eqref{sp2starting}. The three fields in weight $1$ generate a copy of the affine vertex algebra $V^k(\gs\gp_2)$, the fields $\{L,W^{2i} | i\geq 2\}$ in each even weight transform as the trivial $\gs\gp_2$-module, and the three fields $\{X^{2i-1},Y^{2i-1},H^{2i-1} | i \geq 2\}$ in each odd weight transform as the adjoint $\gs\gp_2$-module. It is freely generated and defined over a localization of the polynomial ring $\mathbb{C}[c,k]$.
	
	As in the case of universal $2$-parameter vertex algebras $\Winf$ and $\Wev$ constructed in \cite{Lin} and \cite{KL}, $\Wsp$ will be the universal enveloping vertex algebra of a nonlinear Lie conformal algebra $\nlcalg$, defined over a localization of $\mathbb{C}[c,k]$ with generators $\{L,W^{2i} | i \geq 2\}\cup \{X^{2i-1},Y^{2i-1},H^{2-1} | i\geq 1\}$ and grading $\Delta(X^{2i-1})=\Delta(Y^{2i-1}) = \Delta(H^{2i-1}) =2i-1$, $\Delta(L)=2$ and  $\Delta(W^{2i})=2i$, in the sense of \cite{DSKI}.
	
	We shall work with the OPE rather than lambda-bracket formalism, so the sesquilinearity, skew-symmetry, and Jacobi identities from \cite{DSKI} are replaced with (\ref{conformal identity}-\ref{Jacobi}).
	As explained in \cite[Section 3]{Lin}, specifying a nonlinear Lie conformal algebra in the language of OPEs means specifying generators $\{a^1,a^2,\dots\}$ where field $a^i$ has conformal weight $d_i>0$, and pairwise product expansions 
	\[a^i(z)a^j(w)\sim \sum_{n=0}^{\infty} a^i(w)_{(n)}(w)a^j(w) (z-w)^{-n-1}, \]
	where each term $a^i(w)_{(n)}a^j(w)$ has conformal weight $d_i+d_j-n-1$, and is a normally ordered polynomial in the generators and their derivatives.
	Additionally, for all $a,b,c \in \{a^1,a^2,\dots\}$, (\ref{conformal identity}-\ref{quasi-derivation}) hold, and all Jacobi identities (\ref{Jacobi}) hold as a consequence of (\ref{conformal identity}-\ref{quasi-derivation}) alone.
	This data uniquely determines the universal enveloping vertex algebra which is freely generated by $\{a^1,a^2,\dots\}$.
	
	In this notation, our goal will be to construct the OPE algebra with strong generators $\{L,W^{2i} | i \geq 2\}\cup \{X^{2i-1},Y^{2i-1},H^{2i-1} | i\geq 1\}$ of $\Wsp$, such that identities (\ref{conformal identity}-\ref{quasi-derivation}) are imposed, and the Jacobi identities (\ref{Jacobi}) hold as a consequence of (\ref{conformal identity}-\ref{quasi-derivation}) alone. 
	
	\subsection{Set-up} \label{subsec:setup}
	We postulate that $\Wsp$ has the following features.
	
	\begin{enumerate} 
		\item Weight 1 fields $X^1,Y^1,H^1$ generate the universal affine algebra $V^k(\fr{sp}_2)$ of level $k$. 
		\item Weight 2 field $L$ generates the universal Virasoro algebra $\vir$ of central charge $c$.
		\item Even weight fields $W^{2i}$ each transform as the trivial $\fr{sp}_2$-module.
		\item Odd weight fields $X^{2i-1},Y^{2i-1},H^{2i-1}$ transform as the adjoint $\fr{sp}_2$-module.
		\item Weight 3 fields $X^3,Y^3,H^3$ are primary for $\vir$.
		\item Weight 4 field $W^{4}$ is primary for $\vir$.
		\item Subalgebras $V^k(\fr{sp}_2)$ and $\vir$, together with $W^4$ weakly generate $\Wsp$. Specifically, $W^4$ satisfies the raising property:
		\begin{equation}\label{raising}
			\begin{split}
			W^4_{(1)}X^{2i-1}=X^{2i+1},\quad &W^4_{(1)}Y^{2i-1}=Y^{2i+1},\quad W^4_{(1)}H^{2i-1}=H^{2i+1},\\
			&W^4_{(1)}W^{2i}=W^{2i+2}.
			\end{split}
		\end{equation}
	\end{enumerate}
	
	Since the strong generating type alternates by the parity of conformal weight, there are three types of interactions with some variation in their structure, e.g. see (\ref{ansatzWW}-\ref{ansatzXH}).
	\begin{enumerate}
		\item Even with even. 
		For $n\geq 0$ and $0\leq r\leq 2n-1$ we set 
		\[EE^{2n}_r =\{ L_{(r)} W^{2n-2} , \ W^{4}_{(r)}W^{2n-4} , \ W^{6}_{(r)}W^{2n-6} , \dots\}.\]
		\item Even with odd.
		Let us denote by $EO^{2i,2j-1}_{r}$ the set of following products
		\begin{center}
			\begin{tabular}{|c|c|c|c|}
				\hline
				$(\cdot_{(r)}\cdot)$&$X^{2j-1}$&$Y^{2j-1}$&$H^{2j-1}$\\
				\hline
				$W^{2i}$&$W^{2i}_{(r)}X^{2j-1}$&$W^{2j}_{(r)}Y^{2j-1}$&$W^{2j}_{(r)}H^{2j-1}$\\
				\hline
			\end{tabular}
		\end{center} 
		For $n\geq 0$ and $0\leq r\leq 2n$ we set 
		\[EO^{2n+1}_r =  EO^{2,2n-1}_{r} \cup EO^{4,2n-3}_{r} \cup EO^{6,2n-5}_{r} \cup \dotsb.\]
		\item Odd with odd. 
		Let us denote by $OO^{2i-1,2j-1}_{r}$ the set of following products
		\begin{center}
			\begin{tabular}{|c|c|c|c|}
				\hline
				$(\cdot_{(r)}\cdot)$&$X^{2j-1}$&$Y^{2j-1}$&$H^{2j-1}$\\
				\hline
				$X^{2i-1}$&$X^{2i-1}_{(r)}X^{2j-1}$&$X^{2i-1}_{(r)}Y^{2j-1}$&$X^{2i-1}_{(r)}H^{2j-1}$\\
				
				$Y^{2i-1}$&$Y^{2i-1}_{(r)}X^{2j-1}$&$Y^{2i-1}_{(r)}Y^{2j-1}$&$Y^{2i-1}_{(r)}H^{2j-1}$\\
				
				$H^{2i-1}$&$ H^{2i-1}_{(r)}X^{2j-1}$&$H^{2i-1}_{(r)}Y^{2j-1}$&$H^{2i-1}_{(r)}H^{2j-1}$\\
				\hline
			\end{tabular}
		\end{center}
		For $n\geq 0$ and $0\leq r\leq 2n-1$ we set 
		\[OO^{2n}_r = OO^{1,2n-1} \cup OO^{3,2n-3} \cup \dotsb.\]
	\end{enumerate}
	Let $D^n_r$ denote the set of all $r^{ \T{th} }$ products among generators of total weight $n$; specifically,
	\[D^{2m}_r = OO^{2m}_r\cup EE^{2m}_r,\quad D^{2m+1}_r =EO^{2m+1}_r.\]
	Write $D^n = \bigcup_{r\leq n} D^{n}_r$ and $D_n = \bigcup_{m\leq n} D^m$ for the OPE data among fields of total weight $n$ and not exceeding $n$, respectively. 
	Lastly, let $J^m$ denote the set of all Jacobi identities $J_{r,s}(a,b,c)$ among generating fields $a,b,c$ of total weight exactly $m$, and $J_n = \bigcup_{m\leq n} J^m$.
	
	Our strategy is similar to the one used in \cite{Lin} and \cite{KL}, and consists of four steps. 
	\begin{enumerate}
		\item We begin by writing down the most general OPEs in $D_9$ compatible with conformal weight and $V^k(\fr{sp}_2)$ symmetry. Next, we impose vertex algebra relations (\ref{conformal identity}-\ref{quasi-derivation}) along with Jacobi identities $J_{11}$ to uniquely determine $D_9$ in terms of two parameters $c$ and $k$; see Proposition \ref{prop:base case proposition}.
		\item Next, we use data $D_9$ obtained in the previous step, and raising property of $W^4$ field to explicitly evaluate an infinite set of structure constants, see Proposition \ref{prop:structure constants}.
		\item Next, we proceed inductively; we assume $D^{2n} \cup D^{2n+1}$ to be known, and show that a subset of Jacobi identities in $J^{2n+4}\cup J^{2n+5}$ uniquely determines all OPEs in $D^{2n+2}\cup D^{2n+3}$; see Proposition \ref{induction}.
		Note that we are not checking all the Jacobi identities $J^{2n+4}\cup J^{2n+5}$ at this stage, we leave open the possibility that some of them may not vanish but instead give rise to nontrivial null fields. At the end of induction, we obtain the existence of a possibly degenerate nonlinear Lie conformal algebra $\nlcalg$ over a localization of $\mathbb{C}[c,k]$. 
		We then invoke the De Sole-Kac correspondence \cite{DSKI} to conclude that universal enveloping vertex algebra $\Wsp$ indeed exists; see Theorem \ref{Wsp 2 parameters}.
		\item Lastly, we exhibit a family vertex algebras with known characters to prove free generation of $\Wsp$. Specifically, these are the generalized parafermions in type $C$; see Corollary \ref{Wsp freely generated}.
	\end{enumerate}

	Before we begin with the base case computation, we first investigate some useful consequences of postulated features (1)-(3) in the above list. With foresight of the main result of this section, Theorem \ref{thm:induction}, we often mention the vertex algebra $\Wsp$ during the process of its construction. At this stage it is not yet known if it is a freely generated 2-parameter vertex algebra. However, by the De Sole-Kac correspondence there always exists the universal enveloping vertex algebra $\Wsp$ of the nonlinear Lie conformal algebra $\nlcalg$, possibly degenerate or trivial. All our statements regarding this vertex algebra continue to hold before the construction is complete, even if it were to be degenerate or trivial.
	
	\subsection{PBW monomials and filtration}\label{sect:symmetry}
	We begin by fixing a choice of a lexicographic order on PBW monomials.
	First, order elements of lower conformal weight to be less than those of higher weight, and when they are of the same weight (necessarily odd), we order $X^{2i-1}<Y^{2i-1}<H^{2i-1}$ for each $i\geq 1$. Since $\Wsp$ is of type $\cW(1^3,2,3^3,4,\dots)$, its conformal weight $N$ subspace is spanned by PBW monomials of the form
	\begin{equation}\label{basis}
		\begin{split}
			&:\! \partial^{a_{1}^1} X^1 \dotsb \partial^{a_{e_1}^1} X^1  \  
			\partial^{b_{1}^1} Y^1 \dotsb \partial^{b_{f_1}^1} Y^1  \
			\partial^{c_{1}^1} H^1 \dotsb \partial^{c_{h_1}^1} H^1  \ 
			\partial^{k_{1}^2} L \dotsb \partial^{k_{r_2}^2} L  \dotsb \!:,\\
			&a^1_1\geq \dotsb \geq  a_{e_1}^1 ,\quad b^1_1\geq \dotsb \geq  b_{f_1}^1,\quad  c^1_1\geq \dotsb \geq  c_{h_1}^1, \quad k^2_1\geq \dotsb \geq  k_{r_1}^2, \dots ,\\
			&a_1+\dots +a_{e_1}+b_1+\dots +b_{f_1} + c_1+\dots +c_{h_1} + 2(k_1+\dots +k_{r_2} )+\dots = N.
		\end{split}
	\end{equation}
	Let $\omega$ be a sequence of partitions $(\omega_1,\omega_2,\dots)$.
	We say that a PBW monomial (\ref{basis}) is of type $\omega$ if 
	\[\omega_1 = (\{a^1_1,a^1_2,\dots,a^1_{e_1}\}\cup \{b^1_1,b^1_2,\dots,b^1_{f_1}\}\cup\{c^1_1,c^1_2,\dots,c^1_{h_1}\}),\quad \omega_2 = (k^2_1,k^2_2,\dots,k^2_{r_2}), \dots\]
	where $\omega_1$ is sorted in decreasing order, etc.
	Let $U^{(\omega)}$ be the subspace spanned by all PBW monomials of type $\omega$.
	For example, $U^{(1^1,3)}$ is spanned by monomials $\{:\!\partial X^1X^3\!:,:\!\partial H^1X^3\!:,:\!\partial X^1H^3\!:,:\!\partial X^1Y^3\!:,:\!\partial Y^1X^3\!:,:\!\partial Y^1H^3\!:,:\!\partial H^1Y^3\!:,:\!\partial Y^1Y^3\!:\}$.
	
	\begin{remark}
		Types of monomials are in bijection with a basis of $\cW_{1+\infty}$, which itself is bijective with plane partitions thanks to the formula of MacMahon. 
	\end{remark}
	
	Now we recall Li's canonical decreasing filtration \cite{L}.
	Let $F^p$ be spanned by elements of the form 
	\[:\! \partial^{n_1}\alpha^1 \partial^{n_2}\alpha^2 \dots \partial^{n_r}\alpha^r \!:,\]
	where $\alpha^1,\dots,\alpha^r \in \Wsp$, $n_i\geq 0$, and $n_1+\dotsb +n_r \geq p$. 
	Then we have
	\begin{equation}\label{LiFiltration}
		\Wsp=F^0 \supseteq F^1 \supseteq \dotsb .
	\end{equation}
	Set 
	\[\T{gr}^F(\Wsp) = \bigoplus_{p=0}^{\infty}F^p/F^{p+1},\]
	and for $p\geq 0$ let 
	\begin{equation}\label{projectionLi}
		\pi_p:F^p \to F^p/F^{p+1}\subseteq \T{gr}^F(\Wsp) 
	\end{equation}
	be the projection.
	Then $\T{gr}^F(\Wsp)$ is a graded commutative algebra with product 
	\[\pi_p(\alpha)\pi_q(\beta) = \pi_{p+q}(\alpha_{(-1)}\beta), \quad \alpha \in F^p,\ \beta \in F^q. \]
	
	\subsection{Symmetry}
	By our assumptions, the weight 1 fields $X^1$,$H^1$ and $Y^1$ generate the affine vertex algebra $V^k(\fr{sp}_2)$. In particular, the zero modes $\{X^1_{(0)},H^1_{(0)},Y^1_{(0)}\}$ generate the Lie algebra $\fr{sp}_2 \cong \fr{sl}_2$,
		\[X^1_{(0)}Y^1 = H^1 ,\quad X^1_{(0)}H^1 = -2X^1,\quad  Y^1_{(0)}H^1 = 2Y^1.\] We have an inner automorphism $\sigma$ of $\fr{sp}_2$ that maps
	\[X^1\mapsto Y^1, \quad Y^1\to X^1, \quad H^1\to -H^1, \quad L\mapsto L.\]
	Since $\fr{sp}_2$ acts by derivations on $\Wsp$, thanks to the raising property of $W^4$, we can extend $\sigma$ to all strong generators as follows:
	\[X^{2n-1}\mapsto Y^{2n-1}, \quad Y^{2n-1}\to X^{2n-1}, \quad H^{2n-1}\to -H^{2n-1}, \quad W^{2n} \mapsto W^{2n} ,\quad n\geq 2.\]
Next, it will be convenient to organize the basis \eqref{basis} into irreducible $\fr{sp}_2$-modules. The weight $N$ subspace $\Wsp[N]$ decomposes as
	\begin{equation}\label{sl2 decomposition}
		\Wsp[N] = \bigoplus_{\mu=0}^{N} \mathbb{C}^{M(N,2\mu)} \otimes \rho_{2\mu}.
	\end{equation}
	Here $M(N,\mu)$ is the multiplicity of a highest-weight irreducible $\fr{sp}_2$-module $\rho_{\mu}$ of highest weight $\mu$. We follow the standard choice of scaling for the basis elements of homogeneous weight spaces  \cite{H}.
	We have the weight space decomposition \[\rho_{\mu}=\rho_{\mu,-\mu}\oplus \rho_{\mu,-\mu+2}\oplus\dotsb\oplus \rho_{\mu,\mu-2}\oplus \rho_{\mu,\mu}.\] 
	Here, the weight space $\rho_{\mu, \mu - 2\alpha}$ is spanned by the vector $v_{\mu,\mu-2\alpha}=\frac{1}{\alpha!} (Y^1_{(0)})^{\alpha} v_{\mu,\mu}$, with $v_{\mu,\mu}$ being the highest weight vector.
	With this choice of basis for weight spaces, the $\fr{sp}_2$-action is as follows.
	\begin{equation}\label{sp2 scaling}
		X^1_{(0)} v_{\mu,\alpha} = (\frac{1}{2}\mu+1+\frac{1}{2}\alpha)v_{\mu,\alpha+2},\quad Y^1_{(0)} v_{\mu,\alpha}= (\frac{1}{2}\mu+1-\frac{1}{2}\alpha)v_{\mu,\alpha-2}, \quad H^1_{(0)}v_{\mu,\alpha}= \alpha v_{\mu,\alpha}.
	\end{equation}
	
	We now set up notation to help us organize the strong generators into $\fr{sp}_2$-modules.
	Denote by $U^{n}_{\mu}$ the $\fr{sp}_2$-module isomorphic to $\rho_{\mu}$, spanned by generators of conformal weight $n$, and by $U^{n}_{\mu,\alpha}$ the subspace of Cartan weight $\alpha$.
	By our assumptions, these are either trivial or adjoint $\fr{sp}_2$-modules; specifically we have
	\begin{equation}
		U^{2i}_0= \T{Span}\{W^{2i}\}\cong \rho_0, \quad U^{2i-1}_2=\T{Span}\{ W^{2i-1}_{2,-2}, W^{2i-1}_{2,2},W^{2i-1}_{2,0}\}\cong \rho_2,
	\end{equation}
	where the weight spaces are spanned by the following basis vectors
	\begin{equation}\label{convention Walpha}
		W^{2}_{0,0} = L, \quad W^{2i}_{0,0}= W^{2i},\quad W^{2i-1}_{2,2}= X^{2i-1}, \quad W^{2i-1}_{2,-2}= -Y^{2i-1}, \quad W^{2i-1}_{2,0} = -H^{2i-1},
	\end{equation}
	for $i\geq 1$.
	Note that in the case of strong generators, the conformal weight uniquely specifies the $\fr{sp}_2$-module $\rho_{\bar n}$, where $\bar n = 2( n\ \T{mod} \ 2)$.
	Similarly, we will write 
	\begin{equation}\label{convention more}
		\begin{split}
		W^{(2i-1,2j-1)}_{2,2}= \ :\!X^{2i-1}H^{2j-1}\!:-: \!H^{2i-1}X^{2j-1}\!:, \quad &W^{(2i-1)^d}_{2,2}= \partial^d X^{2i-1},\\
		&W^{(2i)^d}_{0,0}= \partial^d W^{2i},
		\end{split}
	\end{equation}
	for $j\geq i \geq 1$ and $d\geq 1$. 
	The basis for the whole modules $U^{(2i-1,2j-1)}_2$ and $U^{2i-1}_{2}$ is generated in accordance with the $\fr{sp}_2$-action (\ref{sp2 scaling}).
	
	Next, we proceed to organize the rest of the basis (\ref{basis}) into $\fr{sp}_2$-modules. 
	Due to nonassociativity and noncommutativity in $\Wsp$, the subspace $U^{\omega}$ is in general not an $\fr{sp}_2$-module.
	However, its projection by the map (\ref{projectionLi}) on the associated graded $\T{gr}^F({\Wsp})$, again denoted $U^{\omega}$, is an $\fr{sp}_2$-module.
	Then we have an isomorphism of $\fr{sp}_2$-modules 
	\begin{equation}\label{omega}
		U^{\omega} \cong \cS^{\omega_1}(V_1) \otimes  \cS^{\omega_2}(V_0) \otimes \dotsb\otimes  \cS^{\omega_s}(V_{\bar s})\cong \bigoplus_{\mu} U^{\omega}_{\mu},
	\end{equation}
	where functor $\cS^{\lambda}$ defined as 
	\[\cS^{\lambda}(V) = \cS^{m_1}(V)\otimes \cS^{m_2}(V)\otimes \dotsb, \quad \lambda = (1^{m_1},2^{m_2},\dots).\]
	Let $\Omega_{\mu}^{\omega}$ be the index set for some basis of highest weight vectors in $U^{\omega}_{\mu}$ of highest weight $\mu$.
	By the $\fr{sp}_2$-action (\ref{sp2 scaling}) the choice of a basis for $U^{\omega}_{\mu,\mu}$ determines the basis for each weight space $U^{\omega}_{\mu,\alpha}$, and thus the whole $\fr{sp}_2$-module $U^{\omega}_{\mu}$.
	Thus chosen basis vectors for weight spaces $U^{\omega}_{\mu,\alpha}$ are denoted by $W^{\omega_{\chi}}_{\mu,\alpha}$, where $\chi$ indexes the set $\Omega^{\omega}_{\mu}$.
	If $U^{\omega}_{\mu}$ has dimension 1,
	then we simply denote the basis for weight spaces by $W^{\omega}_{\mu,\alpha}$.
	For example, this happens if $\omega = n$, $\omega=(n)^1$ or $\omega = (2i-1,2j-1)$, where we have made the choices of highest weight vectors in (\ref{convention Walpha}) and (\ref{convention more}).
	\begin{remark}
		In general, the decomposition (\ref{omega}) is not multiplicity free.
		It may be interesting to compute dimension $\T{dim}(\Omega_{\mu}^{\omega})$.
	\end{remark}
	For $v \in U^{\omega}_{\mu}$ we conjecture that it is possible to correct for the effects of nonassociativity and noncommutativity in $\Wsp$, in a uniform manner.
	Specifically, we have an $\mathfrak{sp}_2$-equivariant quantization map 
	\begin{equation}\label{correctionsp2map}
		\iota:\T{gr}^F(\Wsp) \to \Wsp,
	\end{equation}
	given by the following.
	\begin{conjecture}\label{correct}
		Let $v \in U^{\omega}_{\mu}$ be a highest weight vector of highest weight $\mu$ in the associated graded $ \T{gr}(\Wsp)$, of type $\omega = (\omega_1,\dots,\omega_s)$.
		Then the following expression is the highest weight vector of highest weight $\mu$ in $\Wsp$.
		\begin{equation}\label{correctionsp2}
			\iota (v) =  v+\sum_{n=1}^{\infty}e_n(\alpha_{1,\mu},\alpha_{2,\mu},\dots, \alpha_{l(\omega)-1,\mu}) (X^1_{(0)}Y^1_{(0)})^n v, \quad \alpha_{n,\mu} = -\frac{1}{(n+2)(n+\mu+1)},
		\end{equation}
		Here $e(x_1,\dots,x_n)$ is an elementary symmetric polynomial of degree $n$, and $l(\omega)$ is the sum of the lengths of partitions in $\omega$.
	\end{conjecture}
	As an example, consider the Sugawara vector $L^{\mathfrak{sp}_2}$.
	In our notation, it is an element of $U^{(1,1)}_0$. 
	The following is the highest weight vector in the associated graded $\T{gr}^F(\Wsp)$. \[v=\frac{1}{2}:\!H^1H^1\!:+2:\!X^1Y^1\!:.\]
	Since length of $(1,1)$ is 2, we only need the $n=1$ correction in (\ref{correctionsp2}), and $\alpha_{1,0}=-\frac{1}{6}$.
	Evaluating $\iota(v)$ we find the well-known $\fr{sp}_2$-invariant vector
	\[\iota (v)=\frac{1}{2}:\!H^1H^1\!:+2:\!X^1Y^1\!:-\partial H^1.\]

	\begin{remark}
		The map $\iota$ in (\ref{correctionsp2map}) and (\ref{correctionsp2}) is not unique. 
		For instance, any linear combination with other highest weight vector in the common subspace $U^{\omega}_{\mu}$ gives a different map which is also $\mathfrak{sp}_2$-equivariant.
	\end{remark}
	
	We have verified Conjecture \ref{correct} for $\Wsp[N]$ with $N\leq 9$, which is what we need for the base computation (Proposition \ref{prop:base case proposition}). Note that in the formula (\ref{correctionsp2}) for $\iota (v)$, the correction terms to $v$ have strictly greater degree in the filtration (\ref{LiFiltration}) than that of $v$. Therefore, we can regard $v$ as the leading term in (\ref{correctionsp2}). By abuse of notation, we use the same notation $U^{\omega}_{\mu}$ to denote the image $\iota (U^{\omega}_{\mu})$, which transforms as $\rho_{\mu}$ as an $\fr{sp}_2$-module. Finally, let $U_{\mu}(n)$ denote the subspace of conformal weight $n$ transforming as the $\fr{sp}_2$-module $\rho_{\mu}$.
	For example, the conformal weight 2 space decomposes as the $\fr{sp}_2$-modules
	\[\Wsp[2]=U_0(2)\oplus U_2(2)\oplus U_4(2)\cong 2\rho_{0}\oplus \rho_2\oplus \rho_4,\]
	where $U_0(2)$ is spanned by $L$ and Sugawara vector $L^{\fr{sp}_2}$, $U_2(2)$ is spanned by highest weight vector $\partial X^1$, and $U_4(2)$ is spanned by the highest weight vector $:\!X^1X^1\!:$.

	Next, we explain how $\fr{sp}_2$-symmetry imposes severe restrictions on the OPEs among strong generators.
	First, recall the following decompositions of $\fr{sp}_2$-modules
	\begin{equation}\label{modules}
		\begin{split}
			\rho_{2}\otimes \rho_2 \cong& \rho_0 \oplus \rho_2\oplus \rho_4,\\
			\rho_0\otimes \rho_2 \cong &\rho_2,\\
			\quad \rho_0\otimes \rho_0 \cong& \rho_0.
		\end{split}
	\end{equation}
	From \eqref{modules}, it follows that:
	\begin{itemize}
		\item only $U^{\omega}_{0}$, $U^{\omega}_{2}$, and $U^{\omega}_{4}$ can arise in the OPEs among the generators of odd conformal weights.
		\item only  $U^{\omega}_{2}$ can arise in the OPEs among the generators of even and odd conformal weights.
		\item only  $U^{\omega}_{0}$ can arise in the OPEs among the generators of even and even conformal weights.
	\end{itemize}
	
	\begin{remark}
		This constraint is analogous to $\mathbb{Z}_2$-symmetry of $\Winf$, which featured prominently in its construction \cite{Lin}.
	\end{remark}
	
	Further restrictions follow.
	Let $v_{\mu,\alpha}\otimes v_{\nu,\beta}$ denote the basis of $\fr{sp}_2$-modules appearing on the left side of the isomorphisms (\ref{modules}), and denote the basis for the $\fr{sp}_2$-modules appearing on the right side by $u^{2,2}_{\mu,\alpha}$, $u^{0,2}_{2,\alpha}$, and $u^{0,0}_{0,0}$, respectively.
	In terms of this basis we have the following relations.
	\begin{equation}\label{sp2mix}
		\begin{split}
			v_{2,\alpha}\otimes v_{2,\beta} =& \epsilon^{2,2}_{0}(\alpha,\beta)u_{0,\alpha+\beta}+ \epsilon^{2,2}_{2}(\alpha,\beta)u_{2,\alpha+\beta}+\epsilon^{2,2}_{4}(\alpha,\beta)u_{4,\alpha+\beta},\\
			v_{0,0}\otimes v_{2,\alpha} =& \epsilon^{0,2}_{2}(0,\alpha)u_{2,\alpha},\\
			v_{0,0}\otimes v_{0,0} =&\epsilon^{0,0}_0(0,0)u_{0,0}.\\
		\end{split}
	\end{equation}
	For each module arising on the right side of (\ref{modules}), we have a choice of the scaling for the highest weight vector.
	The constants $\epsilon^{\mu,\nu}_{\gamma}(\alpha,\beta)$ in (\ref{sp2mix}) are uniquely determined by this choice.
	We can choose the scaling of the $u_{0,0}$ and $u_{2,2}$ so that $\epsilon^{0,0}_0$ and $\epsilon^{0,2}_2$ are identically equal to 1.
	We have chosen the scaling for $u^{2,2}_{0,0}$,$u^{2,2}_{2,2}$ and $u^{2,2}_{4,4}$ so that $\epsilon^{2,2}_{0}(2,-2) =\epsilon^{2,2}_{2}(2,-2) =\epsilon^{2,2}_{4}(2,-2)=1$ in $v_{2,2}\otimes v_{2,-2}$. 
	This determines the remaining displayed in Table \ref{tab:adjointadjoint}.
	\begin{table}\label{tab:epsilon}[h]
		\centering
		\begin{tabular}{|c|c|c|c|}
			\hline
			$\rho_{2}\otimes \rho_{2}$&$v_{2,-2}$&$v_{2,2}$&$v_{2,0}$\\
			\hline
			$v_{2,-2}$&$ 6 v_{4,-4}$&$ v_{0,0}-v_{2,0}+ v_{4,0}$&$2v_{2,-2}+3v_{4,-2}$\\
			\hline
			$v_{2,2}$&$ v_{0,0}+v_{2,0}+ v_{4,0}$&$ 6 v_{4,4}$&$-2v_{2,2}+3v_{4,2}$\\
			\hline
			$v_{2,0}$&$2v_{2,-2}+3v_{4,-2}$&$-2v_{2,2}+3v_{4,2}$&$-2v_{0,0}+4v_{4,0}$\\
			\hline
		\end{tabular}
		\caption{Relationships among structure constants by $\fr{sp}_2$-symmetry in the tensor product $\rho_2\otimes \rho_2$.}\label{tab:adjointadjoint}
	\end{table}
	
	The following is the master OPE among any two strong generators of the vertex algebra $\Wsp$.
	\begin{equation}\label{OPEgeneralG}
		W^{n}_{\mu,\alpha}(z)W^{m}_{\nu,\beta}(w)\sim \sum_{r=0}^{n+m-1}\sum_{\gamma=0,2,4}\sum_{\omega\in U_{\gamma}(r)}\sum_{\chi\in\Omega^{\omega}_{\gamma}} w^{n,m}_{\omega_{\chi},\gamma}(\alpha,\beta)W^{\omega_{\chi}}_{\gamma,\alpha+\beta}(w) (z-w)^{-n-m+r}.
	\end{equation}
	Thanks to $\fr{sp}_2$-symmetry, we have the following factorization of structure constants
	\begin{equation}\label{sp2motion}
		w^{n,m}_{\omega_{\chi},\gamma}(\alpha,\beta)=\epsilon_{\gamma}^{\mu,\nu}(\alpha,\beta)w^{n,m}_{\omega_{\chi},\gamma}, \quad \chi\in\Omega^{\omega}_{\gamma,\alpha+\beta}
	\end{equation}
	From now on we will write (\ref{OPEgeneralG}) in the following shorthand notation.
	\begin{equation}\label{OPEgeneral}
		W^{n}(z)W^{m}(w)  \sim \sum_{r=0}^{n+m-1}\sum_{\gamma=0,2,4}\sum_{\omega\in U_{\gamma}(r)}\sum_{\chi\in\Omega^{\omega}_{\gamma}} \epsilon_{\gamma}^{\bar n,\bar m}w^{n,m}_{\omega_{\chi},\mu}W^{\omega_{\chi}}_{\mu}(w) (z-w)^{-n-m+r}.
	\end{equation}
	To extract the OPEs among fields in (\ref{OPEgeneralG}) from (\ref{OPEgeneral}), we proceed as follows:
	\begin{enumerate}\label{explain}
		\item Determine that $W^n$ transforms as $\rho_{\bar n}$, and $W^m$ as $\rho_{\bar m}$, and recall the decomposition of $\rho_{\bar n}\otimes \rho_{\bar m}$ in (\ref{modules}).
		\item Given a choice of Cartan weights $\alpha$ and $\beta$ appearing in modules $U^n_{\bar n}$ and $U^{m}_{\bar m}$, select the basis of vectors $W^{n}_{\bar n,\alpha}$ and $W^{m}_{\bar m,\beta}$ of weight spaces $U^n_{\bar n,\alpha}$ and $U^n_{\bar m,\beta}$, determined by our choice of the highest weight vector (\ref{convention Walpha}) and $\fr{sp}_2$-action (\ref{sp2 scaling}).
		\item  Given that same choice of Cartan weights $\alpha$ and $\beta$, use Table \ref{tab:adjointadjoint} to evaluate $\epsilon^{\bar n,\bar m}_{\gamma}(\alpha,\beta)$, for $\gamma$ a highest weight of the highest weight module appearing in the decomposition of $\rho_{\bar n}\otimes \rho_{\bar m}$  in (\ref{modules}).
		Note that only in the case of interaction of odd with odd weights this is nontrivial.
		\item  Select the basis of vectors $W^{\omega_{\chi}}_{\mu,\alpha+\beta}$ and of the weight spaces $U^{\omega}_{\mu,\alpha+\beta}$, determined a choice of highest weight vectors in $U^{\omega}_{\mu}$ and the $\fr{sp}_2$-action (\ref{sp2 scaling}).
		Finally, we recover OPE relations (\ref{OPEgeneralG}).
	\end{enumerate}

	We will be imposing Jacobi identities $J_{r,s}(W^u,W^n,W^m)$.
	However, using the shorthand notation (\ref{OPEgeneral}) they still evaluate to expressions involving $\epsilon$ symbols introduced in (\ref{sp2mix}).
	Similarly as in the above discussion regarding how to extract (\ref{OPEgeneralG}) from shorthand (\ref{OPEgeneral}), we can extract from expressions $J_{r,s}(W^u,W^n,W^m)$ the Jacobi identities $J(W^u_{\bar u,\gamma},W^{n}_{\bar n,\mu},W^{m}_{\bar m,\beta})$.
	
	Finally, we note that structure constants $w^{n,m}_{\omega_{\chi},\mu}$ arising in (\ref{OPEgeneral}) depend on the choice of bases indexed by sets $\Omega^{\omega}_{\mu}$. Given such a choice, the structure constants are uniquely determined as rational functions in terms of $c$ and $k$, i.e., the universal algebra $\Wsp$ is a 2-parameter vertex algebra. This feature is independent of the choice of basis.

	\subsection{Step 1: base computation}
	The aim of this subsection is to describe the computation in sufficient detail, that the reader can reproduce our results. The computation was done using the Mathematica package \texttt{OPEdefs} \cite{T}. 
	We begin by setting up the most general OPEs in $D_9$ compatibile with assumptions in Subsection \ref{subsec:setup}.
	Next, we impose conformal and affine symmetry constraints which yield linear relations reducing the number of undetermined constants to just 24.
	Finally, we impose the remaining Jacobi equations in $J_{11}$, which allow us to express all structure constants as rational functions of 2 parameters $c$ and $k$, see Proposition \ref{prop:base case proposition}.

	\subsection{New generators}	
	We begin by introducing a modified strong generating set $\bigcup\{\tilde W^{i}|i\geq 1\}$ that is adapted for the purposes of the base case computation, and is used throughout this subsection. First, the affine fields $W^{1}$ cannot be corrected and remain fixed, thus $\tilde W^{1} =W^{1}$. Next, the conformal vector $W^2$ must be corrected to the $\fr{sp}_2$-coset Virasoro field $\tilde W^2 = W^2- L^{\fr{sp}_2}$, where we recall
	\[L^{\fr{sp}_2}= \frac{1}{2(k+2)}\Big(\frac{1}{2}:\!H^1H^1\!:+2:\!X^1Y^1\!: - \partial H^1\Big).\]
The weight three fields $W^{3}$ all remain the same, since they are primary for $V^k(\fr{sp}_2)$, so $\tilde W^{3} = W^{3}$.
	\begin{lemma}\label{deform}
		Up to a scaling parameter, there is unique correction $\tilde W^4$ of $W^4$ so that it commutes with $V^k(\gs\gp_2)$, and is primary for $\tilde W^2$. It has the explicit form
		\begin{equation}\label{W4deformed}
			\tilde{W}^4 = W^4 - \frac{1}{k+4}\Big(\frac{1}{2}:\!H^1H^3\!:+:\!X^1Y^3\!:+:\!Y^1X^3\!:\Big).
		\end{equation}
	\end{lemma}
	\begin{proof}
		Let $\tilde W^4 = W^4 +\dotsb$, where the omitted terms are normally ordered products of monomials in $\{W^{j}| j =1,2,3\}$ and their derivatives. Imposing the desired constraints, namely that
		\[\tilde{W}^{1}(z) \tilde W^{4} (w) \sim 0,  \quad \tilde W^2 (z) \tilde W^4(w) \sim 4 \tilde W^4(w) (z-w)^{-2}  + \partial \tilde W^4(w) (z-w)^{-1},\]
		we obtain (\ref{W4deformed}).
	\end{proof}
	Next, we define the fields of higher conformal weight inductively via the raising property: 
		\begin{equation}\label{raise}
		\tilde{W}^4_{(1)}\tilde{W}^{n} = \tilde{W}^{n+2},\quad n\geq 3.
	\end{equation}

	We may rewrite our assumptions on the algebra $\Wsp$ in terms of the new generators as follows. 
	\begin{enumerate}
		\item Weight one fields $\tilde{W}^{1}$ generate the universal affine algebra $V^k(\fr{sp}_2)$ of level $k$:
		\[\tilde{W}^{1}(z)\tilde{W}^{1}(w) \sim \epsilon^{2,2}_0k  (z-w)^{-2}+ \epsilon^{2,2}_2 \tilde{W}^{1}(w) (z-w)^{-1}.\]
		\item Weight two field $W^{2}$ generates the universal Virasoro algebra of central charge $c$.
		So $\tilde W^2 = W^2 - L^{\fr{sp}_2}$ generates the Virasoro algebra with $\tilde c = c - \frac{3k}{k+2}$:
		\begin{equation}
			\begin{split}
				\tilde W^2(z)\tilde W^2_{0,0}(w) \sim& \frac{\tilde c}{2}(z-w)^{-4}+2 \tilde W^2(w)(z-w)^{-2}(w)+\partial \tilde W^2(w)(z-w)^{-1}.
			\end{split}
		\end{equation}
		In particular, the action on higher weight generators is as follows
		\begin{equation}
			\begin{split}
				\tilde W^2(z)\tilde W^{2n}(w) \sim\dotsb&+2n \tilde W^{2n}(w)(z-w)^{-2}(w)+\partial \tilde W^{2n}(w)(z-w)^{-1}, \ n\geq 3,\\
				\tilde W^2(z)\tilde W^{2n-1}(w) \sim\dotsb&+(W^2_{(1)}\tilde W^{2n-1}-L^{\fr{sp}_2}_{(1)}\tilde W^{(2n-1)}) (w)(z-w)^{-2}\\
				&+(W^2_{(0)}\tilde W^{2n-1}-L^{\fr{sp}_2}_{(0)}\tilde W^{2n-1} (w)(z-w)^{-1} , \quad n\geq 3.
			\end{split}
		\end{equation}
		\item Using $\tilde W^2=W^2-L^{\fr{sp}_2}$ we compute its action on $\tilde{W}^{3}$:
		\[ \tilde W^2(z) \tilde{W}^{3}(w) \sim \frac{3k+4}{k+2}\tilde{W}^{3}(w)(z-w)^{-2}+ \left(\partial \tilde{W}^{3}-\frac{1}{k+2}W^{(1,3)}_{2}\right)(w)(z-w)^{-1}.\]
		\item By Lemma \ref{deform}, weight four field $\tilde W^4$ is primary for $\tilde W^2$:
		\[\tilde W^2(z)\tilde{W}^4(w) \sim 4\tilde{W}^4(w)(z-w)^{-2}+ \partial\tilde{W}^4(w)(z-w)^{-1}.\]
		\item By (\ref{raise}), fields $\tilde{W}^{2n}$ and $\tilde{W}^{2n+1}$ are primary for $V^k(\fr{sp}_2)$:
		\begin{equation}\label{affine primary}
			\tilde{W}^{1}(z)\tilde{W}^{2n}(w) \sim 0, \quad \tilde{W}^{1}(z)\tilde{W}^{2n+1}(w)\sim \epsilon^{2,2}_2 \tilde{W}^{2n+1}(w)(z-w)^{-1}, \quad n\geq 1.
		\end{equation}
		\item $\tilde W^4$ satisfies the raising property (\ref{raise}):
		\[\tilde W^4(z)\tilde W^{n}(w)\sim \dotsb +\tilde W^{n+2}(w)(z-w)^{-2}+\dotsb,\quad  n\geq 2.\]
	\end{enumerate}
	For the remaining OPEs, we posit they have the most general form compatible with the conformal weight gradation and $\fr{sp}_2$-symmetry, as in (\ref{OPEgeneral}).
	For clarity, we specialize \eqref{OPEgeneral} for each type of interaction
	\begin{equation}\label{OPEs}
		\begin{split}
			\tilde W^{2i}(z)\tilde W^{2j}(w) &\sim \sum_{r=0}^{2i+2j-1}\left(V^{2i,2j}_{0}(r)\right)(w) (z-w)^{-r-1},\\
			\tilde W^{2i}(z)\tilde W^{2j-1}(w) &\sim \sum_{r=0}^{2i+2j-2}  \left(V^{2i,2j-1}_{2}(r)\right)(w) (z-w)^{-r-1},\\
			\tilde W^{2i-1}(z)\tilde W^{2j-1}(w) &\sim\sum_{r=0}^{2i+2j-3}\sum_{\gamma=0,2,4} \left( V^{2i-1,2j-1}_{\gamma}(r)\right)(w) (z-w)^{-r-1}.
		\end{split}
	\end{equation}
	In (\ref{OPEs}), $V^{n,m}_{\mu}(r)$ represents the general linear combination of all PBW monomials of conformal weight $n+m-r$ transforming as the $\rho_{\mu}$ $\fr{sp}_2$-module.
	As an example, we have used Table \ref{tab:adjointadjoint} to write the first order poles among weights 3 and 5 fields.

	\begin{figure}[h]
		\centering
		\begin{footnotesize}
		\begin{tabular}{|c|c|c|c|}
			\hline
			$W^3_{(0)}W^5$&$Y^5$&$X^5$&$H^5$\\
			\hline
			$Y^3$&$ 6 V^{3,5}_{-4,4}(0)$&$ V^{3,5}_{0,0}(0)-V^{3,5}_{2,0}(0)+V^{3,5}_{4,0}(0)$&$2V^{3,5}_{2,-2}(0)+3V^{3,5}_{4,-2}(0)$\\
			\hline
			$X^3$&$V^{3,5}_{0,0}(0)+V^{3,5}_{2,0}(0)+V^{3,5}_{4,0}(0)$&$ 6 V^{3,5}_{4,4}(0)$&$-2V^{3,5}_{2,2}(0)+3V^{3,5}_{4,2}(0)$\\
			\hline
			$H^3$&$2V^{3,5}_{2,-2}(0)+3V^{3,5}_{4,-2}(0)$&$-2V^{3,5}_{2,2}(0)+3V^{3,5}_{4,2}(0)$&$-2V^{3,5}_{2,0}(0)+4V^{3,5}_{4,0}(0)$\\
			\hline
		\end{tabular}
		\end{footnotesize}
	\end{figure}

	\subsection{Affine and conformal symmetry}
	Let $\tilde W^{n}$ and $\tilde W^{m}$ be two generating fields of $\Wsp$.
	In our ansatz (\ref{OPEgeneral}) we have already imposed $\fr{sp}_2$-symmetry and the conformal weight grading conditions. Now we impose the full affine and conformal symmetry constraints on OPEs $\tilde W^{n}(z)\tilde W^{m}(w)$ for all OPEs in $D_9$. This is equivalent to the vanishing of Jacobi identities $J(\tilde{W}^{1},\tilde W^{n},\tilde W^{m})$ and $J(\tilde W^2,\tilde W^{n},\tilde W^{m})$.
	These give rise to linear constraints among the undetermined structure constants.
	Doing so, we find the following.
	\begin{equation}\label{conformal ansatz}
		\tilde W^{n}(z)\tilde W^{m}(w) \sim \sum_{\gamma=0,2,4}\epsilon^{\bar n,\bar m}_{\gamma} \sum_{\omega\in U_{\gamma}(r)}\sum_{\chi\in\Omega^{\omega}_{\gamma}} \tilde w^{n,m}_{\omega_{\chi},\gamma}\sum_{\Lambda} \big(
		\beta^{n,m}_{\gamma,\Delta(\omega_{\chi})}(\Lambda):\! \Lambda \tilde W^{\omega_{\chi}}_{\gamma}\!: (w) (z-w)^{\Delta(\Lambda)+\Delta(\omega)-n-m}\big).
	\end{equation}
	Here $\tilde W^{\omega_{\chi}}_{\gamma}$ are certain fields which we define later in (\ref{data}) and (\ref{fields}).
	Fields $\Lambda$ are contained in $V^k(\fr{sp}_2)\otimes \vir$, and  $\beta^{n,m}_{\gamma,\Delta(\omega_{\chi})}(\Lambda)$ are rational functions of $c$, $k$, and $w^{a,b}_{\alpha_{\xi},\gamma}$ for $a+b\leq \Delta(\omega_{\chi})$, see Remark (\ref{rem:prim}).
	In particular, we have $\beta^{n,m}_{\gamma,\Delta({\omega}_{\chi})}(\B{1})=1$.
	Accordingly, if any additional parameters in the OPE algebra of $\Wsp$ exist, they must arise as structure constants $\tilde w^{n,m}_{\omega_{\xi},\gamma}$.
	Since our aim to show that $\Wsp$ is a 2-parameter vertex algebra, all the relevant data is contained in these constants.
	Therefore we extract the terms in equation (\ref{conformal ansatz}) with $\Lambda=\B{1}$, and use the following shorthand notation, as in \cite{CV}.
	\begin{equation}\label{OPEs refined}
		\tilde W^{n} \times \tilde W^{m} =\sum_{\gamma=0,2,4} \sum_{\omega \in U_{\gamma}}\sum_{\chi \in \Omega^{\omega}_{\gamma}} \epsilon^{\bar{n},\bar{m}}_{\gamma} \tilde{w}^{n,m}_{\omega_{\chi},\gamma}W^{\omega_{\chi}}_{\gamma}.
	\end{equation}
	The above shorthand (\ref{OPEs refined}) represents the notation (\ref{conformal ansatz}), which itself stands for the OPEs (\ref{OPEgeneralG}) among fields which are the basis of Cartan weight spaces in $U^{n}_{\bar n}$ and $U^m_{\bar m}$ as in (\ref{convention Walpha}). To extract fields on the right hand side of (\ref{OPEs refined}), we follow the same procedure as explained in (\ref{explain}). Using the above notation, the imposition of affine and conformal symmetry for OPEs in $D_9$ leaves us with following undetermined constants.
	\begin{equation}\label{data}
		\begin{split}
			\tilde W^{4} \times \tilde W^{4}  =& \tilde w^{4,4}_0 \B{1} + \tilde w^{4,4}_4 \tilde W^{4} + \tilde W^{6},\\
			\tilde W^{4} \times \tilde W^{3}  =& \tilde w^{4,3}_3 \tilde{W}^{3} + \tilde W^{5},\\
			\tilde{W}^{3} \times \tilde{W}^{3}  =& \epsilon^{2,2}_0 \tilde w^{3,3}_0 \B{1} +\epsilon^{2,2}_2 \tilde w^{3,3}_3  \tilde{W}^{3} + \epsilon^{2,2}_0 \tilde w^{3,3}_4\tilde W^{4}+\epsilon^{2,2}_2 \tilde w^{3,3}_5\tilde W^{5},\\
			\tilde W^{4} \times \tilde W^{5}  =& \tilde w^{4,5}_3  \tilde{W}^{3} +\tilde w^{4,5}_5 \tilde W^{5} + \tilde W^{7},\\
			\tilde{W}^{3} \times \tilde W^{6}  =& \tilde w^{3,6}_3  \tilde{W}^{3} + \tilde w^{3,6}_5 \tilde W^{5}+\tilde w^{3,6}_7 \tilde W^{7} + \tilde w^{3,6}_{(3,4),2} \tilde W^{(3,4)}_2\\
			&+ \tilde w^{3,6}_{(3,5),2}\tilde W^{(3,5)}_2+\tilde w^{3,6}_{(1,3,4),2}\tilde W^{(1,3,4)}_2,\\
			\tilde{W}^{3} \times \tilde W^{5}  =&   \epsilon^{2,2}_0 \tilde w^{3,5}_0 \B{1} + \epsilon^{2,2}_2 \tilde w^{3,5}_3  \tilde{W}^{3} +  \epsilon^{2,2}_0 \tilde w^{3,5}_4 \tilde W^{4} +  \epsilon^{2,2}_0 \tilde w^{3,5}_6 \tilde W^{6}+ \epsilon^{2,2}_0 \tilde w^{3,5}_{(3,3),0}\tilde W^{(3,3)}_0 \\
			&+ \epsilon^{2,2}_4 \tilde w^{3,5}_{(3,3),4}\tilde W^{(3,3)}_4 + \epsilon^{2,2}_2 \tilde w^{3,5}_7 \tilde W^{7} +\epsilon^{2,2}_2 \tilde w^{3,5}_{(3,4),2} \tilde W^{(3,4)}_2 + \epsilon^{2,2}_4 \tilde w^{3,5}_{(1,3,3),4}\tilde W^{(1,3,3)}_4.\\
		\end{split}
	\end{equation}
	Here, we have the $\fr{sp}_2$-modules generated by the highest weight vectors
	\begin{equation}\label{fields}
		\begin{split}
			\tilde W^{(3,3)}_{4,4} =& :\!\tilde X^3\tilde X^3\!:,\\
			\tilde W^{(3,3)}_0 =& 2:\!\tilde X^3\tilde Y^3:\!+\frac{1}{2}:\!\tilde H^3\tilde H^3\!:+\dotsb,\\
			\tilde W^{(3,4)}_{2,2} =& :\!\tilde X^3\tilde W^4\!:,\\
			\tilde W^{(1,3,3)}_{4,4} =& :\!X^1\tilde X^3\tilde H^3\!:- :\!H^1\tilde X^3\tilde X^3\!:,\\
			\tilde W^{(1,3,4)}_{2,2} =& :\!X^1 \tilde H^3\tilde W^4\!:- :\!H^1\tilde X^3\tilde W^4\!:,\\
		\end{split}
	\end{equation}
	where the omitted terms in $\tilde W^{(3,3)}_0$ represent normally ordered monomials in $\{\tilde W^{i}|\ i\leq 5\}$ and their derivatives, which are necessary to make $	\tilde W^{(3,3)}_0$ $\fr{sp}_2$-invariant.
	
	\begin{remark}\label{rem:prim}
		It is a well-known fact \cite{BPZ} that conformal symmetry fixes the coefficients of all fields appearing in the OPEs $W^{n}(z)W^{m}(w)$ in terms of the coefficients of Virasoro primaries only, given that fields $W^{n}$ are primary. 
		Our computations suggests that a similar fact holds when symmetry is enlarged to  $V^k(\fr{sp}_2)\otimes \vir$. 
		While the generators $W^n_{\mu}$ for $n\geq 5$ are not primary for the symmetry algebra, they can be uniquely corrected to primary vectors, at least for $n\leq 7$. Then, the functions $\beta^{n,m}_{\gamma,\Delta(\omega_{\chi})}(\Lambda)$ arising in (\ref{conformal ansatz}) are rational functions of $c$ and $k$ alone.
	\end{remark}
	
	\subsection{Nonlinear constraints}
	The imposition of Jacobi identities \[J(\tilde W^{3},\tilde W^{3},\tilde W^{3}), \ J(\tilde{W}^4, \tilde{W}^{3}, \tilde{W}^{3}),\ J(\tilde{W}^4,\tilde{W}^4,\tilde{W}^{3})\] allows us to express all of the structure constants in (\ref{data}) as rational functions in the central charge $c$ and $\fr{sp}_2$-level $k$.
	Specifically, we obtain a system of quadratic equations.
	The solution breaks up in two parts, which is explained by the automorphism $\sigma$.
	We choose our scaling to eliminate any square roots in the OPE algebra, so that it is defined over a localization of polynomial ring; see (\ref{loc}) and (\ref{localization}). Finally, we replace the new generating set $\{\tilde W^{i}|i\geq 1\}$ with the original generating set $\{W^{i} | i\geq 1\}$, and we normalize the fields $W^{3}$ and $W^{4}$ to have the leading poles 
	\begin{equation}\label{scaling}
		\begin{split}
			w^{3,3}_0=&\frac{c (k-1) (k+2) \left(c k+2 c+6 k^2+11 k+4\right)}{576 k ( k+1) (2 k+1) (3 k+4)},\\
			w^{4,4}_0=&\frac{c (k-1) \left(c k+2 c+6 k^2+11 k+4\right)\alpha(c,k)}{20736 (5 c+22) k^2 (k+1)^2 (k+2)^2 (2 k+1)^2 (3 k+4)},\\
			\alpha(c,k) =&c^2(k+2)^2 (2 k-1) (2 k+3) (3 k+4) +c k(k+2)\big(\\
			& \left(192 k^4+1216 k^3+2510 k^2+1961 k+478\right) \\
			&+2 k (2 k+1) \left(96 k^4+1172 k^3+4014k^2+5311 k+2376\right)\big).
		\end{split}
	\end{equation}
	
	\begin{remark}
		Thanks to weight 3 fields $W^{3}$ the difficulty of this computation is comparable to that of $\Winf$ in \cite{Lin}. 
		It is also similar to $\Wev$ in \cite{KL}, in that some quadratic equations must be solved.
	\end{remark}
	
	Although the structure constants in the OPE algebra are not polynomials in $c$ and $k$, they have only finitely many poles.
	Specifically, these are contained in the set
	\begin{equation}\label{loc}
		\{		( 5 c+22), k,(2 k+1), (k+1), (k+2), (k+4),(3k+4),(5k+8),(7k+16)\}.
	\end{equation}
	Let $D$ be the multiplicatively closed set generated by \eqref{loc}, and let 	\begin{equation}\label{localization}
		R=D^{-1}\mathbb{C}[c,k]
	\end{equation}
	be the corresponding localization of $\mathbb{C}[c,k]$.

	\begin{proposition}\label{prop:base case proposition}
		\
		\begin{enumerate}
			\item Data $D_9$ is fully determined in terms of the central charge $c$ and $\fr{sp}_2$-level $k$. Moreover, all Jacobi identities $J_{11}$ vanish. 
			\item Under the scaling (\ref{scaling}) all structure constants appearing in $D_9$ are elements of the ring $R$.
			\item We have $w^{4,3}_{(15),2} =w^{4,5}_{(17),2} =w^{4,5}_{(35),2} =w^{3,6}_{(17),2}=w^{3,6}_{(35),2}=0$, and moreover
			\begin{equation}
				\begin{split}
					W^4_{(1)} W^{3} =& W^{(5)},\quad 	W^4_{(0)} W^{3} = \frac{3}{5}\partial  W^{5} +\dotsb,\\
					W^{4}_{(1)}W^{4} =&W^{6},\quad W^{4}_{(0)}W^{4} =\frac{1}{2}W^{6}+\dotsb,\\
					W^{3}_{(1)} W^{3} =&\epsilon^{2,2}_0 W^{4}+ \dotsb,\quad W^{3}_{(0)} W^{3} =\epsilon^{2,2}_2\frac{3}{5}W^{5} +\epsilon^{2,2}_0\frac{1}{2} \partial W^{4}+ \dotsb\\
					W^{3}_{(1)}W^5 =& \epsilon^{2,2}_0 \frac{5}{4}W^{6}+\dotsb,\quad W^{3}_{(0)}W^5 =\epsilon^{2,2}_2 \frac{3}{7}W^{5}+\epsilon^{2,2}_0 \frac{5}{12}\partial W^{6}+\dotsb,\\
					W^{4}_{(1)}W^{5} =&W^{7}+\dotsb,\quad W^{4}_{(0)}W^{5} = \frac{3}{7}\partial W^{7}+\dotsb,\\
					W^{6}_{(1)}W^{3} =&\frac{6}{5}W^{7}+\dotsb,\quad W^{6}_{(0)}W^{3} = \frac{6}{7}\partial W^{7}+\dotsb,\\
					W^{8}_{(1)}W^{1} =&\frac{16}{5}W^{7}+\dotsb,\quad W^{8}_{(0)}W^{1} = \frac{16}{5}\partial W^{7}+\dotsb,\\
				\end{split}
			\end{equation}
		\end{enumerate}
	\end{proposition}
	
	From now on, we will work with the original strong generating set $\{W^{i}|i\geq 1\}$, unless specified otherwise.
	
	\subsection{Step 2: constant structure constants}	
	Here, we consider consequences of the weak generation property (\ref{raising}).
	In Proposition \ref{prop:structure constants}, we find explicit forms for infinitely many structure constants specified in the following products.
	\begin{equation}
		\begin{split}
			W^{n}_{(1)}W^m=& \epsilon^{\bar n,\bar m}_{\overline{n+m-2}}w^{n,m}_{n+m-2}W^{n+m-2}+\dotsb.\\
		\end{split}
	\end{equation}
	In general, for the structure constants $w^{n,m}_{\omega_{\chi},\mu}$ in (\ref{OPEgeneral}) to be well-defined, we must make explicit the choice of basis indexed by the set $\Omega^{\omega}_{\mu}$ for the subspace $U^{\omega}_{\mu}$.
	Since we only need to solve for some structure constants, and not all, we make only a partial choice of basis.
	We must verify that the structure constants we define are independent of the choice of basis for the complementary subspace.
	To resolve this ambiguity, we introduce a decreasing $\fr{sp}_2$-invariant filtration on $\Wsp$., so that the complementary subspace belongs to a higher filtration degree. We define the {\it degree} of a PBW monomial as follows:
	\[\T{deg}(:\!\partial^{d_1}a_1\dotsb\partial^{d_n}a_n\!:)= n+ \sum_{i=1}^n d_i. \]
	Let $G_n$ denote the span of all PBW monomials of degree greater or equal to $n$.
	Then $G_n$ is a decreasing filtration
	\begin{equation}\label{filtration}
		\Wsp=G_0\supset G_1\supset G_2 \supset\dotsb.
	\end{equation}
	We have surjections
	\begin{equation}\label{surjections}
		\sigma_n:G_{n} \to G_{n}/G_{n+1},
	\end{equation}
	and we have the associated graded algebra
	\begin{equation}\label{graded}
		\T{gr}^G(\Wsp) := \bigoplus_{n=0}^{\infty} G_{n}/G_{n+1}.
	\end{equation}
	This filtration is not {\it good} in the sense of \cite{L}, and $\T{gr}^G(\Wsp) $ is neither commutative nor associative. However, the quasi-derivation identity (\ref{quasi-derivation}) and existence of a conformal weight grading structure implies that 
	\begin{equation}\label{complexity0}
		:\!a_{(-1)}b\!:\ \in G_{n+m},\ \text{for}\  a \in G_n \ \text{and}\ b\in G_m,
	\end{equation}
	and for all nonnegative $r^{\T{th}}$-products we have
	\begin{equation}\label{complexity1}
		a_{(r)}G_n \subset G_{n-r}.
	\end{equation}
	Note that $G_1/G_2$ is the span of all strong generators of the algebra, and $G_2/G_3$ is spanned by all strong generators with one derivative and normally ordered quadratics with no derivatives.
	
	Recall the subspace $U^{\omega}_{\mu}$ discussed in (\ref{sect:symmetry}), and consider $U^{\omega}_{\mu}\cap G_i$. The map $\sigma_i$ restricts to the intersection 
\[\sigma_i: U^{\omega}_{\mu}\cap G_i \to U^{\omega}_{\mu}\cap G_i /  U^{\omega}_{\mu} \cap G_{i+1} \subseteq G_i / G_{i+1}.\]
Fix some basis of $\sigma_i (U^{\omega}_{\mu}\cap G_i)$ and lift it to a set of linearly independent vectors  $\bigcup_{\alpha}\{W^{\omega_{\chi}}_{\mu,\alpha}|\chi \in\Omega^{\omega}_{\mu}\}$ in $U^{\omega}_{\mu}\cap G_i$ so that we have splitting of vector space
\[U^{\omega}_{\mu} = \T{Span}(\bigcup_{\alpha}\{W^{\omega_{\chi}}_{\mu,\alpha}| \chi \in \Omega^{\omega}_{\mu}\})\oplus C.\]
Here the complement $C$ lies in $ G_{i+1}\cap U^{\omega}_{\gamma,\gamma}$. Without the loss of generality, since the degree filtration (\ref{filtration}) is $\fr{sp}_2$-invariant, we can assume that this basis is compatible with the $\fr{sp}_2$-action (\ref{sp2 scaling}). Now the structure constants $w^{n,m}_{\omega_{\chi},\gamma}$ are well-defined, i.e. independent of the choice of  basis for $C$. In what follows we apply this procedure to the cases of $i=1$ and $i=2$, and $U^{\omega}_{\mu,\mu}$ is of dimension 1. Specifically, our partial basis will be given by (\ref{convention Walpha}) and (\ref{convention more}).

	We will need to impose Jacobi identities $J_{2,0}(W^{2i},W^{n},W^{m})$ of the form
	\begin{equation}\label{L2equation}
		J_{2,0}(W^{2i},W^{n},W^{m}) = \varepsilon^{2i,n,m}_{2,0} W^{2i+n+m-4}+\dotsb = 0,
	\end{equation}
	which will give rise to equation (\ref{20identity}).
	Property (\ref{complexity1}) together with the structure of conformal gradation implies that only linear terms with at most one derivative or quadratics terms with no derivatives may contribute to the coefficient of $W^{2i+n+m-4}$ in (\ref{L2equation}).
	Recall that we have already made this choice in (\ref{convention Walpha}) and (\ref{convention more}).
	We proceed to analyze the $3$ types of interactions: even with even, even with odd, and odd with odd weight fields.
	
	\begin{itemize}
		\item Even with even. The first order pole of $W^{2i}$ and $W^{2j}$ has odd conformal weight $2i+2j-1$ and transforms as the trivial $\fr{sp}_2$-module.
		So only $\partial W^{2i+2j-2}$ contributes in (\ref{L2equation}), and we write
		\begin{equation}\label{ansatzWW}
			\begin{split}
				W^{2i}_{(0)}W^{2j}=&v_{0}^{2i,2j}\partial W^{2i+2j-2}+W_{0}^{2i,2j}(0),\\
				W^{2i}_{(1)}W^{2j}=&v_{1}^{2i,2j}W^{2i+2j-2}+W_{0}^{2i,2j}(1),
			\end{split}
		\end{equation}
		where $W_{0}^{2i,2j}(-)$ are some normally ordered polynomial in the generators $\{W^{n}|n\leq 2i+2j-3\}$ and their derivatives.
		Here, we used a new variable for the structure constant $v^{2i,2j}_0= w^{2i,2j}_{(2i+2j-2)^1,0}$ and $v^{2i,2j}_1=w^{2i,2j}_{2i+2j-2,0}$.
		Note that the $\fr{sp}_2$-label is uniquely determined.
		
		\item  Even with odd. The first order pole of $W^{2i}$ and $W^{2j-1}$ has even conformal weight $2i+2j-2$ and transforms as the adjoint $\fr{sp}_2$-module.
		So $\partial W^{(2i+2j-3)}$ and quadratics $W^{(2l-1,2i+2j-2l-3)}_2$ contribute in (\ref{L2equation}), and we write
		\begin{equation}\label{ansatzWH}
			\begin{split}
				W^{2i}_{(0)}W^{2j-1}&=v_{0}^{2i,2j-1}\partial W^{2j+2i-3}+\sum_{l=0}^{\lfloor \frac{i+j}{2}\rfloor}q_{l}^{2i,2j-1}W^{(2l-1,2i+2j-2l-3)}_2+W_{2}^{2i,2j-1}(0),\\
				W^{2i}_{(1)}W^{2j-1}&=v_{1}^{2i,2j-1} W^{2j+2i-3}+W^{2i,2j-1}_{2}(1),
			\end{split}
		\end{equation}
		where $W^{2i,2j-1}_{2}(-)$ are some normally ordered polynomial in generators $\{W^{n}|n\leq 2i+2j-4 \}$.
		Here, we used a new variable for the structure constant $v^{2i,2j-1}_0=	w^{2i,2j-1}_{(2i+2j-1)^1,2}$, $v^{2i,2j-1}_1=w^{2i,2j}_{2i+2j-1,2}$ and $q^{2i,2j-1}_l=w_{(2l-1,2i+2j-2l-3),2}^{2i,2j-1}$.
		Note that the $\fr{sp}_2$-label is uniquely determined.
		\item Odd with odd. The first order pole of $W^{2i-1}$ and $W^{2j-1}$ has odd conformal weight $2i+2j-3$ and transforms under $\fr{sp}_2$ as the $\rho_0\oplus \rho_2 \oplus \rho_4$.
		So only $\partial W^{(2i+2j-3)}$ contributes in (\ref{L2equation}), and we write
		\begin{equation}\label{ansatzXH}
			\begin{split}
				W^{2i-1}_{(0)}W^{2j-1}=&\epsilon^{2,2}_2 a^{2i-1,2j-1} W^{2j+2i-3}+\epsilon^{2,2}_0 v_{0}^{2i-1,2j-1}\partial W^{2i+2j-4}\\
				&+ W_{0}^{2i-1,2j-1}(0)+W_{2}^{2i-1,2j-1}(0)+W_{4}^{2i-1,2j-1}(0),\\
				W^{2i-1}_{(1)}W^{2j-1}=&\epsilon^{2,2}_0v_{1}^{2i-1,2j-1} W^{2i+2j-4}+W_{0}^{2i-1,2j-1}(1)+ W_{2}^{2i-1,2j-1}(1)+ W_{4}^{2i-1,2j-1}(1),
			\end{split}
		\end{equation}
		where terms $W_{-}^{2i-1,2j-1}(-)$ are some normally ordered polynomial in the generators $\{W^{n} | n\leq 2i+2j-4\}$, and their derivatives.
		Here, we used a new variable for the structure constant $v^{2i-1,2j-1}_0=	w^{2i-1,2j-1}_{(2i+2j-4)^1,0}$, $v^{2i-1,2j-1}_1=w^{2i-1,2j-1}_{2i+2j-4,0}$ and $a^{2i-1,2j-1}=w_{2i+2j-3,2}^{2i-1,2j-1}$.
		Note that the $\fr{sp}_2$-label is uniquely determined.
	\end{itemize}
	
	\begin{remark}
		Ansatz (\ref{ansatzWW}) is formally the same as the one used in the construction of $\Wev$; see \cite[Section 3, Eq. 3.2 and Eq. 3.3]{KL}. Unlike the constructions of $\Winf$ and $\Wev$, 
		we see a quadratic $W^{(2l-1,2i+2j-2l-3)}_2$ arising in (\ref{ansatzWH}) and a derivative-free monomial $W^{2j+2i-3}$ in (\ref{ansatzXH}).
	\end{remark}
	Our notation above extends the one used for the structure constants (\ref{sp2motion}). 
	Specifically, if we denote by $W^{n,m}_{\gamma,\alpha+\beta}(r)(\alpha,\beta)$ the normally ordered differential polynomial, arising as above in $W^{n}_{\bar n,\alpha}(z)W^{m}_{\bar m,\beta}(w)$, then affine symmetry affords a factorization
	\[W^{n,m}_{\gamma,\alpha+\beta}(r)(\alpha,\beta) = \epsilon^{\bar n,\bar m}_{\gamma}(\alpha,\beta) W^{n,m}_{\gamma,\alpha+\beta}(r),\quad r=0,1.\]
	Therefore, up to $\fr{sp}_2$-symmetry, all relevant structure is contained in the constants with the form $v^{n,m}_{0}$, $v^{n,m}_{1}$ and $a^{2i-1,2j-1}$, and normally ordered differential polynomials $W^{n,m}_{\gamma,\alpha+\beta}(r)$.
	Denote the double factorial by
	$$a!!=\begin{cases} 
		(2n-1)(2n-3)\dotsb 3, & a=2n-1,\\
		(2n)(2n-2)\dotsb 2, & a=2n.
	\end{cases}$$

\begin{proposition} \label{prop:structure constants} Let the notation be fixed as in (\ref{ansatzWW}-\ref{ansatzXH}). We have the following expressions for the first and second order poles among the strong generators.
		\begin{enumerate}
		\item Even and even weight fields. 
			\begin{equation*}
				\begin{split}
					W^{2i}_{(0)}W^{2j}=&\frac{(2i-1)(2i)!!(2j)!!}{8(2i+2j-2)(2i+2j-4)!!}\partial W^{2i+2j-2}+W^{2i,2j}_{0}(0),\\
					W^{2i}_{(1)}W^{2j}=&\frac{(2i)!!(2j)!!}{8(2i+2j-4)!!}W^{2i+2j-2}+W^{2i,2j}_{0}(1), \quad i\geq2.
				\end{split}
			\end{equation*}
			\item Even and odd weight fields. 
			\begin{equation*}
				\begin{split}
					W^{2i}_{(0)}W^{2j-1}&=\frac{(2i-1)(2i)!!(2j-1)!!}{8(2i+2j-3)(2i+2j-5)!!}\partial W^{2j+2i-3}+W_{2}^{2i,2j-1}(0),\\
					W^{2i}_{(1)}W^{2j-1}&=\frac{(2i)!!(2j-1)!!}{8(2i+2j-5)!!} W^{2j+2i-3}+W_{2}^{2i,2j-1}(1), \quad i\geq2.
				\end{split}
			\end{equation*}
			\item Odd and odd weight fields.
			\begin{equation*}
				\begin{split}
					W^{2i-1}_{(0)}W^{2j-1}=&\epsilon^{2,2}_{2}\frac{(2i-1)!!(2j-1)!!}{(2i+2j-3)!!} W^{2j+2i-3}+\epsilon^{2,2}_{0}\frac{2(2i-2)(2i-1)!!(2j-1)!!}{9(2i+2j-4)(2i+2j-6)!!}\partial W^{2i+2j-4}\\
					&+W_{0}^{2i-1,2j-1}(0)+W_{2}^{2i-1,2j-1}(0)+V_{4}^{2i-1,2j-1}(0),\\
					W^{2i-1}_{(1)}W^{2j-1}=&\epsilon^{2,2}_{0}\frac{2(2i-1)!!(2j-1)!!}{9(2i+2j-6)!!} W^{2i+2j-4}+W_{0}^{2i-1,2j-1}(1)\\
					&+W_{2}^{2i-1,2j-1}(1)+W_{4}^{2i-1,2j-1}(1).
				\end{split}
			\end{equation*}
		\end{enumerate}
	\end{proposition}
	
	\begin{proof}
		We will proceed by induction on $N$.
		Our base case is Proposition \ref{prop:base case proposition}, with $N=3$.
		Inductively, assume that all structure constants defined in (\ref{ansatzWW}-\ref{ansatzXH}) have the form as in Proposition \ref{prop:structure constants} for products in $D^{2N+1}_0\cup D^{2N+2}_0$.
		In particular, it means that all $\{q_{l}^{2i,2j-1}| 2i+2j-1 =2N+1,l\geq 1\}$ vanish.
		First, we will show that constants $q_{l}^{2i,2j+1}$ arising in $D^{2N+3}_0$ vanish.
		
		To this end, let $i,j,l$ be integers such that $2i+2j+2l=2N+6$ and $l\geq 2$, and
		consider the Jacobi identity $J_{0,0}(W^{2l},W^{2i-1},W^{2j-1})$, which has the form
		\[W^{2l}_{(0)}(W^{2i-1}_{(0)}W^{2j-1}) =W^{2i-1}_{(0)}(W^{2l}_{(0)}W^{2j-1})+ (W^{2l}_{(0)}W^{2i-1})_{(0)}W^{2j-1}.\]
		By induction only the left side can contribute a nonzero quadratic expression (\ref{ansatzXH}). However, since $w^{2i-1,2j-1}_{2i+2j-3}\neq 0$, it follows that quadratics do not arise in the product $W^{2l}_{(0)}W^{2i+2j-3}$.
		It remains to show that $W^{2N}_{(0)}W^{3}$ has no nonzero quadratics.
		This can be done by imposing the Jacobi identity $J_{1,0}(W^{4},W^{2N-2},W^{3})$ which has the form
		\begin{equation}
			\begin{split}
				W^{4}_{(1)}(W^{2N-2}_{(0)}W^{3}) =& W^{2N-2}_{(0)}W^{5}+W^{2N-2}_{(0)}W^{3}+(W^{4}_{(0)}W^{2N-2})_{(1)}W^{3}\\
				=&W^{2n-2}_{(0)}W^{5}+(1-v^{4,2N-2}_{0})W^{2N}_{(0)}W^{3}+(W^{4,2N-2}_0(0))_{(1)}W^{3}.
			\end{split}
		\end{equation}
		Since by induction $v^{4,2N-2}_{0} \neq 1$, $W^{2N}_{(0)}W^{3}$ has no nonzero quadratics. Finally, the term $W^{2N+2}_{(0)}W^{1}$ has no quadratics since $W^{2N+2}$ transforms as the trivial $\fr{sp}_2$-module.
		
		Now, assuming that no quadratics arise in $D^{2N+3}_0$, we can show that the structure constants have the desired form in a uniform manner.
		Let $m\neq 1$ and $n\neq 2$, or $m=1$ and $n$ is even, see Remark (\ref{comment}), and so that $2l+n+m\leq 2N+6$.
		Extracting the coefficient of fields $W^{2l+n+m-4}$ in identities $J_{2,0}(W^{2l},W^{n},W^{m})$ gives a relation
		\begin{equation}\label{20identity}
			v^{2l,n+m-2}_{1} v^{n,m}_{0} + (v^{2l, n}_{0} - v^{2l,n}_{1}) v^{2l+n-2, m}_{1}=0.
		\end{equation}
		
		Set $l=1$ in (\ref{20identity}), and recall that the Virasoro action implies that $v^{2,n}_{0}=1$ and $v^{2, n}_1=n$ for $n\geq 1$. 
		Thus we find
		\begin{equation}\label{tempEq}
		v^{n,m}_{0}=\frac{n-1}{n+m-2}v^{n,m}_{1}.
		\end{equation}
		Next, set $l=2$ in (\ref{20identity}) and using the above relation (\ref{tempEq}) we obtain a recurrence 
		\begin{equation}\label{inductLeft}
			v^{n+2, m}_{1}=-\frac{(n-1)v^{4,n+m-2}_{1}}{(n+m-2)(\frac{3}{n+2}v^{4, n}_{1} - v^{4,n}_{1})}v^{n,m}_{1} =\frac{n+2}{n+m-2}v^{n,m}_{1},
		\end{equation}
		where we have used the raising property $v^{4,n}_{1}=v^{4,n+m-2}_{1}=1$.
		Thanks to skew-symmetry we exchange indices $n$ and $m$ to obtain
		\begin{equation}\label{inductRight}
			v^{n,m+2}_{1}=\frac{m+2}{n+m-2}v^{n,m}_{1}.
		\end{equation}
		This proves the inductive hypothesis for $v^{n,m}_{1}$ and $v^{n,m}_{0}$. 
		Lastly, we evaluate $a^{2i-1,2l-1}$.
		Let $i,l$ be integers such that $2l+2i-2 = 2N$.
		We extract the coefficient of $X^{2N-1}$ in Jacobi identity $J_{0,1}(W^{2l},X^{1},H^{2i-1})$, and obtain a relation 
		\begin{equation}\label{solving for x}
			v^{2l,1}_{0} a^{2l-1,2i-1} = -2v^{2l,2i-1}_{1}.
		\end{equation}
		Since we have already determined $v^{2l,1}_{0}$ and $v^{2l,2i-1}_{0}$, the above allows us to solve for $a^{2i-1,2l-1}$.
	\end{proof}

	\begin{remark}\label{comment}
		When $n=2j-1$ and $m=1$, identity (\ref{20identity}) develops a contribution from structure constant of the monomial $:\!\partial H^1 W^{2i+2j-4}\!:$ arising in $W^{2i}_{(0)}W^{2j-1}$.
		Though it can computed exactly, we do not require it for this proof.
	\end{remark}

	\subsection{Step 3: Induction}\label{induction}
	The main result of this subsection is Theorem \ref{thm:induction}. 
	It is proved by induction, and  the process is similar to that of \cite{Lin} and \cite{KL}.
	Our base case is Proposition \ref{prop:base case proposition}.
	By inductive data we mean the set of OPEs in $D_{2n}\cup D_{2n+1}$, and that they are fully expressed in terms of parameters $c$ and $k$. 
	At this stage, the OPEs in $D^{2n+2}\cup D^{2n+3}$ are yet undetermined.
	We will use a subset of Jacobi identities in $J^{2n+4}\cup J^{2n+5}$ to express $D^{2n+2}\cup D^{2n+3}$ in terms of inductive data.
	We write $A\equiv 0$ to denote that $A$ is computable from inductive data $D_{2n}\cup D_{2n+1}$.
	
	\begin{lemma}\label{recursionsSym}
		Let $n$ be a positive integer. Then we have the following.
		\begin{enumerate}
			\item OPEs $W^4(z)W^{n}(w)$ and $L(z)W^{n}(w)$ together determine $L(z)W^{n+2}(w)$.
			\item OPEs $W^{3}(z)W^{n}(w)$ and $W^{1}(z)W^{n}(w)$ together determine $W^{1}(z)W^{n+2}(w)$.
		\end{enumerate}
	\end{lemma}
	
	\begin{proof}
		The Jacobi identity $J_{1,r}(W^4,W^{1},W^{n})$ gives rise to the following relation
		\[ W^{1}_{(r)}W^{n+2}=W^4_{1}W^{1}_{(r)}W^{n}-rW^{3}_{(r)}W^{n}.\]
		Note it also reproduces the affine action when restricted to $r=0$.
		Similarly, the Jacobi identity $J_{1,r}(W^4,L,W^{n})$ gives rise to a relation
		\[L_{(r)}W^{n+2} = W^4_{(1)}L_{(r)}W^{n}-(3r-1)W^4_{(r)}W^{n},\]
		which also reproduces the Virasoro action when restricted to $r=1$ and $r=0$.
	\end{proof}
	
	In Lemmas \ref{lemma:recursions0}-\ref{lemma:recursions3}, we use the raising property of $W^4$ to establish the following.
	\begin{proposition}\label{aa}
		\
		\begin{enumerate}
			\item $W^4(z)W^{2n-2}(w)$ with inductive data $D_{2n}$ together determine $EE^{2n+2}$.
			\item $W^4(z)W^{2n-2}(w)$ and $W^{3}(z)W^{2n-1}(w)$ with inductive data $D_{2n}$ together determine $OO^{2n+2}$.
			\item $W^4(z)W^{2n-1}(w)$ with inductive data $D_{2n+1}$ together determine $EO^{2n+3}$.
		\end{enumerate}
	\end{proposition}
	
	First, we consider the first order poles.
	In Proposition \ref{prop:structure constants}, we have already determined some structure constants arising in the first order poles.
	Therefore, to determine first order poles it is sufficient to analyze the normally ordered differential monomials $W^{i,j}_{0,0}(0)$, $W^{i,j}_{2,0}(0)$, and $W^{i,j}_{4,0}(0)$, defined in (\ref{ansatzWW}-\ref{ansatzXH}).
	Note that by Proposition \ref{prop:structure constants}, the coefficients of $\partial W^{i+j}$ arising in Jacobi identities $J_{1,0}(W^4,W^{i},W^{j})$ vanish.
	
	\begin{lemma}\label{lemma:recursions0}
		Modulo the inductive data, we have that
		\begin{enumerate}
			\item $W^4_{(0)}W^{2n-2}$  determine $\{W^{2i+4}_{(0)}W^{2n-2i-2} | i\geq 1\}$.
			\item $W^4_{(0)}W^{2n-1}$  determine $\{W^{2i+4}_{(0)}W^{2n-2i-1} | i\geq 1\}$.
			\item $W^{3}_{(0)}W^{2n-1}$ and $W^4_{(0)}W^{2n-2}$ determine $\{ W^{2i+3}_{(0)}W^{2n-2i-1}| i\geq 1\}$.
		\end{enumerate}
	\end{lemma}
	
	\begin{proof}
		Consider a general Jacobi identity $J_{1,0}(W^4,W^{i},W^{j})$ which reads
		\begin{equation}\label{Recurion0Full}
			(v^{4,i}_{1}-v^{4,i}_{0})W^{i+2,j}_{\mu}(0) =
			\epsilon^{\bar i,\bar j}_{\mu}v^{i,j}_{0}  W^{4,i+j-2}_{\mu}(0)-W^{i,j+2}_{\mu}(0)+R^{4,i,j}_{1,0}, \quad \mu=0,2,4,
		\end{equation}
		where 
		\begin{equation}
			R_{1,0}^{4,i,j}=\epsilon^{\bar i,\bar j}_{\mu}W^4_{(1)}\left(W^{i,j}_{\mu}(0)\right) - \left(W^{4,i}_{\bar i}(0)\right)_{(1)}W^{j}, \quad \mu=0,2,4.
		\end{equation}
		Note that $R_{1,0}$ is known from inductive data, so we may write
		\begin{equation}\label{Recurion0Partial}
			W^{i+2,j}_{\mu}(0)\equiv \epsilon^{\bar i,\bar j}_{\mu}\frac{v^{i,j}_{1}}{1-v^{4,i}_{0}} W^{4,i+j-2}_{\mu}(0)-\frac{1}{1-v^{4,i}_{0}}W^{i,j+2}_{\mu}(0),\quad \mu=0,2,4.
		\end{equation}
		Note that $1-v^{4,1}_{0} =0$, and hence the recursion (\ref{Recurion0Partial}) is not valid for $i=1$. 
		However, this is not an issue since the products $W^{1}_{(0)}W^{n}$ are known by our assumptions.
		Finally, we iterate recursion (\ref{Recurion0Partial}) and this gives rise to linear relations
		\begin{equation}\label{recursions0}
			\begin{split}
				W^{2l+4,2n-2l-2}_{0,0}(0) \equiv& p_{l}(n)W^{4,2n-2}_{0,0}(0) ,\\
				W^{2l+4,2n-2l-1}_{2,0}(0)  \equiv& q_{l}(n)W^{4,(2n-1)}_{2,0}(0) ,\\
				W^{2l+3,2n-2l-1}_{0,0}(0)  \equiv& h_{l}(n)W^{4,2n-2}_{0,0}(0) +d^{l}W^{3,2n-1}_{0,0}(0) ,\\
				W^{2l+3,2n-2l-1}_{2,0}(0)  \equiv& d^{l}W^{(3),(2n-1)}_{2,0}(0) ,\\
				W^{2l+3,2n-2l-1}_{4,0}(0)  \equiv& d^{l}W^{(3),(2n-1)}_{4,0}(0) ,
			\end{split}
		\end{equation}
		where $p_l(n),q_l(n),h_l(n)$ are some nonzero rational functions in $n$, and $d^{l}$ are constants.
	\end{proof}

	Next, we consider the second order poles.
	As before, by Proposition \ref{prop:structure constants} we have already determined some structure constants arising in the second order poles.
	Therefore, to determine second order poles it is sufficient to analyze the normally ordered differential monomials $W^{i,j}_{0}(1)$, $W^{i,j}_{2}(1)$, and $W^{i,j}_{4}(1)$.
	Moreover by Proposition \ref{prop:structure constants}, the coefficients of $W^{i+j}$ arising in Jacobi identities $J_{1,1}(W^4,W^{i},W^{j})$ vanish.
	\begin{lemma}\label{lemma:recursions1}
		Modulo the inductive data, we have that
		\begin{enumerate}
			\item $W^6_{(1)}W^{2n-4}$ determine $\{ W^{2i+4}_{(1)}W^{2n-2i-2} | i\geq 2\}$.
			\item $W^6_{(1)}W^{2n-3}$ determine $\{ W^{2i+4}_{(1)}W^{2n-2i-3}) | i\geq 2\}$.
			\item $W^{3}_{(1)}H^{2n-1}$ determine $\{W^{2i+3}_{(1)}W^{2n-2i-1})| i\geq 1\}$.
		\end{enumerate}
	\end{lemma}
	
	\begin{proof}
		Consider a general Jacobi identity $J_{1,1}(W^4,W^{i},W^{j})$ which reads
		\begin{equation}\label{Recursion1Full}
			(1-2v^{4,i}_{0})W^{i+2,j}_{\mu}(1) =W^{i,j+2}_{\mu}(1)+R_{1,1}^{4,i,j}, \quad  \mu=0,2,4,
		\end{equation}
		where
		\[R_{1,1}^{4,i,j}=W^4_{(1)}\left(W^{i,j}_{\mu}(1)\right)-\left(W^{4,i}_{\bar i}(0)\right)_{(2)}W^{j}.
		\]
		Note that $R_{1,1}^{4,(i),(j)}$ is known from inductive data, so we may write
		\begin{equation}\label{Recurion1Partial}
			W^{i+2,j}_{\mu}(1) \equiv \frac{1}{1-2v^{4,i}_{0}}W^{i,j+2}_{\mu}(1), \quad  \mu=0,2,4.
		\end{equation}
		Note that $1-2v^{4,4}_{0}=0$ and thus (\ref{Recurion1Partial}) is not valid for $i=2$. 
		However, this is not an issue since we have assumed that the raising property holds.
		Iterating the recursion (\ref{Recurion0Partial}) gives rise to the desired linear relations.
	\end{proof}
	
	Lastly, we consider the higher order products.
	\begin{lemma}\label{lemma:recursions3}
		Let $r>1$. 	Modulo the inductive data, we have that
		\begin{enumerate}
			\item $W^4_{(r)}W^{2n-2}$ determine $\{W^{2i+2}_{(r)}W^{2n-2i} | i\geq 2\}$.
			\item $W^{3}_{(r)}W^{2n}$ determine $\{W^{2i+3}_{(r)}W^{2n-2i-1} | i\geq 1\}$.
			\item $W^{3}_{(r)}W^{2n-1}$ determine $\{W^{2i+3}_{(r)}W^{2n-2i-1} | i\geq 1\}$.
		\end{enumerate}
	\end{lemma}
	
	\begin{proof}
		Consider Jacobi identity $J_{r,1}(W^4,W^{i},W^{j})$ which reads
		\begin{equation}\label{rRecution}
			\begin{split}
				(r-(r+1)v^{4,i}_{0})W^{i+2}_{(r)}W^{j} ={v^{i,j}_{1}} W^4_{(r)}W^{i+j-2}+ R_{r,1}^{4,i,j},
			\end{split}
		\end{equation}
		where \[R_{r,1}^{4,i,j}=W^4_{(r)}W^{i,j}_{\mu}(1)-\left(W^{4,i}_{\bar i}(0)\right)_{(r+1)}W^{j}-\sum_{j=2}^r\binom{r}{j}(W^4_{j}W^{i})_{(r+1-j)}W^{j}.\]
		Note that $R_{r,1}^{4,i,j}$ is known from inductive data, so we may write
		\begin{equation}\label{Recurionrpartial}
			W^{i+2}_{(r)}W^{j} \equiv \frac{v^{i,j}_{1}}{	r-(r+1)v^{4,i}_{0}}W^{4}_{(r)}W^{i+j-2}.
		\end{equation}
		Iterating recursions (\ref{Recurionrpartial}) gives rise to the desired linear relations.
	\end{proof}
	
	From now on we assume that Jacobi identities used in the proof of Proposition \ref{aa} have been imposed.
	In the next series of Lemmas \ref{0product}, \ref{1product}, and \ref{rproduct}, we write down a small set of Jacobi identities to obtain linear relations among the desired products. 
	This reduces our problem to solving a linear system of equations.
	Their proofs are similar, so we only provide an account for the most complicated case, which is part (1) of the following.
	
	\begin{lemma}\label{0product}
		$D^{2n+2}_0$ and $D^{2n+3}_0$ is determined from inductive data with $D^{2n+2}_1$ and $D^{2n+3}_1$. Specifically, assuming Lemma \ref{lemma:recursions0}, we have the following.
		\begin{enumerate}
			\item $J_{0,1}(W^4,H^3,H^{2n-3})$ and $J_{0,0}(X^3,Y^3,H^{2n-3})$ express $W^{4}_{(0)}W^{2n-2}$ and $H^{3}_{(0)}H^{2n-1}$ in terms of inductive data and  $OO^{2n+2}_1$.
			\item $J_{0,1}(W^4,X^3,H^{2n-3})$ expresses $X^3_{(0)}H^{2n-1}$ in terms of inductive data and $OO^{2n+2}_1$.
			\item $J_{0,1}(W^4,W^6,H^{2n-5})$ expresses $W^4_{(0)}H^{2n-3}$ in terms of inductive data together with $EE^{2n+2}_1$ and $EO^{2n+3}_1$.
		\end{enumerate}
	\end{lemma}
	\begin{proof}
		Here we only prove part (1), and the rest is similar.
		To determine all products in $D^{2n+2}_0$, it suffices to determine $W^{4,2n-2}_{0,0}(0)$, $W^{3,2n-1}_{0,0}(0)$, $W^{3,2n-1}_{2,0}(0)$ and $W^{3,2n-1}_{4,0}(0)$. 
		Expanding $J_{0,1}(W^4,H^3,H^{2n-3})$ and $J_{0,0}(X^3,Y^3,H^{2n-3})$, projecting onto $\rho_0$ component, and omitting the inductively known data we obtain two linear relations
		\begin{equation}\label{XYH}
			\begin{split}
				0\equiv&	v^{3,2n-3}_{1} W^{4,2n-2}_{0,0}(0)-v^{4, 2n-3}_{0} W^{3, 2n-1}_{0,0}(0) + 
				v^{4,3}_{0} W^{5, 2n-3}_{0,0}(0) -v^{4, 2n-3}_{0} \partial W^{3, 2n-1}_{0,0}(1)\\
				0\equiv&a^{3,3} W_{0,0}^{5, 2n-3} (0)-2v^{3,2n-3}_{1} W^{3,2n-1}_{0,0}(0).
			\end{split}
		\end{equation}
		Using recurrences obtained in Lemma \ref{recursions0}, we can express $W^{5,2n-3}_{0,0}(0)$ in terms of $W^{3,2n-1}_{0,0}(0)$, modulo inductive data. 
		Finally, we observe that two linear relations (\ref{XYH}) are linearly independent, and provide solutions
		\begin{align*}
			W^{4,2n-2}_{0,0}(0) \equiv&\frac{9 (2 n+3) (n-1)!}{4 n (2 n-1) (2 n+1) \left(\frac{1}{2}\right)_{n-1}} \partial W^{3,2n-1}_{0,0}(1),\\
			W^{3,2n-1}_{0,0}(0) \equiv&\frac{3}{n (2 n+1)} \partial W^{3,2n-1}_{0,0}(1).
		\end{align*}
		Thus products $W^{4}_{(0)}W^{2n-2}$ and $W^{3}_{(0)}W^{2n-3}$ are determined by inductive data together with $\partial W^{3,2n-1}_{0,0}(1) \in OO_1^{2n+2}$. 
		Similarly, the $\rho_2$ and $\rho_4$ components of the Jacobi identity $J_{0,0}(X^3,Y^3,H^{2n-3})$ express $W_{2,0}^{3,2n-1}(0)$ and $W_{4,0}^{3,2n-1}(0)$ in terms of inductive data.
		This completes the proof of part (1).
	\end{proof}
	
	\begin{lemma}\label{1product}
		Data $D^{2n+2}_1$ and $D^{2n+3}_1$ is determined from inductive data with $D^{2n+2}_2$ and $D^{2n+3}_2$. Specifically, we have the following.
		\begin{enumerate}
			\item $J_{0,2}(W^4,H^3,H^{2n-3})$ expresses $H^{3}_{(1)}H^{2n-1}$ in terms of inductive data and  $OO^{2n+2}_2$.
			\item $J_{0,2}(W^4,W^4,W^{2n-4})$ expresses $W^{6}_{(1)}W^{2n-4}$ in terms of inductive data and  $EE^{2n+2}_2$.
			\item  $J_{0,2}(W^4,W^6,H^{2n-5})$ expresses $W^6_{(1)}H^{2n-3}$ in terms of inductive data and $EO^{2n+3}_2$ and $EE^{2n+2}_2$.
		\end{enumerate}
	\end{lemma}
	
	\begin{lemma}\label{rproduct}
		Data $D^{2n+2}_r$ and $D^{2n+3}_r$ is determined from inductive data with $D^{2n+2}_{r+1}$ and $D^{2n+3}_{r+1}$. Specifically, we have the following.
		\begin{enumerate}
			\item Let $r>1$.  $J_{1,r}(W^{2n-1},H^3,H^3)$ and $J_{1,r}(W^4,H^{2n-3},W^{4})$ express $H^{3}_{(r)}H^{2n-1}$ in terms of inductive data and  $OO^{2n+2}_{r+1}$.
			\item Let $r>0$. $J_{r+1,0}(W^4,X^3,H^{2n-3})$ expresses $X^3_{(r)}H^{2n-1}$ in terms of inductive data. 
			\item Let $r>1$. $J_{1,r}(W^{2n-4},W^4,W^{4})$ and $J_{1,r}(W^4,W^{2n-4},W^{4})$ express $W^{4}_{(r)}W^{2n-2}$ in terms of inductive data and  $EE^{2n+2}_{r+1}$.
			\item Let $r>1$. $J_{1,r}(W^4,W^6,H^{2n-5})$ expresses $W^4_{(r)}H^{2n-1}$ in terms of inductive data.		\end{enumerate}
	\end{lemma}
	
	This process terminates after finitely many steps, since all elements of $D^{2n+3}_r\cup D^{2n+4}_r$ vanish for $r\geq 2n+2$. Therefore, we have proven the following.
	
	\begin{theorem}\label{thm:induction}
		There exists a nonlinear conformal algebra $\nlcalg$ over the localized ring $R = D^{-1}\mathbb{C}[c,k]$ given by \eqref{localization}, 
		satisfying the features in Subsection \ref{subsec:setup}, whose universal enveloping vertex algebra $\Wsp$ has the following properties.
		\begin{enumerate}\label{Wsp 2 parameters}
			\item It has conformal weight grading \[\Wsp=\bigoplus_{N=0}^{\infty}\Wsp[N],\quad \Wsp[0]= R.\]
			\item It is strongly generated by fields $\{X^{2i-1},Y^{2i-1},H^{2-1} | i\geq 1\}\cup \{L,W^{2i} | i \geq 2\}$ and satisfies the OPE relations in Proposition \ref{prop:base case proposition}, Jacobi identities in $J_{11}$ and those which appear in Lemmas \ref{lemma:recursions0} -\ref{rproduct}.
			\item It is the unique initial object in the category of vertex algebras with the above properties.
		\end{enumerate} 
	\end{theorem}

	\subsection{Step 4. free generation}
	There are more Jacobi identities than those imposed in Lemmas \ref{lemma:recursions0}-\ref{rproduct}.
	So it is not yet clear that all Jacobi identities among the strong generators hold as a consequence of (\ref{conformal identity}-\ref{quasi-derivation}) alone, or equivalently, that $\nlcalg$ is a nonlinear Lie conformal algebra and $\Wsp$ is freely generated. 
	In order to prove this, we will consider certain simple quotients of $\Wsp$. 
	First, let 
	\[I\subseteq R \cong \Wsp [0] \]
	be an ideal, and let $I\cdot \Wsp$ denote the vertex algebra ideal generated by $I$. 
	The quotient 
	\begin{equation}
		\cW^{\mathfrak{sp}_2,I}_{\infty} = \Wsp / I\cdot \Wsp
	\end{equation}
	has strong generators $\{W^{i} | i\geq 1\}$ satisfying the same OPE relations as the corresponding  generators of $\Wsp$ where all structure constants in $R$ are replaced by their images in $R/I$.
	
	We now consider a localization of $\cW^{\mathfrak{sp}_2,I}_{\infty}$.
	Let $E \subseteq R/I$ be a multiplicatively closed set, and let $S= E^{-1}(R/I)$ denote the localization of $R/I$ along $S$. Thus we have the localization of $R/I$-modules
	\[\cW^{\mathfrak{sp}_2,I}_{\infty, S}= S\otimes_{R/I}\cW^{\mathfrak{sp}_2,I}_{\infty},\]
	which is a vertex algebra over $S$.
	
	\begin{theorem}\label{one-parameter quotients theorem}
		Let $R$, $I$, $E$, and $S$ be as above, and let $\cW$ be a simple vertex algebra over $S$ with the following properties.
		\begin{enumerate}
			\item $\cW$ is generated by affine fields $\bar {X}^1,\bar {H}^1,\bar {Y}^1$, Virasoro field $\bar{L}$ of central charge $c$ and a weight 4 primary field $\bar{W}^4$.
			\item Setting $\bar{W}^{n+2} =\bar{W}^4_{(1)} \bar{W}^{n}$ for all $i\geq 2$, the OPE relations for $\bar{W}^{n}(z)\bar{W}^{m}(z)$ for $n+m\leq 9$ are the same as in $\Wsp$ if the structure constants are replaced with their images in $S$.
		\end{enumerate}
		Then $\cW$ is the simple quotient $\cW^{\mathfrak{sp}_2}_{\infty,S,I}$ of $\cW^{\mathfrak{sp}_2,I}_{\infty,S}$ by its maximal graded ideal $\cI$.
	\end{theorem}
	\begin{proof}
		The assumption that generators $\{W^{i} | i\geq 1\}$ satisfy the above OPE relations is equivalent to the statement that the Jacobi identities in $J_{11}$ hold in the corresponding nonlinear Lie conformal algebra, which is possibly degenerate. 
		Then all OPE relations among the generators $\{ \bar {W}^{i} | i\geq 1\}$ of $\cW^{\mathfrak{sp}_2,I}_{\infty,S}$ must also hold among the fields $\{ \bar {W}^{i} | i\geq 1\}$, since they are formal consequences of these OPE relations together with Jacobi identities, which hold in $\cW$.
		It follows that $\{ \bar {W}^{i} | i\geq 1 \}$ close under OPE and strongly generate a vertex subalgebra $\cW' \subseteq \cW$, which must coincide with $\cW$ since $\cW$ is assumed to be generated by $\{\bar{X}^1,\bar{Y}^1,\bar{H}^1,\bar{L},\bar{W}^4\}$.
		So $\cW$ has the same strong generating set and OPE algebra as $\cW^{\mathfrak{sp}_2,I}_{\infty,S}$.
		Since $\cW$ is simple and the category of vertex algebras over $R$ with this strong generating set and OPE algebra has a unique simple graded object, $\cW$ must be the simple quotient $\cW^{\mathfrak{sp}_2}_{\infty,S,I}$ of $\cW^{\mathfrak{sp}_2,I}_{\infty,S}$.
	\end{proof}

\begin{theorem} \label{weakgeneration:Walgebras} For all $m\geq 2$, $\cC^{\psi}_{BC}(0,m) = \cW^{\psi - 2m-2}(\gs\gp_{2(2m+1)}, f_{2m+1, 2m+1})$, is weakly generated by the fields in weight at most $4$ for all $\psi \in \mathbb{C}$, with the possible exception of 
\begin{enumerate}
\item The critical value $\psi = 0$,
\item $\psi \in \{1, \ \frac{2 m}{1 + 2 m},\ \frac{2( m-1)}{1 + 2 m},\ \frac{2 (1 + 3 m)}{3 (1 + 2 m)}, \ \frac{2 (1 + 5 m)}{5 (1 + 2 m)},\ \frac{2 (7 m-1)}{7 (1 + 2 m)}\}$, which correspond to $k \in \{ -1, -2, -4, -\frac{4}{3}, -\frac{8}{5}, -\frac{16}{7}\}$,
\item The values where the central charge $c_{BC} = -\frac{22}{5}$.
\end{enumerate}
\end{theorem}

\begin{proof} For $f = f_{2m+1, 2m+1}$, $\cW^{\psi - 2m-2}(\gs\gp_{2(2m+1)}, f)$ has a strong generating set which is in bijection with a basis for $\gs\gp_{2(2m+1)}^f$. These generators close nonlinearly under OPE with structure constants given by polynomial functions of $\psi$. 

 Let $\omega^{2i}$ and $\omega^{2i+1}$ be such a choice of generators of $\cW^{\psi - 2m-2}(\gs\gp_{2(2m+1)}, f)$, and let $\xi^{\omega_{2i}}$ and $\xi^{\omega_{2i+1}}$ be the corresponding elements of the $\gs\gp_{2(2m+1)}^f$. In particular, this implies that 
$$\omega^3_{(0)} \omega^{2j+1} = c_{3,2j+1} \omega^{2j+3} + \cdots,$$ where $[\xi^{\omega_{3}}, \xi^{\omega_{2j+1}}] = c_{3,2j+1} \xi^{\omega_{2j+3}}$. Here $c_{3,2j+1}$ are nonzero constants, and the remaining terms are normally ordered monomials in the previous fields $\omega^{a}$ for $a< 2j+3$ and their derivatives.

Without loss of generality, can assume $W^1 = \omega^1$ for $W=X,Y,H$, and by abuse of notation we can also replace $\omega^2$ with $W^2$, without changing the fact for fields $\omega^a, \omega^b$ for $a,b \neq 2$, the structure constants of terms not involving $\omega^2$ are still polynomials in $\psi$. Additionally, we may assume without loss of generality that 
$$H^3 = h^3 + \cdots,\quad X^3 = x^3 + \cdots, \quad Y^3 = y^3 + \cdots,$$ where the remaining terms depend only on $\omega^2, x,y,h$. This is because the leading structure constants appearing in Proposition \ref{prop:structure constants}, which are independent of $c,k$, are independent of the scaling of $W^3$. For $i\geq 2$, we can write
$$W^{2i} = \lambda_{2i} \omega^{2i} + \dots,\quad H^{2i+1} = \lambda_{h, 2i+1} h^{2i+1} + \cdots,\quad X^{2i+1} = \lambda_{x, 2i+1} x^{2i+1} + \cdots, \quad Y^{2i+1} = \lambda_{y, 2i+1} y^{2i+1} + \cdots,$$ where $\lambda_{2i}$, $\lambda_{h,2i+1}$, $\lambda_{x,2i+1}$, $\lambda_{y,2i+1}$ are polynomial functions of $\psi$. Using the fact that $\{x^{2i+1}, h^{2i+1},y^{2i+1}\}$ transform as the adjoint $\gs\gp_2$-module, we can also assume that for each $i \geq 2$, $\lambda_{h,2i+1} = \lambda_{x,2i+1} = \lambda_{y,2i+1}$, so we replace them with $\lambda_{2i+1}$. To prove the theorem, it suffices to show that $\lambda_{2i}$, $\lambda_{2i+1}$ are all nonzero constants.

By Proposition \ref{prop:structure constants},
$$X^{3}_{(0)} Y^{2j - 1} = \frac{3!! (2j-1)!!}{(2j+1)!!} H^{2j+1} + \cdots,$$ while on the other hand
$$ X^{3}_{(0)} Y^{2j - 1} = x^3_{(0)} (\lambda_{2j-1} y^{2j-1} )+ \cdots  = \lambda_{2j-1} c_{3,2j-1} h^{2j+1} + \cdots = \frac{\lambda_{2j-1} c_{3,2j-1}}{\lambda_{ 2j+1}} H^{2j+1} + \cdots.$$
Therefore $\frac{\lambda_{2j-1} c_{3,2j-1}}{\lambda_{2j+1}} = \frac{3!! (2j-1)!!}{(2j+1)!!}$, so by induction we see that the functions $\lambda_{2j-1}$ are all constant.

We need a slightly different argument to show that $\lambda_{2i}$ are all constant. First, 
$$x^3_{(1)} \omega^{2j} = d_{3,2j} x^{2j+1} + \cdots,\qquad  x^3_{(1)} y^{2j+1} = d_{3,2j+1} \omega^{2j+2} .$$
A prioiri, the structure constants $d_{3,2i}$ and $d_{3,2i+1}$ are polynomial functions of $\psi$, but in the course of the proof we will see that they are also constants.
By Proposition \ref{prop:structure constants}, we have
$X^3_{(1)} (X^3_{(1)}  Y^{3}) = X^5 + \cdots$, while on the other hand
$$X^3_{(1)}  (X^{3}_{(1)} Y^3) = x^3_{(1)} (x^3_{(1)} y^3) + \cdots =  d_{3,3} x^3_{(1)} \omega^4 = d_{3,3} d_{3,4} x^5 = \frac{d_{3,3} d_{3,4}}{\lambda_5} X^5.$$
Therefore $\frac{d_{3,3} d_{3,4}}{\lambda_5} = 1$, so $d_{3,3} d_{3,4}$ is a nonzero constant. Since both $d_{3,3}$ and $d_{3,4}$ are polynomials in $\psi$, they must both be nonzero constants as well.

More generally, for $j > 2$, we have $$X^3_{(1)}  (X^{3}_{(1)} Y^{2j-1}) = X^3_{(1)} \frac{6 (2j-1)!!}{(2j-2)!!} W^{2j} = \frac{6 (2j-1)!!}{(2j-2)!!} \frac{3 (2j)!!}{8(2j-1)!!} X^{2j+1}.$$ 
On the other hand, 
\begin{equation} \begin{split} X^3_{(1)}  (X^{3}_{(1)} Y^{2j-1}) & = \lambda_{2j-1} x^3_{(1)} (x^3_{(1)} y^{2j-1}) + \cdots 
\\ & =  \lambda_{2j-1} d_{3,{2j-1}} x^3_{(1)} \omega^{2j} + \cdots
\\ & = \lambda_{2j-1} d_{3,2j-1} d_{3,2j} x^{2j+1} + \cdots
\\ & = \frac{\lambda_{2j-1} d_{3,2j-1} d_{3,2j}}{\lambda_{2j+1}} X^{2j+1} + \cdots.\end{split} \end{equation}
 It follows that  $$\frac{\lambda_{2j-1} d_{3,2j-1} d_{3,2j}}{\lambda_{2j+1}} = \frac{6 (2j-1)!!}{(2j-2)!!} \frac{3 (2j)!!}{8(2j-1)!!},$$ and by the same argument $d_{3,2j-1}$ and $d_{3,2j}$ are nonzero constants for all $j>2$.

Next, we have $X^3_{(1)} W^4 = W^5$, and 
$$X^3_{(1)} W^4 = x^3_{(1)}( \lambda_4 \omega^4) = \lambda_3 d_{3,4} x^5 + \cdots = \frac{\lambda_4 d_{3,4}}{\lambda_5} X^5 + \cdots.$$ Therefore  $\frac{\lambda_4 d_{3,4}}{\lambda_5}  = 1$, so that $\lambda_4$ is a nonzero constant.

Finally, by Proposition \ref{prop:structure constants}, for $j\geq 2$ we have
$$X^3_{(1)} \big(Y^3_{(1)} W^{2j} \big) = \bigg(\frac{2(2j)!!}{8(2j-1)!!}\bigg)\bigg(\frac{6(2j+1)!!}{9(2j)!!} \bigg)W^{2i+2} + \cdots .$$ 
On the other hand, 
\begin{equation} \begin{split} X^3_{(1)} \big(Y^3_{(1)} W^{2j} \big) & = x^3_{(1)} \big(y^3_{(1)} \lambda_{2j} \omega^{2j} \big)  + \cdots
\\ & =  d_{3,2j+1} d_{3,2j} \lambda_{2j}  \omega^{2j+2} + \cdots
\\ & = \frac{d_{3,2j+1} d_{3,2j} \lambda_{2j}}{\lambda_{2j+2}} W^{2j+2} + \cdots. \end{split} \end{equation}
Therefore
$$\frac{d_{3,2j+1} d_{3,2j} \lambda_{2j}}{\lambda_{2i+2}} =  \bigg(\frac{2(2j)!!}{8(2j-1)!!}\bigg)\bigg(\frac{6(2j+1)!!}{9(2j)!!} \bigg),$$ so $\frac{ \lambda_{2j}}{\lambda_{2i+2}}$ is a nonzero constant. Since $\lambda_4$ is a nonzero constant, so is $\lambda_{2j}$ for all $j >2$. \end{proof}

	\begin{corollary} \label{WBC0m:quotient} For all $m \geq 1$, $\cC^{\psi}_{BC}(0,m)$ is a $1$-parameter quotient of $\cW^{\gs\gp_2}_{\infty}$.
		\end{corollary}
		
		\begin{proof} This follows from Theorem \ref{weakgeneration:Walgebras} and Theorem \ref{thm:induction}.
		\end{proof}

	\begin{corollary}\label{Wsp freely generated}
		All Jacobi identities among generators $\{W^{i} | i\geq 1\}$ holds as consequences of (\ref{conformal identity}-\ref{quasi-derivation}) alone, so $\nlcalg$ is a nonlinear Lie conformal algebra with generators $\{W^{i} | i\geq 1\}$. Equivalently, $\Wsp$ is freely generated by $\{W^{i} | i\geq 1\}$ and has graded character
		\begin{equation}\label{character}
			\chi (\Wsp,q) =\sum_{n=0}^{\infty}rank_{R}(\Wsp [n])q^n = \prod_{i,j,l=1}^{\infty}\frac{1}{(1-q^{2i+2j+1})^{3}(1-q^{2i+2l+2})}.
		\end{equation}
		For any prime ideal $I \subseteq R$, $\cW^{\mathfrak{sp}_2,I}_{\infty}$ is freely generated by $\{W^{i} | i\geq 1\}$ as a vertex algebra over $R/I$ and
		\[\chi (\cW^{\mathfrak{sp}_2,I}_{\infty},q) =\sum_{n=0}^{\infty}rank_{R/I}(\cW^{\mathfrak{sp}_2,I}_{\infty}[n])q^n = \prod_{i,j,l=1}^{\infty}\frac{1}{(1-q^{2i+2j+1})^{3}(1-q^{2i+2l+2})}.\]
		For any localization $S=(E^{-1}R)/I$ along a multiplicatively closed set $E\subseteq R/I$, $\cW^{\mathfrak{sp}_2,I}_{\infty,S}$ is freely generated by $\{W^{i} | i\geq 1\}$ and 
		\[\chi (\cW^{\mathfrak{sp}_2,I}_{\infty,S},q) = \sum_{n=0}^{\infty}rank_{S}(\cW^{\mathfrak{sp}_2,I}_{\infty,S}[n])q^n = \prod_{i,j,l=1}^{\infty}\frac{1}{(1-q^{2i+2j+1})^{3}(1-q^{2i+2l+2})}.\]
	\end{corollary}
	\begin{proof}
		If some Jacobi identity among generators $\{W^{i} | i\geq 1\}$ does not hold as a consequence of \eqref{conformal identity}-\eqref{quasi-derivation}, there would be a null vector of weight $N$ in $\Wsp$ for some $N$. 
		Then $\T{rank}_{R}(\Wsp [N])$ would be smaller than that given by (\ref{character}), and the same would hold in any quotient of $\Wsp [N]$, as well as any localization of such a quotient. 
		But since $\cC^{\psi}_{BC}(0,m)$ is a localization of such a quotient and is freely generated of type $\cW(1^3,2,3^3,4,\dots,(2m-1)^3, 2m, (2m+1)^3)$, this is impossible.
	\end{proof}

	\begin{corollary} The vertex algebra $\Wsp$ is simple in the sense of Definition \ref{VOAsimplicity}. 
	\end{corollary}
	
	\begin{proof}
		If $\Wsp$ is not simple, it would have a singular vector $\omega$ in some weight $N$. Let $p\in R$ be an irreducible polynomial and let $I$ be the ideal $(p)\subseteq R$. By rescaling if necessary, we can assume without loss of generality that $\omega$ is not divisible by $p$, and hence descends to a nontrivial singular vector in $\cW^{\mathfrak{sp}_2,I}_{\infty}$. Then for any localization $S$ of $R/I$, the simple quotient of $\cW^{\mathfrak{sp}_2,I}_{\infty,S}$ would have a smaller weight $N$ submodule than $\cW^{\mathfrak{sp}_2,I}_{\infty,S}$ for all such $I$. This contradicts the fact that for all $m$, $\cC^{\psi}_{BC}(0,m)$ is a quotient of $\cW^{\mathfrak{sp}_2,I}_{\infty,S}$ for some $I$ and $S$, and has the same character as $\Wsp$ up to weight $2m+1$.
	\end{proof}
	
	\begin{corollary}
		The vertex algebra $\Wsp$ has full automorphism group $\text{SL}_2(\mathbb{C})$.
	\end{corollary}
	\begin{proof}
		Let $g$ be any automorphism of $\Wsp$ that fixes weight one fields.
		By definition, $g$ preserves the Virasoro generator $L$, and by our assumption it acts as identity on the weight one space. We must have $g(W^4) = a_4 W^4 + \dots$ and $g(X^3) = a_3 X^3 +\dots$, where the omitted terms are normally ordered products in $X^1,Y^1,H^1,X^3,H^3,Y^3, L$ of weights $4$ and $3$ that transform as trivial and adjoint $\fr{sp}_2$-modules, respectively. Imposing the conditions that $g(X^3)$ and $g(W^4)$ are primary with respect to $L$, and that the raising property holds \[g(W^4)(z)X^1(w) \sim g(X^3)(z-w)^{-2} + \partial g(X^3)(z-w)^{-1},\] we find the solution $g(W^4) = a_4 W^4$ and $g(X^3) = a_{3}X^3$.
		Finally, the following vertex algebra products must be respected 
		\begin{equation}
			\begin{split}
		g(W^4_{(7)}W^4) =& g (W^4)_{(7)}g(W^4),\\
		 g(H^3_{(5)}H^3) = &g(H^3)_{(5)}g(H^3),\\
		  g(W^4_{(3)}W^4) =& g (W^4)_{(3)}g(W^4).
			\end{split}
		\end{equation}
		Solving the above constraints fixes the automorphism $g$ to be the identity map.
	\end{proof}

	\subsection{Quotients by maximal ideals of $\Wsp$}
	So far, we have considered quotients of the form $\cW^{\mathfrak{sp}_2}_{\infty,I}$ which are 1-parameter vertex algebras in the sense that $R/I$ has Krull dimension 1. Here, we consider simple quotients of $\cW^{\mathfrak{sp}_2}_{\infty,I}$ where $I \subseteq R$ is a {\it maximal} ideal. Such an ideal always has the form $I=(c-c_0,k-k_0)$ for $c_0,k_0 \in\mathbb{C}$, and $\cW^{\mathfrak{sp}_2}_{\infty,I}$ is a vertex algebra over $\mathbb{C}$.
	We first need a criterion for when the simple quotients of two such vertex algebras are isomorphic.
	\begin{theorem}\label{max quotients}
		Let $c_0,c_1,k_0,k_1$ be complex numbers in the complement of the set $D$ given by (\ref{loc}), and let \[I_0 = (c-c_0,k-k_0),\quad I_1 =(c-c_1,k-k_1)\]
		be the corresponding maximal ideals in $R$.
		Let $\cW_0$ and $\cW_1$ be the simple quotients of $\cW^{\mathfrak{sp}_2,I_0}_{\infty}$ and $\cW^{\mathfrak{sp}_2,I_1}_{\infty}$, respectively. 
		Then $\cW_0 \simeq \cW_1$ if and only if $c_0 = c_1$ and $k_0=k_1$.
	\end{theorem}
	\begin{proof}
		Since the parameters $c$ and $k$ arise as the central charge of the Virasoro field $L$ and level of $\aff$, it follows that if $\cW_0\simeq \cW_1$, then necessarily $c_0=c_1$ and $k_0=k_1$.
	\end{proof}
	
	\begin{corollary}\label{max quotients2}
		Let $I=(p)$ and $J=(q)$ be prime ideals in $R$ such that $\cW^{\gs\gp_2,I}_{\infty}$ and $\cW^{\gs\gp_2,J}_{\infty}$ are not simple. Then any pointwise coincidences between the simple quotients of $\cW^{\mathfrak{sp}_2,I}_{\infty}$ and $\cW^{\mathfrak{sp}_2,J}_{\infty}$ must correspond to intersection points of the truncation curves $V(I)\cap V(J)$.
	\end{corollary}
	
	\begin{corollary}
		Suppose that $\cA$ is a simple, 1-parameter vertex algebra which is isomorphic to the simple quotient of $\cW^{\mathfrak{sp}_2,I}_{\infty}$ for some prime ideal $I\subseteq R$, possibly after localization. 
		Then if $\cA$ is the quotient of $\cW^{\mathfrak{sp}_2,I}_{\infty}$ for some prime ideal $J$, possibly localized, we must have $I=J$.
	\end{corollary}
	\begin{proof}
		This is immediate from Theorem \ref{max quotients} and Corollary \ref{max quotients2}, since if $I$ and $J$ are distinct prime ideals, their truncation curves $V(I)$ and $V(J)$ can intersect in at most finitely many points. The simple quotients of $\cW^{\mathfrak{sp}_2,I}_{\infty}$ and $\cW^{\mathfrak{sp}_2,J}_{\infty}$ therefore cannot coincide as $1$-parameter families. \end{proof}

\subsection{Extending $\cW^{\gs\gp_2}_{\infty}$ over $D$}
Recall that $\cW^{\gs\gp_2}_{\infty}$ is only defined over the localization $D^{-1}\mathbb{C}[c,k]$, so it is not defined along the curves in $\mathbb{C}^2$ where the generators of $D$ vanish. Here we show that it can naturally be extended to each of these curves by suitably rescaling the generators.

First we consider the curve $k = 0$. Before specializing to $k=0$, we define new fields $V^{2i+1} = k^i W^{2i+1}$ and $V^{2i+2}= k^{i}W^{2i+2}$ for $i\geq 0$. One checks that the OPEs among the new fields $V^i$ are now all defined at $k=0$. We can regard the vertex algebra with generators $V^i$ as defining a new $2$-parameter vertex algebra which agrees with $\cW^{\gs\gp_2}_{\infty}$ for all points in the parameter space, but which is also defined along the curve $k=0$. By abuse of notation, we continue to denote this vertex algebra by $\cW^{\gs\gp_2}_{\infty}$. We denote by $\cW^{\gs\gp_2, I_0}_{\infty}$ its quotient by the ideal generated by $I_0 = (k)$, and we denote by $\cW^{\gs\gp_2}_{\infty,I_0}$ the simple quotient of $\cW^{\gs\gp_2, I_0}_{\infty}$ graded by conformal weight. One checks that the fields $V^{2i+1}$ for all $i \geq 0$ lie in the maximal ideal of $\cW^{\gs\gp_2, I_0}_{\infty}$, but the fields $V^{2i+2}$ for $i\geq 0$ survive in the simple quotient. The simple quotient $\cW^{\gs\gp_2}_{\infty, I_0}$ is then of type $\cW(2,4,6,\dots)$ and hence is a $1$-parameter quotient of $\cW^{\text{ev}}_{\infty}$. It can be checked that its truncation curve is given by 
\begin{equation} \label{truncation:k=0} \lambda = \frac{2 + 3 c}{7 (-1 + c) (22 + 5 c)}.\end{equation} It is natural to ask how it is related to the quotients of $\cW^{\text{ev}}_{\infty}$ appearing in \cite{CL4}.
Consider either
 $\cC^{\psi}_{1B}(n,0) \cong \text{Com}(V^k(\gs\go_{2n+1}), V^k(\gs\go_{2n+2}))^{Z_2}$, or $\cC^{\psi}_{1D}(n,0) \cong \text{Com}(V^k(\gs\go_{2n}), V^k(\gs\go_{2n+1}))^{Z_2}$.
For $\cC^{\psi}_{1B}(n,0)$, the truncation curve has a parametrization $(c(\psi), \lambda(\psi))$ appearing in \cite{CL4}, and we have
 $$ \lim_{\psi \rightarrow \infty} c(\psi) = 2n+1,\qquad \lim_{\psi \rightarrow \infty} \lambda(\psi) = \frac{5 + 6 n}{14 n (27 + 10 n)}.$$
Similarly, for $\cC^{\psi}_{1D}(n,0)$, we have
 $$ \lim_{\psi \rightarrow \infty} c(\psi) = 2n,\qquad \lim_{\psi \rightarrow \infty} \lambda(\psi) = \frac{1 + 3 n}{7 (-1 + 2 n) (11 + 5 n)}.$$
 Note that $\lim_{\psi \rightarrow \infty} \cC^{\psi}_{1B}(n,0) \cong \cH(2n+1)^{\text{O}_{2n+1}}$ and $\lim_{\psi \rightarrow \infty} \cC^{\psi}_{1D}(n,0) \cong \cH(2n)^{\text{O}_{2n}}$. The orbifolds $\cH(m)^{\text{O}_m}$ were studied in \cite{LHeis}, and the generators in weights $2,4,6,\dots$ close linearly under OPE and all structure constants are independent of $m$ except for the vacuum terms, which are given by polynomials in $m$. Therefore the OPE algebra of these generators is a Lie conformal algebra over $\mathbb{C}[m]$, and we can replace $m$ with a formal variable to obtain a $1$-parameter vertex algebra with central charge $m$. If we replace $2n+1$ with $c$ in the case of $\cC^{\psi}_{1B}(n,0)$, or replace $2n$ with $c$ in the case of $\cC^{\psi}_{1D}(n,0)$, we recover the truncation curve \eqref{truncation:k=0}. Therefore we can identify $\cW^{\gs\gp_2}_{\infty, I_0}$ with this $1$-parameter quotient of $\cW^{\text{ev}}_{\infty}$. Note that the first relation among the generators of $\cH(m)^{\text{O}_{m}}$ occurs in weight $m^2+3m+2$ \cite{LHeis}. Therefore in the $1$-parameter vertex algebra $\cW^{\gs\gp_2}_{\infty, I_0}$ where $m = c$ is a free parameter, there are no relations among the generators, so it has the same graded character as $\cW^{\text{ev}}_{\infty}$.

Similarly, consider the $k = -\frac{1}{2}$. As above, we define new fields $V^{2i+1} = (k+1/2)^i W^{2i+1}$ and $V^{2i+2}= (k+1/2)^{i}W^{2i+2}$ for $i\geq 0$. The OPEs among the new fields $V^i$ are now all defined at $k=-\frac{1}{2}$, so we regard the vertex algebra with generators $V^i$ as defining an extension of $\cW^{\gs\gp_2}_{\infty}$ to the line $k=-\frac{1}{2}$. By abuse of notation, we continue to denote this vertex algebra by $\cW^{\gs\gp_2}_{\infty}$. We denote by $\cW^{\gs\gp_2, I_{-1/2}}_{\infty}$ its quotient by the ideal generated by $I_{-1/2} = (2k +1)$, and we denote by $\cW^{\gs\gp_2}_{\infty,I_{-1/2}}$ the simple graded quotient of $\cW^{\gs\gp_2, I_{-1/2}}_{\infty}$. In this case, the affine vertex algebra $V^{-1/2}(\gs\gp_2)$ has a singular vector in weight $2$, the fields $V^{2i+1}$ for $i\geq 1$ lie in the maximal ideal, and the fields $V^{2i+2}$ for $i \geq 0$ survive in the simple quotient and commute with $L_{-1/2}(\gs\gp_{2})$. So $\cW^{\gs\gp_2}_{\infty, I_{-1/2}}$ has the form $L_{-1/2}(\gs\gp_2) \otimes \cW$ where $\cW$ is a $1$-parameter quotient of $\cW^{\text{ev}}_{\infty}$. Again, one can check that the truncation curve is given by 
\begin{equation} \label{truncation:k=-1/2} \lambda = -\frac{-34 + 19 c}{49 (-1 + c)^2 (22 + 5 c)^2}.\end{equation} As above, $\cW$ is easily seen to coincide with the large level limit of $\cC^{\psi}_{2D}(n,0)$ where $n$ becomes the parameter, and we set $c=n$. By a similar argument, $\cW$ has the same graded character as $\cW^{\text{ev}}_{\infty}$.
 
For the remaining points in $D$, namely, $k = -1, -2, -4, -\frac{4}{3}, -\frac{8}{5}, -\frac{16}{7}$, and $c = -\frac{22}{5}$, we can similarly extend $\cW^{\gs\gp_2}_{\infty}$ to be defined along each of these curves. Finally, we can perform all these rescalings simultaneously and extend $\cW^{\gs\gp_2}_{\infty}$ so that is is defined over the affine space $\mathbb{C}^2$; equivalently, this vertex algebra is defined over the polynomial ring $\mathbb{C}[c,k]$ and coincides with $\cW^{\gs\gp_2}_{\infty}$ after localizing along $D$.

\section{$1$-parameter quotients of $\cW^{\mathfrak{sp}_2}_{\infty}$} \label{sect:1paramquot}
As mentioned in Subection \ref{subsection:voaring}, we will always regard the $Y$-algebras of type $C$ introduced in Section \ref{sect:YtypeC}, namely $\cC^{\psi}_{XY}(n,m)$, as well as the diagonal cosets $\cC^{\ell}(n)$ introduced in Section \ref{sect:Diagonal}, as $1$-parameter vertex algebras, where $\psi$ and $\ell$ are regarded as formal parameters. In this section, we will prove that these vertex algebras all arise as $1$-parameter quotients of $\cW^{\mathfrak{sp}_2}_{\infty}$. More precisely, this means that there exist localizations $E^{-1} R/I_{XY}(n,m)$ and $F^{-1}\mathbb{C}[\ell]$ such that we have isomorphisms of $1$-parameter vertex algebras
$$ E^{-1} R/I_{XY}(n,m) \otimes_{R/I_{XY}(n,m)} \cW^{\gs\gp_2}_{\infty, I_{XY}(n,m)} \cong \cC^{\psi}_{XY}(n,m),\qquad F^{-1} R/I_n \otimes_{R/I_n} \cW^{\gs\gp_2}_{\infty, I_{n}} \cong \cC^{\ell}(n).$$
Here $E$ and $F$ are the (finite) sets of denominators of structure constants of $\cW^{\gs\gp_2}_{\infty}$ after replacing $c, k$ with the corresponding functions of $\psi$ (respectively $\ell$). Throughout this section, we suppress these localizations from our notation.

The easiest cases are given by the following
\begin{theorem} The following vertex algebras arise as $1$-parameter quotients of $\cW^{\mathfrak{sp}_2}_{\infty}$.
\begin{enumerate}
		\item For $n\geq 1$, $\text{Com}(V^{\ell}(\fr{sp}_{2n}), V^{\ell}(\fr{sp}_{2n+2}))$ (Case $\cC^{\psi}_{BC}(n,0)$),
		\item For $n\geq 1$, $\text{Com}(V^{\ell}(\fr{so}_{2n+1}), V^{\ell}(\fr{osp}_{2n+1|2}))^{\mathbb{Z}_2}$ (Case $\cC^{\psi}_{BB}(n,0)$),
		\item For $n\geq 1$, $\text{Com}(V^{\ell}(\fr{so}_{2n}), V^{\ell}(\fr{osp}_{2n|2}))^{\mathbb{Z}_2}$ (Case $\cC^{\psi}_{BD}(n,0)$),
		\item For $n\geq 1$, $\text{Com}(V^{\ell}(\go\gs\gp_{1|2n}), V^{\ell}(\go\gs\gp_{1|2n+2})^{\mathbb{Z}_2}$ (Case $\cC^{\psi}_{BO}(n,0)$),
		\item For $n\geq 1$, $\text{Com}(V^{\ell}(\fr{sp}_{2n}), V^{\ell-1}(\fr{sp}_{2n}) \otimes \cE(2n))$,
		\item For $n\geq 3$ and $n=2$, $n=4$, and $n=8$, $\text{Com}(V^{\ell}(\fr{so}_{n}), V^{\ell-2}(\fr{so}_{n}) \otimes \cS(n))^{\mathbb{Z}_2}$,
		\item For $n\geq 1$, $\text{Com}(V^{\ell}(\fr{osp}_{1|2n}), V^{\ell-1}(\fr{osp}_{1|2n}) \otimes \cS(1)\otimes \cE(2n))^{\mathbb{Z}_2}$.
	\end{enumerate}
	\end{theorem}
\begin{proof} We only prove this for $\cC^{\psi}_{BC}(n,0)$ since the proof for the other cases is similar. First, $\cC^{\psi}_{BC}(n,0)$ has large level limit
\[\lim_{\psi\to \infty} \cC^{\psi}_{BC}(n,0)  \simeq \cH(3)\otimes \cH(4n)^{\T{Sp}_{2n}},\] where $\cH(m)$ denotes the rank $m$ Heisenberg vertex algebra. Recall that $\cH(4n)$ has full automorphism $\T{O}_{4n}$. There is an embedding of $\T{Sp}_{2n} \hookrightarrow \T{O}_{4n}$, such that the standard representation $\mathbb{C}^{4n}$ of $\T{O}_{4n}$ decomposes as 2 standard representations $\mathbb{C}^{2n}$ of $\T{Sp}_{2n}$.
	We choose generators $\{u^i,u^{-i}| i=1,\dots n \}$ and $\{v^i,v^{-i}| i=1,\dots n \}$ for $\cH(4n)$ which transform as the standard modules in the usual symplectic basis.
	They satisfy the following OPE relations
	\begin{equation}
		\begin{split}
			u^{i}(z)v^{-j}(w)\sim \delta_{i,j} (z-w)^{-2},\quad u^{-i}(z)v^{j}(w)\sim -\delta_{i,j} (z-w)^{-2}.
		\end{split}
	\end{equation}

Next, $\cH(4n)$ has a good increasing filtration \cite{L}, such that the associated graded $\T{gr}^F(\cH(4n)^{\T{Sp}_{2n}})$ is isomorphic to the classical invariant ring 
		\[\Big(\T{Sym}\Big(\bigoplus_{i=0}^{\infty}(V_i \oplus U_i )\Big)\Big)^{\T{Sp}_{2n}},\]
		where $U_i\cong V_i \cong \mathbb{C}^{2n}$ as $\T{Sp}_{2n}$-module.
		Generators for this ring are given by Weyl's first fundamental theorem of invariant theory for the standard representation of $\T{Sp}_{2n}$ \cite{W}. These are quadratics, each corresponding to a pair of distinct modules in $\{U_i,V_i|i\geq 0\}$.
		Let $\{\partial^d u^i,\partial^d u^{-i}  | i=1,\dots,n\}$ and $\{\partial^d v^i,\partial^d v^{-i}  | i=1,\dots,n\}$ be the symplectic bases for each $U_d$ and $V_d$, respectively.
		We have 3 kinds of $\fr{sp}_{2n}$-invariants
		\begin{equation}
			\begin{split}
				X^{a,b} =&\sum_{i=1}^n:\!\partial^{a} u^i\partial^{b}u^{-i}\!: -:\!\partial^{b} u^i\partial^{a}u^{-i}\!:,\quad a> b \geq 0,\\
				Y^{a,b} =&\sum_{i=1}^n:\!\partial^{a} v^i\partial ^{b}v^{-i}\!:-:\!\partial^{b} v^i\partial ^{a}v^{-i}\!:,\quad a> b \geq 0,\\
				H^{a,b} =&\sum_{i=1}^n:\!\partial^{a} u^i\partial ^{b}v^{-i}\!:-:\!\partial^{b} u^i\partial ^{a}v^{-i}\!:,\quad a,b \geq 0,
			\end{split}
		\end{equation}
		which form a strong generating set for $\cH(4n)^{\T{Sp}_{2n}}$.
		Note that $\{X^{a,b} | a> b \geq 0\} \cup \{Y^{a,b} | a> b \geq 0\} \cup \{H^{a,b} | a,b \geq 0\}$ and their derivatives span the same vector space as the smaller set \begin{equation}\label{genset}
			\{X^{2a+1,0} | a\geq 0\}\cup\{Y^{2a+1,0} | a \geq 1\}\cup\{H^{a,0} | a \geq 0\}
		\end{equation}
		and their derivatives. Therefore \eqref{genset} is also a strong generating set for $\cH(4n)^{\T{Sp}_{2n}}$, which therefore has a strong generating set of type $\cW(2,3^3,4,5^3,\dots)$. In fact, by \cite[Theorem 6.6]{Lin2}, $\cH(4n)^{\T{Sp}_{2n}}$ is of type $\cW(2,3^3,4,5^3,\dots, N)$ for some $N$, i.e., it is strongly generated by a finite subset of the generators (\ref{genset}).

		We have the following relations in $\cH(4n)^{\T{Sp}_{2n}}$.
		\begin{equation}
			\begin{split}
				H^{1,0}_{(0)}X^{2a+1,0} =&2 X^{2a+3,0}+\dotsb,\\
				H^{1,0}_{(0)}Y^{2a+1,0} =& -2 Y^{2a+3,0}+\dotsb, \\
				H^{1,0}_{(1)}H^{a,0} =&(a+3) H^{a+1,0}+\dotsb,
			\end{split}
		\end{equation}
		where the omitted terms are in the span of $\{\partial^{2k}X_{2a-3k+3},\partial^{2k}Y_{2a-3k+3},\partial^{k}X_{a-k+1}|k\geq 1\}$. 
		By induction we can get all the generators in (\ref{genset}), thus proving that $\cH(4n)^{\T{Sp}_{2n}}$ is weakly generated by the fields in weight at most $3$. By \cite[Lemma 2.5]{AL}, $\cC^{\psi}_{BC}(n,0)$ is weakly generated by the fields in weight at most $3$ for generic $\psi$. Finally, by Theorem \ref{thm:induction}, $\cC^{\psi}_{BC}(n,0)$ is a $1$-parameter quotient of $\cW^{\gs\gp_2}_{\infty}$.
\end{proof}
	
\begin{remark} For $\text{Com}(V^{\ell}(\fr{so}_{n}), V^{\ell-2}(\fr{so}_{n}) \otimes \cS(n))^{\mathbb{Z}_2}$ when $n=2$, $n=4$, and $n=8$, the level $k$ of the subalgebra $V^k(\gs\gp_2)$ is $-1$, $-2$, and $-4$, respectively. Since the set $D$ contains $(k+1)$, $(k+2)$, and $(k+4)$, the algebra $\cW^{\gs\gp_2}_{\infty}$ is not defined along these curves. But as discussed in Section \ref{sect:main}, it is straightforward to extend $\cW^{\gs\gp_2}_{\infty}$ along these curves by suitably rescaling the generators. Then $\text{Com}(V^{\ell}(\fr{so}_{n}), V^{\ell-2}(\fr{so}_{n}) \otimes \cS(n))^{\mathbb{Z}_2}$ is a $1$-parameter quotient of the extension of $\cW^{\gs\gp_2}_{\infty}$.
\end{remark}

However, this argument does not work for $\cW^{\psi}_{XY}(n,m)$ when $m \geq 1$ since the weak generation property fails in the large level level limit, so we need a different approach.

\begin{lemma} Suppose $\cW$ is a $1$-parameter vertex algebra satisfying the symmetry and strong generation requirements, but not necessarily the weak generation hypotheses. Consider the subalgebra $\tilde{\cW} \subseteq \cW$ generated by the fields in weights at most $4$. Then $\tilde{\cW}$ is of type $\cW(1^3, 2, 3^3, 4,\dots)$, and is a $1$-parameter quotient of $\cW^{\gs\gp_2}_{\infty}$ which need not be simple. 
\end{lemma}

\begin{proof} This is similar to the proof of Lemma 5.10 and Theorem 5.4 of \cite{CL4}, and is omitted.
\end{proof}

\begin{lemma} \label{upto2m+2} For $X = B,C$ and $Y = B,C,D,O$, let $\tilde{\cC}^{\psi}_{XY}(n,m) \subseteq \cC^{\psi}_{XY}(n,m)$ be the subalgebra generated by the fields in weight at most $4$. Suppose that $\tilde{\cC}^{\psi}_{XY}(n,m)$ contains the strong generators up to weight $2m+2$. Then $\tilde{\cC}^{\psi}_{XY}(n,m) = \cC^{\psi}_{XY}(n,m)$, and in particular, $\tilde{\cC}^{\psi}_{XY}(n,m)$ is simple.
\end{lemma}

\begin{proof} If $\tilde{\cC}^{\psi}_{XY}(n,m)$ contains the strong generators up to weight $2m+2$, so does its large level limit $\lim_{\psi \ra \infty} \tilde{\cC}^{\psi}_{XY}(n,m)$. We will show that the subalgebra of $\lim_{\psi \ra \infty} \cC^{\psi}_{XY}(n,m)$ generated by the fields in weights up to $2m+2$ is all of $\lim_{\psi \ra \infty} \cC^{\psi}_{XY}(n,m)$, which proves the lemma. For $X=B$, in the large level limit the fields in weights $1,2,\dots,2m+1$ decouple from the ones in higher weight. These higher weight fields are the generators of the following orbifolds
\begin{enumerate}
\item $\cS_{\text{odd}}(2n+1, 2m+2)^{\text{SO}_{2n+1}}$, for $\cC^{\psi}_{BB}(n,m)$,
\item $\cO_{\text{ev}}(4n, 2m+2)^{\text{Sp}_{2n}}$, for $\cC^{\psi}_{BC}(n,m)$,
\item $\cS_{\text{odd}}(2n, 2m+2)^{\text{SO}_{2n}}$, for $\cC^{\psi}_{BD}(n,m)$,
\item $(\cO_{\text{ev}}(4n, 2m+2) \otimes \cS_{\text{odd}}(1,2m+2))^{\text{Osp}_{1|2n}}$, for $\cC^{\psi}_{BO}(n,m)$.
\end{enumerate}
In all cases, it can be checked easily that these orbifolds are generated by the fields in weights $2m+2$. The proof is similar to the proof of Lemmas 5.1-5.8 of \cite{CL4}, and is omitted. Similarly, in the case $X=C$, the fields in weights $1,2,\dots, 2m$ decouple from the ones in higher weight. The fields in higher weights are the generators of the following orbifolds
\begin{enumerate}
\item $\cS_{\text{ev}}(2n+1, 2m+1)^{\text{SO}_{2n+1}}$, for $\cC^{\psi}_{CB}(n,m)$,
\item $\cO_{\text{odd}}(4n, 2m+1)^{\text{Sp}_{2n}}$, for $\cC^{\psi}_{CC}(n,m)$,
\item $\cS_{\text{ev}}(2n, 2m+1)^{\text{SO}_{2n}}$, for $\cC^{\psi}_{CD}(n,m)$,
\item $(\cO_{\text{odd}}(4n, 2m+1) \otimes \cS_{\text{ev}}(1,2m+1))^{\text{Osp}_{1|2n}}$, for $\cC^{\psi}_{CO}(n,m)$.
\end{enumerate}
Again, it is straightforward to check that these orbifolds are generated by the fields in weights $2m+1$ and $2m+2$.
\end{proof}

Next, recall that for all $m\geq 2$, $\cC^{\psi}_{BC}(0,m) = \cW^{\psi - 2m-2}(\gs\gp_{2(2m+1)}, f_{2m+1, 2m+1})$ has the weak generation property $\cC^{\psi}_{BC}(0,m) = \tilde{\cC}^{\psi}_{BC}(0,m)$ for all $\psi \in \mathbb{C}$ except for an explicit finite set of exceptions which is given by Theorem \ref{weakgeneration:Walgebras}.

\begin{theorem} \label{Yalgebra:Quotientproperty} For $m\geq 1$, $X = B,C$, and $Y = B,C,D,O$, $\tilde{\cC}^{\psi}_{XY}(n,m) = \cC^{\psi}_{XY}(n,m)$ as $1$-parameter vertex algebras. In particular, $\cC^{\psi}_{XY}(n,m)$ is a simple, $1$-parameter quotient of $\cW^{\gs\gp_2}_{\infty}$, and we have 
$$\cC^{\psi}_{XY}(n,m) \cong \cW^{\gs\gp_2}_{\infty, I_{XY,n,m}}.$$ In this notation, $I_{XY,n,m} \subseteq \mathbb{C}[c,k]$ is the ideal generated by $c = c_{XY}$, where $c_{XY}$ is given by \eqref{CD}-\eqref{BO}, and $\cW^{\gs\gp_2}_{\infty, I_{XY,n,m}}$ is the simple quotient of $\cW^{\gs\gp_2}_{\infty} / I_{XY,n,m} \cdot \cW^{\gs\gp_2}_{\infty}$ by its maximal graded ideal. \end{theorem}
	
\begin{proof} First, to show that $\tilde{\cC}^{\psi}_{XY}(n,m) = \cC^{\psi}_{XY}(n,m)$ as $1$-parameter vertex algebras, it is enough to find a level $\psi$ where the simple quotients coincide, and there is no singular vector in weight below $2m+2$. Consider first the case of $\tilde{\cC}^{\psi}_{CD}(n,m)$ for $m\geq 1$. 

The truncation curves for $\tilde{\cC}^{\psi}_{CD}(n,m)$ and $\tilde{\cC}^{\psi'}_{BC}(0,r)$ intersect at the point when $\psi = \frac{2 ( 2 m + n-1)}{1 + 2 m + 2 r}$ and $\psi' = \frac{2 - n + 2 r}{1 + 2 m + 2 r}$, so that $k = - \frac{4 m +  n + 4 m r + 2 n r}{1 + 2 m + 2 r}$. If we choose $r \neq \frac{n-2}{2}$ and so that $k \neq -1, -2, -4, -\frac{4}{3}, -\frac{8}{5}, -\frac{16}{7}$, this is a point where $\cC^{\psi'}_{BC}(0,r)$ has the weak generation property $\tilde{\cC}^{\psi'}_{BC}(0,r) = \cC^{\psi'}_{BC}(0,r)$. 

Suppose first that $n=0$. Then $\cC^{\psi}_{CD}(0,m)$ is freely generated of type $\cW(1^3, 2, 3^3, \dots, (2m-1)^3, 2m)$ for generic $\psi$, so there can be relations below weight $2m+1$ only for finitely many levels. So we can choose $r$ such that there are no relations in $\cC^{\psi}_{CD}(0,m)$ below weight $2m+1$ at the corresponding value of $\psi$.  Since $\tilde{\cC}^{\psi}_{CD}(0,m)$ has the same OPE algebra as $\cC^{\psi'}_{BC}(0,r)$ and there are no relations below weight $2m+1$, $\tilde{\cC}^{\psi}_{CD}(0,m)$ must contain the strong generators of $\cC^{\psi}_{CD}(0,m)$ up to weight $2m$. Then $\tilde{\cC}^{\psi}_{CD}(0,m) = \cC^{\psi}_{CD}(0,m)$ at this value of $\psi$. By Lemma \ref{upto2m+2}, $\tilde{\cC}^{\psi}_{CD}(0,m) = \cC^{\psi}_{CD}(0,m)$ as $1$-parameter vertex algebras.

Next, suppose that $n \geq 1$. Using Weyl's second fundamental of invariant theory for the standard representation of $\text{SO}_{2n}$, one checks that first relation among the generators of $\tilde{\cC}^{\psi}_{CD}(n,m)$ occurs at weight
$(2m+2n+1)(2n+1)$ for generic $\psi$. As above, there can be relations in lower weight only for finitely many values of $k$. Therefore we can find $r$ such that at the corresponding value of $\psi$ the first relation occurs at weight $(2m+2n+1)(2n+1)$. Since $\tilde{\cC}^{\psi}_{CD}(n,m)$ has the same OPE algebra as $\cC^{\psi}_{BC}(0,r)$ and there are no relations below weight $(2m+2n+1)(2n+1)$, $\tilde{\cC}^{\psi}_{CD}(n,m)$ must contain the strong generators of $\cC^{\psi}_{CD}(n,m)$ up to weight $(2m+2n+1)(2n+1) > 2m+2$. By Lemma \ref{upto2m+2}, $\tilde{\cC}^{\psi}_{CD}(n,m) = \cC^{\psi}_{CD}(n,m)$ at this value of $\psi$, and hence this holds as $1$-parameter vertex algebras. The proof for the other families is similar and is based on the following computations.

\begin{enumerate}

\item The truncation curves for $\tilde{\cC}^{\psi}_{CB}(n,m)$, $\tilde{\cC}^{\psi'}_{BC}(0,r)$ intersect when 
$\psi = \frac{-1 + 4 m + 2 n}{1 + 2 m + 2 r},\ \psi' = \frac{3 - 2 n + 4 r}{2 (1 + 2 m + 2 r)}$.

\item The truncation curves for $\tilde{\cC}^{\psi}_{CC}(n,m)$, $\tilde{\cC}^{\psi'}_{BC}(0,r)$ intersect when 
$\psi = \frac{2 (-1 + 2 m - n)}{1 + 2 m + 2 r}$ and $\psi' = \frac{2 + n + 2 r}{1 + 2 m + 2 r}$.

\item The truncation curves for $\tilde{\cC}^{\psi}_{CO}(n,m)$, $\tilde{\cC}^{\psi'}_{BC}(0,r)$ intersect when $\psi =\frac{-1 + 4 m - 2 n}{1 + 2 m + 2 r}$ and $\psi' = \frac{3 + 2 n + 4 r}{2 (1 + 2 m + 2 r)}$.

\item The truncation curves for $\tilde{\cC}^{\psi}_{BB}(n,m)$, $\tilde{\cC}^{\psi'}_{BC}(0,r)$ intersect when $\psi = \frac{1 + 4 m - 2 n}{4 (1 + m + r)}$ and $\psi' = \frac{3 + 2 n + 4 r}{4 (1 + m + r)}$.

\item The truncation curves for $\tilde{\cC}^{\psi}_{BC}(n,m)$, $\tilde{\cC}^{\psi'}_{BC}(0,r)$ intersect when 
$\psi = \frac{1 + 2 m + n}{2 (1 + m + r)}$ and $\psi' = \frac{1 - n + 2 r}{2 (1 + m + r)}$.

\item The truncation curves for $\tilde{\cC}^{\psi}_{BD}(n,m)$, $\tilde{\cC}^{\psi'}_{BC}(0,r)$ intersect when $\psi = \frac{1 + 2 m - n}{2 (1 + m + r)} $ and $\psi' = \frac{1 + n + 2 r}{2 (1 + m + r)}$.

\item The truncation curves for $\tilde{\cC}^{\psi}_{BO}(n,m)$, $\tilde{\cC}^{\psi'}_{BC}(0,r)$ intersect when $\psi = \frac{1 + 4 m + 2 n}{4 (1 + m + r)}$ and $\psi' = \frac{3 - 2 n + 4 r}{4 (1 + m + r)}$. \end{enumerate}\end{proof}
	
\begin{corollary} \label{weakgeneration:finiteness} For generic values of $\psi_0 \in \mathbb{C}$,
$\cC^{\psi_0}_{XY}(n,m)$ is a quotient of $\cW^{\gs\gp_2}_{\infty}$. Similarly, for $n \in \frac{1}{2}\mathbb{Z}$, $\cC^{\ell_0}(n)$ is a quotient of $\cW^{\gs\gp_2}_{\infty}$ for generic values of $\ell_0$. 
\end{corollary}

\begin{corollary} \label{notrialities} There are no isomorphisms between distinct $1$-parameter vertex algebras of the form $\cC^{\psi}_{XY}(n,m)$.
\end{corollary}
	
\begin{proof} Since the algebras $\cC^{\psi}_{XY}(n,m)$ all arise as $1$-parameter quotients of $\cW^{\gs\gp_2}_{\infty}$ of the form $\cW^{\gs\gp_2}_{\infty, I_{XY,n,m}}$, it suffices to show that the ideals $I_{XY,n,m}$ are all distinct. But this is straightforward to check using the formulas for these ideals given by \eqref{CD}-\eqref{BO}.
\end{proof}

\section{Reconstruction} \label{sect:recon} 
Recall the $\fr{sp}_2$-rectangular $\cW$-algebra $\cW^{\psi}_{XY}(n,m)$ from Section \ref{sect:YtypeC} whose coset by $V^{a}(\fr{a})$ is the $Y$-algebra $\cC^{\psi}_{XY}(n,m)$, where $\fr{a}$-level $a$ is determined in terms of $\fr{sp}_2$-level $k$, as in Table \ref{tab:Walgebras}. Theorem \ref{Yalgebra:Quotientproperty} implies that $\cW^{\psi}_{XY}(n,m)$ is an extension of $V^{a}(\mathfrak{a}) \otimes \cW$, where $\cW \cong \cC^{\psi}_{XY}(n,m)$ is a simple, $1$-parameter quotient of $\cW^{\gs\gp_2}_{\infty}$. Moreover, the extension is generated by fields in a fixed weight and parity (determined by $X$ and $Y$), which transform as $\mathbb{C}^2 \otimes \rho_{\mathfrak{a}}$ as a module for $\mathfrak{sp}_2 \oplus \mathfrak{a}$. 

	The main result of this section is a {\it reconstruction theorem}, which states that the full OPE algebra of $\cW^{\psi}_{XY}(n,m)$ is determined by the structure of $\cW$ and the action of $\fr{sp}_2 \oplus\fr{a}$ on the generating fields. This is analogous to \cite[Theorems 9.1 and 9.5]{CL3} and \cite[Theorem 6.3]{CL4}, which were needed in the proof of the key properties of the convolution functor in \cite{CLNS}.
    
    We will use the same notation for the generators of $\cW$; in particular, we have fields of odd weight $\{X^{2i-1},Y^{2i-1},H^{2i-1}|i\geq 1\}$ and fields of even weight $\{W^{2i}|i\geq 1\}$, which we refer to uniformly as $W^n$, see (\ref{convention Walpha}). We would like to handle cases when the nilpotent is of type $B$ and $C$ in a uniform way. Above, $m$ is the rank of the Lie algebra in which the nilpotent element is principal, and it enters our procedure as a degree of the leading pole in (\ref{normalization}). Therefore we adopt the following convention. Let $m$ be the conformal weight of the extension fields, which is either an integer or half-integer. 
	\begin{itemize}
	\item If $m\geq 1$ is an integer, then nilpotent is principal in $\fr{so}_{2m-1}$.
	\item If $m\geq \frac{3}{2}$ is a half-integer, then nilpotent is principal in $\fr{sp}_{2m-2}$.
		If $m=\frac{1}{2}$ then we are in the case of diagonal cosets discussed in Section \ref{sect:Diagonal}.
	\end{itemize}

	\begin{theorem}\label{thm:reconstruction}
		Let $\cA^{\psi}_{XY}(n,m)$ be a simple $1$-parameter vertex (super)algebra with the following properties, which are shared with $\cW^{\psi}_{XY}(n,m)$.
		
		\begin{enumerate}\label{features of extension}\label{assumptions}
			\item 
			$\cA^{\psi}_{XY}(n,m)$ is a conformal extension of $\cW\otimes V^a(\fr{a})$, where $\cW$ is some 1-parameter quotient of $\Wsp$.
		
			\item The extension is generated by fields $P^{\mu,j}$ in weight $m$ which are primary for the conformal vector $L+L^{\fr{a}}$ of $\cA^{\psi}_{XY}(n,m)$, primary for the affine subalgebra $V^k(\fr{sp}_2)\otimes V^{a}(\fr{a})$, and transform as $U = \mathbb{C}^2\otimes \rho_{\fr{a}}$ under $\fr{sp}_2\oplus\fr{a}$.
			
			\item The extension fields have the same parity as the corresponding fields of $\cW^{\psi}_{XY}(n,m)$ appearing in Table \ref{tab:Walgebras}.
			 Algebra $\cA^{\psi}_{XY}(n,m)$ is strongly generated by these fields, together with the generators of  $\cW\otimes V^a(\fr{a})$.
			\item The restriction of the Shapovalov form to the extension fields is non-degenerate.
		\end{enumerate}
		Then $\cA^{\psi}_{XY}(n,m)$ is isomorphic to $\cW^{\psi}_{XY}(n,m)$ as $1$-parameter vertex (super)algebras.
	\end{theorem}
	
	Our strategy is similar to the one used in \cite{CL3} and \cite{CL4}, and consists of 3 steps.
	\begin{enumerate}
		\item We first show that the existence of an extension $\cA^{\psi}_{XY}(n,m)$ satisfying (\ref{features of extension}) uniquely determines the truncation curve expressing $\cW$ as a $1$-parameter quotient of $\Wsp$. It can be uniformly expressed in the form
		\begin{equation}\label{trunk}
			\begin{split}
				c\left(4 \lambda (k+2)+(k+1) (2 k+1)	\right)=4 k  \lambda \left(4  \lambda(k+1) (k+2)-(2 k+1) (2 k+3)\right),
			\end{split}
		\end{equation} where $\lambda$ is given by \eqref{def:lambda}
		
		\item We express the OPEs $W^{3}(z)P^{\mu,j}(w)$ for all  $P^{\mu,j}\in U$, in terms of $\fr{sp}_2$-level $k$.
		\item We argue that the all OPEs of $\cA^{\psi}_{XY}(n,m)$  are determined from OPEs of $\Wsp$ together with $W^{3}(z)P^{\mu,j}(w)$ for all $P^{\mu,j}\in U$.
	\end{enumerate}

	\subsection{Set-up}
	Let $\{q^{\alpha}|\alpha\in S\}$ denote the basis of $\fr{a}$, and $X^{\alpha}(z)$ be the corresponding fields.
	We use generators $\{\tilde W^{i} | i\geq 1\}$ as in (\ref{raising}), to cast our assumptions (\ref{assumptions}) in the OPE form.
	
	\begin{enumerate}\label{def of extension}
		\item $V^{a}(\fr{a})$ is affine subalgebra:
		\[X^{\alpha}(z)X^{\beta}(w)\sim a (q^{\alpha}|q^{\beta})(z-w)^{-2}+(\sum_{\gamma\in S}f^{\alpha,\beta}_{\gamma}X^{\gamma})(w)(z-w)^{-1}.\]
		\item $V^{a}(\fr{a})$ commutes with $\cW$:
		\[\tilde W^{n}(z) X^{\alpha}(w)\sim 0,\quad q^{\alpha}\in\fr{a},\quad n\geq 1.\]
		\item Fields $P^{\mu,j}$ are primary for $V^k(\fr{sp}_2)\otimes V^{a}(\fr{a})$ and transform as $U = \mathbb{C}^2\otimes \rho_{\fr{a}}$:
		\begin{equation*}\label{gP}
			X^{\alpha}(z)P^{\mu,j}(w) \sim (\rho_{\fr{sp}_2\oplus\fr{a}}(q^{\alpha})P^{\mu,j})(w)(z-w)^{-1}, \quad q^{\alpha} \in \fr{sp}_2\oplus \fr{a}.
		\end{equation*}
		\item In $\cW\otimes V^{a}(\fr{a})$ the total Virasoro field is $T=\tilde L+L^{\fr{sp}_2}+L^{\fr{a}}$, so the OPEs of extension fields $L$ are
		\begin{equation*}\label{LP}
			\begin{split}
				\tilde L(z)P^{\mu,j}(w)\sim  \lambda  P^{\mu,j}(w)(z-w)^{-2}+{(\partial P^{\mu,j}-L_{(0)}^{\fr{sp}_2}P^{\mu,j}- L_{(0)}^{\fr{a}}P^{\mu,j})(w)}(z-w)^{-1}
			\end{split}
		\end{equation*}
		where constant 
		\begin{equation} \label{def:lambda} \lambda=m- \frac{3}{4(k+2)}-\frac{\T{Cas}}{a+h^{\vee}_{\fr{a}}}.\end{equation}
		Here $\T{Cas}$ is the eigenvalue of the Casimir in $U(\fr{a})$ of the standard representation $\rho_{\fr{a}}$.
		\item Since $\cA$ is nondegenerate, so we can renormalize fields $P^{\mu,j}$ so that (\ref{normalization}) holds.
	\end{enumerate}
	
	Next, we proceed to set up OPEs among the generators of $\cW$ and extension fields $P^{\mu,j}$.
	For our computation we need only a few structure constants, defined in the following OPEs.
	\begin{equation}\label{ansatz extension}
		\begin{split}
			\tilde W^{2n}(z)P^{-1,1}(w)\sim &p^{2n}_{0} P^{-1,1}(w)(z-w)^{-2n}+(p^{2n}_{1}\partial P^{-1,1}+\dotsb)(w) (z-w)^{-2n+1}\\&
			+(p^{2n}_{1,1}\partial^2 P^{-1,1}+p^{2n}_{2} :\!LP^{-1,1}\!:\!+\dotsb)(w)(z-w)^{-2n+2}+\dotsb,\\
			\tilde X^{2n-1}(z)P^{-1,1}(w)\sim & p^{2n-1}_{0}P^{1,1}(w)(z-w)^{-2n+1}\\
			&+(p^{2n-1}_{1}\partial P^{1,1}+a^{2n-1}L_{(0)}^{\fr{sp}_2}P^{1,1}+b^{2n-1}L_{(0)}^{\fr{a}}P^{1,1})(w)(z-w)^{-2n+2}\\&
			+(p^{2n-1}_{1,1}\partial^2 P^{1,1} +p^{2n}_{2}:\!LP^{1,1}\!:\!+\dots)(w) (z-w)^{-2n+3}+\dotsb.
		\end{split}
	\end{equation}
	From (\ref{ansatz extension}), using Jacobi identities $J_{0,r}(X^{\alpha},W^{n},P^{-1,1})$ for $X^{\alpha} \in V^k(\fr{sp}_2) \otimes V^a(\fr{a})$ one can determine all OPEs $W^n(z)P^{\mu,i}(w)$ for $P^{\mu,j}\in U$.
	Finally, we posit that the OPEs among the extension fields of weight $m$ have the most general form that is compatible with the conformal weight grading and $\fr{sp}_2\oplus\fr{a}$-symmetry.

	\subsection{Step 1: Truncation}
	Computation of truncation curve amounts to the imposition of several Jacobi identities. 
	We will arrive at system consisting of a quadratic, cubic and quartic equations.
	Solving it, we will obtain the formula (\ref{trunk}).
	
	First, we impose conformal symmetry thanks to the Jacobi identities $J(L,W^{3},P^{-1,1})$.
	For the purposes of evaluating the truncation curve, only the structure constants displayed in (\ref{ansatz extension}) are relevant. We find the following expressions.
	\begin{equation}
		\begin{split}
			p^{3}_{1}=&\frac{(3 k+4)p^3_0}{2(k+2)\lambda},\\
			p^{3}_{1,1}=&\frac{ (3 k+4) (2 c k+3 c+4 \lambda k)p^3_0}{2(k+2)^2 \left(2 c
				\lambda+c+16 \lambda^2-10 \lambda\right)\lambda},\\
			p^{3}_{2}=&\frac{2 (3 k+4) (4 \lambda k+8 \lambda-4 k-5)p^3_0}{(k+2)^2 \left(2 c
				\lambda+c+16\lambda^2-10 \lambda\right)},\\
			a^{3}_{1}=&\frac{\left(4 \lambda  k^2-3 k^2+8 \lambda  k-k+4\right)p^3_0}{4 (k-1) (k+2)^2 \lambda},\\
			b^{3}_{1}=&\frac{ \left(2 \lambda  k^2-3 k^2-k-8 \lambda +4\right)p^3_0}{2 (k-1) (k+2)^2 \lambda}.\\
		\end{split}
	\end{equation}
	
	Next, we impose the Jacobi identities $J_{2,1}(X^3,H^3,P^{-1,1})$,
	$J_{1,3}(H^3,H^3,P^{-1,1})$ and $J_{2,4}(W^4,W^4,P^{-1,1})$ to determine the relevant structure constants (\ref{ansatz extension}) arising in the OPEs $W^{4}(z)P^{-1,1}(w)$, 
	$W^{5}(z)P^{-1,1}(w)$ and  $W^{6}(z)P^{-1,1}(w)$  in terms of variables $p^{3}_0$ and $c,k$.
	Extracting the coefficients of $:\!X^1P^{1,1}\!:$, $P^{1,1}$ and $\partial^2 P^{1,1}$ 
	in identities $J_{2,1}(X^3,X^3,P^{-1,1})$, 
	$J_{3,2}(X^3,W^4,P^{-1,1})$ and 
	$J_{1,3}(W^4,W^4,P^{-1,1})$ we obtain a quadratic, cubic and quartic equations in variables $c,k$ and $p^{3}_0$, respectively. 
	Finally, assuming $p^{3}_{0}\neq 0$, and solving this system we uniquely determines $c$ and $p^3_0$ as functions of $k$.
	\begin{equation}
		\begin{split}
			c=&4 k  \lambda\frac{4  \lambda(k+1) (k+2)-(2 k+1) (2 k+3)}{4  \lambda (k+2)+(k+1) (2 k+1)},\\
			p^{3}_0 =&\frac{2 \lambda (k-1) (k+1) (k+2)^2 (2 k+1) (4  \lambda (k+2)+2 k+1)}{(3k+4) \left(4  \lambda (k+2)+(k+1) (2 k+1)\right)}.\\
		\end{split}
	\end{equation}
	In particular, this proves part (1) of Theorem \ref{thm:reconstruction}.
	Moreover, we have the following.
	
	\begin{lemma}\label{lem:truncation}
		The OPEs $W^{3}(z)P^{\mu,j}(w)$ for $P^{\mu,j}\in U$ are expressed as rational functions of the $\fr{sp}_2$-level $k$. 
	\end{lemma}
	
	\begin{remark}
		Substituting the values of $\lambda$ as in (\ref{def:lambda}),
		we recover the formulae for the central charges obtained in (\ref{CD}-\ref{CO}). 
	\end{remark}

	\begin{remark}
		The truncation curve depends on the extension data only via the parameter $\lambda$. A similar feature is exhibited by the $\Winf$ and $\Wev$ algebras.
	\end{remark}

	\subsection{Step 2: Reconstruction}
	By Lemma \ref{lem:truncation}, the OPEs $W^{3}(z)P^{\mu,j}(w)$ for $P^{\mu,j}\in U$ are fully determined in terms of the $\fr{sp}_2$-level $k$. 
	In the following we show that this property propagates to all OPEs $W^{n}(z)P(w)$ for $n\geq 4$.
	\begin{lemma}
For $n\geq 4$, the OPEs $W^{n}(z)P^{\mu,j}(w)$ for $P^{\mu,j}\in U$, are uniquely determined by $W^{3}(z)P^{\nu,i}(w)$ for $P^{\nu,i}\in U$ and the OPEs of $\cW$. 
	\end{lemma}
	\begin{proof}
		We proceed by induction, with our base case the Lemma \ref{lem:truncation}.
        Recall the notation introduced in \eqref{ansatzWW}-\eqref{ansatzXH}.
		Inductively, assume that the OPEs $W^{i}(z)P^{\mu,j}(w)$ for $P^{\mu,j}\in U$ have been expressed in terms of $k$ for $i\leq n$.
		Imposing the Jacobi identity $J_{r,1}(W^{3},W^{n},P^{\mu,j})$ we obtain a relation
		\begin{equation}\label{XPFull}
			\begin{split}
				W^{n+1}_{(r)}P^{\mu,j}=\frac{1}{\epsilon^{2,\bar{n}}_{\bar n}}\frac{1}{rv^{3,n}_{1}-(r+1)v^{3,n}_{0}} \left(W^3_{(r)}W^{n}_{(1)}P^{\mu,j}-W^{n}_{(1)}W^3_{(r)}P^{\mu,j}-R_{r,1}^{3,n}\right),
			\end{split}
		\end{equation}
		where
		\begin{equation}
			R_{r,1}^{3,n}=(W^{3,n}_{\bar n}(0))_{(r+1)}P^{\mu,j}
			-\sum_{j=2}^r\binom{r}{j}(W^3_{(j)}W^{i})_{(r+1-j)}P^{\mu,j}.
		\end{equation}
		By induction, the right side of (\ref{XPFull}) is known, and this completes the proof.
	\end{proof}

    \begin{remark}
        Note that when $n=3$ and $r=1$, the Jacobi identity $J_{r,1}(W^{3},W^{n},P^{\mu,j})$ does not work since the denominator $rv^{3,n}_{1}-(r+1)v^{3,n}_{0}$ vanishes. 
        In this case, one can instead use $J_{r,1}(W^{3,\alpha},P^{\mu,j},W^{n,\beta})$, for example.
        This results in an expression
        \begin{equation*}
			\begin{split}
				P^{\mu,j}_{(1)}W^{4}=\frac{1}{{\epsilon^{2,2}_{0}}(\alpha,\beta)}\left(W^{3,\alpha}_{(1)}P^{\mu,j}_{(1)}W^{3,\beta}+(W^{3,\alpha}_{(0)}P^{\mu,j})_{(2)}W^{3,\beta}+(W^{3,\alpha}_{(1)}P^{\mu,j})_{(1)}W^{3,\beta}-P^{\mu,j}_{(1)}W^{3,\alpha;3,\beta}_{2}(0)\right),
			\end{split}
		\end{equation*} which determines $W^{4}_{(1)} P^{\mu,j}$ by skew-symmetry, and $\epsilon^{2,2}_0(\alpha,\beta)$ are constants defined in Table \ref{tab:epsilon}.
    \end{remark}
    
	Let $i,j \in \{1,\dots,\T{dim}(\rho_{\fr{a}})\}$ be fixed.
	We have the following.
	
	\begin{lemma}\label{PP}
		The OPEs $P^{\mu,i}(z)P^{\nu,j}(w)$ are fully determined from $W^n(z)P^{\alpha,l}(w)$ for $P^{\alpha,l}\in U$, and the OPEs of $\cW$.
	\end{lemma}
	
	\begin{proof}
		We proceed inductively to determine $P^{\mu,i}_{(r)}P^{\nu,j}$ for all $1 \leq r \leq 2m-1$. 
		Our base case is $r=2m-1$, which is known by assumption (\ref{def of extension}).
		Inductively, assume that all products $$\{P^{\alpha,a}_{(s)}P^{\beta,b}|\ \alpha,\beta\in\{-1,1\},a,b \in\{1,\dots,\T{dim}(\rho_{\fr{a}})\}, \ s>r\},$$ are expressed in terms of the $\fr{sp}_2$-level $k$.
		Consider the Jacobi identity $J_{1,r}(X^3, P^{-1,i}, P^{1,j})$, which reads as follows.
		\begin{equation}\label{extensionXP}
			X^3_{(1)}P^{-1,i}_{(r)}P^{1,j}   =
			(X^3_{(1)}P^{-1,i})_{(r)}P^{1,j} +(X^3_{(0)}P^{-1,i})_{(r+1)}P^{1,j}+P^{-1,i}_{(r)}X^3_{(1)}P^{1,j}.
		\end{equation}
		Note that the left side of relation (\ref{extensionXP}) is known by induction, and only the right side gives rise to a contribution of the desired product $P^{-1,i}_{(r-1)}P^{-1,j}$.
		Each such contribution arises from monomials in the products $X^{3}_{(1)}P^{-1,i}, X^{3}_{(0)}P^{-1,i}$ and $X^{3}_{(1)}P^{1,i}$.
		Upon collecting every contribution of the desired product $P^{i}_{(r-1)}P^{j}$, we may express it in terms of inductively data.
	\end{proof}
	This completes the proof of Theorem \ref{thm:reconstruction}.
	
	\begin{remark}
		The Jacobi identities $J_{1,0}(W^{3}, P^{\mu,i}, P^{\nu,j})$ give rise to relations expressing strong generators $W^{n}$ for $n\geq m$, in terms of extension fields.
	\end{remark}
	
	\begin{remark}
		Note that for weight 1 and 2 fields, the Jacobi identities $J(W^1,P^{\mu,i},P^{\nu,j})$ and $J(L,P^{\mu,i},P^{\nu,j})$ express the conformal and affine symmetries.
		More generally, each generator $W^n$ of  $\Wsp$ the Jacobi identity $J(W^n,P^{\mu,i},P^{\nu,j})$ gives rise to a family of recursions among the products $P^{\mu,i}_{(r)}P^{\nu,j}$.
		In this sense, $\fr{sp}_2$-rectangular $\cW$-algebras with a tail possesses a larger symmetry, namely the universal 2-parameter vertex algebra $\Wsp$.
	\end{remark}

\section{Rational quotients of $\cW^{\gs\gp_2}_{\infty}$} \label{sect:rational} 
Recall that the cosets $\cC^{\ell}(n)= \text{Com}(V^{\ell}(\gs\gp_{2n}), V^{\ell-1}(\gs\gp_{2n}) \otimes \cE(2n))$ given by \eqref{diagonalC} are all quotients of $\cW^{\gs\gp_2}_{\infty}$. When $\ell \in \mathbb{N}$, we have an embedding
$L_{\ell}(\gs\gp_{2n}) \rightarrow L_{\ell-1}(\gs\gp_{2n}) \otimes \cE(2n)$, and $\cC_{\ell}(n):= \text{Com}(L_{\ell}(\gs\gp_{2n}), L_{\ell-1}(\gs\gp_{2n}) \otimes \cE(2n))$ coincides with the simple quotient of $\cC^{\ell}(n)$ \cite[Lemma 8.1]{CL2}. The generators $X,Y,H$ of $L_n(\gs\gp_2) \subseteq \cC_{\ell}(n)$ are given by \eqref{sl2generators}, and the conformal vector of $\cC_{\ell}(n)$ is given by $L = L^{\gs\gp_{2n}} + L^{\cE} - L^{\text{diag}}$, where $L^{\cE}$ is the conformal vector in $\cE(2n)$ such that $b^{\pm i}, c^{\pm i}$ have conformal weight $\frac{1}{2}$, and $L^{\text{diag}}$ is the image of the Sugawara vector in the diagonally embedded copy of $L_{\ell}(\gs\gp_{2n})$. 

If we now replace $L$ with the new conformal vector $L + \frac{1}{2} \partial H$, then $b^{\pm i}, c^{\pm i}$ have conformal weights $0,1$ respectively, and $X,H,Y$ have conformal weights $0,1,2$, respectively. By \cite[Corollaries 7.6 and 7.7]{CL4}, for $\ell \in \mathbb{N}$, $\cC_{\ell}(n)$ is then a conformal extension of
\begin{equation} \label{eq:levels}
\cW_{\ell_1}(\gs\gp_{2n}) \otimes \cW_{\ell_2}(\gs\gp_{2n}),\qquad \ell_1 = -(n+1) + \frac{\ell+n+1}{2\ell+2n+1}, \qquad \ell_2 = -(n+1) + \frac{\ell+n}{2\ell+2n+1}.
\end{equation}
 This implies that for $\ell,n \in \mathbb{N}$, $\cC_\ell(n)$ is strongly rational \cite{SY, CMSY}. In this section we will generalize this result by proving that $\cC_{\ell}(n)$ is strongly rational when $\ell-1$ is admissible for $\widehat{\mathfrak{sp}}_{2n}$; see Theorem \ref{thm:coset-rat}. We will also show in Corollary \ref{rationalquotients} that all these vertex algebras are generated by the fields in weight at most $4$, and hence are quotients of $\cW^{\gs\gp_2}_{\infty}$.

\begin{lemma} \label{twocopiesofW(sp)} ${}$
 \begin{enumerate}
\item For $n \in \mathbb{N}$, after a suitable extension of scalars, the $1$-parameter vertex algebra $\cC^{\ell}(n)$ with conformal vector $L + \frac{1}{2} \partial H$, is an extension of the product of universal $\cW$-algebras $$\cW^{\ell_1}(\gs\gp_{2n}) \otimes \cW^{\ell_2}(\gs\gp_{2n}),\qquad \ell_1 = -(n+1) + \frac{\ell+n+1}{2\ell+2n+1}, \qquad \ell_2 = -(n+1) + \frac{\ell+n}{2\ell+2n+1}.$$
\item For generic values of $\ell$, $V^{\ell-1}(\gs\gp_{2n}) \otimes \cE(2n)$ is direct sum of irreducible $\cW^{\ell_1}(\gs\gp_{2n}) \otimes \cW^{\ell_2}(\gs\gp_{2n})$-modules which each have $1$-dimensional top space.
\end{enumerate}
\end{lemma}

\begin{proof} Since $\cC^{\ell}(n)$ has finite-dimensional weight spaces, we can choose a finite basis $m^i_{1}, \dots, m^i_{d_i}$ for the weight $i$ subspace of $\cC^{\ell}(n)$. Let 
$$L_1 = \sum_{j=1}^{d_2} \alpha^2_j m^2_j,\qquad W^4_1 = \sum_{j=1}^{d_4} \alpha^4_j m^4_j,\qquad L_2 = \sum_{j=1}^{d_2} \beta^2_j m^2_j,\qquad W^4_2 = \sum_{j=1}^{d_4} \beta^4_j m^4_j$$ be elements of weights $2$ and $4$, respectively, with undetermined coefficients. We impose the following conditions:
\begin{enumerate}
\item $L_1$ and $L_2$ are commuting Virasoro fields with central charges \begin{equation} \label{pairofccc} c_1 = -\frac{\ell n (3 + 4 \ell + 2 n)}{(1 + \ell + n) (1 + 2 \ell + 2 n)},\qquad c_2 = -\frac{n (1 + \ell + 2 n) (1 + 4 \ell + 6 n)}{(\ell + n) (1 + 2 \ell + 2 n)},\end{equation} respectively
\item $W^4_i$ is primary of weight $4$ for $L_i$ for $i=1,2$, and $W^4_1, W^4_2$ commute.
\item For $i = 1,2$, and $3 \leq k \leq 12$, the fields $W^{2k}_i = (W^4_i)_{(1)} W^{2k-2}_i$ satisfy the OPE relations of $\cW^{\text{ev}}_{\infty}$ along the truncation curves for $\cW^{\ell_1}(\gs\gp_{2n})$ and $\cW^{\ell_2}(\gs\gp_{2n})$, respectively.
\end{enumerate} 
These conditions impose a finite set of algebraic equations in the variables $\alpha^2_j, \alpha^4_j,  \beta^2_j, \beta^4_j$ as well as $\ell$, so the set of solutions to this system is an algebraic variety. Moreover, for all $\ell \in  \mathbb{N}_{\geq 4}$, the system has a solution as above, so this variety is at least $1$-dimensional. Viewing the above system as a system of equations in $\alpha^2_j, \alpha^4_j,  \beta^2_j, \beta^4_j$ with coefficients in the field $\mathbb{C}(\ell)$,  it will have a solution in some (finitely generated) extension $\overline{\mathbb{C}(\ell)}$ of $\mathbb{C}(\ell)$. Substituting these functions of $\ell$ gives us the desired map. Note that structure constants in the OPEs among these fields are rational functions of $\ell$ even though the above coefficients might live in the extension.

Next, we claim that the algebra generated by $\{L_i, W^4_i\}$ is isomorphic to $\cW^{\ell_i}(\gs\gp_{2n})$ rather that some non-simple quotient of $\cW^{\text{ev}}_{\infty}$ along the same truncation curve. For this purpose, it suffices to show that there exists decoupling relation $W^{2n+2}_i = P(L_i, W^4_i,\dots, W^{2n}_i)$ for some normally ordered polynomial $P$ in these fields and their derivatives. We know that there is a singular vector in $\cW^{\text{ev}}_{\infty,I}$ of weight $2n+2$ of the form $W^{2n+2}_i - P(L_i, W^4_i,\dots, W^{2n}_i)$, which generates the maximal ideal. 

Recall that $\cC^{\ell}(n)$ is generically simple and has the same graded character as $\cE(2n)^{\text{Sp}_{2n}}$. In particular, for infinitely many values of $\ell  \in \mathbb{N}$, the graded character of the simple quotient $\cC_{\ell}(n)$ is the same as that of $\cE(2n)^{\text{Sp}_{2n}}$ in weights up to $2n+2$.  Therefore for infinitely many values of $\ell$, this singular vector must vanish, so it must vanish for generic values of $\ell$, and hence is a decoupling relation generically.

Next, we prove part (2) of the lemma. The weight zero subspace $U_0 = \cC^{\ell}(n)[0]$ has dimension $n+1$ and necessarily consists of lowest weight vectors for the action of $\cW^{\ell_1}(\gs\gp_{2n}) \otimes \cW^{\ell_2}(\gs\gp_{2n})$. Let $V_0$ denote the $\cW^{\ell_1}(\gs\gp_{2n}) \otimes \cW^{\ell_2}(\gs\gp_{2n})$-module generated by $U_0$. We claim that the Zhu algebra of $\cW^{\ell_1}(\gs\gp_{2n}) \otimes \cW^{\ell_2}(\gs\gp_{2n})$, which is the polynomial ring generated by the images of $L_i, W^{2r}_i$ for $i = 1,2$ and $r = 2,3,\dots, n$, acts diagonalizably on $U_0$ over some finite extension of $\mathbb{C}(\ell)$ which we can take to be $\overline{\mathbb{C}(\ell)}$ without loss of generality, so we can choose a basis of eigenvectors with coefficients in this field. Moreover, we claim that in weights $m > 0$, there are no lowest weight vectors for $\cW^{\ell_1}(\gs\gp_{2n}) \otimes \cW^{\ell_2}(\gs\gp_{2n})$ in $V_0$ for generic $\ell$. Both statements follow from the semisimplicity of the action of $\cW_{\ell_1}(\gs\gp_{2n}) \otimes \cW_{\ell_2}(\gs\gp_{2n})$ on $\cC_{\ell}(n)$ for all $\ell \in \mathbb{N}$. In particular, a lowest weight vector in $\cC^{\ell}(n)$ would specialize to a lowest weight vector for infinitely many values $\ell \in \mathbb{N}$, which would then descend to a nontrivial lowest weight vector in the simple quotient $\cC_{\ell}(n)$ for $\cW_{\ell_1}(\gs\gp_{2n}) \otimes \cW_{\ell_2}(\gs\gp_{2n})$. It follows that $V_0$ is a sum of $n+1$ modules with $1$-dimensional lowest weight space. 

Next, let $m_1 > 0$ be the first integer for which the space $U_{m_1} \subseteq \cC^{\ell}(n)[m_1]$ of lowest weight vectors of weight $m_1$ for $\cW^{\ell_1}(\gs\gp_{2n}) \otimes \cW^{\ell_2}(\gs\gp_{2n})$ is nontrivial, and let $d_{1} = \text{dim}\ U_{m_1}$. By the same argument, the Zhu algebra of $\cW^{\ell_1}(\gs\gp_{2n}) \otimes \cW^{\ell_2}(\gs\gp_{2n})$ acts diagonalizably on $U_{m_1}$ over a finite extension of $\overline{\mathbb{C}(\ell)}$, and the $\cW^{\ell_1}(\gs\gp_{2n}) \otimes \cW^{\ell_2}(\gs\gp_{2n})$-module $V_{m_1}$ generated by $U_{m_1}$ has no lowest weight vectors in any weight $m > m_1$ for generic $\ell$. It follows that $V_{m_1}$ is a sum of $d_{1}$ simple modules over $\cW^{\ell_1}(\gs\gp_{2n}) \otimes \cW^{\ell_2}(\gs\gp_{2n})$, with $1$-dimensional lowest weight spaces of weight $m_1$, and that $V_0 \cap V_{m_1}$ is trivial. Inductively, over some extension of $\overline{\mathbb{C}(\ell)}$, we can find two sequences of positive integers $m_i, d_i$ for $i\geq 2$ such that $\tilde{\cC}^{\ell}(n)= \bigoplus_{i\geq 0} V_{m_i}$ where each $V_{m_i}$ is a sum of $d_{i}$ simple $\cW^{\ell_1}(\gs\gp_{2n}) \otimes \cW^{\ell_2}(\gs\gp_{2n})$-modules with $1$-dimensional lowest weight spaces of weight $m_i$. \end{proof}

Next, we improve some results from Section 7.5. of \cite{CL4}. Let $\cS(n)$ denote the rank $n$ $\beta\gamma$-system, and recall that
$$\cS(n) \cong L_{-\frac{1}{2}}(\gs\gp_{2n}) \oplus L_{-\frac{1}{2}}(\omega_1),$$
with $L_{-\frac{1}{2}}(\omega_1)$ the simple module corresponding to the standard representation of $\gs\gp_{2n}$ whose highest-weight is the first fundamental weight $\omega_1$. 
Let $P^n_\ell$ denote the set of admissible weights of $L_\ell(\gs\gp_{2n})$ whose finite part is integrable for $\gs\gp_{2n}$. Let $Q$ be the root lattice of $\gs\gp_{2n}$. In general we denote the simple module of highest-weight $\lambda$ by $L_\ell(\lambda)$.
We are interested in $m \in \mathbb Z_{\geq 0}$.
Then 
\begin{equation}
\begin{split}
L_m(\mu) \otimes \cS(n) \cong \bigoplus_{\lambda \in P_{m-\frac{1}{2}}} L_{m-\frac{1}{2}}(\lambda) \otimes M(\lambda, \mu)
\end{split}
\end{equation}
for certain multiplicity spaces which are in fact modules of $\cW_\ell(\gs\gp_{2m})$ with $\ell = -(m+1) + \frac{n+m+1}{2n+2m+1}$ by Theorem 7.8 of \cite{CL4}.
The multiplicity spaces $M(\mu, \mu)$ are non-zero as the top level of $L_m(\mu) \otimes \cS(n)$ is the integrable module $\rho_\mu$ of highest-weight $\mu$ for the Lie algebra $\gs\gp_{2n}$ corresponding to the zero-mode of the subalgebra $L_{m-\frac{1}{2}}(\gs\gp_{2n})$. Similarly, the subspace of conformal weight one half plus the top level is $\rho_\mu \otimes \rho_{\omega_1}$ and it must be the top level for the action of the subalgebra $L_{m-\frac{1}{2}}(\gs\gp_{2n})$. Hence also $M(\lambda, \mu) \neq 0$ for $\rho_\lambda \hookrightarrow \rho_\mu \otimes \rho_{\omega_1}$. 

Consider the minimal quantum Hamiltonian reduction and denote the image of $L_k(\lambda)$ by $\cW_k(\lambda, f_{\text{min}})$. Then $\cW_k(\lambda, f_{\text{min}})$ is an irreducible module for the minimal $\cW$-algebra as long as $k$ is not a positive integer \cite{Ar3}. By Remark 7.3 of \cite{CL4} we have that $\cW_{-\frac{1}{2}}( f_{\text{min}})\cong \mathbb C$ and hence also $\cW_\ell(\omega_1, f_{\text{min}})\cong \mathbb C$.
We can apply Theorem 1 of \cite{AFC} for the minimal nilpotent element to get
\begin{equation}\nonumber
\begin{split}
L_m(\mu) &\cong L_m(\mu)  \otimes \mathbb C \cong L_m(\mu)  \otimes \cW_{-\frac{1}{2}}( f_{\text{min}})  \cong  \bigoplus_{\lambda \in P^n_{m-\frac{1}{2}} \cap Q} \cW_{m-\frac{1}{2}}(\lambda, f_{\text{min}}) \otimes M(\lambda, \mu) \\
L_m(\mu) &\cong L_m(\mu)  \otimes \mathbb C \cong L_m(\mu)  \otimes \cW_{-\frac{1}{2}}( \omega_1,  f_{\text{min}})  \cong  \bigoplus_{\lambda \in P^n_{m-\frac{1}{2}} \cap Q + \omega_1} \cW_{m-\frac{1}{2}}(\lambda, f_{\text{min}}) \otimes M(\lambda, \mu) .
\end{split}
\end{equation}
The highest weight and the conformal weight of the minimal reduction of $L_k(\lambda)$ are given by formula (67)  of \cite{Ar4} and this data uniquely characterizes a simple module of the minimal $W$-algebra \cite[Theorem 6.3.1]{Ar4}. Let $\lambda$ in $P^n_{m-\frac{1}{2}}$ and write $\lambda = \lambda \omega_1 + \bar\lambda$ with $\bar\lambda$ in the orthogonal complement of $\omega_1$. Then the highest-weight of $\cW_k(\lambda, f_{\text{min}})$ is $\bar \lambda$ and the conformal weight is
\[
|\lambda+\rho|^2 -|\rho|^2 -(m+n+\frac{1}{2}) \lambda_1.
\]
In particular for $\mu = \mu_1 \omega_1 + \bar\mu$ and $\bar\mu$ orthogonal to $\omega_1$ one has $\cW_k(\lambda, f_{\text{min}})\cong \cW_k(\mu, f_{\text{min}})$ if and only if $\bar\lambda = \bar\mu$ and $\mu_1 \in \{ \lambda_1, 2m+1 -\lambda_1 \}$. 
This implies 
\begin{corollary}
For $m \in \mathbb Z_{\geq 0}$ and $\gg = \gs\gp_{2n}$, there is a bijection $\tau: P^n_{m-\frac{1}{2}} \cap Q \rightarrow P^n_{m-\frac{1}{2}} \cap (Q + \omega_1)$, mapping $\lambda$ to $\lambda$ to $\lambda + (2m+1 - 2(\lambda, \omega_1))\omega_1$, such that
\begin{enumerate}
\item $\cW_{m-\frac{1}{2}}(\lambda, f_{\text{min}}) \cong \cW_{m-\frac{1}{2}}(\tau(\lambda), f_{\text{min}})$;
\item $L_{m-\frac{1}{2}}(\gs\gp_{2n}) \oplus L_{m-\frac{1}{2}}((2m+1)\omega_1) \subset \text{Com}\left(\cW_\ell(\gs\gp_{2m}), L_m(\gs\gp_{2n}) \otimes \cS(n) \right)$;
\item  the multiplicity spaces satisfy $M(\lambda, \mu) \cong M(\tau(\lambda), \mu)$ as modules for $\cW_\ell(\gs\gp_{2m})$.
\end{enumerate}
\end{corollary}

Next recall symplectic level-rank duality \cite{ORS}. For positive integers $n, m$, there is a bijection $\sigma$ between the sets $P^n_m$ of $\gs\gp_{2n}$ and $P^m_n$ of $\gs\gp_{2m}$, such that
\[
\cE(2nm) \cong \bigoplus_{\mu \in P^n_m} L_n(\sigma(\lambda)) \otimes L_m(\lambda).
\]

\begin{proposition}
For $n, m \in \mathbb Z_{> 0}$ and $k = -(n+1) + \frac{n+m}{2n+2m+1}$, 
\[
L_m(\mathfrak{osp}_{1|2n}) \otimes \cW_k(\mathfrak{sp}_{2n}) \hookrightarrow  L_{m-1}(\mathfrak{sp}_{2n}) \otimes \cE(2n).
\]
\end{proposition}
\begin{proof}
There is a conformal embedding
\[
L_m(\mathfrak{osp}_{1|2n}) \otimes L_{n-\frac{1}{2}}(\mathfrak{sp}_{2m}) \hookrightarrow \cE(2nm) \otimes \cS(n).
\]
This follows from Corollary 2.1 together with Theorem 2.1 and Equation (2.7) of \cite{CLS}.
Since $\cE(nm)$ is an integrable module for  $L_{n}(\mathfrak{sp}_{2m})$, we can apply \cite[Theorem 1]{AFC} for the minimal nilpotent element to get
\[
L_m(\mathfrak{osp}_{1|2n}) \otimes \cW_{n-\frac{1}{2}}(\mathfrak{sp}_{2m}, f_{\text{min}}) \hookrightarrow \cE(2nm) \otimes  \mathbb C^2. 
\]
As we discussed, the two copies, $\mathbb C^2$, come from the fact that the minimal quantum Hamiltonian reduction at this level is two-to-one, in particular we must have an embedding  
\begin{equation}\label{eq:level-rank}
L_m(\mathfrak{osp}_{1|2n}) \otimes \cW_{n-\frac{1}{2}}(\mathfrak{sp}_{2m}, f_{\text{min}}) \hookrightarrow \cE(2nm) . 
\end{equation}
Then \cite[Theorem 7.7]{CL4} tells us that $\cW_{n-\frac{1}{2}}(\mathfrak{sp}_{2m}, f_{\text{min}})$ is strongly rational and that there is a conformal embedding
\[
L_n(\gs\gp_{2m-2}) \otimes \cW_k(\gs\gp_{2n}) \hookrightarrow \cW_{n-\frac{1}{2}}(\mathfrak{sp}_{2m}, f_{\text{min}})
\]
and this embedding extends to a conformal embedding 
\[
L_n(\gs\gp_{2m-2}) \otimes \cW_k(\gs\gp_{2n}) \otimes L_m(\mathfrak{osp}_{1|2n})   \hookrightarrow \cE(2nm).
\]
By \cite[Remark 2.1]{CL3} the embedding of $L_n(\gs\gp_{2m-2})$ is the natural one in the first factor of $\cE(2nm) \cong \cE(2n(m-1)) \otimes \cE(2n)$, in particular its commutant by symplectic-level rank duality is $L_{m-1}(\gs\gp_{2n})  \otimes \cE(2n)$ and our claim follows. 
\end{proof}
We remark that \eqref{eq:level-rank} can be viewed as a novel level-rank duality. It has been recently discovered in connection to mirror symmetry in three-dimensional superCFTs \cite{CGK}.

Note that the condition $\ell = -(n+1) + \frac{n+m}{2n+2m+1}$ can be written as
\begin{equation}\label{eq:relation_levels}
\frac{1}{4} \frac{1}{m+n+\frac{1}{2}} + \ell + n + 1 = \frac{1}{2}
\end{equation}
and let $\tilde \ell$ be the Feigin-Frenkel dual level of $\ell$, that is $2(\ell+ n+1)(\tilde \ell + 2n-1)=1$, i.e. $\cW^\ell(\gs\gp_{2n}) \cong W^{\tilde \ell}(\gs\go_{2n+1})$.
Then 
\[
\frac{1}{2m+2n+1} +\frac{1}{\tilde \ell + 2n-1} = 1.
\] 
\begin{corollary}
For $n\in \mathbb Z_{>0}$ amd $m, \ell$ related via \eqref{eq:relation_levels}
\[
V^m(\mathfrak{osp}_{1|2n}) \otimes \cW^\ell(\mathfrak{sp}_{2n}) \hookrightarrow  V^{m-1}(\mathfrak{sp}_{2n}) \otimes \cE(2n)
\]
\end{corollary}
\begin{proof}
The generating singular vector of $V^{m-1}(\mathfrak{sp}_{2n}) $ for $m \in \mathbb Z_{>0}$ has conformal weight $m$ and so the OPE algebra of $V^{m-1}(\mathfrak{sp}_{2n})  \otimes \cE(2n)$ 
and $L_{m-1}(\mathfrak{sp}_{2n}) \otimes \cE(2n)$ coincides for fields whose conformal weight is less than $m/2$. Since the structure constants of the OPE algebra of  $V^{m-1}(\mathfrak{sp}_{2n})  \otimes \cE(2n)$ are rational functions in $m$ and there are only countably many structure constants, these structure constants are completely determined by those of $L_{m-1}(\mathfrak{sp}_{2n}) \otimes \cE(2n)$ for $m \in \mathbb Z_{>0}$. 
Hence the claim.
\end{proof}

The following lemma follows from \cite{CMY} if $k$ is irrational and in general its proof is the same as for simple Lie algebras, which is \cite[Lemma 2.1]{ACL}.
Recall that the Kazhdan-Lusztig category KL$_k(\gg)$ is the category of finite length, smooth $\hat\gg$-modules at level $k$ that are locally finite for the horizontal subalgebra $\gg$.
\begin{lemma}\textup{\cite[Lemma 2.1]{ACL}}\label{lem:proj}
    Let $k + h^\vee \in \mathbb C\setminus \mathbb Q_{\leq 0}$. Then $V^k(\mathfrak{osp}_{1|2n})$ is projective in KL$_k(\mathfrak{osp}_{1|2n})$.
\end{lemma}
\begin{proof}
    Let $\gg = \mathfrak{osp}_{1|2n}$ and $\gh$ a Cartan subalgebra and as usual we identify $\gg$ with $\gg \otimes t^0 \in \hat\gg$. Let $L_0$ be the Virasoro zero-mode of  $V^k(\gg)$. For an object $M$ in KL$_k(\gg)$ and $\lambda \in \gh^*$, we set
    \begin{equation}\nonumber
        \begin{split}
            M^\lambda &= \left\{ m \in M | hm = \lambda(h)m,  \ \forall h \in \gh \right\} \\
            M_{[\Delta]} &= \left\{ m \in M | (L_0 - \Delta)^rm=0, \ r >\hspace{-1mm}> 0 \right\}, \qquad  M_{\Delta} = \left\{ m \in M | L_0 m =  \Delta m\right\}\\
             M_{[\Delta]}^\lambda &= M^\lambda \cap  M_{[\Delta]}.
        \end{split}
    \end{equation}
    In particular, since $\gg$ acts locally finitely, $\gh$ has to act semisimply and so
    \[
    M = \bigoplus_{\substack{ \lambda \in \gh^* \\ \Delta \in \mathbb R}} M_{[\Delta]}^\lambda.
    \]
    We denote by $V^k(\lambda)$ the Verma module whose top level is the  simple $\gg$-module $E_\lambda$ of highest-weight $\lambda$. If $V^k(\lambda)$ is in KL$_k(\gg)$ then its top level needs to be finite dimensional and so $\lambda in P^+$ needs to be dominant integral. Any simple object $M$ in KL$_k(\gg)$ needs to be lower-bounded by smoothness and its top level needs to be a finite-dimensional $\gg$-module. Let $E_\lambda$ be a direct summand of the top level then by the universal property of Verma modules there is a non-zero homomorphism $V^k(\lambda) \rightarrow M$ and so $M$ must be the simple quotient of $V^k(\lambda)$ which we denote by $L^k(\lambda)$. The conformal weight of the top level is 
    \[
    h_\lambda = \frac{(\lambda +2\rho, \lambda)}{2 (k + h^\vee)}
    \]
    with $\rho$ the Weyl vector of $\gg$ and $h^\vee$ the dual Coxeter number. Since 
    $(\lambda +2\rho, \lambda) \in \mathbb Q_{\geq 0}$ for all $\lambda \in P^+$ and $(\lambda +2\rho, \lambda) = 0$  only if $\lambda = 0$ we have that $h_\lambda = 0$ if and only if $\lambda = 0$. Moreover, for any object $M \in$ KL$_k(\gg)$ we have $M_{[-n]} = 0$ for $n \in \mathbb Z_{>0}$. Thus $L_0$ acts as $\frac{\Omega}{2(k+h^\vee)}$ on $M_{[0]}$ with $\Omega$ the Casimir of $\gg$. Since $\gg$ acts semisimply on finite-dimensional $\gg$-modules we have $M_{[0]}= M_{0}$. Further $M_0$ has to be a direct sum of trivial $\gg$-modules and so $M_{[0]} = M_{0}^{\gg[t]}:= M_0 \cap M^{\gg{t}}$. Thus by Frobenius reciprocity
    \[
    \text{Hom}{\hat\gg}\left(V^k(\gg), M\right) \cong \text{Hom}{\gg[t]}\left( E_0, M\right) \cong M_0^{\gg[t]} = M_{[0]}. 
    \]
    The functor $M \mapsto M_{[0]}$ is exact and hence $V^k(\gg)$ is projective. 
\end{proof}

\begin{theorem}\label{thm:coset-rat}
Let $\ell -1$ be an admissible level for $\widehat{\gs\gp}_{2n}$. Then $\cC_\ell(n)$ is strongly rational. 
\end{theorem}
\begin{proof}
We need to show that $\cC_\ell(n)$ is a conformal extension of 
\begin{equation} \label{eq:levels}
\cW_{\ell_1}(\gs\gp_{2n}) \otimes \cW_{\ell_2}(\gs\gp_{2n}),\qquad \ell_1 = -(n+1) + \frac{\ell+n+1}{2\ell+2n+1}, \qquad \ell_2 = -(n+1) + \frac{\ell+n}{2\ell+2n+1}.
\end{equation}
If $\ell - 1$ is admissible for $\gs\gp_{2n}$, then both $\ell_1$ and $\ell_2$ are non-degenerate principal or co-principal admissible levels and hence $\cW_{\ell_1}(\gs\gp_{2n})$ and $\cW_{\ell_2}(\gs\gp_{2n})$ are strongly rational \cite{Ar1, Ar2}. By \cite{SY, CMSY} the same is then true for $\cC_\ell(n)$ as a simple vertex algebra that is a conformal extension of a strongly rational one.

The simple admissible level affine vertex algebra $L_\ell(\gs\gp_{2n})$ embeds into $L_{\ell-1}(\gs\gp_{2n}) \otimes \cE(2n)$ by \cite[Cor. 4.1]{KW89}. Hence by \cite[Theorem 4.5.2]{GS} as explained in the proof  of \cite[Theorem 5.3]{GS} the maximal ideal of $V^\ell(\go\gs\gp_{1|2n})$ is generated by the one of the $V^\ell(\gs\gp_{2n})$ subalgebra and so also $L_\ell(\go\gs\gp_{1|2n})$ acts on $L_{\ell-1}(\gs\gp_{2n}) \otimes \cE(2n)$.
Theorem 4.1 of \cite{ACK} is stated for Lie algebras, but it also holds for $\gg = \go\gs\gp_{1|2n}$, since this Lie superalgebra has the property that $V^k(\gg)$ is projective in $KL_k(\gg)$ as long as $k  + h^\vee \notin \mathbb Q_{\leq 0}$, see Lemma \ref{lem:proj}. Thus by \cite[Theorem 4.1]{ACK}
the coset $\text{Com}(L_\ell(\go\gs\gp_{1|2n}), L_{\ell-1}(\gs\gp_{2n}) \otimes \cE(2n))$  is simple and by \cite[Theorem 8.1]{CL2} it is isomorphic to $\cW_{\ell_2}(\gs\gp_{2n})$, $\text{Com}(L_\ell(\go\gs\gp_{1|2n}), L_{\ell-1}(\gs\gp_{2n}) \otimes \cE(2n))\cong \cW_{\ell_2}(\gs\gp_{2n})$. 
By \cite[Theorem 4.2]{CGL} $L_\ell(\go\gs\gp_{1|2n})$ is a conformal extension of $L_\ell(\gs\gp_{2n}) \otimes \cW_{\ell_1}(\gs\gp_{2n})$.
It follows that $\cC_\ell(n)$ is a conformal extension of $\cW_{\ell_1}(\gs\gp_{2n}) \otimes \cW_{\ell_2}(\gs\gp_{2n})$. Since both levels are non-degenerate principal or coprincipal admissible and hence $\cW_{\ell_1}(\gs\gp_{2n}) \otimes \cW_{\ell_2}(\gs\gp_{2n})$ is strongly rational, $\cC_\ell(n)$ is simple and as a conformal extension of a strongly rational vertex algebra. $L_{\ell-1}(\gs\gp_{2n}) \otimes \cE(2n))$ is $\mathbb Z_{\geq 0}$-graded by conformal weight, hence the same must be true for its coset $\cC_\ell(n)$. Thus the assumptions of \cite[Theorem 4.10]{CMSY} are satisfied and so $\cC_\ell(n)$ is strongly rational 
as well by \cite{SY, CMSY}.
\end{proof}

Since $\cC^\ell(n)$ is a $1$-parameter quotient of $\cW^{\gs\gp_2}_{\infty}$, it is immediate from Corollary \ref{weakgeneration:finiteness} that for each $n$, $\cC_\ell(n)$ is a quotient of $\cW^{\gs\gp_2}_{\infty}$ for generic values of $\ell$. In the next subsection, we will show that $\cC_\ell(n)$ is generated by the fields in weights up to $4$ for all $\ell$ such that $\ell -1$ is admissible for $\widehat{\gs\gp}_{2n}$. Therefore all the above strongly rational vertex algebras $\cC_\ell(n)$ are quotients of $\cW^{\gs\gp_2}_{\infty}$.

\subsection{Branching Rules}
Gurbir Dhillon, Shigenori Nakatsuka and the first named author have a current project on algebras of chiral differential operators (CDOs) $\cD^{\text{ch}}_{G,k}$ for a supergroup $G$. In particular in that work a relation between the CDO of $\go\gs\gp_{1|2n}$ at level $\ell$ and the CDO of $\gs\gp_{2n}$ at level $\ell-1$ will be shown, 
\[
\mathcal D^{\text{ch}}_{\text{Osp}_{1|2n}, \ell + n + \frac{1}{2}} \cong \mathcal D^{\text{ch}}_{\text{Sp}_{2n}, \ell + n} \otimes \cE(2n).
\]
Let $P^+_{\go\gs\gp_{1|2n}}, P^+_{\gs\gp_{2n}}$ be the sets of dominant weights. 
Let $V^\ell(\lambda)$ denote the universal Weyl module of $V^\ell(\gg)$ whose top level is the irreducible highest-weight representation $E_\lambda$  of highest-weight $\lambda$ of $\gg$. Let $L_\ell(\lambda)$ be its simple quotient.   
Then for generic level $\ell$ the CDOs decompose as
\[
\mathcal D^{\text{ch}}_{\text{Osp}_{1|2n}, \ell + n + \frac{1}{2}} \cong \bigoplus_{\lambda \in P^+_{\go\gs\gp_{1|2n}}} V^\ell(\lambda) \otimes V^{-\ell - 2n-1}(\lambda), \qquad
\mathcal D^{\text{ch}}_{\text{Sp}_{2n}, \ell + n} \cong \bigoplus_{\lambda \in P^+_{\gs\gp_{2n}}} V^{\ell-1}(\lambda) \otimes V^{-\ell - 2n-1}(\lambda).
\]
The twisted quantum Hamiltonian reductions \cite{AF} are labelled by dominant coweights $\mu \in \check{P}^+$ and the reduction functor is denoted by $H^0_{DS, \mu}$. Set $T^k_{\lambda, \mu} := H^0_{DS, \mu}(V^k(\lambda))$. 
As before set 
$\ell_1 = -(n+1) + \frac{\ell+n+1}{2\ell+2n+1}$ and $\ell_2 = -(n+1) + \frac{\ell+n}{2\ell+2n+1}$.
Then by \cite[Theorems 3.1 and 3.2]{CGL} for generic $\ell$
\[
 V^\ell(\mu) \cong \bigoplus_{\lambda \in P^+_{\gs\gp_{2n}}}  V^\ell(\lambda) \otimes T^{\ell_1 }_{\lambda, \mu}
\] 
and using that coweights of $\gs\gp_{2n}$ are naturally identified with weights of $\gs\go_{2n+1}$ and dominant weights of $\go\gs\gp_{1|2n}$ are in one-to-one correspondence with dominant weights of $\gs\go_{2n+1}$ that lie in the root lattice. 
It follows that 
\begin{equation}\nonumber
\begin{split}
\mathcal D^{\text{ch}}_{\text{Osp}_{1|2n}, \ell + n + \frac{1}{2}} &\cong \bigoplus_{\lambda \in P^+_{\go\gs\gp_{1|2n}}} V^\ell(\lambda) \otimes V^{-\ell - 2n-1}(\lambda) 
\cong \bigoplus_{\lambda \in P^+_{\go\gs\gp_{1|2n}}}  \bigoplus_{\mu \in P^+_{\gs\gp_{2n}}}
 V^\ell(\lambda)  \otimes T^{\ell_2  }_{\mu, \lambda} \otimes V^{-\ell -2n-1}(\mu)
\end{split}
\end{equation}
On the other hand 
\[
\mathcal D^{\text{ch}}_{\text{Osp}_{1|2n}, \ell + n + \frac{1}{2}} \cong \mathcal D^{\text{ch}}_{\text{Sp}_{2n}, \ell + n} \otimes \cE(2n) \cong 
\bigoplus_{\mu \in P^+_{\gs\gp_{2n}}} V^{\ell-1}(\mu) \otimes V^{-\ell - 2n-1}(\mu) \otimes \cE(2n)
\]
and so by comparing coefficients we get the generic branching rules
\[
V^{\ell-1}(\mu)  \otimes \cE(2n) \cong \bigoplus_{\lambda \in P^+_{\go\gs\gp_{1|2n}}}   
 V^\ell(\lambda)  \otimes T^{\ell_2  }_{\mu, \lambda}. 
\]
Let $\ell -1$ be admissible and denote by $L_{\ell_2} (\mu, \lambda)$ the simple quotient of $T^{\ell_2 }_{\mu, \lambda}$. 
We want to determine the branching rules at these admissible levels. First we note the following. 
In \cite{C} the fusion rules of principal $\cW$-algebras at non-degenrate (co)-principal admissible levels were computed. In particular it was noted that the fusion rules of the $L_{\ell_2  }(0, \lambda)$ coincide with the ones of the one of the semisimplification of $U_q(\gs\go_{2n+1})$ for $q = e^{\frac{\pi i }{2}(\ell_2 + n + 1)}$. But the tensor product of the semisimplification is in principle computable, see \cite[Prop. 5]{Saw}.  In particular in type $B$ the tensor product with the standard representation $E_{\omega_1}$ is given in Example 4.3 of \cite{Row}\footnote{\cite{Row} states the result for odd roots of unity as he is using an older result than \cite[Prop.5]{Saw} which was stated in lesser generality. Since \cite[Prop.5]{Saw} holds for all roots of unity $q$, the same is then true for Example 4.3 of \cite{Row}.} and the whole tensor ring is generated by it.

With this information and (2.3) and (2.4) of \cite{CGL} exactly the same reasoning as the proof of Theorem 4.2 of \cite{CGL} can be done, and one gets that
\[
L_{\ell-1}(\mu)  \otimes \cE(2n) \cong \bigoplus_{\lambda \in P^B_Q}   
 L_\ell(\lambda)  \otimes L_{\ell_2  }(\mu, \lambda) 
\]
with $\ell = - (n+1) + \frac{u}{v}$ and $P^B = P(u+v, 2u +v)$ the set of admissible weights of $\gs\go_{2n+1}$ at level $\ell_2$ and $P^B_Q = P^B \cap Q$ the subset of admissible weights that lie in the root lattice $Q$. 
In particular 
\[
L_{\ell-1}(\gs\gp_{2n})  \otimes \cE(2n) \cong \bigoplus_{\lambda \in P^B_Q}   
 L_\ell(\lambda)  \otimes L_{\ell_2  }(0, \lambda). 
\]
 By the main theorem of \cite{CKM} it means that $L_{\ell-1}(\gs\gp_{2n})  \otimes \cE(2n)$ is generated under OPE by $L_\ell(\go\gs\gp_{1|2n})  \otimes \cW_{\ell_2  }(\gs\gp_{2n})\oplus L_\ell(\omega_1)  \otimes L_{\ell_2  }(0, \omega_1)$. 
The same reasoning now applies to $C_\ell(n)$, that is 
\[
C_\ell(n) \cong  \bigoplus_{\lambda \in P^B_Q}   
 L_{\ell_1}(0, \lambda)  \otimes L_{\ell_2  }(0, \lambda)
\]
is generated under OPE by $\cW_{\ell_1}(\gs\gp_{2n})  \otimes \cW_{\ell_2  }(\gs\gp_{2n})\oplus L_{\ell_1}(0, \omega_1)  \otimes L_{\ell_2  }(0, \omega_1)$. This is stronger than the property that we need, namely that $\cC_{\ell}(n)$ is generated by the fields in weights up to $4$, so we obtain

\begin{corollary} \label{rationalquotients} For all $n\geq 1$ and all levels $\ell$ such that $\ell -1$ is admissible for $\widehat{\gs\gp}_{2n}$,  $\cC_\ell(n)$ is a quotient of $\cW^{\gs\gp_2}_{\infty}$.
\end{corollary}

\subsection{Unitarity}

The notion of unitarity of a vertex algebra and its modules was introduced in \cite{DL14}. Examples of strongly rational vertex algebras that are strongly unitary in the sense that all of its modules are unitary are free fermions and affine vertex algebras at positive integer level. Unitarity is important for the physics of conformal field theory, i.e. it is necessary for the CFTs involved in the correspondence to higher spin gravity theories on Anti-de-Sitter space in three dimensions; but also the connection to conformal nets can only be made for unitary vertex algebras \cite{Te}.

Cosets provide a way to find new unitary vertex algebras, namely if $W \subseteq V$ are unitary, then $\text{Com}(W, V)$ is unitary as well \cite[Cor. 2.8]{DL14}. By the same argument as the proof of that corollary, the $\text{Com}(W, V)$-module $\text{Hom}_{W-\text{mod}}(M, N)$ is unitary if $M$ is a unitary $W$-module and $N$ a unitary $V$-module. We obtain 
\begin{corollary}
For $\ell \in \mathbb Z_{\geq 1}$,  $\cC_\ell(n)$  is unitary and any  $\cC_\ell(n)$-module that appears in the decomposition of any module of $L_{\ell-1}(\gs\gp_{2n})  \otimes \cE(2n)$ is unitary as well. 
\end{corollary}
We conjecture that $\cC_\ell(n)$ is strongly unitary for $\ell \in \mathbb Z_{\geq 1}$. We note that the family $\cC_\ell(n)$ for $\ell, n$ positive integers should exactly be the one corresponding to higher spin gravity on AdS$_3$ with $\text{Sp}_2$-restricted matrix extension \cite{CHY}. 

The GKO coset construction was used to show that a simple, principal $\cW$-algebra of a simply-laced Lie algebra at non-degenerate admissible level is unitary if and only if the level is of the form $k = - h^\vee + \frac{p}{q}$ and $|p-q|=1$ \cite[Theorem 12.6]{ACL}. Unitary, strongly rational $\cW$-algebras are rare, and we will now explain that up to collapsing levels where $\cW_k(\gg) \cong \mathbb C$, the minimal unitary series of simply-laced principal $\cW$-algebras are the only unitary, strongly rational principal $\cW$-algebras. In particular, the $\cW$-algebras $\cW_{\ell_1}(\gs\gp_{2n})$ and $\cW_{\ell_2}(\gs\gp_{2n})$ appearing in the description \eqref{eq:levels} of $\cC_{\ell}(n)$, are not unitary.

In the proof of \cite[Theorem 12.6]{ACL}, a necessary condition for unitarity of a strongly rational vertex algebra was given, namely that the effective central charge coincides with the actual central charge. This criterion can be used to show that in the non simply-laced case, the only non-degenerate admissible levels $k$ where $\cW_k(\gg)$ is unitary are the collapsing levels where $\cW_k(\gg) \cong \mathbb C$. We verify this case by case, (because of Feigin-Frenkel duality \cite{FF} it is enough to consider principal non-degenerate admissible levels).

Let $\cW_k(\gg)$ be the simple principal $\cW$-algebra of $\gg$ at level $k$ and let $k$ be non-degenerate principal admissible, that is $k = - h^\vee + \frac{p}{q}$ with $p, q$ co-prime positive integers and $p \geq h^\vee, q \geq h$ and $q$ is co-prime to the lacity $r^\vee$ of $\gg$. 
In this case $\cW_k(\gg)$ is strongly rational \cite{Ar1, Ar2} and the central charge is \cite{Ar2}
\[
c = - \frac{\ell ((h+1)p - h^\vee q)(r^\vee h^\vee_{{}^L\gg}p - (h+1)q)}{pq}
\]
with $\ell$ the rank of $\gg$. 
While the effective central charge for principal non-degenerate admissible level is  \cite{KW4}
\[
c_{\text{eff}} = \ell - \frac{h^\vee \dim \gg}{pq}.
\]
In the case of $\gg = \gs\gp_{2n}$ one has $\ell = n, r^\vee =2, h^\vee = n+1, h =2n, h^\vee_{{}^L\gg} = 2n-1$. Hence $c = c_{\text{eff}}$ if and only if 
\[
0 = -(n+1) (2n+1) + p^2 2(4n^2-1) + q^2(n+1)(2n+1) - pq 2(n+1)(4n-1). 
\]
Set $p = \frac{q}{2} + m$. In particular since $k$ is principal admissible $q$ is odd and so $m \in \mathbb Z+ \frac{1}{2}$. 
The condition simplifies to
\[
0 = \frac{3}{2}q^2 - 6mnq + (2n+1) ( 2m^2(2n-1) -(n+1)) 
\]
whose solutions are given by 
\[
2mn \pm \sqrt{D}, \qquad D = 4 m^2n^2  - \frac{2}{3} (2n+1) (2m^2(2n-1) -(n+1)) = \frac{2}{3}(n+1)(2n+1 -2m^2(n-1)). 
\]
We analyze $D$. First if $n=1$, then $D =4$ and so $q = 2m \pm 2$, i.e. $p = q \pm 1$, recovering the well-known unitary Virasoro minimal model series. 
If $n>1$, then we get a solution for $m^2 = \frac{1}{4}$, namely $D = (n+1)^2$. In this case $q = \epsilon n \pm \epsilon' (n+1)$ for $\epsilon, \epsilon' \in \{ \pm 1\}$ and the only non-degenerate admissible level of those four is  $q = 2n+1$ and hence $p = n+1$. In this case it is well-known that $\cW_k(\gs\gp_{2n}) \cong \mathbb C$, this follows e.g. from Theorem 1.1. of \cite{AMP} using that the central charge vanishes.
If $|m| \geq \frac{3}{2}$, then $D<1$, so there is no real solution to our condition. 

In the case of $\gg_2$, one needs the solutions of 
$84p^2 +28q^2 -96pq = 28$ with unique non-degenerate principal admissible level solution $q = 7, p=4$. Finally, in the case of $\gf_4$, one needs the solutions of 
$234p^2 +117q^2-330pq = 117$ with unique non-degenerate principal admissible level solution $q = 13, p=9$. In both cases we verify that $c=0$ and hence again by Theorem 1.1. of \cite{AMP} the $\cW$-algebra collapses to $\mathbb C$.

\subsection{Some exceptional $\cW$-algebras} Let $\gg$ be a simple Lie algebra and $k = -h^{\vee} + \frac{p}{q}$ an admissible level for $\widehat{\gg}$. The associated variety of $L_k(\gg)$ is a nilpotent orbit closure $\overline{\mathbb{O}_q}$ which depends only on the denominator $q$. If $f\in \gg$ is a nilpotent lying in $\mathbb{O}_q$, the simple $\cW$-algebra $\cW_k(\gg,f)$ is known to be non-zero and lisse \cite{Ar1}. Such pairs $(f,q)$ are called {\it exceptional pairs}; they generalize the notion of exceptional pair in \cite{KW4,EKV}. The corresponding $\cW$-algebras are also called exceptional, and were conjectured by Arakawa to be rational in \cite{Ar1}, generalizing the original conjecture of \cite{KW4}. This was proven for principal $\cW$-algebras by Arakawa in \cite{Ar2}, for all type $A$ cases by Arakawa and van Ekeren \cite{AvE}, and in full generality by McRae \cite{McR}.

Here we give a new proof of the rationality of certain exceptional $\cW$-algebras by exhibiting them as extensions of the tensor product of two rational principal $\cW$-algebras. First recall that $\cW_{r}(\gs\gp_{2(2m+1)}, f_{2m+1, 2m+1})$ is exceptional when $$r= -((2m+1)+1) + \frac{p}{2m+1},\qquad p \geq 2m+1,\qquad p,\ 2m+1 \ \text{coprime}.$$
 If $k$ is a positive integer, we have an intersection of the truncation curves for $\cW^{r}(\gs\gp_{2(2m+1)}, f_{2m+1, 2m+1}) \cong \cC^{\psi}_{BC}(0,m)$ and $\cC^{\ell}(k)$ at the point
$$r = -(2m+2) + \frac{2 + k + 2 m}{2m+1} = -(k+1) + \ell.$$
 The shifted levels for the subalgebras $\cW^{\ell_1}(\gs\gp_{2k})$ and $\cW^{\ell_2}(\gs\gp_{2k})$ of $\cC^{\ell}(k)$ are 
$\frac{1 + k}{3 + 2 k + 2 m}$ and $\frac{2 + k + 2 m}{3 + 2 k + 2 m}$. It follows that when  $k +1$ and $2m+1$ are coprime, $\cC_{\ell}(k)$ is strongly rational.

By Theorem \ref{weakgeneration:Walgebras}, $\cW_{r}(\gs\gp_{2(2m+1)}, f_{2m+1, 2m+1})$ is generated by the fields of weight at most $4$ for all such values of $r$. This implies that $\cW_{r}(\gs\gp_{2(2m+1)}, f_{2m+1, 2m+1}) \cong \cC_{\ell}(k)$ at these points, which provides a new proof that these exceptional $\cW$-algebras are strongly rational.

Similarly, recall that $\cW_{r}(\gs\go_{4m+1}, f_{2m,2m})$ is exceptional when $r = -(4m-1) + \frac{p}{2m+1}$ for $p \geq 4m-1$ and $p,2m$ coprime. If $k$ is a positive integer, we have an intersection of the truncation curves for $\cW_{r}(\gs\go_{4m+1}, f_{2m,2m})^{\mathbb{Z}_2} \cong \cC^{\psi}_{CB}(0,m)$ and $\cC^{\ell}(k)$ at the point
\begin{equation} \label{orbifold:valuesofr} r =-(4m-1) + \frac{1 + 2 k + 4 m}{2 m} = -(k+1) + \ell.\end{equation}
 The shifted levels for the subalgebras $\cW^{\ell_1}(\gs\gp_{2k})$ and $\cW^{\ell_2}(\gs\gp_{2k})$ of $\cC^{\ell}(k)$ are 
$\frac{1 + 2 k }{2 (1 + 2 k + 2 m)}$ and $\frac{1 + 2 k + 4 m}{2 (1 + 2 k + 2 m)}$. It follows that when $1+2k$, $2m$ are coprime, $\cC_{\ell}(k)$ is strongly rational. 

Recall that $\cW^{r}(\gs\go_{4m+1}, f_{2m,2m})$ is freely generated of type $\cW\big(1^3, 2, 3^3, \dots, (2m-1)^3,2m, \big(\frac{2m+1}{2}\big)^2\big)$, and that its $\mathbb{Z}_2$-orbifold is a $1$-parameter quotient of $\cW^{\gs\gp_2}_{\infty}$. By the same argument as the proof of Theorem \ref{weakgeneration:Walgebras}, as long as $k$ is not in the set $\{ -1, -2, -4, -\frac{4}{3}, -\frac{8}{5}, -\frac{16}{7}\}$ and the central charge $c \neq -\frac{22}{5}$, the fields in weight at most $3$ weakly generate all generators of $\cW^{r}(\gs\go_{4m+1}, f_{2m,2m})$ except for the two fields $P^{\pm}$ of weight $\frac{2m+1}{2}$, which have eigenvalue $-1$ with respect to the $\mathbb{Z}_2$-action. The orbifold $\cW^{r}(\gs\go_{4m+1}, f_{2m,2m})^{\mathbb{Z}_2}$ is strongly generated by the fields in weights $1,2,\dots, 2m$ together with the composite fields for $n\geq 0$:
\begin{equation} \label{CBorbifoldgen}
\begin{split} & w^{2n+2m+2} =\ :\!\partial^{2n+1}P^{+}P^{-}\!: \ -\  :\!P^{+}\partial^{2n+1} P^{-}\!:, \qquad h^{2n+2m+1} =\ :\!\partial^{2n}P^{+}P^{-}\!: \ +\ :\!P^{+}\partial^{2n} P^{-}\!:,
\\ & x^{2n+2m+1} =\ :\!\partial^{2n}P^{+}P^{+}\!:, \qquad y^{2n+2m+1} =\ :\!\partial^{2n}P^{-}P^{-}\!:.
\end{split} \end{equation}
Let $W^{2n+2m+2}, X^{2n+2m+1}, Y^{2n+2m+1}, H^{2n+2m+1}$ be the images of the corresponding generators of $\cW^{\gs\gp_2}_{\infty}$ in $\cW^{r}(\gs\go_{4m+1}, f_{2m,2m})^{\mathbb{Z}_2}$. We have
\begin{equation} \begin{split} & W^{2n+2m+2} = \lambda_{2n+2m+2} w^{2n+2m+2} + \cdots, \qquad H^{2n+2m+1} = \lambda_{2n+2m+1} h^{2n+2m+1} + \cdots,
\\ & X^{2n+2m+1} = \lambda_{2n+2m+1} x^{2n+2m+1} + \cdots,\qquad Y^{2n+2m+1} = \lambda_{2n+2m+1} y^{2n+2m+1} + \cdots,
\end{split}
\end{equation}
 for some structure constants $\lambda_{2n+2m+1}, \lambda_{2n+2m+2}$, where the remaining terms depend only on the generators in weights $1,2,\dots, 2m$ and their derivatives.

Since the fields $X^{2m+1}, Y^{2m+1}, H^{2m+1}$ are nontrivial in Zhu's commutative algebra of $\cW^{r}(\gs\go_{4m+1}, f_{2m,2m})$, and the Poisson structure is independent of $r$, the same argument in the proof of Theorem \ref{weakgeneration:Walgebras} that shows that $\lambda_{2m+1}$ is a nonzero constant for all $i < m$, we also get that $\lambda_{2m+1}$ is a nonzero constant. (It is nonzero because for generic values of $r$, $\cW^{r}(\gs\go_{4m+1}, f_{2m,2m})^{\mathbb{Z}_2}$ satisfies the weak generation property). Therefore $h^{2m+1}, x^{2m+1}, y^{2m+1}$ are generated by the fields in weight at most $3$ when $k \notin \{ -1, -2, -4, -\frac{4}{3}, -\frac{8}{5}, -\frac{16}{7}\}$ and $c \neq -\frac{22}{5}$. 

It is not true $\lambda_{2n+2m+1}, \lambda_{2n+2m+2}$ are independent of $r$ for $n >0$. However, using the OPEs $H^3(z) P^{\pm}(w)$ and $X^3(z) P^{\pm}(w)$, up to the normalization of $H^3, X^3$ we compute
\begin{equation}
\begin{split}
X^3_{(0)}h^{2n+2m+1} =& \frac{4  k (k+1) (4 k n+2 k+2 n-2 m+1) (2 k n+2 k+n+m+1)}{(n+1) (2 n+1) (2 k+4 m+1)}x^{2n+2m+3}+\dotsb,\\
H^3_{(1)}w^{2n+2m+2} =&\frac{2 k (k+1) (2 k n+2 k+n+m+1) (4 k n+4 k m+6 k+2 n+6 m+3)}{(n+1) (2 k+4 m+1)}h^{2n+2m+3}+\dotsb,\\
H^3_{(1)}h^{2n+2m+1} =& \frac{4 k (k+1) (2 k+1) (m+n+1) (4 k n+2 k+2 n-2 m+1)}{(2 n+1) (2 k+4 m+1)}w^{2n+2m+2}+\dotsb,\\
\end{split}
\end{equation}
where the remaining terms depend on generators of lower weight. Again, by the same argument in the proof of Theorem \ref{weakgeneration:Walgebras}, $\lambda_{2n+2m+1}, \lambda_{2n+2m+2}$ are nonzero whenever the numerator and denominator in the above expressions do not vanish. This never happens for positive integer values of $k$, so all the fields \eqref{CBorbifoldgen} in  $\cW^r(\gs\go_{4m+1}, f_{2m,2m})^{\mathbb{Z}_2}$ are generated by the fields in weight at most $3$ for the values of $r$ given by \eqref{orbifold:valuesofr}. Therefore the simple quotient $\cW_{r}(\gs\go_{4m+1}, f_{2m,2m})^{\mathbb{Z}_2}$ has the weak generation property, and hence is isomorphic to $\cC_{\ell}(k)$. Since $\cW_{r}(\gs\go_{4m+1}, f_{2m,2m})$ is a simple current extension of $\cC_{\ell}(k)$, this provides a new proof that these $\cW$-algebras are strongly rational.

\subsection{New rational $\cW$-superalgebras}
Now we consider $\cW_{r}(\go\gs\gp_{1|2(2m+1)}, f_{2m+1,2m+1})$. If $k$ is a positive integer, we have an intersection of the truncation curves for $\cW_{r}(\go\gs\gp_{1|2(2m+1)}, f_{2m+1,2m+1})^{\mathbb{Z}_2} \cong \cC^{\psi}_{BO}(0,m)$ and $\cC^{\ell}(k)$ at the point
$$r = -(2m+\frac{3}{2}) + \frac{3 + 2 k + 4 m}{2 (1 + 2 m)}  = -(k+1) + \ell.$$

 The shifted levels for the subalgebras $\cW^{\ell_1}(\gs\gp_{2k})$ and $\cW^{\ell_2}(\gs\gp_{2k})$ of $\cC^{\ell}(k)$ are 
$\frac{3 + 2 k + 4 m}{4 (1 + k + m)}$ and $\frac{1 + 2 k}{4 (1 + k + m)}$. It follows that when $2m+1$ and $2k+1$ are coprime, $\cC_{\ell}(k)$ is strongly rational.

\begin{theorem} \label{newrationalsuper} For all $m \geq 1$, $\cW_{r}(\go\gs\gp_{1|2(2m+1)}, f_{2m+1,2m+1})$ is strongly rational for all but finitely many values $r = -(2m+\frac{3}{2}) + \frac{3 + 2 k + 4 m}{2 (1 + 2 m)}$, when $2m+1$ and $2k+1$ are coprime.
\end{theorem}

\begin{proof} Since we have an intersection of truncation curves for $\cC^{\psi}_{BO}(0,m)$ and $\cC^{\ell}(k)$ at these values of $r$, the simple quotients of the specializations of  $\cC^{\psi}_{BO}(0,m)$ and $\cC^{\ell}(k)$ to the corresponding values $\psi = \frac{3 + 2 k + 4 m}{2 (1 + 2 m)}$ and $\ell =  (k+1) + \frac{3 + 2 k + 4 m}{2 (1 + 2 m)}$ are isomorphic. 

As above, $\cW^{r}(\go\gs\gp_{1|2(2m+1)}, f_{2m+1,2m+1})$ is freely generated of type $\cW\big(1^3, 2, 3^3, \dots, 2m,(2m+1)^3, (m+1)^2\big)$, where the two fields $P^{\pm}$ in weight $m+1$ are odd and have eigenvalue $-1$ with respect to the $\mathbb{Z}_2$-action. The orbifold $\cW^{r}(\go\gs\gp_{1|2(2m+1)}, f_{2m+1,2m+1})^{\mathbb{Z}_2}$ is a $1$-parameter quotient of $\cW^{\gs\gp_2}_{\infty}$, and is strongly generated by the fields in weights $1,2,\dots, 2m+1$ together with the composite fields for $n\geq 0$:
\begin{equation} \label{BOorbifoldgen}
\begin{split} & w^{2n+2m+2} =\ :\!\partial^{2n}P^{+}P^{-}\!: \ + \ :\!P^{+}\partial^{2n} P^{-}\!:, \qquad h^{2n+2m+3} =\ :\!\partial^{2n+1} P^{+}P^{-}\!:\ -\ :\!P^{+}\partial^{2m+1} P^{-}\!:,
\\ & x^{2n+2m+3} =\ :\!\partial^{2n+1}P^{+}P^{+}\!:, \qquad y^{2n+2m+3} =\ :\!\partial^{2n+1}P^{-}P^{-}\!:. 
\end{split} \end{equation}
Let $W^{2n+2m+2}, X^{2n+2m+1}, Y^{2n+2m+1}, H^{2n+2m+1}$ be the images of the corresponding generators of $\cW^{\gs\gp_2}_{\infty}$ in $\cW^{r}(\go\gs\gp_{1|2(2m+1)}, f_{2m+1,2m+1})^{\mathbb{Z}_2}$. As above, we have
\begin{equation} \begin{split} & W^{2n+2m+2} = \lambda_{2n+2m+2} w^{2n+2m+2} + \cdots, \qquad H^{2n+2m+1} = \lambda_{2n+2m+1} h^{2n+2m+1} + \cdots,
\\ & X^{2n+2m+1} = \lambda_{2n+2m+1} x^{2n+2m+1} + \cdots,\qquad Y^{2n+2m+1} = \lambda_{2n+2m+1} y^{2n+2m+1} + \cdots,
\end{split}
\end{equation} where the remaining terms depend only on the generators in weights $1,2,\dots, 2m+1$ and their derivatives.

Again by the proof of Theorem \ref{weakgeneration:Walgebras}, all even generators in weights $1,\dots, 2m+1$ are generated by the fields in weight at most $3$ for $k \notin \{ -1, -2, -4, -\frac{4}{3}, -\frac{8}{5}, -\frac{16}{7}\}$ and $c \neq -\frac{22}{5}$. Since $\cW^{r}(\go\gs\gp_{1|2(2m+1)}, f_{2m+1,2m+1})^{\mathbb{Z}_2}$ is generated by the fields in weight at most $3$ for generic values of $r$, the coefficient $\lambda_{2m+2}$ is a nonzero rational function of $r$, so it can vanish for at most finitely many values of $r$. However, since $w^{2m+2}$ vanishes in Zhu's commutative algebra, we cannot use the above argument to conclude that $\lambda_{2m+2}$ is a constant. 

A similar computation using the OPEs $H^3(z) P^{\pm}(w)$ and $X^3(z) P^{\pm}(w)$ shows that up to the normalization of $H^3, X^3$,
\begin{equation}
\begin{split}
    H^3_{(1)}w^{2n+2m+2} =&\frac{4 k (k+1) (2 k n+k+n+m+1) (4 k n+4 k m+6 k+2 n+6 m+5)}{(2 n+1) (2 k+4 m+3)}h^{2n+2m+3}+\dotsb\\
    H^3_{(1)}h^{2n+2m+3} =& \frac{2 k (k+1) (2 k+1) (n+m+2) (4 k n+4 k+2 n-2 m+1)}{(n+1) (2 k+4 m+3)}w^{2n+2m+4}+\dotsb\\
     X^3_{(0)}h^{2n+2m+3} =&\frac{4 k (k+1) (4 k n+4 k+2 n-2 m+1) (2 k n+3 k+n+m+2)}{(n+1) (2 n+3) (2 k+4 m+3)}x^{2n+2m+5}+\dotsb.
\end{split}
\end{equation}
As long as $\lambda_{2m+2} \neq 0$, and none of the above numerators and denominators vanish (which never happens for $k\in \mathbb{N}$), this shows that $\lambda_{2n+2m+2}, \lambda_{2n+2m+3}$ are also nonzero. Therefore $\cW_{r}(\go\gs\gp_{1|2(2m+1)}, f_{2m+1,2m+1})^{\mathbb{Z}_2}$ has the weak generation property whenever $\lambda_{2m+2} \neq 0$, and therefore is isomorphic to $\cC_{\ell}(k)$ for all but finitely many of the above values of $r$. Since it is strongly rational, so is the simple current extension $\cW_{r}(\go\gs\gp_{1|2(2m+1)}, f_{2m+1,2m+1})$. \end{proof}

\begin{remark} We expect that $\lambda_{2m+2} \neq 0$ for all $k\in \mathbb{N}$, so that Theorem \ref{newrationalsuper} holds for all $k\in \mathbb{N}$. \end{remark}

\subsection{New level-rank dualities} 
By the proof of \cite[Cor. 7.6]{CL4}, for positive integers $n,\ell$, we have the following level-rank duality for $\cC_{\ell}(n)$.
$$\cC_{\ell}(n) \cong  \cC_{BC, n + \ell + 1}(\ell-1,0) = \text{Com}(L_n(\mathfrak{sp}_{2\ell-2}), L_n(\mathfrak{sp}_{2\ell})).$$
By Corollary \ref{rationalquotients}, $\cC_{\ell}(n)$ is a quotient of $\cW^{\gs\gp_2}_{\infty}$ at these points. We observe that for each $\ell \in \mathbb{N}$, the truncation curves for $\cC^{\ell}(n)$ and $\cC^{\psi}_{CD}(n+2\ell,1)$ intersect at the point where $\psi = 2 n + 2 \ell +2$. Therefore we expect isomorphisms
\begin{equation}\label{levelrank:C} \cC_{\ell}(n) \cong   \cC_{BC, n + \ell + 1}(\ell-1,0) \cong \cC_{CD,2 n + 2 \ell +2}(n+2\ell,1), \end{equation}
In this subsection, we conjecture similar level-rank dualities for the algebras $\cC_{\ell}(n-\frac{1}{2})$ and $\cC_{\ell}(-\frac{n}{2})$.

First, recall that for $n \in \mathbb{N}$, $\cC^{\ell}(n-\frac{1}{2}) = \text{Com}(V^{\ell}(\mathfrak{osp}_{1|2n}), V^{\ell-1}(\mathfrak{osp}_{1|2n}) \otimes \cE(2n) \otimes \cS(1))^{Z_2}$. As above, based on intersections of truncation curves, we expect that for all integers $n\geq 1$ and $\ell \geq 2$, we have the isomorphisms
\begin{equation} \label{levelrank:O} \cC_{\ell}\big(n-\frac{1}{2}\big) \cong \cC_{BC, n + \ell + \frac{1}{2}}(\ell-1,0) \cong  \cC_{CB,2\ell + 2n + 1}(n+2\ell-1, 1),\end{equation}
where we recall that 
\begin{equation*} \begin{split} \cC_{BC, n + \ell + \frac{1}{2}}(\ell-1,0) & \cong \text{Com}(L_{n-1/2}(\mathfrak{sp}_{2\ell-2}), L_{n-1/2}(\mathfrak{sp}_{2\ell})),
\\ \cC_{CB,2\ell + 2n + 1}(n+2\ell-1, 1) & \cong \text{Com}(L_{-2\ell + 2}(\gs\go_{2n+4\ell-1}), \cW_{-2\ell}(\gs\go_{2n+4\ell +3}, f_{\text{min}}))^{\mathbb{Z}_2}. \end{split} \end{equation*}

Next, recall that for $n \in \mathbb{N}$, 
$\cC_{2\ell}(-n) = \text{Com}(V^{\ell}(\mathfrak{so}_{2n}), V^{\ell+2}(\mathfrak{so}_{2n}) \otimes \cS(2n))^{Z_2}$. Based on intersections of truncation curves,
we expect that for positive integers $n,\ell$,
\begin{equation} \label{levelrank:D} \cC_{2\ell}(-n) \cong \cC_{BD,-n-\ell+1}(\ell+1,0) \cong \cC_{CD,2 (n + \ell)}(3 n + 2 \ell -2, 1),\end{equation}
where 
\begin{equation*} \begin{split} \cC_{BD,-n-\ell+1}(\ell+1,0) & \cong \text{Com}(L_{2n}(\mathfrak{so}_{2\ell+2}), L_{-n}(\mathfrak{osp}_{2\ell +2|2}))^{\mathbb{Z}_2},
\\ \cC_{CD,2 (n + \ell)}(3n+2\ell-2, 1) & \cong \text{Com}(L_{-2 (2 n + \ell - 2)}(\gs\go_{6n+4\ell-4}), \cW_{-2 (2 n + \ell -1)}(\gs\go_{6n+4\ell}, f_{\text{min}}))^{\mathbb{Z}_2}.\end{split} \end{equation*}

Finally, recall that for $n \in \mathbb{N}$, $\cC_{2\ell+1}\big(- n - \frac{1}{2}\big) = \text{Com}(V^{\ell}(\mathfrak{so}_{2n+1}), V^{\ell+2}(\mathfrak{so}_{2n+1}) \otimes \cS(2n+1))^{Z_2}$. Again, based on intersections of truncation curves, we expect that for positive integers $n,\ell$,
\begin{equation} \label{levelrank:B} \cC_{2\ell+1}\big(- n - \frac{1}{2}\big) \cong \cC_{BB,-n-\ell}(\ell+1,0) \cong \cC_{CB, 2 (n +\ell+1)}(3 n + 2 \ell, 1),\end{equation}
where 
\begin{equation*} \begin{split} \cC_{BB,-n-\ell}(\ell+1,0) & \cong \text{Com}(L_{2n+1}(\mathfrak{so}_{2\ell+3}), L_{-n-1/2}(\mathfrak{osp}_{2\ell +3|2}))^{\mathbb{Z}_2},
\\ \cC_{CB,2 (n +\ell+1)}(3n+2\ell, 1) & \cong \text{Com}(L_{- 4 n - 2 \ell +1}(\gs\go_{6n+4\ell+1}), \cW_{ - 4 n - 2 \ell -1}(\gs\go_{6n+4\ell+5}, f_{\text{min}}))^{\mathbb{Z}_2}.\end{split} \end{equation*}

\section{A completion of $\cW^{\gs\gp_2}_{\infty}$ and its relation to $\cW^{\rm{ev}}_{\infty}$} \label{sect:completion}
	
Lemma \ref{twocopiesofW(sp)} suggests that there is a close connection between $\cW^{\gs\gp_2}_{\infty}$ and $\cW^{\rm{ev}}_{\infty}$, which was denoted by $\cW^{\text{ev}}(c,\lambda)$ in \cite{KL}. Recall that the orthosymplectic $Y$-algebras of \cite{GR} were denoted by $\cC^{\psi}_{iZ}(a,b)$ in \cite{CL4}, where $i = 1,2$, $Z = B,C,D,O$, and $a,b \in \mathbb{Z}_{\geq 0}$. They arise as $1$-parameter quotients of $\cW^{\rm{ev}}_{\infty}$ via the following procedure. There is an ideal $I_{iZ,a,b} \subseteq \mathbb{C}[c,\lambda]$ given in parametric form in terms of the parameter $\psi$ in \cite{CL4}. This generates a vertex algebra ideal $I_{iZ,a,b} \cdot  \cW^{\text{ev}}_{\infty}$, and
$\cC^{\psi}_{iZ}(a,b)$ is isomorphic to the simple graded quotient of $$\cW^{\text{ev}}_{\infty} / I_{iZ, a,b} \cdot \cW^{\text{ev}}_{\infty},$$ i.e., the quotient of $\cW^{\text{ev}}_{\infty}$ by its maximal ideal $\cI_{iZ,a,b}$ containing $I_{iZ,a,b}$. For example, in the case $I_{2C,0,n}$, 
$$\cC^{\psi}_{2C}(0,n) \cong \cW^k(\gs\gp_{2n}),\qquad k + n+1 = \psi.$$
For later use, we sometimes consider $1$-parameter quotients of $\cW^{\text{ev}}_{\infty} / I_{iZ, a,b} \cdot \cW^{\text{ev}}_{\infty}$ along ideals which are not maximal, i.e., quotients of $\cW^{\text{ev}}_{\infty}$ along non-maximal ideals $\cI$ containing $I_{iZ,a,b}$. We denote such a quotient by $\tilde{\cC}^{\psi}_{iZ}(a,b)$; it is apparent that as a $1$-parameter vertex algebra, $\cC^{\psi}_{iZ}(a,b)$ is the simple quotient of $\tilde{\cC}^{\psi}_{iZ}(a,b)$. For $n \in \mathbb{N}$, let $I_n \subseteq \mathbb{C}[c,k]$ be the ideal generated by $(k-n)$, and let $I_n \cdot \cW^{\gs\gp_2}_{\infty}$ be the ideal in $\cW^{\gs\gp_2}_{\infty}$ generated by $I_n$, so that $\cW^{\gs\gp_2, I_n}_{\infty} = \cW^{\gs\gp_2}_{\infty} / I_n \cdot \cW^{\gs\gp_2}_{\infty}$, which has simple graded quotient $\cW^{\gs\gp_2}_{\infty, I_n} \cong \cC^{\ell}(n)$.

We now consider the tensor product $\cW^{\rm{ev}}_{\infty} \otimes \cW^{\rm{ev}}_{\infty}$. Here the first copy of $\cW^{\rm{ev}}_{\infty}$ has parameters $c,\lambda$, and the second copy has parameters $c', \lambda'$, which we suppress from the notation. We use the notation $I_{iX, n,m}$ and $I'_{iX, n,m}$, for ideals in the first (respectively second) copy of $\cW^{\rm{ev}}_{\infty}$. If we give $\cW^{\gs\gp_2}_{\infty, I_n}$ the conformal vector $L + \frac{1}{2} \partial H^1$, Lemma \ref{twocopiesofW(sp)} has the following immediate corollary.

\begin{corollary} \label{cor:twocopiesofW(sp)} For $n \in \mathbb{N}$, there is a homomorphism
\begin{equation} \Psi_n: \cW^{\rm{ev}}_{\infty} \otimes \cW^{\rm{ev}}_{\infty} \rightarrow \cW^{\gs\gp_2}_{\infty, I_n},
\end{equation} which induces the conformal embedding
\begin{equation}  \label{inducedhomo} \cC^{\psi}_{2C}(0,n) \otimes \cC^{\psi'}_{2C}(0,n) \rightarrow \cC^{\ell}(n)\end{equation} given by Lemma \ref{twocopiesofW(sp)}. Here the ideals $I_{2C,0,n}$ and $I'_{2C,0,n}$ are parametrized by functions $(c(\psi), \lambda(\psi))$ and $(c'(\psi'), \lambda'(\psi'))$, and the shifted levels $\psi$ and $\psi'$ satisfy
\begin{equation} \label{twopsiparameters} \psi = \frac{\ell+n+1}{2\ell+2n+1}, \qquad \psi' = \ell_2 = (n+1) + \frac{\ell+n}{2\ell+2n+1}.\end{equation}
Via \eqref{twopsiparameters}, $\cC^{\psi}_{2C}(0,n) \otimes \cC^{\psi'}_{2C}(0,n)$ is a $1$-parameter vertex algebra with parameter $\ell$, so \eqref{inducedhomo} is a homomorphism of $1$-parameter vertex algebras.
\end{corollary}

It is natural next to ask whether there exists a vertex algebra homomorphism $\Psi: \cW^{\text{ev}}_{\infty} \otimes \cW^{\text{ev}}_{\infty} \rightarrow \cW^{\gs\gp_2}_{\infty}$ which induces the homomorphisms $\Psi_n$ for $n \in \mathbb{N}$ in the sense that $\Psi_n$ is the  composition of $\Psi$ with the quotient map $\cW^{\gs\gp_2}_{\infty} \rightarrow \cW^{\gs\gp_2}_{\infty, I_n} \cong \cC^{\ell}(n)$. In fact, no such map $\Psi$ exists, and it is easy to check this by computer. The reason is that the dimensions of each weight space of $\cW^{\gs\gp_2}_{\infty, I_n}$ increase without bound as $n$ increases, and the number of terms in the formula for the images of $L_i, W^4_i$ for $i=1,2$ also increases without bound. So if such a homomorphism were to exist, we would need to replace $\cW^{\gs\gp_2}_{\infty}$ with a completion in which certain infinite sums make sense. The main result in this section is that there exists a completion $\tilde{\cW}^{\gs\gp_2}_{\infty}$ of $\cW^{\gs\gp_2}_{\infty}$ with the following properties:

\begin{enumerate}
\item For $n \in \mathbb{N}$, we can take the quotient $\tilde{\cW}^{\gs\gp_2, I_n}_{\infty}$ along the ideal generated by $(k-n)$, as well as its simple graded quotient $\tilde{\cW}^{\gs\gp_2}_{\infty, I_n}$, and we have an isomorphism $\tilde{\cW}^{\gs\gp_2}_{\infty, I_n} \cong \cW^{\gs\gp_2}_{\infty, I_n}$.

\item There exists a homomorphism
\begin{equation} \Psi: \cW^{\rm{ev}}_{\infty} \otimes \cW^{\rm{ev}}_{\infty} \rightarrow \tilde{\cW}^{\gs\gp_2}_{\infty},
\end{equation} which induces the maps $\Psi_n$ in the sense that $\Psi_n$ is the composition of $\Psi$ with the quotient map $ \tilde{\cW}^{\gs\gp_2}_{\infty} \rightarrow \cW^{\gs\gp_2}_{\infty, I_n}$.
\end{enumerate}

\subsection{Completion of $\cW^{\gs\gp_2}_{\infty}$}
First, we replace the conformal vector $L \in \cW^{\gs\gp_2}_{\infty}$ with $\tilde{L} = L + \frac{1}{2}\partial H$, so that $X^1,H^1,Y^1$ have conformal weights $0,1,2$, respectively. Similarly, $X^{2i+1}, H^{2i+1}, Y^{2i+1}$ have conformal weights $2i, 2i+1, 2i+2$ with respect to $\tilde{L}$, respectively. Then the strong generating type of $\cW^{\gs\gp_2}_{\infty}$ becomes $\cW(0,1,2^3,3,4^3,5,\dots)$. The weight zero space is infinite-dimensional since it is isomorphic to $\mathbb{C}[X^1]$, and the higher weight spaces are linearly isomorphic to finite rank modules over $\mathbb{C}[X^1]$. We define the completion $\tilde{\cW}^{\gs\gp_2}_{\infty}$ by allowing the coefficients to be in the ring of power series $\mathbb{C}[[X^1]]$, i.e., we allow infinite sums of elements with fixed conformal weight. Since $X^1$ has Cartan weight $2$, the algebra is filtered but not graded by Cartan weight. It is more convenient to work with the eigenvalue of $\frac{1}{2} H^1_{(0)}$, which is half the Cartan weight. For $N \in \mathbb{Z}$, let $\cG^d_N := \cG^d_N(\tilde{\cW}^{\gs\gp_2}_{\infty})$ be the subspace of elements consisting of infinite sums of weight $d$ and $\frac{1}{2} H^1_{(0)}$-eigenvalue at least $N$, and $\cG_N:= \bigoplus_{d \geq 0}\cG^d_N$. Clearly we have 
$$\cG^d_N \supset \cG^d_{N+1},$$ so this is a decreasing filtration, and we have the associated graded algebra
$$\text{gr}(\tilde{\cW}^{\gs\gp_2}_{\infty}) = \bigoplus_{N \in \mathbb{Z}} \cG_N / \cG_{N+1}$$ and similarly for $d$th weighted component. For each weight $d\geq 0$, the minimum $\frac{1}{2} H^1_{(0)}$-eigenvalue is $-\frac{d}{2}$, and each weighted component of $\text{gr}(\tilde{\cW}^{\gs\gp_2}_{\infty})$ is finite-dimensional.

We need the following property of the filtration which is easy to check. 
\begin{equation} \label{filt:behavior}  \begin{split} & \text{For}\ a \in \cG_N, \ b \in \cG_M,\ \text{and}\ n \in \mathbb{Z},
\ a_{(n)} b \in \cG_{N+M},
\\ & \text{For}\ a \in \cG^d_{-d/2}, \ b \in \cG^e_{-e/2}, \ \text{and}\ n\geq 0, \ a_{(n)} b \in \cG^{d+e-n-1}_{-(d+e-1)/2}.
\end{split}\end{equation} 
For each weight $d \geq 0$ and any $M > N$ we have the projection maps 
$$\pi^d_{M,N}:\cG^d_{-d/2} / \cG^d_M \rightarrow \cG^d_{-d/2} / \cG^d_N.$$

Note that $\tilde{\cW}^{\gs\gp_2}_{\infty}$ has a well-defined vertex algebra structure in which locality continues to hold. For $n \in \mathbb{N}$, recall that $\cI_n \subseteq \cW^{\gs\gp_2}_{\infty}$ denotes the maximal ideal containing $(k-n)$, and that $\cC^{\ell}(n) \cong \cW^{\gs\gp_2}_{\infty, I_n}$. Similarly, we denote by $\tilde{\cW}^{\gs\gp_2}_{\infty, I_n}$ the simple quotient of $\tilde{\cW}^{\gs\gp_2,I_n}_{\infty} = \tilde{\cW}^{\gs\gp_2}_{\infty}/ I_n \cdot \tilde{\cW}^{\gs\gp_2}_{\infty}$ by its maximal ideal graded by conformal weight. Since $L_n(\gs\gp_2) \subseteq \tilde{\cW}^{\gs\gp_2,I_n}_{\infty}$, every power series in  $\mathbb{C}[[X^1]]$ truncates to a polynomial of degree at most $n$, so $\tilde{\cW}^{\gs\gp_2}_{\infty, I_n} \cong \cC^{\ell}(n) \cong \cW^{\gs\gp_2}_{\infty, I_n}$.

\begin{theorem} \label{completion:twocopies} After a suitable extension of scalars, we have a vertex algebra homomorphism
\begin{equation} \label{2paramembedding} \Psi: \cW^{\text{ev}}_{\infty} \otimes \cW^{\text{ev}}_{\infty} \rightarrow \tilde{\cW}^{\gs\gp_2}_{\infty}\end{equation} with the following properties:
\begin{enumerate}
\item For all $n \in \mathbb{N}$, composing this map with the quotient map $\tilde{\cW}^{\gs\gp_2}_{\infty} \rightarrow \cW^{\gs\gp_2}_{\infty, I_n}$, we recover the homomorphism from Lemma \ref{twocopiesofW(sp)}.
 \item As a module over $\cW^{\text{ev}}_{\infty} \otimes \cW^{\text{ev}}_{\infty}$, $\tilde{\cW}^{\gs\gp_2}_{\infty}$ decomposes as a direct sum of simple modules with $1$-dimensional lowest weight spaces. 
 
 \item For each $n \in \mathbb{N}$, under the quotient map $\tilde{\cW}^{\gs\gp_2}_{\infty} \rightarrow \cW^{\gs\gp_2}_{\infty, I_n}$, we recover the decomposition of $\cW^{\gs\gp_2}_{\infty, I_n}$ given by Lemma \ref{twocopiesofW(sp)} (2).
\end{enumerate}
 \end{theorem}

\begin{proof} Fix an integer $N$. As in the proof of Lemma \ref{twocopiesofW(sp)}, for each weight $i\geq 0$, we can choose a finite basis $m^i_{1}, \dots, m^i_{d_i}$ for $\cG^{i}_{-i/2} / \cG^i_{N+10}$ consisting of elements of weight $i$ with $\frac{1}{2} H^1_{(0)}$-eigenvalue at most $N+9$. By abuse of notation, we use the same notation for lifts of these elements of $\cG^{i}_{-i/2} / \cG^i_{N+10}$ to elements of $\cG^i_{-i/2}$. Let
$$L_{1,N} = \sum_{j=1}^{d_2} \alpha^2_j m^2_j,\qquad W^4_{1,N} = \sum_{j=1}^{d_4} \alpha^4_j m^4_j,\qquad L_{2,N} = \sum_{j=1}^{d_2} \beta^2_j m^2_j,\qquad W^4_{2,N} = \sum_{j=1}^{d_4} \beta^4_j m^4_j$$ be elements of weights $2$ and $4$ in $\cG^2_{-1}$ and $G^4_{-2}$, respectively, with undetermined coefficients. We impose the following conditions:
\begin{enumerate}
\item Modulo the space $\cG_{N+1}$, $L_{1,N}$ and $L_{2,N}$ are commuting Virasoro fields with central charges $c_1, c_2$ given by \eqref{pairofccc}.
\item Modulo the space $\cG_{N+1}$, $W^4_{i,N}$ is primary of weight $4$ for $L_{i,N}$ for $i=1,2$, and $W^4_{1,N}, W^4_{2,N}$ commute.
\item For $i = 1,2$, and $r \geq 3$, define the fields $W^{2r}_{i,N} = (W^4_{i,N})_{(1)} W^{2r-2}_{i,N}$. Modulo the space $\cG_{N+1}$, for $r \leq 6$ and $s+t \leq 7$, $L_{i,N}(z) W_{i,N}^{2r}(w)$ and $W_{i,N}^{2s}(z)W_{i,N}^{2t}(w)$ satisfy the OPE relations of $\cW^{\text{ev}}_{\infty}$ along the truncation curves for $\cW^{\ell_1}(\gs\gp_{2k})$ and $\cW^{\ell_2}(\gs\gp_{2k})$, respectively.
\end{enumerate}

Since $W^{2r}_{i,N} \in \cG^{2r}_{-r}$, it follows from \eqref{filt:behavior} that the only contributions modulo $\cG_{N+1}$ to the OPEs 
$$L_{i,N}(z) W^{2r}_{j,N}(w),\qquad W^{2s}_{i,N}(z) W^{2t}_{j,N}(w),\qquad i,j = 1,2,\qquad r \leq 6,\qquad s+t \leq 7,$$ can come from the terms appearing in $L_{i,N}$, $W^{2r}_{i,N}$, $W^{2s}_{i,N}$, and $W^{2t}_{i,N}$, modulo $\cG_{N+10}$. The above conditions impose a finite set of algebraic equations in the variables $\alpha^2_j, \alpha^4_j,  \beta^2_j, \beta^4_j$, as well as $\ell,k$, so the set of solutions to this system is an algebraic variety which we denote by $V_N$

For each $k \geq N$, the projection of $V_N$ onto the $\ell$-axis is $1$-dimensional by Corollary \ref{cor:twocopiesofW(sp)}, so the projection of $V_N$ onto the $k\ell$-plane is $2$-dimensional. Therefore we have a parametrization (not necessarily with rational functions) of a $2$-dimensional subset of this variety. For each $M > N$ we can construct elements $L_{i,M}, W^4_{i,M}$ which satisfy the above conditions modulo $\cG_{M+1}$, which agree with $L_{i,N}, W^4_{i,N}$ modulo $\cG_{N+1}$ Therefore the direct limit $L_i = L_{i,\infty}, W^4_i = W^4_{i,\infty}$ of these terms are well-defined elements of $\tilde{\cW}^{\gs\gp_2}_{\infty}$ and satisfy the desired OPEs. This completes the proof of statement (1).

Next we prove statement (2). Since $W_i^{2a} \in \cG^{2a}_{-a}$, the operators $(W_i^{2a})_{(2a-1)}$ for $2 \leq a \leq N$, map the space $\cG^0_{2N+1}$ to $\cG^0_{N+1}$. Therefore they induce maps of finite-dimensional spaces 
$$(W_i^{2a})_{(2a-1)}: \cG^0_0 / \cG^0_{2N+1} \rightarrow \cG^0 /\cG^0_{N+1}.$$ 
The space $\cG^0_0 / \cG^0_{2N+1}$ has dimension $2N+1$, with basis $(X^1)^j$ for $0 \leq j \leq 2N$. We claim that the operators $(W^{2a}_i)_{(2a-1)}$ for $2\leq a \leq N$ are diagonalizable over an extension $K_{0,N}$ of the field $\mathbb{C}(k,\ell)$, up to order $N$. This has the following meaning: we can choose a basis of $\cG^0_0 / \cG^0_{2N+1}$ consisting of vectors $v_t = \sum_{j=0}^{2N} \alpha^t_j (X^1)^j$ with coefficients in $K_{0,N}$, such that for $a \leq N$,
$$(W^{2a}_i)_{(2a-1)} (\sum_{j=0}^{2N} \alpha^t_j (X^1)^j) = \lambda_{i,a,t,N} (\sum_{j=0}^{N} \alpha^t_j (X^1)^j),\qquad \lambda_{i,a,t,N} \in K_{0,N}.$$ Equivalently, we can write this condition as $(W^{2a}_i)_{(2a-1)} (v_t) = \lambda_{i,a,t,N}\pi^0_{2N+1,N+1}(v_t)$. This is because for infinitely many values of $k \in \mathbb{N}$, this holds for generic $\ell$ by Lemma \ref{twocopiesofW(sp)} (2).

Via the projection map $\cG^0_0 / \cG^0_{2M+1} \rightarrow \cG^0_0 / \cG^0_{2N+1}$ for $N<M$, an eigenvector $v_t$ up to order $M$ as above for the operators $\{(W^{2a}_i)_{(2a-1)}| \ a \leq M\}$, projects to an eigenvector for $\{(W^{2a}_i)_{(2a-1)}| \ a \leq N\}$ of order $N$, with the same eigenvalue. This map is surjective on eigenspaces in the above sense. Therefore for fixed $N$, we can assume the vectors $v_t$ above (of which there are finitely many) all come from power series that are honest eigenvectors. Since $N$ is arbitrary, we get such a basis of eigenvectors for all of $\cG^0_0 = \tilde{\cW}^{\gs\gp_2}_{\infty}[0]$.

Let $V_0$ be the $\cW^{\text{ev}}_{\infty} \otimes \cW^{\text{ev}}_{\infty}$-module generated by the span of these eigenvectors. We claim that in weights $m> 0$, there are no lowest weight vectors for the action of $\cW^{\text{ev}}_{\infty} \otimes \cW^{\text{ev}}_{\infty}$. This follows from the fact that for $k \in \mathbb{N}$, the action of $\cW_{\ell_1}(\gs\gp_{2k}) \otimes \cW_{\ell_2}(\gs\gp_{2k})$ on the simple quotient is semisimple for generic $\ell$. Therefore $V_0$ is a sum of simple modules with $1$-dimensional lowest weight space.

Next, let $m_1 > 0$ be the first integer for which the space $U_{m_1} \subseteq \tilde{\cW}^{\gs\gp_2}_{\infty} [m_1]$ of lowest weight vectors of weight $m_1$ for $\cW^{\text{ev}}_{\infty} \otimes \cW^{\text{ev}}_{\infty}$ is nontrivial. As above, for each fixed $N$, the operators 
$(W^{2a}_i)_{(2a-1)}$ for $2\leq a \leq N$ map the space $\cG^{m_1}_{2N+1}$ to $\cG^{m_1}_{N+1}$, so they induces maps of finite-dimensional spaces 
$$(W^{2a}_i)_{(2a-1)}:(\cG^{m_1}_{-m_1/2} \cap U_{m_1} )/  (\cG^{m_1}_{2N+1} \cap U_{m_1}) \rightarrow ( \cG^{m_1}_{-m_1/2} \cap U_{m_1} ) / (\cG^{m_1}_{N+1}  \cap U_{m_1}).$$ 
Similarly, the operators 
$(W^{2a}_i)_{(2a-1)}$ for $2\leq a \leq N$ are diagonalizable on $(\cG^{m_1}_{-m_1} \cap U_{m_1} )/  (\cG^{m_1}_{2N+1} \cap U_{m_1})$ over an extension $K_{1,N}$ of the field $K_0$, up to order $N$, i.e., we can choose a basis of $(\cG^{m_1}_{-m_1} \cap U_{m_1} )/  (\cG^{m_1}_{2N+1} \cap U_{m_1})$ consisting of vectors $v_1,\dots, v_{d_1}$ such that 
$$(W^{2a}_i)_{(2a-1)} (v_t) = \lambda_{i,a,t,N} \pi^{m_1}_{2N+1,N}(v_t).$$ 

For $N<M$, via the projection map 
$$\pi^{m_1}_{2M+1,2N+1}: (\cG^{m_1}_{-m_1/2} \cap U_{m_1}) / (\cG^{m_1}_{2M+1} \cap U_{m_1})  \rightarrow (\cG^{m_1}_{-m_1/2} \cap U_{m_1})/ ( \cG^{m_1}_{2N+1} \cap U_{m_1}),$$ an eigenvector $v_t$ up to order $M$ as above for the operators $\{(W^{2a}_i)_{(2a-1)}| \ a \leq M\}$, projects to an eigenvector for $\{(W^{2a}_i)_{(2a-1)}| \ a \leq N\}$ up to order $N$, with the same eigenvalue. Therefore for fixed $N$, we can assume without loss of generality that the vectors $v_1,\dots, v_{d_1}$ above all come from power series that are honest eigenvalues, with coefficients in some extension $K_1$ which the union of the fields $K_{1,N}$. Since $N$ is arbitrary, we get such a basis of eigenvectors for all of $U_{m_1}$

Let $V_{m_1}$ be the $\cW^{\text{ev}}_{\infty} \otimes \cW^{\text{ev}}_{\infty}$-module generated by the span of these eigenvectors in $U_{m_1}$. As above, in weights $m> m_1$, there are no lowest weight vectors for the action $\cW^{\text{ev}}_{\infty} \otimes \cW^{\text{ev}}_{\infty}$ in $V_{m_1}$. This follows from the semisimplicity of the action of $\cW_{\ell_1}(\gs\gp_{2k}) \otimes \cW_{\ell_2}(\gs\gp_{2k})$ on the simple quotient for each $k \in \mathbb{N}$ and generic $\ell\in \mathbb{C}$. Therefore $V_{m_1}$ is a sum of simple modules with $1$-dimensional lowest weight space, and $V_0 \cap V_1$ is trivial. Inductively, over some extension of $\mathbb{C}(\ell,k)$, we can find two sequences of positive integers $m_i, d_i$ for $i\geq 2$ such that $\tilde{\cW}^{\gs\gp_2}_{\infty} = \bigoplus_{i\geq 0} V_{m_i}$ where each $V_{m_i}$ is a sum of $d_i$ simple $\cW^{\text{ev}}_{\infty} \otimes \cW^{\text{ev}}_{\infty}$-modules, each of which has $1$-dimensional lowest-weight space of weight $m_i$. 

Finally, it is apparent from the proof of (2) that statement (3) holds as well. \end{proof}

\begin{remark} Since $\cW^{\rm{ev}}_{\infty}$ is simple for generic values of the parameters by \cite[Corollary 3.14]{KL}, $\Psi$ is injective when restricted to either the first or second copy of $\cW^{\rm{ev}}_{\infty}$. However, $\Psi$ is not injective since $\cW^{\rm{ev}}_{\infty} \otimes \cW^{\rm{ev}}_{\infty}$ has $4$ parameters (namely, the parameters $c, \lambda$ for the first copy, and $c' ,\lambda'$ for the second copy), whereas $\tilde{\cW}^{\gs\gp_2}_{\infty}$ has only $2$ parameters. The relations between these parameters are determined completely by the fact that $\Psi$ is compatible with the maps $\Psi_n$ given by Lemma \ref{twocopiesofW(sp)}. Let $I_{\Psi}$ be the vertex algebra ideal generated by these relations in $\mathbb{C}[c, c', \lambda, \lambda']$, regarded as the weight zero component of $\cW^{\rm{ev}}_{\infty} \otimes \cW^{\rm{ev}}_{\infty}$. Clearly $I_{\psi} \subseteq \text{Ker}\ \psi$. It is apparent from the compatibility of $\Psi$ and $\Psi_n$ that the image of $\Psi$ has the same character as $\cW^{\rm{ev}}_{\infty} \otimes \cW^{\rm{ev}}_{\infty}$. Since $\cW^{\mathfrak{sp}_2}_{\infty}$ has $2$ parameters, we must have $I_{\psi} = \text{Ker}\ \psi$ after a suitable localization.
\end{remark}

\subsection{Completion of $Y$-algebras of type $C$}
For $X = B,C$ and $Y = B,C,D,O$, we can complete $\cC^{\psi}_{XY}(n,m)$ in the same way by changing conformal vector $L$ to $L + \frac{1}{2} \partial H^1$, and then allowing power series in the field $X^1$, so that we have infinite-dimensional weight spaces. We denote the completion by $\tilde{\cC}^{\psi}_{XY}(n,m)$, and it is straightforward to check that $\tilde{\cC}^{\psi}_{XY}(n,m)$ is just the quotient
$\tilde{\cW}^{\gs\gp_2}_{\infty, I_{X,Y,n,m}}$, i.e., the process of taking the simple graded quotient along the ideal $I_{XY,n,m}$ commutes with completion. Let \begin{equation} \label{defpsixynm} \Psi_{XY,n,m}: \cW^{\text{ev}}_{\infty} \otimes \cW^{\text{ev}}_{\infty} \rightarrow \tilde{\cC}^{\psi}_{XY}(n,m)\end{equation} denote the composition of $\Psi$ with the quotient map $\tilde{\cW}^{\gs\gp_2}_{\infty} \rightarrow \tilde{\cW}^{\gs\gp_2}_{\infty, I_{X,Y,n,m}}$.

Recall that $\tilde{\cC}^{\psi}_{iZ}(a,b)$ denotes a possibly non-simple $1$-parameter quotient of $\cW^{\text{ev}}_{\infty}$ along a (possibly nonmaximal) ideal $\cI$ containing $I_{iZ,a,n}$, and that $\cC^{\psi}_{iZ}(a,b)$ is the simple quotient of $\tilde{\cC}^{\psi}_{iZ}(a,b)$ as $1$-parameter vertex algebras.

\begin{theorem} \label{pairsofcurves} For each of the $Y$-algebra $\cC^{\psi}_{XY}(n,m)$ of type $C$, there is a pair of orthosymplectic $Y$-algebras $\cC^{\psi'}_{i_1 Z_1}(a_1, b_1)$ and $\cC^{\psi''}_{i_2 Z_2}(a_2, b_2)$ where $i_1, i_2 \in \{1,2\}$ and $Z_1, Z_2 \in \{B,C,D,O\}$, such that $\Psi_{XY,n,m}$ descends to a map of $1$-parameter vertex algebras
\begin{equation} \label{psiXYnmInducedtilde} \tilde{\cC}^{\psi'}_{i_1 Z_1}(a_1, b_1) \otimes \tilde{\cC}^{\psi''}_{i_2 Z_2}(a_2, b_2)  \hookrightarrow \tilde{\cC}^{\psi}_{XY}(n,m).\end{equation}
They are listed as follows:
\begin{equation} \begin{split} \label{pairofcurves:case1} 
\tilde{\cC}^{\psi}_{1D}(m+n+1,m-1) \otimes \tilde{\cC}^{\psi-1}_{1B}(n,m) & \hookrightarrow \tilde{\cC}^{\psi}_{CB}(n,m),
\\  \tilde{\cC}^{\psi}_{1B}(m+n,m-1) \otimes \tilde{\cC}^{\psi-1}_{1D}(n,m) & \hookrightarrow \tilde{\cC}^{\psi}_{CD}(n,m),
\\  \tilde{\cC}^{\psi}_{2C}(m+n+1,m) \otimes \tilde{\cC}^{\psi-1/2}_{2C}(n,m+1) & \hookrightarrow \tilde{\cC}^{\psi}_{BC}(n,m),
\\  \tilde{\cC}^{\psi}_{2O}( m + n + 1, n) \otimes \tilde{\cC}^{\psi-1/2}_{2O}(n,m+1) & \hookrightarrow \tilde{\cC}^{\psi}_{BO}(n,m). \end{split} \end{equation}

\begin{equation} \begin{split} \tilde{\cC}^{\psi}_{2B}(n-m-1,  m) & \otimes \tilde{\cC}^{\psi-1/2}_{2B}(n,m+1) \hookrightarrow \tilde{\cC}^{\psi}_{BB}(n,m), \quad n \geq m+1,
\\ \tilde{\cC}^{\psi'}_{2O}(m-n+1,  m) & \otimes \tilde{\cC}^{\psi-1/2}_{2B}(n,m+1) \hookrightarrow \tilde{\cC}^{\psi}_{BB}(n,m),\quad n \leq m. \end{split} \end{equation}

\begin{equation} \begin{split} 
\\ \tilde{\cC}^{\psi}_{1O}(n-m,  m-1) & \otimes \tilde{\cC}^{\psi-1}_{1C}(n,m) \hookrightarrow \tilde{\cC}^{\psi}_{CC}(n,m),\quad n \geq m+1,
\\ \tilde{\cC}^{\psi}_{1B}(m-n,  m-1) & \otimes \tilde{\cC}^{\psi-1}_{1C}(n,m) \hookrightarrow \tilde{\cC}^{\psi}_{CC}(n,m),\quad n \leq m.  \end{split} \end{equation}

\begin{equation} \begin{split} 
\tilde{\cC}^{\psi}_{1C}(n-m-1,  m-1) & \otimes \tilde{\cC}^{\psi-1}_{1O}(n,m)\hookrightarrow \tilde{\cC}^{\psi}_{CO}(n,m),
\quad n \geq m+1,
\\ \tilde{\cC}^{\psi}_{1D}(m-n+1,  m-1) & \otimes \tilde{\cC}^{\psi-1}_{1O}(n,m)\hookrightarrow \tilde{\cC}^{\psi}_{CO}(n,m),\quad n \leq m.
\end{split} \end{equation}

\begin{equation} \begin{split} \tilde{\cC}^{\psi}_{2D}(n-m-1,m) & \otimes \tilde{\cC}^{\psi-1/2}_{2D}(n,m+1) \hookrightarrow \tilde{\cC}^{\psi}_{BD}(n,m), \quad n \geq m+1,
\\ \tilde{\cC}^{\psi}_{2C}(m-n+1,  m) & \otimes \tilde{\cC}^{\psi-1/2}_{2D}(n,m+1) \hookrightarrow \tilde{\cC}^{\psi}_{BD}(n,m), \quad n \leq m.\end{split} \end{equation}
\end{theorem}

\begin{proof} By using the data in Appendix B of \cite{CL4} on intersections between the truncation curves for $\cC^{\psi}_{iZ}(a,b)$ and $\cC^{\psi}_{2C}(0,n) \cong \cW^k(\gs\gp_{2n})$, we can check that for each $Y$-algebra of type $C$, $\cW^{\psi}_{XY}(n,m)$, there are infinitely many values of $k \in \mathbb{N}$ such that the truncation curves for $\cW^{\ell_1}(\gs\gp_{2k})$ and $\cW^{\ell_1}(\gs\gp_{2k})$ are coincident with the truncation curves for the pair of algebras $\cC^{\psi'}_{i_1 Z_1}(a_1,b_1)$ and $\cC^{\psi''}_{i_2 Z_2}(a_2,b_2)$ on the list above. This proves that the kernel of the map $\Psi_{X,Y,m,n}$ contains the ideal $I_{i_1Z_1, a_1, b_1} \otimes I_{i_2Z_2, a_2, b_2}$ given above. \end{proof}

We conjecture that the map \eqref{psiXYnmInducedtilde} descends to the simple quotient, i.e., it induces a map
\begin{equation} \label{psiXYnmInduced} \cC^{\psi'}_{i_1 Z_1}(a_1, b_1) \otimes \cC^{\psi''}_{i_2 Z_2}(a_2, b_2)  \hookrightarrow \tilde{\cC}^{\psi}_{XY}(n,m).\end{equation}

Although we are unable to prove this in all cases, we have the following argument that works for $3$ of the $8$ families.

\begin{theorem} \label{pairsofcurves:refined}
For cases $\cC^{\psi}_{CB}(n,m)$, $\cC^{\psi}_{BC}(n,m)$, and $\cC^{\psi}_{BO}(n,m)$,
\eqref{pairofcurves:case1} descends to a map $\cC^{\psi'}_{i_1 Z_1}(a_1, b_1) \otimes \cC^{\psi''}_{i_2 Z_2}(a_2, b_2)  \hookrightarrow \tilde{\cC}^{\psi}_{XY}(n,m)$. \end{theorem}

\begin{proof} We first consider $\cC^{\psi}_{CB}(n,m)$ with $n,m$ fixed. For each $k \in \mathbb{N}$, we have an intersection of the truncation curves for 
$\cC^{\psi}_{CB}(n,m)$ and $\cC^{\ell}(k)$ at the point
$$\psi = \frac{1 + 2 k + 4 m + 2 n}{2 m},\qquad \ell = -(k+1) + \frac{1 + 2 k + 4 m + 2 n}{4 m}.$$
For these values of $\ell$, the shifted levels of $\cW^{\ell_1}(\gs\gp_{2k})$ and $\cW^{\ell_2}(\gs\gp_{2k})$ are 
$$ \frac{1+2k+4m+2n}{2(1+2k+2m+2n)}, \qquad \frac{1+2k+2n}{2(1+2k+2m+2n)}.$$ Since $\ell-1$ is admissible for $\widehat{\gs\gp}_{2k}$ for infinitely many values of $k$, $\cC_{\ell}(k)$ is an extension of the simple quotient $\cW_{\ell_1}(\gs\gp_{2k}) \otimes \cW_{\ell_2}(\gs\gp_{2k})$ at these values. By Corollary \ref{rationalquotients}, $\cC_{\ell}(k)$ has the weak generation property and is a quotient of $\cW^{\gs\gp_2}_{\infty}$ for these values of $\ell$. Since the image $\Psi_{CB,n,m}(\cW^{\text{ev}}_{\infty} \otimes \cW^{\text{ev}}_{\infty})$ is the simple vertex algebra $\cW_{\ell_1}(\gs\go_{2r})^{\mathbb{Z}_2} \otimes \cW_{\ell_2}(\gs\go_{2r})^{\mathbb{Z}_2}$, it follows that the maximal ideals $\cI_{1D,m+n+1,m-1}$ in the first copy of $\cW^{\text{ev}}_{\infty}$ and  $\cI_{1B,n,m}$ in the second copy $\cW^{\text{ev}}_{\infty}$, lie in the kernel of $\Psi_{CB,n,m}$ at the corresponding point in the parameter space. Since these two maximal ideals lie in this kernel at infinitely many points along the truncation curve for $\tilde{\cC}^{\psi}_{CB}(n,m)$, it follows that $\tilde{\cC}^{\psi}_{1D}(m+n+1,m-1)$ and $\tilde{\cC}^{\psi-1}_{1B}(n,m)$ must be simple.

For $\cC^{\psi}_{BC}(n,m)$, for all $k \in \mathbb{N}$, we have an intersection of the truncation curves for $\cC^{\psi}_{BC}(n,m)$ and $\cC^{\ell}(k)$ for $\ell = -(k+1) + \frac{2 + k + 2 m + 2n}{1 + 2 m}$, such that the shifted levels of $\cW^{\ell_1}(\gs\gp_{2k})$ and $\cW^{\ell_2}(\gs\gp_{2k})$ are $$ \frac{2 + k + 2 m + n}{3 + 2 k + 2 m + 2 n}, \qquad \frac{1 + k + n}{3 + 2 k + 2 m + 2 n}.$$ As above, $\ell-1$ is admissible for infinitely many values of $k$, and the rest of the argument is the same.

Finally, for $\cC^{\psi}_{BO}(n,m)$, for all $k \in \mathbb{N}$, we have an intersection of the truncation curves for $\cC^{\psi}_{BO}(n,m)$ and $\cC_{\ell}(k)$ for $\ell = -(k+1) + \frac{3 + 2 k + 4 m + 2 n}{2 (1 + 2 m)}$, such that the shifted levels of $\cW^{\ell_1}(\gs\gp_{2k})$ and $\cW^{\ell_2}(\gs\gp_{2k})$ are $$ \frac{3 + 2 k + 4 m + 2 n}{4 (1 + k + m + n)}, \qquad \frac{1 + 2 k + 2 n}{4 (1 + k + m + n)}.$$ Since $\ell-1$ is admissible for infinitely many values of $k$, the claim follows as above.  \end{proof}

Similarly, for $n \in \frac{1}{2} \mathbb{Z}$, $n \neq 0,-\frac{1}{2}$, recall the diagonal cosets $\cC^{\ell}(n)$ from Section \ref{sect:Diagonal}, which are simple $1$-parameter quotients of $\cW^{\gs\gp_2}_{\infty} / I_n  \cdot \cW^{\gs\gp_2}_{\infty}$. Here $I_n \subseteq \mathbb{C}[c,k]$ is the ideal $(k - n)$, and $c$ a rational function of $\ell$. As above, we may replace $L$ with $L + \frac{1}{2} \partial H^1$ and take the completion $\tilde{\cC}^{\ell}(n)$. When $n \in \mathbb{N}$, we have $\cC^{\ell}(n) = \tilde{\cC}^{\ell}(n)$. In all other cases, $\cC^{\ell}(n) \neq \tilde{\cC}^{\ell}(n)$, but for all $n \in \frac{1}{2} \mathbb{Z}$, we have the homomorphism $$\Psi_{n}: \cW^{\text{ev}}_{\infty} \otimes \cW^{\text{ev}}_{\infty} \rightarrow \tilde{\cC}^{\ell}(n)$$ given by the composition of $\Psi$ with the quotient map $\tilde{\cW}^{\gs\gp_2}_{\infty} \rightarrow  \tilde{\cC}^{\ell}(n)$.

\begin{theorem} For $n \in \frac{1}{2}\mathbb{Z}$, $n \neq 0,-\frac{1}{2}$ there is a pair of orthosymplectic $Y$-algebras $\cC^{\psi'}_{i_1 Z_1}(a_1, b_1)$ and $\cC^{\psi''}_{i_2 Z_2}(a_2, b_2)$ where $i_1, i_2 \in \{1,2\}$ and $Z_1, Z_2 \in \{B,C,D,O\}$, such that $\Psi_{n}$ descends to a map of $1$-parameter vertex algebras 
\begin{equation} \label{psiXYnmInduced} \tilde{\cC}^{\psi'}_{i_1 Z_1}(a_1, b_1) \otimes \tilde{\cC}^{\psi''}_{i_2 Z_2}(a_2, b_2)  \hookrightarrow \tilde{\cC}^{\ell}(n).\end{equation}
They are listed as follows:
\begin{equation} \begin{split} \tilde{\cC}^{\psi'}_{1D}(1, n) \otimes \tilde{\cC}^{\psi''}_{1D}(1,n) & \hookrightarrow \tilde{\cC}^{\ell}(n-\frac{1}{2}),
 \quad \psi' = \frac{1 + n + \ell}{1 + 2 n + 2 \ell}, \quad   \psi'' = \frac{1 + n + \ell}{3 + 2 n + 2 \ell},
 \\  \tilde{\cC}^{\psi'}_{1D}(n,0) \otimes \tilde{\cC}^{\psi''}_{1D}(n,0) & \hookrightarrow \tilde{\cC}^{\ell}(-n), \quad
\psi' = 3 - 2 n - \ell,  \quad \psi'' = -3 + 2 n + \ell,
\\ \tilde{\cC}^{\psi'}_{1B}(n,0) \otimes \tilde{\cC}^{\psi''}_{1B}(n,0)&  \hookrightarrow \tilde{\cC}^{\ell}(-n-\frac{1}{2}), \quad
\psi' = 2 - 2 n - \ell, \quad  \psi'' = -2 + 2 n + \ell. \end{split}\end{equation} 
  \end{theorem}

 \begin{proof} The argument is similar to the proof of Theorem \ref{pairsofcurves}. It makes use of the data in \cite{CL4} on intersections between the truncation curves for $\cC^{\psi}_{iZ}(a,b)$, so we omit it. 
 \end{proof}
 
 \begin{remark} By Corollary \ref{cor:twocopiesofW(sp)}, for $n \in \mathbb{N}$, $\Psi_n$ induces a map on the simple quotient $\cC^{\psi'}_{i_1 Z_1}(a_1, b_1) \otimes \cC^{\psi''}_{i_2 Z_2}(a_2, b_2)  \hookrightarrow \tilde{\cC}^{\ell}(n)$. We conjecture that this holds for all $n \in \frac{1}{2} \mathbb{Z}$, $n \neq 0,-\frac{1}{2}$. \end{remark}

To summarize, $\cW^{\gs\gp_2}_{\infty}$ is not an extension of the tensor product of two copies of $\cW^{\text{ev}}_{\infty}$ in a way that is compatible with the maps $\Psi_n$ for $n \in \mathbb{N}$ given by Corollary \ref{cor:twocopiesofW(sp)} but the completion $\tilde{\cW}^{\gs\gp_2}_{\infty}$ does have this property. This allows us to attach a pair of orthosymplectic $Y$-algebras $\cC^{\psi'}_{i_1Z_2}(a_1,b_1)$ and $\cC^{\psi''}_{i_2Z_2}(a_2,b_2)$ to each $\cC^{\psi}_{XY}(n,m)$, as well as each of the diagonal cosets $\cC^{\ell}(n)$ for $n\in \frac{1}{2}\mathbb{Z}$.

We conclude by mentioning another connection between $\cW^{\gs\gp_2}_{\infty}$ and a tensor product of two copies of $\cW^{\text{ev}}_{\infty}$, which is suggested by the following result of Fasquel, Nakatsuka, and the second author \cite{FKN}. Recall that $\cW^{\ell}(\mathfrak{sp}_{2(2n+1)}, f_{2n+1, 2n+1})$ has a copy of $V^{k}(\mathfrak{sp}_2)$ for $k= (2n+1)\ell + 4n(n+1)$. Let $f$ be a nonzero nilpotent inside this copy of $\mathfrak{sp}_2$.

\begin{theorem} \label{FKN} As $1$-parameter vertex algebras, $H_f(\cW^{\ell}(\mathfrak{sp}_{2(2n+1)}, f_{2n+1, 2n+1}))\cong \cW^{\ell}(\mathfrak{sp}_{2(2n+1)}, f_{2n+2, 2n})$. 
\end{theorem}
This statement is completely analogous to the first isomorphism in \eqref{fks:typeA}; the only difference is that $f_{n,n} \in \gs\gl_{2n}$ is decomposable, where is $f_{2n+1,2n+1} \in \gs\gp_{2(2n+1)}$ is not. The nilpotent $f_{2n+2, 2n}\in \gs\gp_{2(2n+1)}$ decomposes as the sum of hook-type nilpotents $f_{2n+2, 1^{2n}}$ and $f_{2n, 1^{2n+2}}$, so we expect by Conjecture \ref{reductionstages} that $\cW^{\ell}(\mathfrak{sp}_{2(2n+1)}, f_{2n+2, 2n})$ is an extension of the tensor product 
\begin{equation} \label{2Cstages} \cC^{\psi'}_{2C}(n, n+1) \otimes \cC^{\psi''}_{2C}(0,n),\qquad \psi' = \ell + 2n+2,\qquad \psi'' = \ell + 2n + \frac{3}{2}.\end{equation}
In particular, for each $n$, the reduction of $\cW^{\ell}(\mathfrak{sp}_{2(2n+1)}, f_{2n+1, 2n+1})$ should be an extension of a tensor product of two $1$-parameter quotients of $\cW^{\rm{ev}}_{\infty}$. This strongly suggests that we can apply the reduction functor $H_f$ to the $2$-parameter vertex algebra $\cW^{\gs\gp_2}_{\infty}$ itself, and the resulting vertex algebra, which is of type $\cW(2^3, 3, 4^3, 5, 6^3, 7,\dots)$, should be an extension of the product of two copies of $\cW^{\text{ev}}_{\infty}$ where the parameters of these two copies are related. The relation between the parameters is uniquely determined by the requirement that it is compatible with the isomorphisms given by Theorem \ref{FKN} for all $n \geq 3$. In fact, we can first take a $1$-parameter quotient $\cC^{\psi}_{XY}(n,m)$ of $\cW^{\gs\gp_2}_{\infty}$ and then apply $H_f$, and this should be equivalent to first applying $H_f$ to $\cW^{\gs\gp_2}_{\infty}$ and then passing to a $1$-parameter quotient of $H_f(\cW^{\gs\gp_2}_{\infty})$, where both copies of $\cW^{\text{ev}}_{\infty}$ truncate. We can thus associate a pair of truncation curves for $\cW^{\text{ev}}_{\infty}$ to each truncation curve for $\cW^{\gs\gp_2}_{\infty}$.

Rather surprisingly, this perspective does {\it not} lead to the same correspondence as Theorem \ref{pairsofcurves} between $Y$-algebras of type $C$, and pairs of orthosymplectic $Y$-algebras. Replacing $\cW^{\ell}(\mathfrak{sp}_{2(2n+1)}, f_{2n+1, 2n+1})$ with $\cC^{\psi}_{BC}(0,n)$ and replacing the parameter $\ell$ in \eqref{2Cstages} with $\psi$, Theorem \ref{FKN} suggests that we have a conformal embedding
 $$\cC^{\psi}_{2C}(n, n+1) \otimes \cC^{\psi-1/2}_{2C}(0,n) \hookrightarrow H_f(\cC^{\psi}_{BC}(0,n)).$$ On the other hand, by Theorems \ref{pairsofcurves} and \ref{pairsofcurves:refined}, we have the conformal embedding 
 $$\cC^{\psi}_{2C}(n+1,n) \otimes \cC^{\psi-1/2}_{2C}(0,n+1)  \hookrightarrow \tilde{\cC}^{\psi}_{BC}(0,n),$$ which corresponds to iterated reduction in the opposite order. Similarly, one verifies that for all orthosymplectic $Y$-algebras $\cC^{\psi}_{XY}(n,m)$, the conformal embeddings 
$$\cC^{\psi'}_{i_1 Z_1}(a_1, b_1) \otimes \cC^{\psi''}_{i_2 Z_2}(a_2, b_2)  \hookrightarrow H_f(\cC^{\psi}_{XY}(n,m))$$ coming from this perspective are related to the ones in Theorem \ref{pairsofcurves} by iterated reduction in the opposite order. It is an interesting question to find a conceptual explanation for this phenomenon.

\end{document}